\documentclass[11pt,a4paper,reqno]{amsart}

\makeatletter \renewcommand{\tocsection}[3]{%
  \indentlabel{%
    \@ifnotempty{#2}{%
      \makebox[2.0em][l]{\ignorespaces#1 #2.}%
    }%
  }%
  #3%
} \makeatother
  
\usepackage{amssymb}
\usepackage{amsfonts}
\usepackage{mathtools}
\usepackage{enumitem}
\usepackage{bm}
\setlist{itemsep=1mm}                                         
\usepackage[hidelinks]{hyperref}
\usepackage{url}
\usepackage{pdfsync}
\usepackage{comment}
\usepackage[all]{xy}
\usepackage[textwidth=27mm, textsize=scriptsize, color=green!20]{todonotes}
\newcommand{\change}[1]{#1}

\newtheorem{theorem}{Theorem}[section]
\newtheorem{proposition}[theorem]{Proposition}
\newtheorem{lemma}[theorem]{Lemma}
\newtheorem{corollary}[theorem]{Corollary}

\newtheorem{introtheorem}{Theorem}

\theoremstyle{definition}
\newtheorem{definition}[theorem]{Definition}
\newtheorem{proposition-definition}[theorem]{Proposition-Definition}
\newtheorem{remark}[theorem]{Remark}
\newtheorem{example}[theorem]{Example}

\newtheorem{question}[theorem]{Question}
\numberwithin{equation}{section}

\DeclareMathOperator{\Spec}{\mathrm{Spec}}
\DeclareMathOperator{\Div}{\mathrm{Div}}
\DeclareMathOperator{\Rat}{\mathrm{Rat}}

\DeclareMathOperator{\DSP}{\mathrm{DSP}}
\DeclareMathOperator{\Gal}{\mathrm{Gal}}

\DeclareMathOperator{\vol}{\mathrm{vol}}

\DeclareMathOperator{\can}{can}
\DeclareMathOperator{\Id}{Id}

\DeclareMathOperator{\Hom}{\mathrm{Hom}}
\DeclareMathOperator{\val}{\mathrm{val}}
\DeclareMathOperator{\supp}{\mathrm{supp}}
\DeclareMathOperator{\MI}{\mathrm{MI}}

\DeclareMathOperator{\MV}{\mathrm{MV}}
\DeclareMathOperator{\integ}{\mathrm{int}}
\DeclareMathOperator{\model}{\mathrm{mod}}
\DeclareMathOperator{\adel}{\mathrm{adel}}

\newcommand{\longhookrightarrow}{\lhook\joinrel\longrightarrow}

\renewcommand{\and}{\quad \text{and} \quad}

\newcommand{\abs}{\mathrm{abs}}
\newcommand{\ess}{\mathrm{ess}}

\newcommand{\Gm}{\mathbb{G}_{\mathrm{m}}}

\newcommand{\bR}{\mathbb{R}}
\newcommand{\bC}{\mathbb{C}}
\newcommand{\bQ}{\mathbb{Q}}
\newcommand{\bN}{\mathbb{N}}

\newcommand{\bZ}{\mathbb{Z}}

\newcommand{\cO}{\mathcal{O}}

\newcommand{\an}{\mathrm{an}}
\newcommand{\rmr}{r}
\newcommand{\rmw}{w}

\renewcommand\div{\operatorname{div}}

\newcommand{\Dcan}{ {\overline{D}} \vphantom{L}^{\can}}

\newcommand{\Mcan}{ {\overline{M}} \vphantom{L}^{\can}}
\newcommand{\Ncan}{ {\overline{N}} \vphantom{L}^{\can}}

\newcommand{\Etor}{{\overline{E}} \vphantom{L}^{\mathrm{tor}}}

\newcommand{\Ecan}{ {\overline{E}} \vphantom{L}^{\can}}

\begin{document}

\title[Approximation of adelic divisors and
equidistribution]{Approximation of adelic divisors and
  equidistribution of small points}

\author[Balla\"y]{Fran\c{c}ois Balla\"y}
\address{\hspace*{-6.3mm} Universit\'e Caen Normandie, CNRS, Normandie Univ, LMNO UMR6139, F-14000 CAEN, FRANCE \vspace*{-2.8mm}}
\address{\hspace*{-6.3mm} {\it Email address:} {\tt francois.ballay@unicaen.fr}}

\author[Sombra]{Mart{\'\i}n~Sombra}
\address{\hspace*{-6.3mm} Instituci\'o Catalana de Recerca i Estudis Avan\c{c}ats, 
  08010 Barcelona, Spain \vspace*{-2.8mm}}
\address{ 
\hspace*{-6.3mm}  Departament de Matem\`atiques i
  Inform\`atica, Universitat de Barcelona,  08007
  Bar\-ce\-lo\-na, Spain \vspace*{-2.8mm}}
\address{\hspace*{-6.3mm} Centre de Recerca Matem\`atica, 08193 Bellaterra, Spain \vspace*{-2.8mm}}
\address{\hspace*{-6.3mm} {\it Email address:} {\tt martin.sombra@icrea.cat}}

\date{\today} \subjclass[2020]{Primary 14G40; Secondary 11G50, 14M25,
  37P30} \keywords{Adelic divisor, height of points, essential
  minimum, equidistribution, toric variety, dynamical system,
  semiabelian variety}

\begin{abstract}
  We study the asymptotic distribution of the Galois orbits of generic
  sequences of algebraic points of small height in a projective
  variety over a number field. Our main result is a generalization of
  Yuan's equidistribution theorem that applies to heights for which
  Zhang's lower bound for the essential minimum is not necessarily an
  equality. It extends to all projective varieties a theorem of Burgos
  Gil, Philippon, Rivera-Letelier and the second author for toric
  varieties.  It also applies to sums of canonical heights for an
  algebraic dynamical system, and in particular it recovers Kühne's
  semiabelian equidistribution theorem.  We also generalize previous
  work of Chambert-Loir and Thuillier to obtain new logarithmic
  equidistribution results. Finally we extend our main result to the
  quasi-projective setting recently introduced by Yuan and Zhang.
\end{abstract}

\maketitle

\thispagestyle{empty}

\vspace{-6mm}

\setcounter{tocdepth}{1}
\tableofcontents

\vspace{-6mm}


\section{Introduction}

In their seminal work \cite{SUZ}, Szpiro, Ullmo and Zhang
used Arakelov theory to prove an equidistribution theorem for the
Galois orbits of algebraic points of a projective variety over a
number field. Their result applies to generic sequences of points in an
abelian variety with Néron-Tate heights converging to zero, and is at
the heart of the proof of the Bogomolov conjecture for abelian
varieties \cite{Ullmo, Zhang:equi}. It has been developed in several
directions  by many authors, culminating with the celebrated
equidistribution theorem of Yuan \cite{Yuan:big}.

In this paper we generalize Yuan's theorem to allow more flexibility
in the choice of the height function. In the same spirit we generalize
the logarithmic equidistribution theorem of Chambert-Loir and
Thuillier \cite{CLT}.

Our results extend to all projective varieties the toric
equidistribution theorem of Burgos Gil, Philippon, Rivera-Letelier and
the second author \cite{BPRS:dgopshtv}. Moreover they strengthen it,
showing that this result holds without any semipositivity assumption
and with respect to functions admitting logarithmic singularities
along some specific effective divisors.  They also apply in the
dynamical setting to sums of canonical height functions, giving an
equidistribution theorem allowing test functions with logarithmic
singularities along preperiodic hypersurfaces. In particular, this
recovers the semiabelian equidistribution theorem of Kühne
\cite{Kuhne:pshsv} and strengthens its statement to include
convergence with respect to functions with logarithmic singularities
along torsion and boundary hypersurfaces.

We also give a partial converse to our main result, showing that in
some situations the imposed condition on the height function is
necessary for \change{the}  equidistribution to occur. This includes the
semipositive toric case, thus recovering the reciprocal of the toric
equidistribution theorem from \cite{BPRS:dgopshtv}.  Finally we extend
our main result to the setting of adelic line bundles on
quasi-projective varieties introduced by Yuan and Zhang
\cite{YuanZhang:quasiproj}.

\subsection*{Background}

Arakelov geometry provides a very general and powerful framework to
define and study heights of algebraic points. Classical heights in
Diophantine geometry such as Néron-Tate heights on abelian varieties
are special cases of height functions associated to metrized line
bundles in the sense of Zhang \cite{Zhang:spam}.

Let $K$ be a number field with a fixed algebraic closure
$ \overline{K}$. Let $X$ be a projective variety of dimension
$d \ge 1$ over $K$ and $\overline{D}$ an adelic divisor on $X$. The
latter consists of a Cartier divisor $D$ on $X$ with an adelic family
of Green functions, and is a datum essentially equivalent to that of a
metrized line bundle on $X$. Let
$h_{\overline{D}} \colon X(\overline{K}) \rightarrow \bR$ be the
associated height function, and denote by $\mu^{\abs}(\overline{D})$
and $\mu^{\ess}(\overline{D})$ its absolute and essential minima \change{(see Section \ref{subsec:minima} for the definitions).}

A fundamental inequality of Zhang \cite{Zhang:positive} asserts that if
$D$ is ample and $\overline{D}$ is semipositive then
\begin{equation}\label{eq:Zhangineq}
\mu^{\ess}(\overline{D}) \ge \frac{(\overline{D}^{d+1})}{(d+1)(D^d)},
\end{equation} 
where $(D^{d})$ and  $(\overline{D}^{d+1})$ denote the top intersection numbers of $D$ and 
 $\overline{D}$, respectively.

For every generic sequence $(x_{\ell})_{\ell}$ in $X(\overline{K})$ we
have
$\liminf_{\ell \to \infty} h_{\overline{D}}(x_{\ell}) \ge
\mu^{\ess}(\overline{D})$, and there exist generic sequences for which
 \change{equality} holds.  Following \cite{BPRS:dgopshtv}, we say that
$(x_{\ell})_{\ell}$ is
\emph{$\overline{D}$-small} if the heights of these points converge to
the smallest possible value, that is
\begin{displaymath}
  \lim_{\ell \to \infty} h_{\overline{D}}(x_{\ell}) =
\mu^{\ess}(\overline{D}).
\end{displaymath}
For each place $v \in \mathfrak{M}_K$ we denote by $X_v^{\an}$ the
$v$-adic analytification of $X$. It is a Berkovich space over $\bC_v$,
the completion of an algebraic closure of the local field $K_v$.  For
an algebraic point $x \in X(\overline{K})$ we denote by
$\delta_{O(x)_v}$ the uniform probability measure on
$O(x)_v \subset X_v^{\an}$, the image in $X_{v}^{\an}$ of the Galois
orbit of $x$.

With this notation, Yuan's theorem \cite{Yuan:big} can be stated as
follows.

\begin{introtheorem}[Yuan]\label{thm:YuanIntro}
  Let $\overline{D} \in \widehat{\Div}(X)$ be a semipositive adelic
  divisor on $X$ with ample geometric divisor $D$, and assume that
\begin{equation}\label{eq:conditionYuan}
\mu^{\ess}(\overline{D}) = \frac{(\overline{D}^{d+1})}{(d+1)(D^d)}.
\end{equation}
Then for every $v \in \mathfrak{M}_K$ and every $\overline{D}$-small
generic sequence $(x_{\ell})_{\ell}$ in $X(\overline{K})$ the
sequence of probability measures $(\delta_{O(x_{\ell})_{v}})_{\ell}$
on $X_v^{\an}$ converges weakly to
$c_1(\overline{D}_v)^{\wedge d}/(D^d)$.
\end{introtheorem}

In other words, this result asserts that if Zhang's lower bound
\eqref{eq:Zhangineq} is an equality then $\overline{D}$ has the
equidistribution property at every place $v\in \mathfrak{M}_{K}$, in
the sense that the Galois orbits of points in $\overline{D}$-small
generic sequences equidistribute in $X_v^{\an}$ for every~$v$, and
that moreover the $v$-adic equidistribution measure is the normalized
$v$-adic Monge-Ampère measure of $\overline{D}$. When $X$ is a curve,
this theorem is due to Autissier \cite{Autissier} and Chambert-Loir
\cite{Chambert-Loir:Mesuresetequi}.

Yuan's theorem encompasses in a unified way the (Archimedean) theorems
of Szpiro, Ullmo and Zhang for abelian varieties~\cite{SUZ}, of Bilu
for canonical heights on toric varieties \cite{Bilu} and of
Chambert-Loir for canonical heights on isotrivial semiabelian
varieties~\cite{Chambert-Loir:pphvs}, as well as their non-Archimedean
analogues by Chambert-Loir \cite{Chambert-Loir:Mesuresetequi}. The
positivity assumptions in Theorem \ref{thm:YuanIntro} can be weakened,
and in fact Yuan's proof remains valid when $D$ is big but not
necessarily ample. Moreover, Berman and Boucksom
\cite{BermanBoucksom:2010} and later Chen \cite{Chen:diffvol}
generalized this theorem for the Archimedean places to the
non-semipositive case.

By a result of the first author, the fact that Zhang's lower
bound is an equality is equivalent to the equality between the
essential and the absolute minima \cite[Theorem~6.6]{Ba:Okounkov}.
This is a very restrictive
condition 
that is nevertheless satisfied in the important case of canonical
heights on polarized dynamical systems \cite{Zhang:spam}, which
includes the canonical heights on toric varieties and the Néron-Tate
heights on abelian varieties.

To our knowledge, there is no general result ensuring the
equidistribution property for an adelic divisor on a projective
variety over a number field when Zhang's inequality is strict.
However, there are two remarkable situations where results in this
direction are known.
 
In \cite{BPRS:dgopshtv}, Burgos Gil, Philippon, Rivera-Letelier and
the second author achieved a systematic description of this property
in the toric setting. Their result gives a criterion in terms of
convex analysis, and shows that there are plenty of toric adelic
divisors for which Zhang's inequality is strict but that nevertheless
satisfy the property. This provides a wealth of new equidistribution
phenomena previously out of reach, as well as situations where it does
not occur.

In \cite{Kuhne:pshsv}, Kühne proved the long standing semiabelian
equidistribution conjecture, showing that this
property holds for canonical heights on semiabelian varieties.  As
shown by Chambert-Loir \cite{Chambert-Loir:pphvs}, Theorem
\ref{thm:YuanIntro} does not apply in this case \change{as
condition} \eqref{eq:conditionYuan} can fail when the semiabelian
variety is not isotrivial. Kühne's theorem allowed him to give a
purely Arakelov-geometric proof of the semiabelian Bogomolov
conjecture, previously established by David and Philippon with other
methods~\cite{DP:stvs}.
   
\subsection*{Main theorem}
The results of \cite{BPRS:dgopshtv} and \cite{Kuhne:pshsv} raise the
following question: on an arbitrary projective variety, what can be
said regarding the equidistribution property for an adelic divisor
when Zhang's lower bound is strict? More precisely, can we identify a
condition weaker than \eqref{eq:conditionYuan} that \change{under
  suitable positivity assumptions} guarantees this property for a
given adelic divisor? Our main contribution is a generalization of
Yuan's theorem that answers affirmatively this question.

As in \cite{SUZ, Chambert-Loir:Mesuresetequi, Yuan:big, Chen:diffvol},
the positivity properties of adelic divisors play a central role in
our approach.  We work with the more general notion of adelic
$\bR$-divisors developed by Moriwaki, as it provides a particularly
efficient framework to study positivity
\cite{BMPS:positivity,Moriwaki:MAMS}. Adelic $\bR$-divisors are better
behaved on normal varieties, and so \change{from now on we assume that
  $X$ is a normal projective variety of dimension $d\ge 1$ over $K$.}

Let $\overline{D}$ be an adelic $\bR$-divisor on $X$. A
\emph{semipositive approximation} of $\overline{D}$ is a pair
$(\phi, \overline{Q})$ consisting of a normal modification
$\phi \colon X' \rightarrow X$ and a semipositive adelic $\bR$-divisor
$\overline{Q}$ on $X'$ such that the $\bR$-divisor $Q$ is big and
$\phi^*\overline{D} - \overline{Q}$ is pseudo-effective.  It is a
variant of the notion of admissible decomposition introduced by  Chen
\cite{Chen:diffvol}.

Given two big $\bR$-divisors $P,A$ on $X$, the \emph{inradius} of $P$
with respect to $A$ is the positive real number defined as 
\begin{displaymath}
r(P;A) = \sup\{ \lambda \in \bR \ | \, P- \lambda A \text{ is big} \}.
\end{displaymath}
This geometric invariant was introduced by Teissier \cite{Teissier}
and measures the bigness of $P$, see 
Section \ref{subsection:inradius} for more details.  Our main result is the following.  

\begin{introtheorem}\label{thm:MainIntro}
  Let $\overline{D}$ be an adelic $\mathbb{R}$-divisor on $X$ with $D$
  big. Assume that there exists a sequence
  $(\phi_n \colon X_n \to X,\overline{Q}_n)_n$ of semipositive
  approximations of $\overline{D}$ such that
\begin{equation}\label{eq:condMainIntro}
\lim_{n\to\infty} \frac{\mu^{\ess}(\overline{D}) - \mu^{\abs}(\overline{Q}_n)}{r(Q_n;\phi_n^*D)} = 0.
\end{equation}
Let $v \in\mathfrak{M}_K$, and for each $n \ge 1$ let $\nu_{n,v}$ be
the pushforward to $X_v^{\an}$ of the normalized $v$-adic Monge-Ampère
measure $c_1(\overline{Q}_{n,v})^{\wedge d}/(Q_n^d)$ on
$X_{n,v}^{\an}$. Then
\begin{enumerate}[label=(\roman*), leftmargin=*, widest=ii]
\item the sequence $(\nu_{n,v})_n$ converges weakly to a probability
  measure $\nu_{v}$ on $X_v^{\an}$,
\item for every $\overline{D}$-small generic sequence
  $(x_{\ell})_{\ell}$ in $X(\overline{K})$, the sequence of probability measures
  $(\delta_{O(x_{\ell})_{v}})_{\ell}$ on $X_v^{\an}$ converges weakly
  to~$\nu_{v}$. 
\end{enumerate}
\end{introtheorem}

\change{As already mentioned, condition \eqref{eq:conditionYuan} in Yuan's
theorem is equivalent to the equality
$\mu^{\ess}(\overline{D}) = \mu^{\abs}(\overline{D})$ by
\cite[Theorem~6.6]{Ba:Okounkov}. Therefore} Theorem \ref{thm:YuanIntro}
follows immediately from Theorem \ref{thm:MainIntro} applied with the
constant sequence $(\phi_n,\overline{Q}_n) = (\Id_X,\overline{D})$,
$n\in\bN$. Theorem \ref{thm:MainIntro} also implies Chen's
equidistribution theorem (Corollary \ref{coro:Chenequi}).

We actually  prove a stronger result (Theorem \ref{thm:MainSequences})
showing that under  condition~\eqref{eq:condMainIntro} for every
$\overline{D}$-small generic sequence $(x_{\ell})_{\ell}$ in
$X(\overline{K})$ and
$\overline{E} \in \widehat{\Div}(X)_{\bR}$~we~have
\begin{equation}\label{eq:HCPintro}
  \lim_{\ell \to \infty} h_{\overline{E}}(x_{\ell}) = \lim_{n\to \infty} \, \frac{(\overline{Q}_n^d \cdot \phi_n^* \overline{E}) - d\,  \mu^{\ess}(\overline{D})
    \, (Q_n^{d-1} \cdot \phi_n^* E)}{(Q_n^d)}.
\end{equation}
In particular, both limits exist in $\bR$ and the second does not
depend on the choice of the sequence
$(\phi_n,\overline{Q}_n)_n$. Theorem \ref{thm:MainIntro} follows by
specializing \eqref{eq:HCPintro} to the adelic divisors $\overline{E}$
over the zero divisor of $X$ associated to continuous real-valued
functions on $X_v^{\an}$.

We also obtain a generalization of Chambert-Loir and Thuillier's
logarithmic equi\-distribution theorem \cite{CLT}, showing that in the
situation of Theorem~\ref{thm:MainIntro} the equidistribution property
extends to test functions with logarithmic singularities along
effective divisors satisfying a numerical condition (Theorem
\ref{thm:logequi} and Corollary~\ref{cor:logequi}).

Note that it is always possible to construct a sequence
$(\phi_n,\overline{Q}_n)_n$ of semipositive approximations of
$\overline{D}$ with $\mu^{\abs}(\overline{Q}_n)$ converging to
$\mu^{\ess}(\overline{D})$, see for instance
Lemma~\ref{lem:goodapprox}.  However such a sequence does not necessarily satisfy 
condition~\eqref{eq:condMainIntro}
 since, as already explained, there are situations where
the equidistribution property fails.

\subsection*{Toric varieties}\label{subsec:toricintro}
We first apply our results in the toric setting. To this end, let $X$
be a projective toric variety over $K$ with torus
$\mathbb{T} \simeq \Gm^{d}$ and $\overline{D}$ a toric adelic
$\mathbb{R}$-divisor on~$X$ with big geometric $\bR$-divisor $D$.  Let
$\Delta_{D}$ be the $d$-dimensional polytope associated to this
$\mathbb{R}$-divisor. Following \cite{BPS:asterisque}, the family of
Green functions of~$\overline{D}$ induces a family of concave
functions $\vartheta_{\overline{D},v}\colon \Delta_{D}\to \mathbb{R}$,
$v\in \mathfrak{M}_{K}$, called \textit{local roof functions}, whose
weighted sum gives the \textit{global roof function}
\begin{displaymath}
\vartheta_{\overline{D}} \colon \Delta_{D}\to \mathbb{R}.
\end{displaymath}
These concave functions convey a lot of information about the height
function of~$\overline{D}$. For instance, its essential minimum
coincides with the maximum of $\vartheta_{\overline{D}}$, and if
$\overline{D}$ is semipositive then its absolute minimum coincides
with the minimum of this concave function. This readily implies that
the only toric adelic $\mathbb{R}$-divisors to which Yuan's theorem
applies are those whose associated global roof function is constant.

The global roof function $\vartheta_{\overline{D}}$ is said to be
\emph{wide} if the width of its sup-level sets remains relatively
large as the level approaches its maximum, see  the Appendix for details. When this is the case, there
is a unique balanced family of vectors $u_{v}$,
$v\in \mathfrak{M}_{K}$, such that each $u_{v}$ is a sup-gradient of
the $v$-adic roof function $\vartheta_{\overline{D},v}$. Then for each
$v$ we can associate to $u_{v}$ a probability measure on $X_{v}^{\an}$
that we denote by $\eta_{\overline{D},v}$. \change{When $v$ is
  Archimedean it is the Haar probability measure on the translate
  $\val_v^{-1}(u_v)$ of the compact torus of $\mathbb{T}_{v}^{\an}$,
  whereas if $v$ is non-Archimedean it is the Dirac measure at the
  Gauss point of the translate $\val_v^{-1}(u_{v})$ of this $v$-adic
  analytic torus, see Section~\ref{sec:inrad-posit-arithm} for
  precisions.}

The following is our main result in this setting.

\begin{introtheorem}
  \label{thm:toricIntro}
  Let $\overline{D}$ be a toric adelic $\mathbb{R}$-divisor on $X$
  with $D$ big, and assume that $\vartheta_{\overline{D}}$ is
  wide. Then for every $v\in \mathfrak{M}_{K}$ and
every  $\overline{D}$-small generic sequence $(x_{\ell})_{\ell}$ in
  $X(\overline{K})$ the sequence of probability measures
  $(\delta_{O(x_{\ell})_{v}})_{\ell}$ on $X_v^{\an}$ converges weakly
  to $ \eta_{\overline{D},v}$.
\end{introtheorem}

We actually show a stronger result (Theorem \ref{thm:1}): under the
assumptions of Theorem \ref{thm:toricIntro}, for every
$\overline{D}$-small generic sequence $(x_{\ell})_{\ell}$ in
$X(\overline{K})$ and every
$\overline{E} \in \widehat{\Div}(X)_{\mathbb{R}}$ with $E$ toric we
have
  \begin{displaymath} 
\lim_{\ell\to \infty} h_{\overline{E}}(x_{\ell}) = \sum_{v\in \mathfrak{M}_{K}}n_{v}\int_{X_{v}^{\an}}g_{\overline{E},v} \, d \eta_{\overline{D},v},
\end{displaymath}
where $g_{\overline{E},v}$ denotes the $v$-adic Green function \change{of $\overline{E}$.}

As an application of our logarithmic equidistribution result we
strengthen the previous to allow test functions with logarithmic
singularities along the boundary and some specific translates of
subtori (Theorem \ref{thm:2}).  In the semipositive case we can
combine it with the characterization of the Bogomolov property
in~\cite[Section~5]{BPRS:dgopshtv} to obtain the following consequence
(Corollary \ref{cor:3}).

\begin{introtheorem}
  \label{thm:toriclogequiIntro} 
  Let $\overline{D}$ be a semipositive toric adelic
  $\mathbb{R}$-divisor on $X$ with $D$ big and such that
  $\vartheta_{\overline{D}}$ is wide.  Let $E$ be an effective divisor
  on $X$ such that each irreducible component $V$ of $E$ that is not
  contained in $X\setminus \mathbb{T}$ satisfies
  $\mu^{\ess}(\overline{D}|_{V})=\mu^{\ess}(\overline{D})$. Then for
  every $v\in \mathfrak{M}_{K}$ and every $\overline{D}$-small generic
  sequence $(x_{\ell})_{\ell}$ in $X(\overline{K})$ we have
 \begin{displaymath}
   \lim_{\ell \to \infty}\int_{X_v^{\an}} \varphi \, d\delta_{O(x_\ell)_v} =
\int_{X_{v}^{\an}}\varphi \, d \eta_{\overline{D},v}
 \end{displaymath}
 for any function
 $\varphi \colon X_v^{\an} \rightarrow \bR \cup \{\pm \infty \}$ with
 at most logarithmic singularities along~$E$ \change{(Definition \ref{def:3})}.
\end{introtheorem}

We also show that when $\overline{D}$ is semipositive, the converse of
Theorem \ref{thm:toricIntro} holds: in this situation, the condition
that $\vartheta_{\overline{D}}$ is wide is necessary for the
equidistribution of the Galois orbits of points in
$\overline{D}$-small generic sequences (Theorem \ref{thm:5}).
Together with Theorem \ref{thm:toricIntro}, this fully recover the
main theorem of \cite{BPRS:dgopshtv}.  We refer to Remark~\ref{rem:9}
for a more detailed comparison between our results and those in
\emph{loc. cit.}.

\subsection*{Dynamical systems and semiabelian
  varieties}

Yuan's theorem gives the equidistribution property for the canonical
metrized line bundles associated to polarized dynamical systems, a
result with vast consequences in arithmetic dynamics.  In a similar
vein, our result implies this property for adelic
$\mathbb{R}$-divisors that are sums of several canonical adelic
$\mathbb{R}$-divisors with different regimes with respect to a
dynamical system that is not necessarily polarized.

Let $\phi \colon X\to X$ be a surjective endomorphism of a normal
projective variety over~$K$ of dimension $d\ge 1$.  Then $\phi$ is
finite and we denote by $\deg(\phi)$ its degree. For $i=1,\dots, s$
let $D_{i} \in \Div(X)_{\mathbb{R}}$ with $\phi^*D_i$ \change{linearly equivalent to} $q_iD_i$
for a real number $q_{i}>1$ and set
\begin{displaymath}
\overline{D} = \sum_{i = 1}^s \overline{D}_i^{\can},
\end{displaymath}
where $\overline{D}_i^{\can}\in \widehat{\Div}(X)_{\mathbb{R}}$
denotes the canonical adelic $\mathbb{R}$-divisor over $D_{i}$. For
simplicity here we assume that $D_{i}$ is effective for every $i$ and
that $D$ is ample, although our results are valid under slightly
weaker positivity assumptions.

\begin{introtheorem}
  \label{introthm:equidyn}
  Let $v\in \mathfrak{M}_{K}$ and denote by $\mu_{v}$ the normalized
  Monge-Ampère measure of any semipositive adelic $\mathbb{R}$-divisor
  over $D$.
Then
\begin{enumerate}[label=(\roman*), leftmargin=*, widest=ii]
\item \label{item:1} the sequence
  $\displaystyle{\Big(\frac{\phi_{v}^{\circ n,\an, *} \mu_{v}
    }{\deg(\phi)^n}\Big)_{n}}$ converges weakly to a probability
  measure $\nu_{v}$ on $X_v^{\an}$,
\item \label{item:8} for every $\overline{D}$-small generic sequence
  $(x_{\ell})_{\ell}$ in $X(\overline{K})$ the sequence of
  probability measures $(\delta_{O(x_{\ell})_{v}})_{\ell}$ on
  $X_v^{\an}$ converges weakly to~$\nu_{v}$.
\end{enumerate}
\end{introtheorem}

In this statement we denote by $\phi_{v}^{\circ n,\an, *} \mu_{v}$ the
pullback of the probability measure~$\mu_{v}$ on $X_{v}^{\an}$ with
respect to the $n$-th iterate of the $v$-adic analytification of
$\phi$, which is well-defined because $\phi$ is finite.  Note that a
semipositive adelic $\mathbb{R}$-divisor over $D$ always exists by
ampleness.  

The probability measure $\nu_v$ in Theorem \ref{introthm:equidyn} is called the $v$-adic
\emph{equilibrium measure} of $\phi$ with respect to $D$, and this result shows that it is well-defined in this context.
 If the $v$-adic Green function of $\Dcan_i$ is semipositive
for every $i$ then this measure coincides with the normalized $v$-adic
Monge-Ampère measure of~$\overline{D}$,~that~is
\begin{displaymath}
\nu_{v} = \frac{c_1(\overline{D}_v)^{\wedge d}}{(D^d)}.  
\end{displaymath}
These results are contained in Theorem \ref{thm:equidyn}, which also
shows the existence and gives an explicit expression for the limit
height $ \lim_{\ell \to \infty} h_{\overline{E}}(x_{\ell})$ for any
$\overline{E}\in \widehat{\Div}(X)_{\mathbb{R}}$.

Theorem \ref{introthm:equidyn} is a straightforward consequence of
Theorem \ref{thm:MainIntro}. When $\overline{D}$ is semipositive, it
is obtained by considering the sequence of adelic
$\mathbb{R}$-divisors on $X$ defined as
\begin{displaymath}
  \overline{Q}_n = q^{-n}\, \phi^{\circ n, *}\overline{D}, \quad n\in \mathbb{N},
\end{displaymath}
with $q=\max_{j}q_{j}$. From the dynamical properties of
$\overline{D}_{i}^{\can}$, $i=1,\dots, s$, one can check that
$(\Id_{X},\overline{Q}_{n})$ is a semipositive approximation of
$\overline{D}$ whose absolute minimum decreases much faster than its
inradius as $n\to \infty$, and conclude then with
Theorem~\ref{thm:MainIntro}.

As another application of the logarithmic equidistribution theorem, we
strengthen the previous result in the semipositive case to allow test
functions with logarithmic singularities along preperiodic
hypersurfaces.

\begin{introtheorem}
  \label{introthm:dynlogEP} 
  Assume that $\Dcan_{i}$ is semipositive for every $i$.  Let
  $(x_{\ell})_{\ell}$ be a $\overline{D}$-small generic sequence in
  $X(\overline{K})$ and $E$ an effective divisor on $X$ such that each
  of its irreducible components is preperiodic. Then for every
  $v\in \mathfrak{M}_{K}$ we have
   \begin{displaymath}
     \lim_{\ell \to \infty}\int_{X_v^{\an}} \varphi \, d\delta_{O(x_\ell)_v} = \int_{X_v^{\an}} \varphi \, \frac{c_1(\overline{D}_v)^{\wedge d}}{(D^{d})}
   \end{displaymath}
   for any function
   ${\varphi \colon X_v^{\an} \rightarrow \bR \cup \{\pm \infty \}}$
   with at most logarithmic singularities along $E$ \change{(Definition \ref{def:3})}.
 \end{introtheorem}

 As a particular case of these results we recover K\"uhne's semiabelian
 equidistribution theorem \cite{Kuhne:pshsv} and extend it to include
 the computation of the corresponding limit heights (Theorem
 \ref{thm:semiab}).  Furthermore, we can also strengthen it to show
 the convergence of the Galois orbits of points in generic small
 sequences with respect to test functions with logarithmic
 singularities along torsion hypersurfaces (Theorem
 \ref{thm:logequisemiab}).

\subsection*{Further results and questions}

In Theorem \ref{thm:MainProduct} we give an alternative formulation of
Theorem \ref{thm:MainIntro} in terms of the arithmetic positive
intersection numbers introduced by Chen \cite{Chen:diffvol}. Its
statement is more intrinsic in the sense that it does not rely on the
choice of a specific sequence of semipositive approximations.

In \cite[Theorem 5.4.3]{YuanZhang:quasiproj}, Yuan and Zhang extended
Yuan's theorem to the setting of adelic line bundles on
quasi-projective varieties. In Section \ref{sec:quasiproj} we present
a generalization of Theorem \ref{thm:MainIntro} in this context that
recovers this result.

In the toric setting the condition in Theorem
\ref{thm:MainIntro} translates into a convex analysis statement
involving the global roof function. We have a similar result on an
arbitrary projective variety in terms of the arithmetic Okounkov
bodies introduced and studied by Boucksom and Chen \cite{BC}. Since
this is beyond the scope of this text, it will appear in a subsequent
manuscript.

As already explained, the condition in Theorem \ref{thm:MainIntro} is
optimal in the semipositive toric case: for a semipositive toric
adelic $\bR$-divisor the equidistribution property at every place is
equivalent to the existence of a sequence of semipositive
approximations satisfying  condition \eqref{eq:condMainIntro}. It
is natural to ask whether this remains true in general, that is if
Theorem \ref{thm:MainIntro} actually gives a criterion for the
equidistribution property (Question~\ref{question:converse}). In
Proposition \ref{prop:conversequasican} we give an affirmative answer
under an additional technical condition, which is always satisfied for
semipositive toric adelic
$\bR$-divisors.

\subsection*{Comments on the proof}
Our initial  motivation was to generalize the toric equidistribution theorem from
\cite{BPRS:dgopshtv} to all projective varieties. A major
obstacle was that both the statement and the proof of this result rely
heavily on notions and tools that are specific to toric adelic $\bR$-divisors,
with no clear extension outside of the toric setting. A first step was
to produce a non-trivial reformulation of the toric equidistribution
theorem that we could translate in terms of the arithmetic and
geometric properties of the algebra of global sections of the adelic
$\bR$-divisor, leading to the statement of Theorem
\ref{thm:MainIntro}.
 
The proof of Theorem \ref{thm:MainIntro} is based on Szpiro, Ullmo and
Zhang's variational principle and Yuan's arithmetic analogue of Siu's
inequality. Compared with \cite{Yuan:big}, we do not apply the latter
directly to $\overline{D}$ but rather to the sequence of semipositive
approximations satisfying  condition \eqref{eq:condMainIntro}. The
main difficulty is that we need to keep a very precise control of the
error terms that arise in this asymptotic process.  The core of our
proof is a new consequence of the arithmetic Siu's inequality with an
error term involving an inradius (Corollary \ref{coro:ineqSiu}). It is
based on precise estimates for arithmetic intersection numbers based
on the arithmetic Hodge index theorem of Yuan and
Zhang~\cite{YuanZhang:hodge}.  This technical feature is not needed in
the case of curves, for which most of the difficulties disappear.

We next outline the proof in this
situation.  Assume that $X$ is a smooth projective curve over~$K$ with
an adelic $\mathbb{R}$-divisor $\overline{D}$ such that $D$ is big.
Shifting the Green functions of $\overline{D}$ if necessary we assume
$\mu^{\ess}(\overline{D}) > 0$.

The assumption of Theorem~\ref{thm:MainIntro} boils down to
 \begin{equation}
   \label{eq:83}
   \lim_{n\to \infty}\frac{\mu^{\ess}(\overline{D})-\mu^{\abs}(\overline{Q}_{n})}{(Q_{n})}=0
 \end{equation}
 for a sequence $(\overline{Q}_{n})_{n}$ of semipositive adelic
 $\mathbb{R}$-divisors on $X$ with $Q_{n}$ big and
 $\overline{D}-\overline{Q}_{n}$ pseudo-effective for every
 $n$. Indeed, since $d=1$ every modification of $X$ is an isomorphism,
 and moreover $r(Q_{n};D)=(Q_{n})/(D)$ and so the inradius in that
 condition can be replaced by the intersection number~$(Q_{n})$.

 Let $v \in \mathfrak{M}_K$ and $(x_{\ell})_{\ell}$ a
 $\overline{D}$-small generic sequence in $X(\overline{K})$. To prove
 that $\overline{D}$ satisfies the equidistribution property at $v$ we
 need to show that for every continuous function
 $\varphi\colon X_v^{\an} \to \mathbb{R} $ we have
\begin{equation}\label{eq:ideaproofint}
\lim_{\ell \to \infty} \int_{X_v^{\an}} \varphi \, d\delta_{O(x_{\ell})_v} = \lim_{n \to \infty} \int_{X_v^{\an}} \varphi \,  \frac{c_{1}(\overline{Q}_{n,v})^{\wedge d}}{(Q_{n})},
\end{equation}
including the existence of both limits. By a standard density
argument, it suffices to consider the case where $\varphi$ is
$\Gal(\overline{K}_v/K_v)$-invariant. This allows to associate to
$\varphi$ an adelic divisor $\overline{E} = \overline{0}^{\varphi}$
over the zero divisor of $X$ such that
\eqref{eq:ideaproofint} translates into
\begin{equation}\label{eq:ideaproofheight}
\lim_{\ell \to \infty} h_{\overline{E}}(x_{\ell}) = \lim_{n\to \infty} \frac{(\overline{Q}_n \cdot \overline{E})}{(Q_n)}.
\end{equation}
By density, we can furthermore assume that $\overline{E}$ is the
difference of two semipositive adelic divisors on $X$.

Let $ \lambda  \in (0,1)$. Then
\begin{displaymath}
 \mu^{\ess}(\overline{D}) + \lambda\liminf_{\ell \to \infty } h_{\overline{E}}(x_{\ell}) = \liminf_{\ell \to \infty } h_{\overline{D} + \lambda \overline{E}}(x_{\ell}) \ge \mu^{\ess}(\overline{D} + \lambda\overline{E})
\end{displaymath}
by the linearity of heights and the definition of the essential
minimum.  Now the assumption that $\overline{D} - \overline{Q}_n$ is
pseudo-effective together with Zhang's lower bound for the essential
minimum in terms of the $\chi$-volume (Theorem \ref{thm:Zhang}) gives
\begin{displaymath}
  \mu^{\ess}(\overline{D} + \lambda\overline{E}) \ge \mu^{\ess}(\overline{Q}_n + \lambda \overline{E}) \ge \frac{\widehat{\vol}_{\chi}(\overline{Q}_n + \lambda \overline{E})}{2\, (Q_n)}.
\end{displaymath}

Condition \eqref{eq:83} implies $\mu^{\abs}(\overline{Q}_n) > 0$
for $n$ large enough, and therefore $\overline{Q}_n$ is nef.  Then a
classical consequence of Yuan's arithmetic version of Siu's inequality
\cite{Yuan:big} shows that there exists a constant $c \ge 0$ such that
\begin{equation}\label{eq:ideaproofSiu}
  \widehat{\vol}_{\chi}(\overline{Q}_n + \lambda \overline{E}) \ge
  (\overline{Q}_n^2) + 2 \,\lambda \, (\overline{Q}_n\cdot \overline{E}) - c\, \lambda^2
\end{equation}
for every  sufficiently large $n$.  
On the other hand, Zhang's theorem on successive minima gives
$ {(\overline{Q}_n^2)} \ge 2\, (Q_n) \, \mu^{\abs}(\overline{Q}_n)$.
Combining these inequalities gives
\begin{displaymath}
\liminf_{\ell \to \infty } h_{\overline{E}}(x_{\ell}) \ge \frac{\mu^{\abs}(\overline{Q}_n)- \mu^{\ess}(\overline{D})}{\lambda}   + \frac{(\overline{Q}_n \cdot \overline{E})}{(Q_n)} - \frac{c\, \lambda}{2\, (Q_n)}.
\end{displaymath}
Using \eqref{eq:83}, we can apply this to a suitable sequence of real
numbers $(\lambda_{n})_{n}$ to obtain
\begin{displaymath}
  \liminf_{\ell \to \infty } h_{\overline{E}}(x_{\ell}) \ge
\limsup_{n\to \infty} \frac{(\overline{Q}_n \cdot \overline{E})}{(Q_n)},
\end{displaymath}
and then \eqref{eq:ideaproofheight} by applying this inequality to
$-\overline{E}$. This concludes the proof in the one-dimensional case.

In higher dimensions this argument breaks down as a lower bound of the
form~\eqref{eq:ideaproofSiu} is not precise enough to take advantage
of the condition in Theorem~\ref{thm:MainIntro}. This is where we need
the finer estimate from Corollary \ref{coro:ineqSiu}.

When specialized to semiabelian varieties, our proof of Theorem
\ref{thm:MainIntro} is different from Kühne's, although both
ultimately rely on an asymptotic use of the arithmetic Siu's
inequality: indeed, we do not need to modify the semiabelian variety
through a sequence of isogenies as in \cite{Kuhne:pshsv}, but rather
produce a suitable sequence of semipositive approximations sitting on
the given semiabelian variety.

As emphasized in \cite{Chen:diffvol}, the variational principle
reduces the equidistribution property to the differentiability of
invariants associated to adelic $\mathbb{R}$-divisors. For example,
Yuan's theorem is a consequence of the differentiability of the
$\chi$-volume function~\cite{Yuan:big} whereas Chen's equidistribution
theorem follows from the differentiability of the arithmetic volume
function \cite{Chen:diffvol}.  Similarly, our main results concern the
differentiability of the essential minimum function, which is
equivalent to the height convergence property in \eqref{eq:HCPintro}
and thus implies Theorem \ref{thm:MainIntro}.  We refer the reader to
Section~\ref{sec:variational} for a review of the correspondence
between differentiability and equidistribution.

\subsection*{Organization} Sections \ref{sec:Rdiv} and
\ref{section:adelicdiv} contain the material we need about
$\bR$-divisors and adelic $\bR$-divisors. In Section \ref{sec:diffvol}
we study the relationship between Fujita approximations and positive
intersection numbers of adelic
$\mathbb{R}$-divisors. We state our main theorem in
Section \ref{sec:Mainthm} and prove it in Section \ref{sec:Proof}
together with some complements like the partial converse and the
logarithmic equidistribution theorem.  In Sections
\ref{sec:toric-varieties} and~\ref{sec:dynamical-heights} we apply
these results in the settings of toric varieties and of dynamical
systems, including semiabelian varieties. Finally, in Section
\ref{sec:quasiproj} we extend our main theorem to quasi-projective
varieties. The Appendix contains the auxiliary
results in convex analysis that are needed for the application to
toric varieties.

 \subsection*{Acknowledgments}
 
 We thank Jos\'e Ignacio Burgos Gil, Huayi Chen, \'Eric Gaudron,
 Thomas Gauthier, Souvik Goswami, Roberto Gualdi, Walter Gubler, Shu Kawaguchi,
 Robert Wilms, Junyi Xie and De-Qi Zhang for helpful
 discussions. Special thanks go to the participants of the Oberseminar
 on Arakelov theory at the university of Regensburg for their
 attentive reading and their suggestions on a previous version of this
 text. We are very grateful to the referee for numerous remarks and suggestions that helped us clarify and improve the text.

 Part of this work was done while we met at the universities of
 Barcelona and Caen, and we thank these institutions for their
 hospitality. Fran\c{c}ois Balla\"{y} was partially supported by the ANR project AdAnAr (Projet-ANR-24-CE40-6184).
  Mart\'{\i}n Sombra was partially supported by the
 MICINN research project PID2023-147642NB-I00, the AGAUR research
 project 2021-SGR-01468 and the AEI project CEX2020-001084-M of the
 Mar\'{\i}a de Maeztu program for centers and units of excellence in
 R\&D.

\section{Notation and conventions} 

We let $K$ be a number field and $\overline{K}$ a fixed algebraic
closure of $K$. We also let $\mathcal{O}_{K}$ be the ring of integers
of $K$.  We denote by $\mathfrak{M}_K$ the set of places of $K$ and by
$\mathfrak{M}_{K}^{\infty}\subset \mathfrak{M}_{K}$ the subset of
Archimedean places. For each $v \in \mathfrak{M}_ K$ we denote by
$K_v$ the completion of $K$ with respect to $v$, and by
$\mathbb{Q}_{v}$ the closure of $\mathbb{Q}$ in $K_{v}$. We endow
$K_{v}$ with the unique absolute value $|\cdot|_{v}$ extending the
usual absolute value on $\mathbb{Q}_{v}$ and  set
\begin{displaymath}
n_v = \frac{[K_v:\bQ_v]}{[K:\bQ]}.
\end{displaymath}
These absolute values and weights
verify the \emph{product formula}:
\begin{displaymath}
  \sum_{v \in \mathfrak{M}_K} n_v \log|\alpha|_v = 0 \quad \text{ for } \alpha \in K^{\times}.
\end{displaymath} 
Furthermore, for any place $v_0 \in \mathfrak{M}_{\bQ}$ we have
$\sum_{v \mid v_0} n_v = 1$, the sum being over the places
$v \in \mathfrak{M}_K$ above $v_0$.

For each $v\in \mathfrak{M}_{K}$ we fix an algebraic closure
$\overline{K}_v$ of $K_{v}$, which admits a unique absolute value
extending $|\cdot|_{v}$. We  denote by $\mathbb{C}_{v}$ the
completion of $\overline{K}_{v}$, and we still denote by
$|\cdot|_{v}$ the induced absolute value. We also fix an embedding
\begin{math}
\overline{K} \hookrightarrow \overline{K}_v \subset \mathbb{C}_{v}.  
\end{math}

A \emph{variety} is a separated and integral scheme of finite type
over a field. \change{Throughout the text we} let $X$ be a normal projective variety over $K$ of
dimension $d \ge 1$. A \emph{modification} of $X$ is a birational
projective morphism $\phi \colon X'\rightarrow X$, that is said
\emph{normal} if so is $X'$.  We write
$X_{K'} = X \times_K \Spec (K')$ for any field extension
$K \subset K'$.  The elements of $X(\overline{K})$ are called the
\emph{algebraic points} of $X$.

A \emph{measure} is a Radon measure on a locally compact Hausdorff
space. A \emph{signed measure} is a difference of two \change{finite
  measures. For a sequence of measures $(\nu_n)_n$ on a locally
  compact Hausdorff space we write $\lim_n \nu_n = \nu$ to say that it
  converges weakly to a measure $\nu$ on that space.}

\numberwithin{equation}{section}

\section{The inradius of an \texorpdfstring{$\bR$}{R}-divisor}
\label{sec:Rdiv}

Throughout the text we assume some working knowledge of
$\mathbb{R}$-divisors and their positivity properties as presented in
\cite[Chapters 1 and 2]{LazI}. Nevertheless, in this section we recall
the basic objects and notations. We pay special attention to the
inradius of an $\mathbb{R}$-divisor, explaining its properties and
relation to other invariants.

\subsection{Intersection numbers and global sections of
  \texorpdfstring{$\bR$}{R}-divisors}
\label{sec:inters-numb-glob}

We denote by $\Div(X)$ the Abelian group of Cartier divisors on $X$
and by $\Div(X)_\bR = \Div(X)\otimes_{\bZ} \bR$ the real vector space
of $\bR$-Cartier divisors. Since $X$ is normal, there is an injective
morphism from $\Div(X)$ to the free Abelian group of Weil divisors of
$X$. Therefore $\Div(X)$ has no torsion and there is an inclusion
\begin{math}
\Div(X) \subset \Div(X)_{\bR}.
\end{math}
Since we will be mainly concerned with Cartier divisors and
$\mathbb{R}$-Cartier divisors, we just call them divisors and
$\mathbb{R}$-divisors for short.

We denote by $\Rat(X)$ the field of rational functions of $X$ and we
set $\Rat(X)^{\times}_{\bR} = \Rat(X)^{\times}\otimes_{\bZ}\bR$.  For
a nonzero rational function $f \in \Rat(X)^{\times}$ we denote by
$\div(f) \in \Div(X)$ its associated principal divisor. This
assignment extends by linearity to a map
\begin{displaymath}
\div\colon \Rat(X)^{\times}_\bR \longrightarrow \Div(X)_{\bR}.  
\end{displaymath}
Given two $\mathbb{R}$-divisors $D,D'$ on $X$ we write $D \equiv D'$
when they are linearly equivalent, that is when $D' = D + \div(f)$ for
some $f \in \Rat(X)^{\times}_\bR$.

For $D \in \Div(X)_{\mathbb{R}}$ we denote by $[D]$ its associated
$\mathbb{R}$-Weil divisor. The \emph{support} of~$D$, denoted by
$\mathrm{supp}(D)$, is defined as the support of this $\bR$-Weil
divisor. It is a closed subset of $X$.

For an integer $0\le r \le d$, an $r$-dimensional cycle $Z$ of $X$ and
a family of $\mathbb{R}$-divisors $D_{i}\in \Div(X)_\bR$,
$i=1,\dots r$, we denote their intersection number by
\begin{displaymath}
  (D_1 \cdots D_r \cdot Z) \in \mathbb{R}.
\end{displaymath}
When $Z =X$ we simply write this quantity as
\begin{math}
  (D_1 \cdots D_d).
\end{math}
It depends only on the linear equivalence classes of the
$\bR$-divisors and is invariant under normal modifications.  If
$D\in \Div(X)_{\bR}$ is nef then the intersection number~$(D^d)$
coincides with the volume $\vol(D)$ of $D$.

Given $D_{1},\dots, D_{d-1},D_{d},D'_{d} \in \Div(X)_{\mathbb{R}}$
with $D_{i}$ nef for $1 \le i\le d-1$ and $D_{d}-D_{d}'$
pseudo-effective, the corresponding intersection numbers compare as
\begin{equation}
  \label{eq:comparegeometriccap}
  (D_1 \cdots D_{d-1}\cdot D_d')\le (D_1 \cdots  D_{d-1}\cdot D_d).
\end{equation}

 The \emph{space of global sections} of an $\bR$-divisor $D$  is defined as 
\begin{displaymath}
\Gamma(X,D) = \{ (f, D) \ | \ f \in \Rat(X)^{\times}, \ D + \div(f) \ge 0\} \cup \{0\}.
\end{displaymath}
It is a finite dimensional \change{$K$-vector space}.  Given
$s = (f, D) \in \Gamma(X,D) \setminus \{0\}$ we set
\begin{displaymath}
 \div(s) = D + \div(f) \in \Div(X)_\bR 
\end{displaymath}
for the corresponding $\mathbb{R}$-divisor in the linear equivalence class
of $D$.

Given global sections $s_1 = (f_1,D_1)$ and $s_2 = (f_2,D_2)$ of
$\bR$-divisors $D_1,D_2 \in \Div(X)_{\bR}$, their product is defined
as
\begin{displaymath}
s_1 \otimes s_2 = (f_1f_2, D_1 + D_2) \in \Gamma(X,D_1 + D_2).
\end{displaymath}
The \change{$K$-\emph{algebra of sections}} of $D$ is then defined as the direct sum 
\begin{displaymath}
 R(D)= \bigoplus_{m \in \bN} \Gamma(X,mD)
\end{displaymath}
endowed with the structure of graded algebra induced by this product.

\subsection{Definition and basic properties of the inradius}
\label{subsection:inradius}

The notion of inradius of a line bundle with respect to another one
was introduced by Tessier as a measure of its
bigness~\cite{Teissier}. As explained in \emph{loc. cit.}, the
terminology is inspired by its interpretation in the toric setting,
where it identifies with the usual inradius in the sense of convex
geometry (Proposition~\ref{prop:1}).

\begin{definition}
  \label{def:8}
  Let $P,A$ be big $\mathbb{R}$-divisors on $X$. The \emph{inradius}
  of $P$ with respect to~$A$ is defined as
\begin{displaymath}
  r(P;A)  = \sup\{ \lambda \in \bR \ | \ P - \lambda A \text{ is big} \}.
\end{displaymath}
\end{definition}

We have $r(P;A)<\infty$: indeed, choosing any ample divisor $H$ then
for any $\lambda\in \mathbb{R}$ such that $P-\lambda A$ is
\change{big} we have $\lambda\le (H^{d-1}\cdot P)/(H^{d-1}\cdot A)$ by
\eqref{eq:comparegeometriccap}.  On the other hand, since $P$ is big
we also have $r(P;A)>0$ by the continuity of the volume
function. Hence $r(P;A)$ is a positive real number. \change{Note that
  $ P-r(P;A) A $ is pseudo-effective~and}
\begin{displaymath}
\change{r(P;A)  =   \max\{ \lambda \in \bR \ | \ P - \lambda A \text{ is pseudo-effective} \}. }
 \end{displaymath}
Moreover
\begin{displaymath}
r(\delta P;A) = \delta \, r(P;A) \and
r(P;\delta A) = \delta^{-1} \, r(P;A) \quad \text{ for } \delta \in \bR_{> 0},   
\end{displaymath}
and $r(\phi^*P;\phi^*A) = r(P;A)$ for any
normal modification $\phi \colon X' \rightarrow X$.  This latter
property follows from the fact that the volume is invariant under
birational morphisms.

Changing the big $\mathbb{R}$-divisor of reference $A$ can only modify
the inradius $r(P;A)$ up to a bounded factor, as we next show.

\begin{lemma}\label{lemma:compareinradii} Let
  $P, A_1, A_2$ be big $\mathbb{R}$-divisors on $X$. Then
\begin{displaymath}
r(A_2;A_1) \, r(P;A_2) \le r(P;A_1) \le \frac{1}{r(A_1;A_2)} \, r(P;A_2).
\end{displaymath}
\end{lemma}

\begin{proof} Since both $P - r(P;A_2)A_2$ and $A_2 - r(A_2;A_1) A_1$
  are pseudo-effective, so is $P- r(P;A_2) \, r(A_2;A_1) \, A_1$.
  Therefore $r(A_2;A_1) \, r(P;A_2) \le r(P;A_1)$, which gives the
  first inequality. The second follows by interchanging the roles of
  $A_1$ and $A_2$.
\end{proof}

The next key lemma shows that the inradius can be controlled up to a
constant factor by a quotient of intersections numbers.

\begin{lemma}\label{lemma:inradiusandcap} Let $P, A$ be big and
  nef $\mathbb{R}$-divisors on $X$. Then
\begin{displaymath}
\frac{(P^d)}{d\, (P^{d-1} \cdot A)} \le r(P;A) \le \frac{(P^d)}{(P^{d-1} \cdot A)}.
\end{displaymath}
\end{lemma}

\begin{proof} For $\lambda \in \bR$, by Siu's inequality \cite[Theorem
  2.2.15]{LazI} we have that $P-\lambda A$ is big whenever
\begin{displaymath}
(P^{d})> \lambda \, d \, (P^{d-1} \cdot A).
\end{displaymath}
This gives the first inequality.  The second follows from the facts
that $P$ is nef and $P - r(P;A) A$ is pseudo-effective, which imply
\begin{displaymath}
(P^d) - r(P;A) \, (P^{d-1}\cdot A) = (P^{d-1}\cdot (P - r(P;A) A)) \ge 0
\end{displaymath}
by inequality \eqref{eq:comparegeometriccap}.
\end{proof}

\begin{lemma}\label{lemma:lowerboundinradiusvol}
  Let $P, A$ be big  and nef $\mathbb{R}$-divisors on $X$ with   $A - P$
  pseudo-effective. Then
 \begin{displaymath}
   r(P;A) \ge \frac{(P^d)}{d \, (A^d)}.
 \end{displaymath}
\end{lemma}

\begin{proof}
  Since $A - P$ is pseudo-effective and $A,P$ are nef, we get
  $(P^{d-1} \cdot A) \leq (A^d)$ by iteratively
  applying \eqref{eq:comparegeometriccap}. The result then follows from Lemma
  \ref{lemma:inradiusandcap}.
\end{proof}

The inradius allows to control the behavior of the volume function
with respect to perturbations, under suitable positivity assumptions.

\begin{lemma}\label{lemma:comparevolsinradius} Let
  $P, E,A \in \Div(X)_{\bR}$ with $P,A$ big and $A \pm E$
  pseudo-effective. Then for every $\lambda\in \mathbb{R}$ with
  $0\le \lambda\le r(P;A)$ we have
\begin{displaymath}
\Big( 1 - \frac{\lambda}{r(P;A)} \Big)^d \vol(P) \le \vol(P + \lambda E) \le \Big( 1 + \frac{\lambda}{r(P;A)} \Big)^d \vol(P).
\end{displaymath}
\end{lemma}

\begin{proof} Since $A - E$ is pseudo-effective we have that
\begin{displaymath}
\frac{1}{r(P;A)} P - E = \frac{1}{r(P;A)} (P - r(P;A) A) + A - E
\end{displaymath}
is also pseudo-effective.  Therefore 
\begin{displaymath}
\vol( P + \lambda E) \le \vol \Big(P + \frac{\lambda}{r(P;A)}P \Big) = \Big( 1 + \frac{\lambda}{r(P;A)} \Big)^d \vol(P)
\end{displaymath}
because the volume function is positive homogeneous of degree $d$ and
increases along pseudo-effective directions.  Similarly,
\begin{displaymath}
\frac{1}{r(P;A)} P + E = \frac{1}{r(P;A)} (P - r(P;A) A) + A + E
\end{displaymath}
is pseudo-effective and so 
\begin{displaymath}
\Big( 1 - \frac{\lambda}{r(P;A)} \Big)^d \vol(P) = \vol \Big(P - \frac{\lambda}{r(P;A)}P\Big) \le \vol(P + \lambda E)
\end{displaymath}
as stated. 
\end{proof}

\section{Adelic \texorpdfstring{$\bR$}{R}-divisors}\label{section:adelicdiv}

In this section we recall the definition and basic facts concerning
adelic $\bR$-divisors, referring to \cite{BPS:asterisque,
  BMPS:positivity, Moriwaki:MAMS} for the proofs and more details.

 \subsection{First definitions} \label{subsec:adelicdef}

 For $v \in \mathfrak{M}_ K$ we denote by $X_v^{\mathrm{an}}$ the
 analytification of $X_{\mathbb{C}_{v}}$ in the sense of Berkovich,
 see~\cite[Section 1.2]{BPS:asterisque} and \cite[Section
 1.3]{Moriwaki:MAMS} for short introductions sufficient for our
 purposes. There is an injective map
 $X(\bC_v) \hookrightarrow X_v^{\an}$ that induces an inclusion
\begin{displaymath}
\iota_v \colon X(\overline{K}) \longhookrightarrow X_v^{\an}
\end{displaymath}
via the chosen embedding
$\overline{K} \hookrightarrow \overline{K}_v \subset \bC_v$.

For an algebraic point $x \in X(\overline{K})$ we denote by
$O(x) \subset X(\overline{K})$ its orbit under the action of the
absolute Galois group $\Gal(\overline{K}/K)$, and by
\begin{displaymath}
  O(x)_v = \iota_v(O(x)) \subset X_{v}^{\an}
\end{displaymath}
its image under $\iota_{v}$. \change{Note that $O(x)$ is a finite set
  and that $O(x)_{v}$ does not depend on the choice of the embedding.}

The local Galois group $G_v = \Gal(\overline{K}_v/K_v)$ acts on
$X_v^{\an}$.  We denote by $C(X_v^{\an})$ the space of continuous
real-valued functions on $X_v^{\an}$ and by
$C(X_v^{\an})^{G_v} \subset C(X_v^{\an})$ the subspace of those
functions that are $G_v$-invariant.

Let $D \in \Div(X)_{\bR}$ and $v \in \mathfrak{M}_ K$. A
\emph{continuous $v$-adic Green function} for $D$ is a $G_v$-invariant
function
\begin{displaymath}
g_v \colon X_v^{\mathrm{an}} \setminus \mathrm{supp}(D)_v^{\mathrm{an}} \rightarrow \bR
\end{displaymath}
with the property that for each open subset $U\subset X$ where $D$ is
defined by an $\mathbb{R}$-rational function $f$ we have that
$g_{v}+ \log |f_{v}^{\an}|_v$ extends to a continuous function
on~$U_{v}^{\an}$, with $f_{v}^{\an}$ the $v$-adic analytification of
$f$.  In this text we only consider continuous $v$-adic Green
functions, and so we call them $v$-adic Green functions for short.

Let $U \subset \Spec(\cO_K)$ be a nonempty open subset. A \emph{model}
of $X$ over $U$ is a normal integral projective scheme
$\mathcal{X} \rightarrow U$ such that
$X = \mathcal{X} \times_U \Spec(K)$. For $D \in \Div(X)_\bR$, a
\emph{model} of $(X,D)$ over $U$ is a pair
$(\mathcal{X}, \mathcal{D})$ where $\mathcal{X}$ is a model of $X$
over $U$ and $\mathcal{D}$ is an $\bR$-divisor on $\mathcal{X}$ whose
restriction to $X$ coincides with $D$. Such a model induces a $v$-adic
Green function for $D$ for each place $v \in U$ that we denote by
$g_{\mathcal{D},v}$, see \cite[Section 2.1]{Moriwaki:MAMS} for
details.

\begin{definition}
  \label{def:adelicdiv}
  An \emph{adelic $\bR$-divisor} on $X$ is a pair
  $\overline{D} = (D,(g_v)_{v \in \mathfrak{M}_ K}) $ with
  $D \in \Div(X)_\bR$ and $g_v$ a $v$-adic Green function for $D$ for
  each $v \in \mathfrak{M}_{K}$, such that there is a model
  $(\mathcal{X},\mathcal{D})$ of $(X,D)$ over a nonempty open subset
  $U \subset \Spec(\cO_K)$ with $g_v = g_{\mathcal{D},v}$ for all
  $v \in U$. When $D$ is a divisor we say that $\overline{D}$ is an
  \emph{adelic divisor}.

  We say that $\overline{D}$ is an adelic
  $\mathbb{R}$-divisor \emph{over} $D$ and conversely, we say that
  $D$ is the \emph{geometric $\mathbb{R}$-divisor} of $\overline{D}$.
  \end{definition}

  Unless otherwise stated, given an adelic
  $\mathbb{R}$-divisor $\overline{D}$ on
  $X$ we use the same letter~$D$ to denote its geometric
  $\mathbb{R}$-divisor and we denote by $(g_{\overline{D},v})_{v \in
    \mathfrak{M}_K}$ its family of Green functions.
  
  The set of adelic $\bR$-divisors forms a real vector space that we
  denote by $\widehat{\Div}(X)_{\mathbb{R}}$, and the subgroup of
  adelic divisors is denoted by $\widehat{\Div}(X)$. \change{It is worth to
  note that
  $\widehat{\Div}(X)_{\mathbb{R}} \ne \widehat{\Div}(X) \otimes_{\bZ}
  \bR$ although there is a natural surjective morphism of vector
  spaces
  $\widehat{\Div}(X) \otimes_{\bZ} \bR \to \widehat{\Div}(X)_{\bR}$,
  see for instance \cite[Definition 3.3]{BMPS:positivity}.}

\begin{example}
    \label{ex:divon0}
    Let $(\varphi_v)_{v \in \mathfrak{M}_K}$ with
    $\varphi_v \in C(X_v^{\an})^{G_v}$ for each
    $v \in \mathfrak{M}_{K}$ and $\varphi_v = 0$ for all except
    finitely many $v$. Then
    $(0,(\varphi_v)_{v \in \mathfrak{M}_K})$ is an adelic
    divisor over $0 \in \Div(X)$. Every adelic divisor over
    the zero divisor of $X$ is of this form.
\end{example}

Denote by
$[\infty] = (0, (\varphi_v)_{v\in \mathfrak{M}_K}) \in
\widehat{\Div}(X)$ the adelic divisor over $0\in \Div(X)_{\mathbb{R}}$
given by the constant functions $\varphi_v = 1$ if $v$ is Archimedean,
and $\varphi_v = 0$ if $v$ is non-Archidemedean. Then for
$\overline{D}\in\widehat{\Div}(X)_{\mathbb{R}}$ and $t\in \mathbb{R}$
we set
\begin{equation}
\label{eq:1}
\overline{D}(t)=\overline{D}-t\, [\infty] \in\widehat{\Div}(X)_{\mathbb{R}}.
\end{equation}

Given $f \in \Rat(X)^{\times}$ we set
$ \widehat{\div}(f) = (\div(f), (-\log|f_v^{\an}|_v)_{v\in
  \mathfrak{M}_K})$ for its associated principal divisor.
This assignment extends by linearity to a map
 \begin{displaymath}
\widehat{\div} \colon \Rat(X)^{\times}_{\bR} \longrightarrow  \widehat{\Div}(X)_{\bR}..
 \end{displaymath}
 Two adelic $\bR$-divisors
 $\overline{D}, \overline{D'} \in \widehat{\Div}(X)_{\bR}$ are said 
 \emph{linearly equivalent}, denoted
 $\overline{D} \equiv \overline{D'}$, when there is
 $f \in \Rat(X)^{\times}_\bR$ such that
 $\overline{D'} = \overline{D} + \widehat{\div}(f)$.

For   $\overline{D} \in \widehat{\Div}(X)_{\bR}$ and a
 morphism
$\phi \colon X' \rightarrow X$ from a normal projective variety $X'$
whose image is not contained in $\mathrm{supp}(D)$, the
\emph{pullback} $\phi^*\overline{D} \in \widehat{\Div}(X')_{\bR}$ is
defined as the $\bR$-divisor $\phi^*D$ equipped at each place
$v \in \mathfrak{M}_K$ with the pullback to $(X')^{\an}_v$ of
$g_{\overline{D},v}$ by the $v$-adic analytification of~$\phi$.

For $D\in \Div(X)_{\mathbb{R}}$ and $v\in \mathfrak{M}_{K}$ let
$g_{v}$ be a $v$-adic Green function for $D$. We say that $g_{v}$ is
\emph{semipositive} when it is of ($C^{0} \, \cap$ PSH)-type in the
sense of \cite[Section~1.4 and Definition 2.1.6]{Moriwaki:MAMS}.  In
the Archimedean case this means that $g_{v}$ is plurisubharmonic,
whereas in the non-Archimedean case it means that $g_{v}$ can be
uniformly approached by the $v$-adic Green functions of a sequence of
vertically nef models. On the other hand, we say that $g_{v}$ is
\emph{DSP} (short for difference of semipositive) if there are
semipositive $v$-adic Green functions $g_{1,v}$ and $g_{2,v}$ such
that $g_{v}=g_{1,v}-g_{2,v}$. Note that if $D$ admits a semipositive
$v$-adic Green function then it is nef, and in particular
$\vol(D) = (D^d)$.

 Passing to the global situation, an adelic $\mathbb{R}$-divisor
 $\overline{D}$ is \emph{semipositive} if all its $v$-adic Green
 functions are semipositive, and is \emph{DSP} if there are two
 semipositive adelic $\bR$-divisors $\overline{D}_1$, $\overline{D}_2$
 such that $\overline{D} = \overline{D}_1 - \overline{D}_2$. We denote
 by
  \begin{displaymath}
\widehat{\DSP}(X)_{\bR} \subset \widehat{\Div}(X)_{\bR}
  \end{displaymath}
  the subspace of DSP adelic $\bR$-divisors on $X$. 

\begin{remark}
  \label{rem:4} \change{If $\overline{D} \in \widehat{\Div}(X)$ is an
    adelic divisor then there is a model $(\mathcal{X},\mathcal{D})$
    of $(X,D)$ over a nonempty open subset $U \subset \Spec(\cO_K)$
    such that $\mathcal{D} \in \Div(\mathcal{X})$ and
    $g_{\overline{D},v} = g_{\mathcal{D},v}$ for all $v \in
    U$. Moreover,} one can associate to $\overline{D}$ a metrized line
  bundle $\overline{L} = (\cO_X(D), (\|.\|_v)_{v \in \mathfrak{M}_K})$
  on $X$ in the sense of Zhang \cite{Zhang:spam}, and every such
  metrized line bundle can be constructed in this way
  \cite[Proposition 3.8]{BMPS:positivity}.  The metrized line bundle
  $\overline{L}$ is semipositive in the sense of \cite{Zhang:spam} if
  and only if $\overline{D}$ is semipositive, and it is integrable in
  the sense of \cite{Zhang:spam} if and only if $\overline{D}$ is DSP.
\end{remark}

\subsection{Heights of points and cycles}
\label{subsec:defheight} 

Given $\overline{D} \in \widehat{\Div}(X)_{\mathbb{R}}$
and $x \in X({\overline{K}})$,
the \emph{height} of $x$ with respect to
$\overline{D}$ is defined as 
\begin{displaymath}
 h_{\overline{D}}(x) =\sum_{v \in  \mathfrak{M}_ K} \frac{n_v}{\# O(x)_v} \sum_{y \in O(x)_v}g_{\overline{D'},v}(y) \in \mathbb{R}
\end{displaymath}
for any $\overline{D'}\in \widehat{\Div}(X)_{\mathbb{R}}$ with
$\overline{D'}\equiv \overline{D}$ such that
$x \notin \mathrm{supp}( D')(\overline{K})$.  This quantity does not
depend on the choice of $\overline{D'}$ thanks to
the product formula.

For an integer $0\le r\le d$ let $Z$ be an $r$-dimensional cycle of
$X$ and $\overline{D}_i$, $i=1,\dots, r$, a family of adelic
$\mathbb{R}$-divisors on $X$.  Let $v \in \mathfrak{M}_K$ and assume
that $g_{\overline{D}_{i},v}$ is DSP for all~$i$.  Then
there is a signed measure on $X_v^{\an}$, denoted by
\begin{equation} \label{eq:22} 
c_1(\overline{D}_{1,v}) \wedge \cdots \wedge c_1(\overline{D}_{r,v}) \wedge \delta_{Z_v^{\an}},
\end{equation}
supported on $Z_v^{\an}$ and with total mass
$(D_1 \cdots D_r \cdot Z)$.  When the $D_i$'s are divisors, it is
defined using the complex Monge-Ampère operator in the Archimedean
case, whereas in the non-Archimedean case it is the signed measure
introduced by Chambert-Loir~\cite{Chambert-Loir:Mesuresetequi}. This
construction extends by multilinearity and continuity to the general
case of adelic $\bR$-divisors with DSP $v$-adic Green functions.  When
these $v$-adic Green functions are semipositive, it is actually a
measure.

In the case when $Z = X$ and
$\overline{D}_1 = \cdots = \overline{D}_d=\overline{D}$, this signed
measure is called the \emph{$v$-adic Monge-Ampère measure}
of~$\overline{D}$ and denoted by~$c_1(\overline{D}_{v})^{\wedge d}$.

Assume now that $\overline{D}_1, \ldots, \overline{D}_r$ are DSP, and
let $\overline{D}_{r+1} $ be a further adelic $\mathbb{R}$-divisor on
$X$. The \emph{height}
$h_{\overline{D}_1, \ldots, \overline{D}_{r+1}}(Z)$ of $Z$ with
respect to $\overline{D}_{1},\dots, \overline{D}_{r+1}$ is defined by
induction on the dimension $r$ of the cycle. When $r=0$ it is given by
linearity from the previous definition of height of points, whereas
when $r>0$ it is given by
\begin{multline}\label{eq:Bezout}
h_{\overline{D}_1, \ldots, \overline{D}_{r+1}}(Z) = h_{\overline{D}_1,\ldots,\overline{D}_r}(D' \cdot Z)\\ + \sum_{v \in \mathfrak{M}_K} n_v \int_{X_v^{\an}} g_{\overline{D'},v} \, c_1(\overline{D}_{1,v}) \wedge \cdots \wedge c_1(\overline{D}_{r,v}) \wedge \delta_{Z_v^{\an}}
\end{multline}
for any $\overline{D'}\in \widehat{\Div}(X)_{\mathbb{R}}$ with
$\overline{D'}\equiv \overline{D}_{r+1}$ such that $D'$ intersects $Z$
properly,  where $D' \cdot Z$ is the intersection cycle.
This formula is well-defined because the $v$-adic Green
function $g_{\overline{D'},v}$ is integrable with respect to the
signed measure therein \cite[Theorem 4.1]{CLT}. Furthermore,  it does
not depend on the choice of $\overline{D'}$ and it is multilinear in
$\overline{D}_1, \ldots, \overline{D}_{r+1}$
\cite[Section~1.5]{BPS:asterisque}, \cite[page
225]{BMPS:positivity}. If
$\overline{D}_{r+1} \in \widehat{\mathrm{DSP}}(X)_{\bR}$ this
construction is symmetric in
$\overline{D}_1, \ldots, \overline{D}_{r+1}$.
For $\overline{D}\in \widehat{\DSP}(X)_{\mathbb{R}}$ we write
$h_{\overline{D}}(Z)$ for the height of $Z$ with respect to $r+1$
copies of $\overline{D}$. \change{Following \cite[page 225]{BMPS:positivity} 
we refer to identity  \eqref{eq:Bezout} as the \emph{arithmetic Bézout formula}.}

For
$\overline{D}_{i}\in \widehat{\DSP}(X)_{\mathbb{R}}$, $i=1,\dots, d$,
and $\overline{D}_{d+1}\in \widehat{\Div}(X)_{\mathbb{R}}$, we define their 
\emph{arithmetic intersection number}  as 
\begin{displaymath}
(\overline{D}_1 \cdots \overline{D}_{d+1}) = h_{\overline{D}_1, \ldots, \overline{D}_{d+1}}(X) \in \mathbb{R}.
\end{displaymath}
This quantity only depends on the linear equivalence classes of these
adelic $\mathbb{R}$-divisors, and for any normal modification
$\phi\colon X' \rightarrow X$ we have
\begin{displaymath}
(\phi^*\overline{D}_1 \cdots \phi^*\overline{D}_{d+1}) = (\overline{D}_1 \cdots \overline{D}_{d+1}).
\end{displaymath}

It follows from these definitions that for any adelic divisor
$\overline{0}=(0,(\varphi_v)_{v \in \mathfrak{M}_K})$ over the zero
divisor of $X$ we have
\begin{equation}
  \label{eq:12}
  (\overline{D}_{1} \cdots \overline{D}_{d}\cdot \overline{0} )
  = \sum_{v \in \mathfrak{M}_{K}} n_v\int_{X_v^{\an}} \varphi_{v} \ c_1(\overline{D}_{1,v}) \wedge \cdots \wedge c_1(\overline{D}_{d,v}).
\end{equation}
In particular, 
\begin{equation}\label{eq:arithmeticvsgeometricintersection}
(\overline{D}_{1} \cdots \overline{D}_{d}\cdot [\infty]) = \sum_{v \in \mathfrak{M}_{K}^{\infty}} n_v\int_{X_v^{\an}} c_1(\overline{D}_{1,v}) \wedge \cdots \wedge c_1(\overline{D}_{d,v}) = (D_1 \cdots D_{d}).
\end{equation}

\subsection{Arithmetic volumes and positivity}
\label{subsec:arithmvol}

Let $\overline{D} \in \widehat{\Div}(X)_{\bR}$.  Given a nonzero
global section $s = (f, D) \in \Gamma(X,D) \setminus \{0\}$, for each
point
$y \in X_v^{\mathrm{an}} \setminus
\mathrm{supp}(\div(s))_v^{\mathrm{an}}$~we~set
\begin{displaymath}
  \|s(y)\|_{\overline{D},v} = \exp(-g_{\overline{D} + \widehat{\div}(f),v}(y)).
\end{displaymath} 
By \cite[Propositions 1.4.2 and 2.1.3]{Moriwaki:MAMS}, this assignment
can be uniquely extended to a continuous real-valued function
$y \in X_v^{\mathrm{an}} \mapsto \|s(y)\|_{\overline{D},v} \in \bR$,
called the \emph{$v$-adic norm} of $s$.  We also set
\begin{equation*} 
\|s\|_{\overline{D}, v,\sup}= \sup_{y \in X_v^{\mathrm{an}}} \|s(y)\|_{\overline{D},v}.
\end{equation*}
The height of a point outside the support of a global section can be
bounded from below in terms of these sup-norms: for  $m \in \bN_{>0}$
and $s \in \Gamma(X,mD) \setminus \{0\}$ we have
 \begin{equation}
    \label{eq:heightlowerbound}
    h_{\overline{D}}(x)\ge -\frac{1}{m} \sum_{v\in \mathfrak{M}_{K}} n_v\log\|s\|_{\overline{D},v,\sup} \quad   \text{ for all }
    x\in X(\overline{K})\setminus
    \mathrm{supp}(\div(s))(\overline{K}).
  \end{equation}

  For $s \in \Gamma(X,D) \setminus \{0\} $ we say that $s$ is
  \emph{small} if $\|s\|_{\overline{D},v,\sup} \le 1$ for all
  $v \in \mathfrak{M}_ K$, and by convention we agree that
  $0 \in \Gamma(X,D)$ is small. We denote by
  \begin{math} 
\widehat{\Gamma}(X,\overline{D})    
  \end{math}
  the set of small global sections of $\overline{D}$.

Let $\mathbb{A}_K$ be the ring of ad\`{e}les of $K$ and
  consider the adelic unit~ball
\begin{displaymath}
  \mathbb{B}_{\overline{D}} = \{ (s_v)_{v\in \mathfrak{M}_K} \in \Gamma(X,D) \otimes_K \mathbb{A}_K \ | \ \|s_v\|_{\overline{D},v,\sup} \leq 1 \ \text{ for all }
  v \in \mathfrak{M}_K\}.
\end{displaymath}
We have that $\Gamma(X,D) \otimes_K \mathbb{A}_K$ is a locally compact
group and $\Gamma(X,D)$ is a lattice within it \emph{via} the diagonal
embedding $K \hookrightarrow \mathbb{A}_K$.  We denote by $\mu$ the
unique Haar measure on $\Gamma(X,D) \otimes_K \mathbb{A}_K$ satisfying
\begin{displaymath}
\mu((\Gamma(X,D) \otimes_K \mathbb{A}_K)/\Gamma(X,D))=1
\end{displaymath}
and we set
$ \widehat{\chi}(X,\overline{D}) = \log (\mu(\mathbb{B}_{\overline{D}}))$.

\begin{definition}
  The \emph{arithmetic volume} and the \emph{$\chi$-volume} of
  $\overline{D}$ are respectively defined as
  \begin{align*}
    \widehat{\vol}(\overline{D}) &= \frac{1}{[K:\bQ]}\limsup_{m \rightarrow \infty} \frac{ \log(\# \widehat{\Gamma}(X,m\overline{D}))}{m^{d+1}/(d+1)!},\\
  \widehat{\vol}_{\chi}(\overline{D}) & = \frac{1}{[K:\bQ]}\limsup_{m \rightarrow \infty} \frac{\widehat{\chi}(X,m\overline{D})}{m^{d+1}/(d+1)!}.    
  \end{align*}
\end{definition}

 We now recall classical positivity notions for adelic $\bR$-divisors. 
 
 \begin{definition}\label{def:positivity} We say that $\overline{D}$
   is
 \begin{enumerate}[leftmargin=*]
 \item \emph{effective} if $D$ is effective and 
    $g_{\overline{D},v} \ge 0$ on
   $X_v^{\an} \setminus \mathrm{supp}(D)_v^{\an}$ for all $v \in \mathfrak{M}_K$,
 \item \emph{big} if $\widehat{\vol}(\overline{D}) > 0$,
 \item \emph{pseudo-effective} if $\overline{D} + \overline{B}$ is big
   for every big $\overline{B} \in \widehat{\Div}(X)_{\bR}$,
 \item \emph{nef} if $\overline{D}$ is semipositive and $h_{\overline{D}}(x) \ge 0$ for all $x \in X(\overline{K})$.
 \end{enumerate}
 \end{definition}

 In the sequel we recall the basic properties of the arithmetic volume
 and the $\chi$-volume, referring to~\cite{Moriwaki:MAMS} for the
 details and proofs.

\begin{remark}
  \label{rem:7}
Unlike \cite{Moriwaki:MAMS}, here we do not assume that $X$ is
  geometrically irreducible. In this more general situation
  the Stein factorization shows that there is a finite extension
  $K'/K$ with the property that the structural morphism
  $X \rightarrow \Spec(K)$ factors through a morphism
  $X \rightarrow \Spec(K')$ and that $X$ is geometrically connected
  over $K'$. Since $X$ is normal, it is geometrically irreducible as a
  variety over $K'$ and all the cited results from \emph{loc. cit.}
  extend directly to our setting.
\end{remark}

 For any  normal modification
$\phi\colon X' \to X$ we have
\begin{displaymath}
\widehat{\vol}(\phi^*\overline{D}) = \widehat{\vol}(\overline{D})
\and
\widehat{\vol}_{\chi}(\phi^*\overline{D}) =
\widehat{\vol}_{\chi}(\overline{D}),   
\end{displaymath}
and moreover
$\widehat{\vol}(\overline{D}) \ge \widehat{\vol}_{\chi}(\overline{D})$
\cite[Section 4.3]{Moriwaki:MAMS}.
By \cite[Theorem
5.2.1]{Moriwaki:MAMS} we also have the following continuity property:
for every $\overline{E} \in \widehat{\Div}(X)_{\bR}$,
 \begin{displaymath}
 \widehat{\vol}(\overline{D}) = \lim_{\lambda \to 0} \widehat{\vol}(\overline{D} + \lambda \overline{E}) \and  \widehat{\vol}_{\chi}(\overline{D}) = \lim_{\lambda \to 0} \widehat{\vol}_{\chi}(\overline{D} + \lambda \overline{E}).
 \end{displaymath}

Now let $t\in \mathbb{R}$, and for each $m \in \bN$   denote by
 $R_m^t(\overline{D})$ the $K$-linear subspace of
 $\Gamma(X,mD)$ generated by $\widehat{\Gamma}(X,m \overline{D}(t))$
 with $\overline{D}(t)=\overline{D} - t[\infty]$ as in \eqref{eq:1}.
Note
that $\widehat{\Gamma}(X,m \overline{D}(t))$ is
the set of global sections $s \in \Gamma(X,mD)$ such that
\begin{displaymath}
  \log\|s\|_{\overline{D},v,\sup} \le
  \begin{cases}
 -mt & \text{ if } v \text{ is Archimedean},\\
 0 & \text{ otherwise.} 
  \end{cases}
\end{displaymath}
Then we set
\begin{displaymath}
R^t(\overline{D}) = \bigoplus_{m \in \bN}
R_m^t(\overline{D}), 
\end{displaymath}
that is a graded subalgebra of the algebra of sections $R(D)$. Its
volume is the quantity
\begin{displaymath}
\vol(R^t(\overline{D})) = \limsup_{m\rightarrow \infty} \frac{\dim_K R_m^t(\overline{D})}{m^d/d!}.
\end{displaymath}

The next theorem is due to Chen, and allows to express the arithmetic
volume of a big adelic $\mathbb{R}$-divisor in terms of the volumes of
these graded subalgebras.

\begin{theorem}
  \label{thm:ChenIntegral}
  If $\overline{D}$ is big then 
\begin{math}
\displaystyle{\widehat{\vol}(\overline{D}) = (d+1) \int_0^{\infty} \vol(R^t(\overline{D})) \ dt.}
\end{math}
\end{theorem}

\begin{proof}
  When $\overline{D}$ is an adelic divisor, this formula is given by
  \cite[Theorem 3.8]{Chen:Fujita} (see also \cite[Formula (5.2)]{Chen:diffvol}). It is also a consequence of the
  results of Boucksom and Chen on arithmetic Okounkov bodies
  \cite[Theorems 1.11 and 2.8]{BC}, whose proof can be carried out for
  adelic $\bR$-divisors using the extension of this theory in
  \cite[Section 7.3]{Moriwaki:MAMS}.
\end{proof}

Arithmetic $(\chi$-)volumes coincide with arithmetic intersection
numbers under suitable positivity conditions.

\begin{theorem}[{\cite[Theorem
    5.3.2]{Moriwaki:MAMS}}]\label{thm:arithmHodgeMoriwaki} If
  $\overline{D}$ is
  semipositive then
  \begin{displaymath}
   \widehat{\vol}_{\chi}(\overline{D}) = (\overline{D}^{d+1}). 
  \end{displaymath}
  If moreover $\overline{D}$ is nef, then
  $\widehat{\vol}(\overline{D})=\widehat{\vol}_{\chi}(\overline{D}) =
  (\overline{D}^{d+1})$.
\end{theorem}

The existence of arithmetic Fujita approximations for adelic divisors
was established independently by Yuan \cite{Yuan:volumes} and  Chen
\cite{Chen:Fujita}. We will use the following extension to adelic
$\mathbb{R}$-divisors from \cite[Theorem 5.1.6]{Moriwaki:MAMS}.

\begin{theorem}\label{thm:arithmeticFujita} Assume that
  $\overline{D}$ is big. For each $\varepsilon > 0$ there exists a normal
  modification $\phi\colon X'\rightarrow X$ and a nef adelic
    $\mathbb{R}$-divisor $\overline{P}$ on $X'$ such that
  $\phi^*\overline{D} - \overline{P}$ is pseudo-effective and 
\begin{displaymath}
 (\overline{P}^{d+1}) = \widehat{\vol}(\overline{P}) \ge \widehat{\vol}(\overline{D}) - \varepsilon.
\end{displaymath}
\end{theorem}

Yuan's arithmetic analogue of Siu's inequality is a key ingredient for
equidistribution results in Arakelov geometry \cite[Theorem
2.2]{Yuan:big}.  We will use the next extension to adelic
$\mathbb{R}$-divisors, which follows from the original one by the
continuity of the $\chi$-volume function as in \cite[Proof of Theorem
7.5]{ChenMoriwaki:DysDirichlet}.
  
\begin{theorem}
  \label{thm:Yuan}
  Let $\overline{P}_1, \overline{P}_2$ be nef adelic $\bR$-divisors on
  $X$. Then
\begin{displaymath}
\widehat{\vol}_{\chi}(\overline{P}_1 - \overline{P}_2) \ge (\overline{P}_1^{d+1}) - (d+1)\, (\overline{P}_1^{d} \cdot \overline{P}_2).
\end{displaymath}
\end{theorem}

We now turn to other positivity aspects of adelic
  $\mathbb{R}$-divisors. First note that if there is
  $s \in \widehat{\Gamma}(X,\overline{D}) \setminus \{0\}$ then for
  every $\overline{B} \in \widehat{\Div}(X)_{\bR}$ and
  $m \in \bN_{>0}$ we have an inclusion
 \begin{displaymath}
  \widehat{\Gamma}(X,m\overline{B}) \longhookrightarrow
 \widehat{\Gamma}(X,m(\overline{D} +\overline{B})) 
\end{displaymath}
given by the multiplication by $s^{\otimes m}$. \change{Therefore
$\widehat{\vol}(\overline{D} + \overline{B}) \ge
\widehat{\vol}(\overline{B})$. It follows that if
$R^0(\overline{D}) \ne \{0\}$ then $\overline{D}$ is
pseudo-effective.} In particular, an effective adelic $\bR$-divisor is
pseudo-effective.  \change{We also note that the inequality
\begin{equation*}
\widehat{\vol}(\overline{D} + \overline{B}) \ge
\widehat{\vol}(\overline{B})
\end{equation*}
holds for every $\overline{B} \in \widehat{\Div}(X)_{\bR}$ whenever
$\overline{D}$ is pseudo-effective. This follows from the effective
case by the continuity of the arithmetic volume function. In
particular the sum of a big adelic $\bR$-divisor with a
pseudo-effective one is big.}  Moreover a nef adelic $\bR$-divisor is
pseudo-effective \cite[Proposition 4.4.2(2)]{Moriwaki:MAMS}.

\change{We next give several characterizations for the bigness of an
  adelic $\bR$-divisor. In particular they imply that if
  $\overline{D}$ is big then its geometric $\bR$-divisor $D$ is big.

\begin{lemma}\label{lem:charactbigness} The following conditions are equivalent:
\begin{enumerate}[leftmargin=*]
\item\label{item:big1} $\overline{D}$ is big,
\item\label{item:big2} there exists $\varepsilon > 0$ such that $\vol(R^{\varepsilon}(\overline{D})) > 0$,
\item\label{item:big3} $D$ is big and there exists $\varepsilon > 0$ such that $R^{\varepsilon}(\overline{D}) \ne \{0\}$,
\item\label{item:big4} $D$ is big and there exists $\varepsilon > 0$ such that $\overline{D} - \varepsilon [\infty]$ is pseudo-effective.
\end{enumerate}
\end{lemma}

\begin{proof}
  By Theorem \ref{thm:ChenIntegral} we have that \eqref{item:big1}
  implies \eqref{item:big2}. 
  If \eqref{item:big2} holds then
  $\vol(D) \ge \vol(R^{\varepsilon}(\overline{D})) > 0$, and in
  particular we have \eqref{item:big3}. Now assume that
  \eqref{item:big3} holds. Then
  $R^0(\overline{D} - \varepsilon[\infty]) =
  R^{\varepsilon}(\overline{D}) \ne \{0\}$ and so \eqref{item:big4}
  holds by the discussion preceding the lemma. The fact that
  \eqref{item:big4} implies \eqref{item:big1} is given by
  \cite[Proposition 4.4.2(3)]{Moriwaki:MAMS}.
\end{proof}

 When $D$ is big, a
 sufficiently large shift of~$\overline{D}$~is~big.

\begin{lemma} \label{lem:6} If $D$ is big then  $\overline{D}+ c\,[\infty]$ is big for any sufficiently large $c > 0$.
\end{lemma}

\begin{proof} 
  If $D$ is big then there exists a nonzero global section
  $s \in \Gamma(X,mD)$ with $m>0$.  After multiplying $s$ by a
  suitable nonzero scalar we can assume that
  $\|s\|_{\overline{D},v,\sup} \le 1$ for every non-Archimedean place
  $v$. Taking any $c >0$ such that
  $ \log\|s\|_{\overline{D},v,\sup}< mc$ for every Archimedean place
  $v$ we have that $s \in R^\varepsilon(\overline{D} + c[\infty])$ for
  some $\varepsilon > 0$, and therefore $\overline{D} + c\,[\infty]$
  is big by Lemma~\ref{lem:charactbigness}.
\end{proof}

Pseudo-effectiveness is stable under taking limits.

\begin{lemma}\label{lem:pseffclosed} Let
  $\overline{E} \in \widehat{\Div}(X)_{\bR}$ and $\varepsilon >0$ such
  that $\overline{D} + \lambda \overline{E}$ is pseudo-effective for
  every $\lambda \in \mathbb{R}$ with $0<\lambda\le \varepsilon$. Then
  $\overline{D}$ is pseudo-effective.
\end{lemma}

\begin{proof} Let $\overline{B} $ be a big adelic $\mathbb{R}$-divisor
  on $X$. By the continuity of the arithmetic volume there exists
  $\lambda \in (0,\varepsilon)$ such that
  $\overline{B} - \lambda \overline{E}$ is big. Hence
\begin{displaymath}
\overline{D} + \overline{B} = \overline{D} + \lambda\overline{E} + \overline{B} - \lambda\overline{E}
\end{displaymath}
is big, being the sum of a big adelic $\bR$-divisor and a
pseudo-effective one. Therefore $\overline{D}$ is pseudo-effective, as
stated.
\end{proof}}

We need the following version of the well-known fact that adelic
$\bR$-divisors can be approximated by DSP adelic $\bR$-divisors.

\begin{lemma}
  \label{lemma:approachDSP}
  Let $\overline{E} \in \widehat{\Div}(X)_{\bR}$. Then for each
  $\varepsilon > 0$ there exists
  $\overline{E}' \in \widehat{\mathrm{DSP}}(X)_{\bR}$ over $E$ such
  that $\overline{E} - \overline{E}'$ is effective and
  $\overline{E}' + \varepsilon[\infty] - \overline{E}$ is
  pseudo-effective.
\end{lemma}

\begin{proof} By \cite[Theorem 4.1.3]{Moriwaki:MAMS} there is a finite
  set $S \subset \mathfrak{M}_K$ containing
  $\mathfrak{M}_{K}^{\infty}$ such that for each $\varepsilon' >0$
  there is a model $(\mathcal{X},\mathcal{E})$ of $(X,E)$ over
  $\Spec(\cO_K)$ such that setting
  $g_{\overline{E}',v} = g_{\mathcal{E},v}$ and
  $\varphi_v = g_{\overline{E},v} - g_{\overline{E}',v}$ for every
  non-Archimedean place $v $ we have $\varphi_v = 0$ for $v \notin S$
  and $0 \le \varphi_v \le \varepsilon'$ for
  $v \in S \setminus \mathfrak{M}_{K}^{\infty}$. Moreover,
  \change{using the density of smooth functions among continuous ones
    one can show that} for every Archimedean place~$v$ there is a
  smooth $v$-adic Green function $g_{\overline{E}',v}$ on $E$ such
  that $\varphi_v = g_{\overline{E},v} - g_{\overline{E}',v}$ also
  satisfies $0 \le \varphi_v \le \varepsilon'$. 

  By construction,
  $\overline{E}' = (E,(g_{\overline{E}',v})_{v\in \mathfrak{M}_K})$ is
  a DSP adelic $\mathbb{R}$-divisor over~$E$ and
  $\overline{E} - \overline{E}' = (0, (\varphi_v)_{v \in
    \mathfrak{M}_K})$ is effective.  Moreover, taking $\varepsilon'$
  sufficiently small with respect to $\varepsilon$ we get from
  \cite[Lemma 1.11]{BMPS:positivity} that a sufficiently high multiple
  of
\begin{displaymath}
\overline{E}' + \varepsilon[\infty] - \overline{E} =  \varepsilon[\infty] - (0,(\varphi_v)_v)
\end{displaymath}
has a nonzero small global section. Thus
$\overline{E}' + \varepsilon[\infty] - \overline{E}$ is
pseudo-effective.
\end{proof}

\begin{lemma}
  \label{lemma:intersectionpseff}
  Let $\overline{D}_1, \dots, \overline{D}_{d}$ be nef adelic
  $\bR$-divisors on $X$. For every
  $\overline{D}_{d+1},\overline{D'}_{\hspace{-0.2em}d+1} \in
  \widehat{\Div}(X)_{\bR}$ with
  $\overline{D}_{d+1}-\overline{D'}_{\hspace{-0.2em}d+1} $
  pseudo-effective we have
\begin{displaymath}
  (\overline{D}_1 \cdots \overline{D}_{d} \cdot \overline{D'}_{\hspace{-0.2em}d+1} ) \le
  (\overline{D}_1 \cdots \overline{D}_{d} \cdot \overline{D}_{d+1} ).
\end{displaymath}
\end{lemma}

\begin{proof}
  Setting
  $\overline{E}=\overline{D}_{d+1}-\overline{D'}_{\hspace{-0.2em}d+1}$,
  the inequality is equivalent to the fact that
  \begin{displaymath}
   (\overline{D}_1 \cdots \overline{D}_{d} \cdot \overline{E} )\ge 0.  
  \end{displaymath}
  The case when $\overline{E}$ is DSP is given by \cite[Proposition
  4.5.4(3)]{Moriwaki:MAMS}. For the general case, let $\varepsilon >0$
  and take $\overline{E}'$ as in Lemma \ref{lemma:approachDSP}.  Then
  $\overline{E}-\overline{E}'$ is effective and
  $\overline{E}'+\varepsilon[\infty] $ is both DSP and
  pseudo-effective, and so from the
  formulae~\eqref{eq:12}~and~\eqref{eq:arithmeticvsgeometricintersection}
  we get
  \begin{displaymath}
    (\overline{D}_1 \cdots \overline{D}_{d} \cdot \overline{E} )\ge  
    (\overline{D}_1 \cdots \overline{D}_{d} \cdot \overline{E}' )\ge
    (\overline{D}_1 \cdots \overline{D}_{d} \cdot (-\varepsilon[\infty])) =-\varepsilon \, (D_{1}\cdots D_{d}).
      \end{displaymath}
      The result follows by letting $\varepsilon\to 0$.
\end{proof}
 
We will also need the following auxiliary result.

\begin{lemma}
  \label{lem:1}
  Let $\overline{E} \in \widehat{\mathrm{DSP}}(X)_{\bR}$. Then there
  exist big and nef
  $\overline{N}_1,\overline{N}_2 \in \widehat{\Div}(X)_{\bR}$ such
  that $\overline{E} = \overline{N}_1 - \overline{N}_2$.
\end{lemma}

\begin{proof}
  By definition, there are semipositive
  $\overline{D}_1,\overline{D}_2 \in \widehat{\Div}(X)_{\bR}$ such
  that $\overline{E} = \overline{D}_1 - \overline{D}_2$. Take any
  semipositive $\overline{A} \in \Div(X)_{\bR}$ with $A$ ample.  Then
  both $D_{1}$ and $D_{2}$ are nef, and so both $A + D_1$ and
  $A+D_{2}$ are ample. It follows from inequality
  \eqref{eq:heightlowerbound} that the height functions
  \begin{displaymath}
h_{\overline{D}_1 + \overline{A}},h_{\overline{D}_2 + \overline{A}} \colon X(\overline{K})
  \longrightarrow \bR    
  \end{displaymath}
  are bounded from below by a real number. Taking a sufficiently large
  $c \in \bR$ and letting
  $\overline{N}_i = \overline{D}_i + \overline{A} + c\, [\infty]$,
  $i =1,2$, we have that $\overline{N}_i$ is nef \change{by formula
  \eqref{eq:twistheight} below,} and it is big by Lemma \ref{lem:6}.
  Since $\overline{E} = \overline{N}_1 - \overline{N}_2$, the lemma is
  proven.
\end{proof}

\subsection{Absolute and essential minima}
\label{subsec:minima}
Let $\overline{D} \in \widehat{\Div}(X)_{\bR}$. Its  \emph{absolute minimum} is
\begin{displaymath}
\mu^{\abs}(\overline{D}) = \inf_{x \in X(\overline{K})} h_{\overline{D}}(x) .
\end{displaymath}
Clearly $\mu^{\abs}(\phi^*\overline{D}) = \mu^{\abs}(\overline{D})$
for any surjective morphism $\phi\colon X'\rightarrow X$ with $X'$
projective and normal. We also have 
$\mu^{\abs}(\overline{D}) >-\infty$ whenever $D$ is semiample.

By definition, $\overline{D}$ is nef if and only if it is semipositive
and $\mu^{\abs}(\overline{D}) \ge 0$.  Note that for $t\in \mathbb{R}$
we have
\begin{equation}\label{eq:twistheight}
 h_{\overline{D}(t)}(x) =
 h_{\overline{D}}(x) -t \quad \text{ for } x\in X(\overline{K}),
\end{equation}
and so when $\overline{D}$ is semipositive we have
\begin{equation}
  \label{eq:28}
\mu^{\mathrm{abs}}(\overline{D})  = \sup \{ t \in \bR \ | \ \overline{D}(t) \text{ is nef}\, \}.
\end{equation}

The following lower bound for the height of effective cycles is a
consequence of Zhang's theorem on minima \cite[Theorem
5.2]{Zhang:positive}.
  
\begin{lemma}\label{lemma:ineqheightnef}
 Let $Z$ be an effective cycle of $X$ of dimension $r$. If $\overline{D}$ is semipositive then
  \begin{displaymath}
   h_{\overline{D}}(Z) \ge (r+1)\, \mu^{\abs}(\overline{D}) \, (D^{r}\cdot Z). 
  \end{displaymath}
 In particular, if $\overline{D}$ is nef then $h_{\overline{D}}(Z) \ge 0$.
\end{lemma}  
  
\begin{proof} By linearity we may assume that $Z$ is a
  subvariety. When $D$ is ample, the result follows then from
  \cite[Corollary 2.9]{Ba:Okounkov}.  For the general case, choose
  $\overline{A} \in \widehat{\Div}(X)_{\bR}$ semipositive with $A$
  ample. Then for any $\varepsilon > 0$ we have that $D + \varepsilon A$ is ample
  and so
\begin{align*} 
h_{\overline{D} + \varepsilon \overline{A}}(Z)& \ge (r+1)\, \mu^{\abs}(\overline{D} + \varepsilon \overline{A}) \, ((D+ \varepsilon A)^{r}\cdot Z)\\
&  \ge (r+1)\, (\mu^{\abs}(\overline{D}) + \varepsilon \mu^{\abs}( \overline{A})) \, ((D+ \varepsilon A)^{r}\cdot Z).
\end{align*}
We conclude by letting $\varepsilon \to 0$ and using multilinearity.
\end{proof}  
  
The \emph{essential minimum} of $\overline{D}$ is the quantity
\begin{displaymath}
\mu^{\mathrm{ess}}(\overline{D}) = \sup_{Y \varsubsetneq X} \inf_{x \in (X \setminus Y)(\overline{K})} h_{\overline{D}}(x),
\end{displaymath}
where the supremum is over the closed subsets $Y \varsubsetneq X$. We
have $\mu^{\ess}(\phi^*\overline{D})=\mu^{\ess}(\overline{D})$ for any dominant and generically finite morphism $\phi \colon X' \rightarrow X$ with $X'$ projective and normal \cite[Proposition 3.4]{BPS:smthf}. We also have
$\mu^{\mathrm{ess}}(\overline{D}) < \infty$, and
$\mu^{\ess}(\overline{D}) \in \bR$ whenever $R(D) \ne \{0\}$
\cite[Proposition 2.6]{BC}.

\begin{lemma}
  \label{lemma:propertiesessmin}
  Let
  $\overline{D}_1, \overline{D}_2 \in \widehat{\Div}(X)_{\mathbb{R}}$.
  Then
\begin{enumerate}[leftmargin=*]
\item\label{item:essminsuperadd} 
  $\mu^{\mathrm{ess}}(\overline{D}_1 + \overline{D}_2) \ge
  \mu^{\mathrm{ess}}(\overline{D}_1) +
  \mu^{\mathrm{ess}}(\overline{D}_2)$,
\item\label{item:essmincont} if $D_1$ is big then
  $\lim_{\lambda \to 0} \mu^{\ess}(\overline{D}_1 + \lambda
  \overline{D}_2) = \mu^{\ess}(\overline{D}_1)$,
\item\label{item:essminlowerbound} for every $t \in \bR$ such that
  $R^{t}(\overline{D}_1)\ne \{0\}$ we have
  $\mu^{\ess}(\overline{D}_1) \ge t$,
  \item\label{item:essminincreases} if $D_1$ is big and
    $\overline{D}_1 - \overline{D}_2$ is pseudo-effective then
    $\mu^{\mathrm{ess}}(\overline{D}_1) \ge
    \mu^{\mathrm{ess}}(\overline{D}_2)$,
    \item\label{item:essmineffdiff} \change{If $\overline{D}_1-\overline{D}_2$ is effective then $\mu^{\mathrm{ess}}(\overline{D}_1) \ge
    \mu^{\mathrm{ess}}(\overline{D}_2)$.}
\end{enumerate}
\end{lemma}

\begin{proof}
  The first two points can be found in \cite[Lemma 3.15]{Ba:Essmin},
  whereas the third is a direct consequence of  inequality
  \eqref{eq:heightlowerbound}.

  For \eqref{item:essminincreases} choose
  $\overline{B} \in \widehat{\Div}(X)_{\bR}$ big and
  $\varepsilon > 0$. Then
  $\overline{D}_1 + \varepsilon \overline{B} - \overline{D}_2 $ is big
  and therefore
  $R^0(\overline{D}_1 + \varepsilon \overline{B} - \overline{D}_2) \ne
  \{0\}$. We  deduce from  \eqref{item:essminsuperadd} and
  \eqref{item:essminlowerbound}
\begin{displaymath}
\mu^{\ess}(\overline{D}_1 + \varepsilon \overline{B}) \ge \mu^{\ess}(\overline{D}_1 + \varepsilon \overline{B} - \overline{D}_2) +   \mu^{\ess}(\overline{D}_2) \ge \mu^{\ess}(\overline{D}_2)
\end{displaymath}
and  conclude by letting $\varepsilon\to 0$ and using
\eqref{item:essmincont}. \change{In the case where $\overline{D}_1 - \overline{D}_2$ is effective, we can apply the same argument with $\varepsilon = 0$ to get \eqref{item:essmineffdiff}.}
\end{proof}

The next result characterizes the essential mimimum of an adelic
$\mathbb{R}$-divisor with big geometric $\mathbb{R}$-divisor. It was
obtained by the first author \cite[Theorem~1.1]{Ba:Essmin} as a
consequence of a theorem of Ikoma \cite{Ikoma:IMRN}, assuming that
$\overline{D}$ is semipositive. This condition was later removed
thanks to the independent works of Qu and Yin \cite{QuYin} and
Szachniewicz \cite{Szachniewicz}. 

\begin{theorem}
  \label{thm:Essmin}
We have
\begin{equation}\label{eq:10}
\mu^{\ess}(\overline{D}) \le \sup\{ t \in \bR \, | \ \overline{D}(t) \text{ is pseudo-effective}\},
\end{equation}  
with equality if $D$ is big. In that case we also have
\begin{displaymath}
  \mu^{\mathrm{ess}}(\overline{D})   = \sup \{ t \in \bR \, | \ \overline{D}(t) \text{ is big}\}= \sup \{ t \in \bR \, | \ R^t(\overline{D})
  \ne \{0\}\}.
\end{displaymath}
\end{theorem}

\begin{proof}\change{When all non-Archimedean Green functions of
    $\overline{D}$ are given by a single model, the first statement is
    a consequence of \cite[Lemma 3.3.5 and
    Theorem~3.3.7]{Szachniewicz}. The general case follows from this
    one by combining the approximation result of \cite[Theorem
    4.1.3]{Moriwaki:MAMS} with Lemma
    \ref{lemma:propertiesessmin}\eqref{item:essmineffdiff}.  The
    second statement follows from the first using Lemma
    \ref{lem:charactbigness}.}
\end{proof}

\change{
  \begin{remark}\label{rema:pseffmax}
By Lemma \ref{lem:pseffclosed} and Theorem \ref{thm:Essmin} we
    have that $\overline{D} - \mu^{\ess}(\overline{D})[\infty]$ is
    pseudo-effective whenever $\mu^{\ess}(\overline{D}) >
    -\infty$. This holds for example when $R(D) \ne \{0\}$.
    If moreover $D$ is big then
    \begin{math}
\mu^{\ess}(\overline{D})= \max\{ t \in \bR \, | \ \overline{D}(t) \text{ is pseudo-effective}\}.      
    \end{math}
\end{remark}
}

We  will also use  the next variants of Zhang's lower bound for the
essential minimum \cite[Theorem~5.2]{Zhang:positive}.

\begin{theorem}
  \label{thm:Zhang}
  If $D$ is big then
\begin{displaymath}
 \mu^{\mathrm{ess}}(\overline{D}) \ge \frac{\widehat{\vol}_{\chi}(\overline{D})}{(d+1) \vol(D)}.
\end{displaymath}
Moreover, if $\overline{D}$ is big then 
\begin{displaymath}
\mu^{\ess}(\overline{D}) \ge \frac{\widehat{\vol}(\overline{D})}{(d+1) \vol(R^0(\overline{D}))} \ge \frac{\widehat{\vol}(\overline{D})}{(d+1) \vol(D)}.
\end{displaymath}
\end{theorem}

\begin{proof} \change{When $X$ is geometrically irreducible over $K$, the first inequality
  is \cite[Theorem 7.2(1)]{Ba:Essmin}, noting that our normalization
  for $\widehat{\vol}_{\chi}(\overline{D})$ differs from that in
  \emph{loc. cit.}  by a factor $[K:\bQ]$. The inequality holds in
  general because both sides remain the same when applying the Stein
  factorization as in Remark \ref{rem:7}.

 The second statement follows} 
  from Theorem \ref{thm:ChenIntegral} and Lemma
  \ref{lemma:propertiesessmin}\eqref{item:essminlowerbound}, which
  imply
  \begin{displaymath}
    \widehat{\vol}(\overline{D})= (d+1) \int_{0}^{\mu^{\ess}(\overline{D})} \vol(R^t(\overline{D})) dt  \le  
    (d+1)\vol(R^0(\overline{D})) \, \mu^{\mathrm{ess}}(\overline{D}).
\end{displaymath}
\end{proof}

Under mild positivity assumptions, the condition that Zhang's
lower bound is an equality is equivalent to the fact that the essential
 and the absolute minima coincide. 
\begin{theorem}
  \label{thm:criterionextremal}
  Assume that $\overline{D}$ is semipositive and that $D$ is
  big. Then
\begin{displaymath}
 \mu^{\mathrm{ess}}(\overline{D}) = \frac{(\overline{D}^{d+1})}{(d+1) \, (D^d)} \quad \text{ if and only if } \quad \mu^{\ess}(\overline{D}) = \mu^{\abs}(\overline{D}).
\end{displaymath}
\end{theorem}

\begin{proof}
  \change{This is {\cite[Theorem 6.6]{Ba:Okounkov}} when $X$ is geometrically
  irreducible over $K$. The general case follows from this one using
  the Stein factorization as in the proof of Theorem \ref{thm:Zhang}.}
\end{proof}

\subsection{Positive linear functionals on ${G_v}$-invariant  functions}\label{subsec:positivelinearfunctionals}

Let $v \in \mathfrak{M}_K$ and recall that $ C(X_v^{\an})^{G_v}$
denotes the space of $G_{v}$-invariant continuous real-valued
functions on $X_{v}^{\an}$.  By the Riesz representation theorem, any
positive linear functional on $C(X_v^{\an})$ corresponds to a measure
on $X_{v}^{\an}$. We need the following variant of this result for
positive linear functionals on $C(X_v^{\an})^{G_v}$.

\begin{lemma}
  \label{lem:extendlinearfunctionals}
  Let $\Lambda \colon C(X_v^{\an})^{G_v} \rightarrow \bR$ be a
  positive linear functional. 
\begin{enumerate}[leftmargin=*]
\item\label{item:extendlinearfunctional} There is a unique
  $G_{v}$-invariant measure $\nu$ on $X_{v}^{\an}$ such that
  \begin{equation}
    \label{eq:11}
   \Lambda(\varphi)=  \int_{X_v^{\an}} \varphi \, d\nu \quad \text{ for all } \varphi \in  C(X_v^{\an})^{G_v}.
  \end{equation} 
\item\label{item:weakconvergenceGalinv} If $(\nu_{n})_{n}$ is a
  sequence of $G_v$-invariant measures on $X_v^{\an}$ such that
\begin{displaymath}
\lim_{n \to \infty} \int_{X_v^{\an}} \varphi \, d\nu_{n} = \Lambda(\varphi) \quad \text{ for all } \varphi \in  C(X_v^{\an})^{G_v}, 
\end{displaymath}
then $\lim_{n\to\infty}\nu_{n}=\nu$. 
\end{enumerate}
\end{lemma}

\begin{proof} For \eqref{item:extendlinearfunctional}, let
  $V \subset C(X_v^{\an})$ be the subspace of continuous real-valued
  functions on $X_v^{\an}$ that are
  $\Gal(\overline{K}_v/K'_v)$-invariant for some finite extension
  $K'/K$. By \cite[``Equivalence'' at page 638]{Yuan:big} (see also
  \cite[Proposition 2.11 and Theorem 2.13]{GualdiMartinez}) this
  subspace is dense in $C(X_v^{\an})$ with respect to the supremum
  norm.

  For each $\varphi \in V$ let $K'/K$ be a finite extension such that
  $\varphi$ is $\Gal(\overline{K}_v/K'_v)$-invariant and set
\begin{displaymath}
  \widetilde{\varphi} = \frac{1}{\# \Gal(K'_v/K_v)} \sum_{\sigma \in \Gal(K'_v/K_v)} \sigma^*\varphi .
\end{displaymath}
This is an element of $C(X_{v}^{\an})^{G_{v}}$ that does not depend on
the choice of this finite extension.  We extend $\Lambda$ to a
$G_v$-invariant positive linear functional
$\widetilde{\Lambda}\colon V \to \mathbb{R}$ by setting
$\widetilde{\Lambda}(\varphi) = \Lambda(\widetilde{\varphi})$, and
then by continuity to a $G_v$-invariant positive linear functional on
$C(X_v^{\an})$. By the Riesz representation theorem there is a
$G_{v}$-invariant measure~$\nu$ on $X_{v}^{\an}$ that satisfies
\eqref{eq:11}.

If $\nu'$ is another $G_{v}$-invariant measure on $X_{v}^{\an}$
satisfying this equality on $C(X_{v}^{\an})^{G_{v}}$ then it also does
on $V$. Therefore $\nu'=\nu$ by the density of this subspace.

To prove  \eqref{item:weakconvergenceGalinv}, for each $\varphi \in V$ we have 
\begin{displaymath}
  \int_{X_v^{\an}} \varphi \, d\nu_{n} =  \int_{X_v^{\an}} \widetilde{\varphi} \, d\nu_{n} \  \xlongrightarrow[n \to \infty]{} \
\int_{X_v^{\an}} \widetilde{\varphi} \, d\nu =  \int_{X_v^{\an}} \varphi \, d\nu
\end{displaymath}
by the $G_v$-invariance of the measures $\nu_{n}$,
$n\in \mathbb{N}$, and $\nu$. The statement then follows from the
density of $V$.
\end{proof} 
 
Finally we recall a basic instance of Lemma
\ref{lem:extendlinearfunctionals}\eqref{item:extendlinearfunctional}. For
$\varphi \in C(X_v^{\an})^{G_v}$ we have that $\varphi$ is a $v$-adic
Green function over the zero divisor of $X$, and we define the adelic
divisor $\overline{0}^{\varphi}$ by equipping this divisor with $\varphi$
at $v$ and the zero function at every other place. The assignment
\begin{equation}
  \label{eq:27}
\varphi \longmapsto \overline{0}^\varphi 
\end{equation}
gives an inclusion
$C(X_v^{\an})^{G_v} \hookrightarrow \widehat{\Div}(X)_{\bR}$.  Now if
$\overline{D}$ is a semipositive adelic $\mathbb{R}$-divisor on $X$,
for each $\varphi\in C(X_{v}^{\an})^{G_{v}}$ we have
\begin{equation}\label{eq:intersectionvarphi}
(\overline{D}^d \cdot \overline{0}^{\varphi}) = n_v \int_{X_v^{\an}} \varphi \, c_1(\overline{D}_v)^{\wedge d}
\end{equation} 
by the formula in \eqref{eq:12}. Hence the positive linear functional
$C(X_v^{\an})^{G_v} \to \mathbb{R}$ defined by
$\varphi\mapsto (\overline{D}^d \cdot \overline{0}^{\varphi}) $ is
represented by the $G_{v}$-invariant measure
$ n_v \, c_1(\overline{D}_v)^{\wedge d}$. 

\section{Fujita approximations and positive intersection numbers}\label{sec:diffvol}

In this section we revisit Chen's theorem on the differentiability of
the arithmetic volume \cite{Chen:diffvol} to express the derivative of
this function at a big adelic $\mathbb{R}$-divisor in terms of
arithmetic Fujita approximations. This will allow us to characterize
in terms of such approximations the arithmetic positive intersection
numbers that are relevant to our results.
 
\subsection{Differentiability of concave functions}\label{sec:diff-conc-funct}

We first recall the notion of differentiability of functions on real
vector spaces and its relation to concavity.  Our ambient \change{space} will be a
real vector space $V$ endowed with the topology defined by declaring
that a subset $U\subset V$ is open whenever its restriction to any
finite-dimensional affine subspace $H\subset V$ is open with respect
to the Euclidean topology of $H$.

\begin{definition}
  \label{def:diff}  
  Let $\Phi\colon U \to \mathbb{R}$ be a function on an open subset
  $U\subset V$. For a point $x\in U$ and a linear subspace
  $W\subset V$ we say that $\Phi$ is \emph{differentiable at $x$ along
    $W$} if for all $z\in W$ the one-sided derivative
  \begin{displaymath}
    \partial_{z} \Phi(x)=\lim_{\lambda \to 0^{+}} \frac{\Phi(x+\lambda\,z)-\Phi(x)}{\lambda}
  \end{displaymath}
  exists in $\bR$ and the map
  $ z\in W\to \partial_{z} \Phi(x) \in \bR$ is linear. When $W=V$, we
  simply say that $\Phi$ is differentiable at $x$.
\end{definition}

A real-valued function $\Phi$ on a subset $ U\subset V$ is
\emph{concave} when $U$ is convex and
\begin{displaymath}
  \Phi(\lambda x+(1-\lambda)y) \ge \lambda\Phi(x)+(1-\lambda)\Phi(y) \quad \text{ for all }  x,y\in U \text{ and }  0\le \lambda\le 1.
\end{displaymath}

\begin{lemma}\label{lemma:diffconcave} 
  Let $\Phi \colon U\to \mathbb{R}$ be a concave function on an open
  convex subset of $ V$ and let $x\in U$. Then
  \begin{enumerate}[leftmargin=*]
  \item\label{item:diffconcaveineq} for all $z\in V$ the one-sided
    derivative $\partial_{z} \Phi(x)$ exists in $\bR$ and 
     $\partial_{z} \Phi(x) \le -\partial_{-z} \Phi(x)$,
  \item\label{item:diffconcavealldirections} the function $\Phi$ is
    differentiable at $x$ along a linear subspace $W\subset V$ if and only if
    $\partial_{z} \Phi(x) = -\partial_{-z} \Phi(x)$ for all $z\in W$.
\end{enumerate}
\end{lemma}

\begin{proof}
  For \eqref{item:diffconcaveineq} let $a > 0$ be a real number
  sufficiently small so that $x+\lambda z \in U$ for all
  $-a< \lambda < a$. Let $\iota\colon (-a,a)\to V$ be the inclusion
  map defined as $\iota(\lambda)=x+\lambda z$, so that the function
  $\iota^{*}\Phi \colon (-a,a)\to \mathbb{R}$ is concave. Then the
  functions~$r_{-},r_{+}\colon (0,a)\to \mathbb{R}$ respectively
  defined as
  \begin{displaymath}
    r_{-}(\lambda)= \frac{\iota^{*}\Phi(-\lambda)-\iota^{*}\Phi(0)}{-\lambda} \and
      r_{+}(\lambda)= \frac{\iota^{*}\Phi(\lambda)-\iota^{*}\Phi(0)}{\lambda} 
  \end{displaymath}
  verify that $r_{-}$ is non-decreasing, $r_{+}$ is non-increasing,
  and $r_{-}(\lambda)\ge r_{+}(\lambda)$ for every $\lambda$.  Hence
  both converge when $\lambda\to 0^{+}$ and 
  \begin{displaymath}
-       \partial_{-z} \Phi(x)=\lim_{\lambda \to 0^{+}} r_{-}(\lambda)\ge \lim_{\lambda \to 0^{+}} r_{+}(\lambda) 
       =\partial_{z} \Phi(x).
  \end{displaymath}

  For \eqref{item:diffconcavealldirections} it is clear that the
  differentiability of $\Phi$ at $x$ along $W$ implies that
  $\partial_{z} \Phi(x) = -\partial_{-z} \Phi(x)$ for all $z\in
  W$. Conversely, let $z_{1},z_{2}\in W$ and suppose that this
  condition holds for these two vectors. For $\lambda>0$ small we
  deduce from the concavity of $\Phi$ that
  \begin{displaymath}
    \frac{\Phi(x+\lambda(z_{1}+z_{2}))-\Phi(x)}{\lambda} \ge 
    \frac{\Phi(x+2\lambda z_{1})-\Phi(x)}{2\lambda} +\frac{\Phi(x+2\lambda z_{2})-\Phi(x)}{2\lambda}.
  \end{displaymath}
  Letting $\lambda\to 0^{+}$ we get
  \begin{math}
   \partial_{z_{1}+z_{2}}\Phi(x)\ge
  \partial_{z_{1}}\Phi(x)+\partial_{z_{2}}\Phi(x). 
  \end{math}
  Applying this inequality to $-z_{1},-z_{2}$ together with
  \eqref{item:diffconcaveineq} we obtain
  \begin{displaymath}
    -\partial_{-z_{1}}\Phi(x)-\partial_{-z_{2}}\Phi(x)\ge -\partial_{-z_{1}-z_{2}}\Phi(x)\ge    \partial_{z_{1}+z_{2}}\Phi(x)\ge
  \partial_{z_{1}}\Phi(x)+\partial_{z_{2}}\Phi(x).   
  \end{displaymath}
  By assumption the extremes in these inequalities coincide, and so
  $ \partial_{z_{1}+z_{2}}\Phi(x)=
  \partial_{z_{1}}\Phi(x)+\partial_{z_{2}}\Phi(x)$. In addition, the
  one-sided derivative is positive homogeneous of degree $1$ and so
  linear, as stated.
\end{proof}

\subsection{The differentiability of the arithmetic volume}\label{subsec:diffarithmvol}

Let $\overline{D} $ be a big adelic $\bR$-divisor on $X$.

\begin{definition}
  \label{def:9}
A \emph{Fujita approximation sequence} of $\overline{D}$ is a sequence 
$(\phi_n, \overline{P}_n)_{n}$ satisfying the following conditions:
\begin{enumerate}[leftmargin=*]
\item \label{item:7} $\phi_{n} \colon X_{n} \rightarrow X$ is a normal modification,
\item \label{item:11} $\overline{P}_{n}$ is a nef adelic $\bR$-divisor on $X_n$ with
  $\phi^*\overline{D} - \overline{P}_{n}$ pseudo-effective,
\item \label{item:12} $\displaystyle{\lim_{n \to \infty} (\overline{P}_n^{d+1}) =
  \widehat{\vol}(\overline{D})}$.
\end{enumerate}
\end{definition} 

The existence of Fujita approximation sequences is warranted by
Theorem \ref{thm:arithmeticFujita}. The next result is a variant of
Chen's differentiability theorem~\cite{Chen:diffvol} and shows that
the derivative of the arithmetic volume function can be realized in
terms of any such sequence. Its proof follows closely that of Chen.

\begin{theorem}
   \label{thm:Chendiff}
   The arithmetic volume function is differentiable at $\overline{D}$,
   and for any Fujita approximation sequence
   $(\phi_n, \overline{P}_n)_{n}$ of $\overline{D}$ we have
\begin{displaymath}
\partial_{\overline{E}}\, \widehat{\vol}(\overline{D}) = (d+1) \lim_{n\to \infty} (\overline{P}_n^d \cdot \phi_n^*\overline{E})
\quad  \text{ for all } \overline{E} \in \widehat{\Div}(X)_{\bR}.
\end{displaymath}
\end{theorem}

We need the following consequence of the arithmetic Siu's
inequality (Theorem~\ref{thm:Yuan}).

\begin{lemma}\label{lemma:ineqYuanv0} Let $\overline{P}$ and
  $\overline{E}$ be two adelic $\bR$-divisors on $X$ with
  $\overline{P}$ nef, and $\overline{A}$ a big and nef adelic
  $\bR$-divisor on $X$ such that $\overline{A} - \overline{P}$ is
  pseudo-effective and $\overline{A} \pm \overline{E}$ are nef. There
  exists a constant $c_{d}$ depending only on $d$ such that
\begin{displaymath}
  \widehat{\vol}(\overline{P} + \lambda \overline{E}) \ge (\overline{P}^{d+1})
  + (d+1) \, (\overline{P}^d \cdot \overline{E}) \, \lambda - c_d\, \widehat{\vol}(\overline{A})  \, \lambda^2 
  \quad \text{ for all } \lambda\in [0,1].
\end{displaymath}
\end{lemma}

\begin{proof} \change{This is a formal consequence of Theorem~\ref{thm:Yuan},
  as observed by Chen in the setting of Hermitian line bundles
  \cite[Lemma 4.2]{Chen:diffvol}. Using the same arguments, Ikoma
  proved the lemma in the case when all non-Archimedean Green
  functions are induced by an integral model \cite[Proposition
  5.1(1)]{Ikoma:IMRN}. His proof remains valid in our setting.}
\end{proof}

\begin{proof}[Proof of Theorem \ref{thm:Chendiff}] Consider the
  function $\Phi = \widehat{\vol}^{\frac{1}{d+1}} $ on the big cone of
  $ \widehat{\Div}(X)_{\mathbb{R}}$. It is positive homogeneous of
  degree $1$ and super-additive by the Brunn-Minkowski inequality, and
  its domain is an open cone \cite[Theorems 5.2.1 and
  5.3.1]{Moriwaki:MAMS}. Therefore it is a concave function on a
  convex open subset of a real vector space.

  Let $\overline{E} \in \widehat{\Div}(X)_{\bR}$.  Applying Lemma
  \ref{lemma:diffconcave} to this concave function and noting that
  $\widehat{\vol}=\Phi^{d+1}$ we deduce that the one-sided derivative
  $\partial_{\overline{E}}\, \widehat{\vol}(\overline{D})$ exists in
  $\mathbb{R}$ and moreover
  $-\partial_{-\overline{E}}\, \widehat{\vol}(\overline{D}) \ge
  \partial_{\overline{E}}\, \widehat{\vol}(\overline{D})$. We claim
  that
\begin{equation}\label{eq:lowerpartialvol} 
\partial_{\overline{E}}\, \widehat{\vol}(\overline{D}) \ge (d+1) \, \limsup_{n\to \infty} \, (\overline{P}_n^d \cdot \phi_n^*\overline{E}).
\end{equation}
To prove this, we first assume that $\overline{E}$ is DSP.  By Lemma
\ref{lem:1}, \change{there are big and nef adelic $\mathbb{R}$-divisors
$\overline{N}_1$, $\overline{N}_2$ on $X$ such that
$\overline{E} = \overline{N}_1 - \overline{N}_2$. Letting
$\overline{A}$ be a sufficiently large multiple of
$\overline{N}_1 + \overline{N}_2$, we have that
$\overline{A} \pm \overline{E}$ are nef and
$\overline{A} - \overline{D}$ is big.} This implies that
$\phi_n^*\overline{A} - \overline{P}_n$ is pseudo-effective for all
$n$.
Then by Lemma \ref{lemma:ineqYuanv0} there exists $c >0$ such that
\begin{displaymath}
\widehat{\vol}(\overline{D}+ \lambda\overline{E}) \ge \widehat{\vol}(\overline{P}_n + \lambda \, \phi_n^*\overline{E}) \ge (\overline{P}_n^{d+1}) + (d+1) \lambda \, (\overline{P}_n^{d} \cdot \phi_n^*\overline{E}) - c\,  \lambda^2
\end{displaymath}
for all $0<\lambda\le 1$ and $n \in \bN$.  Taking the $\limsup$ with respect to $n$
we deduce
\begin{displaymath}
  \frac{\widehat{\vol}(\overline{D} + \lambda \overline{E}) - \widehat{\vol}(\overline{D})}{\lambda} \ge  (d+1) \, 
  \limsup_{n\to \infty} \, (\overline{P}_n^d \cdot \phi_n^*\overline{E}) - c\,  \lambda,
\end{displaymath}
and we obtain \eqref{eq:lowerpartialvol} in this case by letting
$\lambda\to 0$.

For the general case, let $\varepsilon > 0$. By Lemma
\ref{lemma:approachDSP} there is
$\overline{E}' \in \widehat{\DSP}(X)_{\bR}$ such that
$\overline{E} - \overline{E}'$ and
$\overline{E}' + \varepsilon[\infty] - \overline{E}$ are
pseudo-effective. Then
$\partial_{\overline{E}}\, \widehat{\vol}(\overline{D}) \ge
\partial_{\overline{E}'}\, \widehat{\vol}(\overline{D})$, and using  Lemma
\ref{lemma:intersectionpseff}, the formula in
\eqref{eq:arithmeticvsgeometricintersection} and the fact that
$(P_n^d) = \vol(P_n)$ we obtain
 \begin{displaymath}
 (\overline{P}_n^d \cdot \phi_n^*\overline{E}') \ge (\overline{P}_n^d \cdot \phi_n^*(\overline{E} -\varepsilon [\infty])) = (\overline{P}_n^d \cdot \phi_n^*\overline{E}) - \varepsilon\vol(P_n) \ge (\overline{P}_n^d \cdot \phi_n^*\overline{E}) - \varepsilon\vol(D) 
 \end{displaymath}
 for all $n$. From the DSP case we deduce
\begin{displaymath}
  \partial_{\overline{E}}\, \widehat{\vol}(\overline{D}) \ge (d+1) \, \Big(\limsup_{n\to \infty} \, (\overline{P}_n^d \cdot \phi_n^*\overline{E}) -\varepsilon \vol(D)\Big),
\end{displaymath}
and so \eqref{eq:lowerpartialvol} follows by letting
$\varepsilon \to 0$.  Applying this inequality to $-\overline{E}$ we
obtain
\begin{displaymath}
(d+1) \,\liminf_{n\to \infty}  \, (\overline{P}_n^d \cdot \phi_n^*\overline{E}) \ge -\partial_{-\overline{E}}\, \widehat{\vol}(\overline{D}) \ge \partial_{\overline{E}}\, \widehat{\vol}(\overline{D}) \ge (d+1) \,  \limsup_{n\to \infty} \, (\overline{P}_n^d \cdot \phi_n^*\overline{E}).
\end{displaymath}
\change{By Lemma \ref{lemma:diffconcave}} this implies that
$ \partial_{\overline{E}}\, \widehat{\vol}(\overline{D})=(d+1) \,
\lim_{n\to \infty} \, (\overline{P}_n^d \cdot \phi_n^*\overline{E})$,
including the fact that this limit exists in $\mathbb{R}$ and that
$\widehat{\vol}$ is differentiable at $\overline{D}$.
\end{proof}

Combining this with Chen's formula for the arithmetic volume
(Theorem~\ref{thm:ChenIntegral}) we obtain useful information about
the asymptotics of Fujita approximation sequences.

\begin{proposition}
  \label{prop:orthogonal}
  Let $(\phi_n, \overline{P}_n)_{n}$ be a Fujita approximation
  sequence of $\overline{D}$. Then 
\begin{displaymath}
  \lim_{n\to\infty} (\overline{P}_n^d \cdot \change{\phi_n^*}\overline{D}) = \widehat{\vol}(\overline{D}), \quad \lim_{n\to\infty} (P_n^d) = \vol(R^0(\overline{D})), \quad
   \lim_{n\to\infty} \mu^{\ess}(\overline{P}_n) = \mu^{\ess}(\overline{D}).
\end{displaymath}
\end{proposition}

\begin{proof}
  The first formula follows from Theorem \ref{thm:Chendiff} after
  noting that
  $\partial_{\overline{D}}\, \widehat{\vol}(\overline{D}) =
  (d+1)\widehat{\vol}(\overline{D})$ by the positive homogeneity of
  the arithmetic volume.

  \change{We now prove the second one. Since $\overline{D}$ is big,
    there exists $\varepsilon > 0$ such that
    $\overline{D}+\lambda[\infty]$ is big for every
    $\lambda > -\varepsilon$.  By Theorem \ref{thm:ChenIntegral} for
    every such $\lambda$ we have}
  \begin{displaymath}
    \widehat{\vol}(\overline{D}+\lambda[\infty]) 
    = (d+1) \int_{0}^{\infty} \vol(R^{t}(\overline{D}+\lambda[\infty])) \, dt
  = (d+1) \int_{-\lambda}^{\infty} \vol(R^{t}(\overline{D})) \, dt
\end{displaymath}
 Since the
function $t\mapsto \vol(R^{t}(\overline{D}))$ is non-increasing, this
integral formula implies
  \begin{displaymath}
    \partial_{[\infty]}\, \widehat{\vol}(\overline{D}) = (d+1) \lim_{\lambda\to 0^{-}}  \vol(R^{\lambda}(\overline{D})), \quad 
-  \partial_{-[\infty]}\, \widehat{\vol}(\overline{D}) = (d+1) \lim_{\lambda\to 0^{+}}  \vol(R^{\lambda}(\overline{D})).
  \end{displaymath}
By Theorem \ref{thm:Chendiff} these two quantities coincide and so  
  $ \partial_{[\infty]}\, \widehat{\vol}(\overline{D})= (d+1)
  \vol(R^{0}(\overline{D})) $, and moreover   this derivative equals
  $(d+1) \lim_{n\to \infty} (\overline{P}_n^d \cdot [\infty]) = (d+1)
  \lim_{n \to\infty} (P_n^d)$.
  
  For the third, assume by contradiction that there exists
  $ \gamma< \mu^{\ess}(\overline{D}) $ such that
  $\mu^{\ess}(\overline{P}_n) \le \gamma$ for an arbitrarily large
  $n$ and set
\begin{displaymath}
c = (d+1) \int_{\gamma}^{\mu^{\ess}(\overline{D})} \vol(R^t(\overline{D})) \, dt. 
\end{displaymath}
We have $c=\widehat{\vol}(\overline{D}(\gamma))>0$ by Theorem
\ref{thm:Essmin}.  
Since $\overline{D}-\phi_{n}^{*} \overline{P}_{n}$ is
pseudo-effective, by Theorem \ref{thm:ChenIntegral} and
Lemma~\ref{lemma:propertiesessmin}\eqref{item:essminlowerbound} we have 
\begin{displaymath}
\widehat{\vol}(\overline{P}_n) = (d+1) \int_0^{\mu^{\ess}(\overline{P}_n)} \hspace{-3mm}\vol(R^t(\overline{P}_n)) \, dt \le  (d+1) \int_0^{\gamma} \vol(R^t(\overline{D})) \, dt
 \le  \widehat{\vol}(\overline{D}) - c,
\end{displaymath}
which contradicts the last condition in Definition \ref{def:9}.
\end{proof}

\begin{definition}
  \label{def:13}
  For each $\overline{E} \in \widehat{\Div}(X)_{\bR}$, the \emph{arithmetic
    positive intersection number}
  $(\langle \overline{D}^d \rangle \cdot \overline{E})$ is defined~as
\begin{displaymath}
(\langle \overline{D}^d \rangle \cdot \overline{E}) = \lim_{n\to \infty}  (\overline{P}_n^d \cdot \phi_n^*\overline{E})
\end{displaymath}
for any Fujita approximation sequence $(\phi_n, \overline{P}_n)_{n}$
of $\overline{D}$.  By Theorem \ref{thm:Chendiff} this limit exists and
does not depend on the choice of the sequence.  
\end{definition}

\begin{remark}\label{rem:13}
  Our definition of the quantity
  $(\langle \overline{D}^d \rangle \cdot \overline{E}) $ agrees with
  that of Chen~\cite{Chen:diffvol} since both coincide with
  $(d+1)^{-1} \partial_{\overline{E}}\, \widehat{\vol}(\overline{D})$.
  The alternative approach from Definition~\ref{def:13} is better
  suited for our purposes.
\end{remark}

\begin{remark}
\label{rem:geometricpositiveintersection} 
One can similarly adapt the proof of Boucksom, Favre and Jonsson's
differentiability theorem \cite{BFJ:diffvol} to show that for
$D, E \in \Div(X)_{\bR}$ with $D$ big, the geometric positive
intersection number $(\langle D^{d-1} \rangle \cdot E)$ introduced in
\emph{loc. cit.}  can be expressed as
\begin{displaymath}
(\langle D^{d-1} \rangle \cdot E) = \lim_{n \rightarrow \infty} (P_n^{d-1}\cdot \phi_n^*E)
\end{displaymath}
for any sequence $(\phi_n, P_n)_n$ such that $P_n$ is a nef
$\bR$-divisor on a normal modification
$\phi_n \colon X_n \rightarrow X$ with $\phi_n^*D - P_n$
pseudo-effective and $\lim_{n \to\infty} (P_n^d) = \vol(D)$.
\end{remark}

With this definition, the first and second formulae in Proposition
\ref{prop:orthogonal} become
\begin{equation}
  \label{eq:40}
  (\langle \overline{D}^d \rangle \cdot \overline{D}) =\widehat{\vol}(\overline{D}) \and
  (\langle \overline{D}^d \rangle \cdot [\infty]) = \vol(R^{0}(\overline{D})).
\end{equation}
The first of them is \cite[Corollary 4.4]{Chen:diffvol} and is usually
interpreted as an asymptotic orthogonality for the Fujita
approximations of $\overline{D}$.

\change{Recall that $\vol(R^0(\overline{D}))> 0$ by Lemma
\ref{lem:charactbigness}.} Hence arithmetic positive intersections
numbers allow to define the linear functional
\begin{equation}
  \label{eq:41}
\Omega_{\overline{D}} \colon \widehat{\Div}(X)_{\bR} \longrightarrow \bR, \quad \overline{E} \longmapsto \frac{(\langle \overline{D}^d \rangle \cdot \overline{E})}{\vol(R^0(\overline{D}))}.
\end{equation}
By Proposition \ref{prop:orthogonal}, for any Fujita approximation
sequence $(\phi_n, \overline{P}_n)_{n}$ of $\overline{D}$ we have
\begin{equation}\label{eq:Omegalimit}
\Omega_{\overline{D}}(\overline{E}) = \lim_{n \to \infty} \frac{(\overline{P}_n^d \cdot \phi_n^*\overline{E})}{(P_n^d)} \quad \text{ for all } \overline{E} \in \widehat{\Div}(X)_{\bR}.
\end{equation}

\begin{remark}
  \label{rem:14}
  This linear functional verifies that
  $\Omega_{\overline{D}}(\overline{E}) \ge 0$ for every
  $\overline{E}\in \widehat{\Div}(X)_{\mathbb{R}}$ pseudo-effective,
  $\Omega_{\overline{D}} (\widehat{\div}(f)) = 0$ for every
  $f \in \Rat(X)_{\bR}^{\times}$ and
  $\Omega_{\overline{D}}([\infty]) = 1$, as it follows respectively
  from Lemma \ref{lemma:intersectionpseff}, the invariance of
  arithmetic intersection numbers with respect to linear equivalence,
  and the second formula in \eqref{eq:40}. Hence
  $\Omega_{\overline{D}}$ is a normalized GVF functional in the sense
  of \cite{Szachniewicz}.
  \change{In particular
    \begin{displaymath}
      \Omega_{\overline{D}}(\overline{M}) \ge \mu^{\ess}(\overline{M}) \quad \text{ for all }  \quad
      \overline{M} \in \widehat{\Div}(X)_{\bR}      
    \end{displaymath}
    because $\overline{M}-\mu^{\ess}(\overline{M})[\infty]$ is
    pseudo-effective whenever $\mu^{\ess}(\overline{M}) > -\infty$ by
    Remark \ref{rema:pseffmax}.}
\end{remark}

For each $v\in \mathfrak{M}_K$ we have that $\Omega_{\overline{D}}$
induces a positive linear functional on $C(X_{v}^{\an})^{G_{v}}$
through the inclusion in \eqref{eq:27}. Hence by
Lemma~\ref{lem:extendlinearfunctionals}\eqref{item:extendlinearfunctional}
there is a unique $G_{v}$-invariant measure $\omega_{\overline{D},v}$
on $X_{v}^{\an}$ such that
\begin{equation}
  \label{eq:36}
\Omega_{\overline{D}}(\overline{0}^{\varphi})=    \frac{(\langle \overline{D}^d \rangle \cdot \overline{0}^\varphi)}{\vol(R^0(\overline{D}))} = n_v \int_{X_v^{\an}} \varphi \, d\omega_{\overline{D},v} \quad \text{ for all } \varphi \in C(X_v^{\an})^{G_v}.
\end{equation}
It follows from the limit formula in \eqref{eq:Omegalimit} together
with \eqref{eq:intersectionvarphi} that
for any Fujita approximation sequence $(\phi_n,\overline{P}_{n})_{n}$
of $\overline{D}$ we have
\begin{displaymath}
\int_{X_{v}^{\an}} \varphi \, d \omega_{\overline{D},v}
  =\lim_{n\to \infty} \int_{X_{v}^{\an}} \varphi \, d \omega_{n,v} \quad \text { for all } \varphi \in C(X_v^{\an})^{G_v}
\end{displaymath}
with $\omega_{n,v}$ the pushforward to $X_{v}^{\an}$ of
$c_{1}(\overline{P}_{n,v})^{\wedge d}/ (P_{n}^{d})$. Hence by Lemma
\ref{lem:extendlinearfunctionals}\eqref{item:weakconvergenceGalinv} we
have
\begin{math}
  \omega_{\overline{D},v}=\lim_{n\to \infty} \omega_{n,v}. 
\end{math}
In particular, $\omega_{\overline{D},v}$ is a probability measure.

\section{The differentiability of the essential minimum}\label{sec:Mainthm}

The differentiability of the essential minimum at an adelic
$\mathbb{R}$-divisor $\overline{D}$ is closely related to the
asymptotic behavior of the $\overline{D}$-small generic sequences of
algebraic points.  After explaining this relation we state our central
result, giving the differentiability of this function assuming the
existence of suitable semipositive approximations
of~$\overline{D}$. We then explain how it implies both Yuan's and
Chen's equidistribution theorems~\cite{Yuan:big,
  Chen:diffvol}. Finally we give a reformulation of this result in
terms of arithmetic positive intersection numbers.

Throughout this section we let  $\overline{D}$ be an adelic
$\bR$-divisor on $X$ over a big $\bR$-divisor~$D$.

\subsection{The variational approach to limit heights and equidistribution}\label{sec:variational}

As noted  in Section \ref{subsec:minima}, the essential minimum function 
\begin{displaymath}
  \mu^{\ess}\colon \widehat{\Div}(X)_{\bR}\longrightarrow \mathbb{R}\cup\{-\infty\}
\end{displaymath}
takes finite values on the open cone
$C \subset \widehat{\Div}(X)_{\bR}$ of adelic $\bR$-divisors with big
geometric $\bR$-divisor. Moreover, it is positive homogeneous of
degree $1$ and super-additive by Lemma
\ref{lemma:propertiesessmin}. Therefore it is a concave function on
$C$.  By Lemma \ref{lemma:diffconcave}, for every
$\overline{E}\in \widehat{\Div}(X)_{\bR}$ the one-sided derivative
$\partial_{\overline{E}}\, \mu^{\ess}(\overline{D})$ exists in $\bR$
and we have
  \begin{displaymath}
 -\partial_{-\overline{E}}\, \mu^{\ess}(\overline{D}) \ge
\partial_{\overline{E}}\, \mu^{\ess}(\overline{D}). 
\end{displaymath}

\begin{definition}
  \label{def:10}
  A sequence $(x_{\ell})_{\ell}$ in $X(\overline{K})$ is
  \emph{generic} if for every closed subset $Y \varsubsetneq X$ there
  is $\ell_0 \in \bN$ such that $x_{\ell} \notin Y(\overline{K})$ for
  all $\ell \ge \ell_0$.  When this is the case, we say that
  $(x_{\ell})_{\ell}$ is $\overline{D}$-\emph{small} if
\begin{displaymath}
\lim_{\ell \to \infty} h_{\overline{D}}(x_{\ell}) = \mu^{\ess}(\overline{D}).
\end{displaymath}
\end{definition} 

In the sequel, every considered generic sequence of algebraic points
lies in $X(\overline{K})$ unless otherwise stated.

By  \cite[Proposition 2.5]{BPRS:dgopshtv} for every  generic sequence
$(x_{\ell})_{\ell}$ we have 
\begin{equation}\label{eq:liminfessmin}
 \liminf_{\ell \to \infty} h_{\overline{D}}(x_{\ell}) \ge
\mu^{\ess}(\overline{D}),
\end{equation}
and moreover there are generic sequences that are
$\overline{D}$-small. A similar conclusion holds for the heights of
$\overline{D}$-small generic sequences with respect to another adelic
$\bR$-divisor~$\overline{E}$, replacing the quantity
$\mu^{\ess}(\overline{D})$ by the one-sided derivative
$\partial_{\overline{E}}\, \mu^{\ess}(\overline{D})$.

\begin{lemma}
\label{lem:sequencesandderivative} 
Let
  $\overline{E} \in \widehat{\Div}(X)_{\bR}$. For every
  $\overline{D}$-small generic sequence $(x_{\ell})_{\ell}$~we~have
\begin{displaymath}
-\partial_{-\overline{E}}\, \mu^{\ess}(\overline{D}) \ge \limsup_{\ell \to \infty} h_{\overline{E}}(x_{\ell}) \ge \liminf_{\ell \to \infty} h_{\overline{E}}(x_{\ell}) \ge \partial_{\overline{E}}\, \mu^{\ess}(\overline{D}).
\end{displaymath}
Moreover, there exists a $\overline{D}$-small generic sequence $(x_{\ell})_{\ell}$
such that
\begin{displaymath}
\lim_{\ell \to \infty} h_{\overline{E}}(x_{\ell}) = \partial_{\overline{E}}\, \mu^{\ess}(\overline{D}).
\end{displaymath}
\end{lemma}

\begin{proof}
For  a $\overline{D}$-small generic
  sequence   $(x_{\ell})_{\ell}$ and  $\lambda > 0$ we have 
  $\liminf_{\ell \to \infty} h_{\overline{D} + \lambda
    \overline{E}}(x_{\ell}) \ge \mu^{\ess}(\overline{D} + \lambda
  \overline{E})$ \change{by the definition of the essential minimum,} and so
 \begin{displaymath}
 \liminf_{\ell \to \infty} h_{\overline{E}}(x_{\ell})  = \liminf_{\ell \to \infty} \frac{ h_{\overline{D} + \lambda \overline{E}}(x_{\ell}) -  h_{\overline{D}}(x_{\ell})}{\lambda} \ge \frac{\mu^{\ess}(\overline{D} + \lambda \overline{E})  - \mu^{\ess}(\overline{D})}{\lambda}.
 \end{displaymath}
 Therefore
 $ \liminf_{\ell \to \infty} h_{\overline{E}}(x_{\ell}) \ge
 \partial_{\overline{E}}\, \mu^{\ess}(\overline{D})$, and we complete
 the proof of the first statement by applying this to $-\overline{E}$.

 We now pass to the second. By homogeneity, after multiplying
 $\overline{E}$ by a sufficiently small positive constant we assume
 without loss of generality that $D + E$ is big. For each
 $n \in \bN_{> 0}$ we choose a generic sequence $(x_{n,\ell})_{\ell}$
 such that
 \begin{equation}
   \label{eq:limvar}
\lim_{\ell \to \infty} h_{\overline{D} + \frac{1}{n} \overline{E}}(x_{n,\ell}) = \mu^{\ess}\Big(\overline{D} + \frac{1}{n}\overline{E}\Big).
\end{equation}
Let $H_i$, $i \in \bN$, be the collection of all the hypersurfaces of
$X$ in an arbitrary order. Then for each $n$ we choose
$\ell(n) \in \bN$ such that the algebraic point
$y_n = x_{n,\ell(n)} \in X(\overline{K})$ satisfies the conditions:
\begin{enumerate}[leftmargin=*]
\item \label{item:23} $y_n  \notin \bigcup_{i = 0}^n H_i$,
\item \label{item:27} $h_{\overline{D} + \frac{1}{n} \overline{E}}(y_n) \le
  \mu^{\ess}(\overline{D} + \frac{1}{n}\overline{E}) + 1/n^{2}$, 
\item \label{item:28} $h_{\overline{D}}(y_n) \ge \mu^{\ess}(\overline{D}) -1/n^{2}$.
\end{enumerate}
\change{This  can be done by using the genericity of the sequence
$(x_{n,\ell})_{\ell}$ together with \eqref{eq:liminfessmin} and~\eqref{eq:limvar}.}

By \eqref{item:23} the sequence $(y_n)_{n}$ is generic. Since
$D + E$ is big, this implies that there is $c \in \bR$ such that
$h_{\overline{D} + \overline{E}}(y_n) \ge c$ for all $n$, thanks to
the lower bound~\eqref{eq:heightlowerbound}.  Then it follows from
 condition \eqref{item:27} that
\begin{displaymath}
h_{\overline{D}}(y_n)  =  \frac{n}{n-1}  \Big( h_{\overline{D} + \frac{1}{n} \overline{E}}(y_n) - \frac{1}{n}h_{\overline{D} + \overline{E}}(y_n) \Big)
 \le \frac{n}{n-1} \mu^{\ess} \Big(\overline{D} +  \frac{1}{n}\overline{E}\Big) +  \frac{1}{n(n-1)} - \frac{c}{n-1}.
\end{displaymath}
Therefore the sequence $(y_n)_{n}$ is also $\overline{D}$-small by
the continuity of the essential minimum (Lemma
\ref{lemma:propertiesessmin}\eqref{item:essmincont}).  On the other
hand,  conditions \eqref{item:27} and \eqref{item:28} give
\begin{multline*}
  - \frac{1}{n^2} + \frac{1}{n} h_{\overline{E}}(y_n) \le
  h_{\overline{D}}(y_n)+ \frac{1}{n} h_{\overline{E}}(y_n) -
  \mu^{\ess}(\overline{D}) \\ = h_{\overline{D} + \frac{1}{n}
    \overline{E}}(y_n) - \mu^{\ess}(\overline{D}) \le
  \mu^{\ess}\Big(\overline{D} + \frac{1}{n}\overline{E}\Big) -
  \mu^{\ess}(\overline{D}) + \frac{1}{n^2},
\end{multline*}
which implies
\begin{displaymath}
  \limsup_{n \to \infty} h_{\overline{E}}(y_n) \le \lim_{n\to \infty} n\, \Big( \mu^{\ess}(\overline{D} + \frac{1}{n}\overline{E})
  - \mu^{\ess}(\overline{D})\Big) = \partial_{\overline{E}}\, \mu^{\ess}(\overline{D}).
\end{displaymath}
We conclude that
$ \lim_{n \to \infty} h_{\overline{E}}(y_n) =
\partial_{\overline{E}}\, \mu^{\ess}(\overline{D})$, as stated.
\end{proof}

\begin{proposition}
  \label{prop:diffvsHCP}
  Let $\overline{E} \in \widehat{\Div}(X)_{\bR}$. The following
  conditions are equivalent:
\begin{enumerate}[leftmargin=*]
\item\label{item:variationalHCP} for every $\overline{D}$-small
  generic sequence $(x_{\ell})_{\ell}$ the limit
  $\lim_{\ell \to \infty} h_{\overline{E}}(x_{\ell})$ exists,
\item\label{item:variationaldiff}
  $\partial_{-\overline{E}}\, \mu^{\ess}(\overline{D}) =
  -\partial_{\overline{E}}\, \mu^{\ess}(\overline{D})$.
\end{enumerate}
If they are satisfied, then
$\lim_{\ell \to \infty} h_{\overline{E}}(x_{\ell})=
\partial_{\overline{E}}\, \mu^{\ess}(\overline{D})$ for every
$\overline{D}$-small generic sequence $(x_{\ell})_{\ell}$.
\end{proposition}

\begin{proof} The implication
  \eqref{item:variationaldiff} $\Rightarrow$
  \eqref{item:variationalHCP} is given by the first part of Lemma
  \ref{lem:sequencesandderivative}.  Conversely, suppose that
  $\eqref{item:variationalHCP}$ holds. Then
  $\lim_{\ell \to \infty} h_{\overline{E}}(x_{\ell}) =
  \partial_{\overline{E}}\, \mu^{\ess}(\overline{D})$ for every
  $\overline{D}$-small generic sequence $(x_{\ell})_{\ell}$. Indeed,
  if the limit were different then using the second part of
  Lemma~\ref{lem:sequencesandderivative} we could construct a
  $\overline{D}$-small generic sequence $(x'_{\ell})_{\ell}$ such that
  $(h_{\overline{E}}(x'_{\ell}))_{\ell}$ does not converge. Moreover,
  applying the latter to $-\overline{E}$ gives that every
  $\overline{D}$-small generic sequence $(x_{\ell})_{\ell}$ verifies
\begin{displaymath}
 \partial_{-\overline{E}}\, \mu^{\ess}(\overline{D}) = \lim_{\ell \to \infty} h_{-\overline{E}}(x_{\ell}) = -\lim_{\ell \to \infty} h_{\overline{E}}(x_{\ell}) = - \partial_{\overline{E}}\, \mu^{\ess}(\overline{D}),
\end{displaymath}
which gives \eqref{item:variationaldiff}.
\end{proof}

The next result summarizes the relation between limit heights for
$\overline{D}$-small sequences of algebraic points and the
differentiability of the essential minimum function.  It is a direct
consequence of the previous one together with Lemma
\ref{lemma:diffconcave}\eqref{item:diffconcavealldirections}.

\begin{proposition}
\label{prop:11}
The following conditions are equivalent:
\begin{enumerate}[leftmargin=*]
\item \label{item:21} for every $\overline{D}$-small generic sequence
  $(x_{\ell})_{\ell}$ in $X(\overline{K})$ and 
  $\overline{E} \in \widehat{\Div}(X)_{\bR}$ the limit
  $\lim_{\ell \to \infty} h_{\overline{E}}(x_{\ell})$ exists,
\item \label{item:20}   the essential minimum function is differentiable at $\overline{D}$.
\end{enumerate}
If they are satisfied, then for any $\overline{D}$-small generic sequence
  $(x_{\ell})_{\ell}$ in $X(\overline{K})$ we have 
\begin{displaymath}
  \lim_{\ell \to \infty} h_{\overline{E}}(x_{\ell})=
  \partial_{\overline{E}}\, \mu^{\ess}(\overline{D})  \quad \text{ for all } \overline{E} \in \widehat{\Div}(X)_{\bR}.
\end{displaymath}
\end{proposition}

Let $v \in \mathfrak{M}_K$. For $x \in X(\overline{K})$ we denote by
$\delta_{O(x)_v}$ the uniform probability measure on \change{the $v$-adic
Galois orbit $O(x)_v$ of this point in $X_v^{\an}$,} that is
 \begin{displaymath}
 \delta_{O(x)_v}= \frac{1}{\# O(x)_v} \sum_{y \in O(x)_v} \delta_y.
\end{displaymath}

\begin{definition}
\label{def:11}
We say that $\overline{D}$ satisfies the \emph{equidistribution
  property} at $v$ if there is a probability measure
$\nu_{\overline{D},v}$ on $X_v^{\an}$ such that for every
$\overline{D}$-small generic sequence~$(x_{\ell})_{\ell}$, the
sequence of measures $(\delta_{O(x_{\ell})_v})_{\ell}$ converges to
$\nu_{\overline{D},v}$. When this holds, $\nu_{\overline{D},v}$ is
called the \emph{(v-adic) equidistribution measure} of $\overline{D}$.
\end{definition}

Note that if $\overline{D}$ satisfies the equidistribution property at
a place $v$ then its equidistribution measure $\nu_{\overline{D},v}$
is $G_v$-invariant, being the limit of a sequence of $G_v$-invariant
discrete probability measures.

The next result gives the relation between the
equidistribution property at~$v$ and the differentiability of the
essential minimum function along the subspace of $G_{v}$-invariant
continuous real-valued functions on $X_{v}^{\an}$.

\begin{proposition}
  \label{prop:diffvsEP}  The following
  conditions are equivalent:
\begin{enumerate}[leftmargin=*]
\item $\overline{D}$ satisfies the equidistribution property at $v$,
\item the essential minimum function is differentiable at
  $\overline{D}$ along $C(X_v^{\an})^{G_v}$.
\end{enumerate}
If they are satisfied, then  the equidistribution measure
$\nu_{\overline{D},v}$ is the unique $G_v$-invariant measure on
$X_v^{\an}$ such that
\begin{displaymath}
n_v \int_{X_v^{\an}} \varphi \, d\nu_{\overline{D},v} = \partial_{\overline{0}^{\varphi}}\, \mu^{\ess}(\overline{D})\quad \text{ for all } \varphi \in  C(X_v^{\an})^{G_v}.
\end{displaymath}
\end{proposition}

\begin{proof}
  For every $\varphi \in C(X_v^{\an})^{G_v}$ and
  $x \in X(\overline{K})$ we have
\begin{displaymath}
  h_{\overline{0}^{\varphi}}(x) = \frac{n_v}{\# O(x)_v} \sum_{y \in O(x)_v} \varphi(y) = n_v\int_{X_v^{\an}} \varphi \, d\delta_{O(x)_v}.
\end{displaymath}
The statement  then follows from Proposition \ref{prop:diffvsHCP}
 with Lemmas
\ref{lemma:diffconcave}\eqref{item:diffconcavealldirections}
and~\ref{lem:extendlinearfunctionals}.
\end{proof}

\subsection{Main theorem}\label{subsec:mainthm}

Recall that $\overline{D}$ is an adelic $\mathbb{R}$-divisor on
\change{the $d$-dimensional normal projective variety $X$,} with $D$
big.

\begin{definition}
  \label{def:12}
  A \emph{semipositive approximation} of
  $\overline{D}$ is a pair $(\phi, \overline{Q})$ where 
\begin{enumerate}[leftmargin=*]
\item $\phi \colon X' \rightarrow X$ is a normal modification,
\item $\overline{Q}$ is a semipositive adelic $\bR$-divisor on $X'$
  with big geometric $\bR$-divisor $Q$,
\item $\phi^*\overline{D} - \overline{Q}$ is pseudo-effective.
\end{enumerate} 
When $\phi$ is the identity on $X$, we simply denote by $\overline{Q}$
the semipositive approximation of $\overline{D}$ corresponding to the
pair $(\Id_{X},\overline{Q})$.
\end{definition} 

\begin{remark}
  \label{rem:15}
  \change{For a pair $(\phi, \overline{Q})$ as above, the geometric
    $\bR$-divisor~$Q$ is nef by semipositivity. Since it is also big,
    we have $(Q^d) = \vol(D) > 0$.}
\end{remark}

  The following is the central result of this text. 

\begin{theorem}\label{thm:MainSequences}
  Assume that there exists a sequence
  $(\phi_n, \overline{Q}_n)_{n}$ of semipositive approximations
  of $\overline{D}$ such that
\begin{equation}\label{eq:conditioninradiussequences}
 \lim_{n \rightarrow \infty} \frac{\mu^{\ess}(\overline{D}) - \mu^{\abs}(\overline{Q}_n)}{r(Q_n;\phi_n^*D)} = 0.
\end{equation}
Then the essential minimum function is
differentiable at $\overline{D}$ and 
  \begin{displaymath}
    \partial_{\overline{E}}\,  \mu^{\ess}(\overline{D}) 
    = \lim_{n \rightarrow \infty} \, \frac{(\overline{Q}_n^d \cdot \phi_n^* \overline{E}) - d\,  \mu^{\ess}(\overline{D}) \, (Q_n^{d-1} \cdot \phi_n^* E)}{(Q_n^d)} \quad \text{ for all } \overline{E}\in \widehat\Div(X)_{\mathbb{R}}. 
\end{displaymath}
In particular, the limit on the right-hand side exists in $\bR$ and
does not depend on the choice of the sequence.
\end{theorem}

\change{The proof will be given in Section \ref{sec:Proof}.} This result
together with Proposition \ref{prop:11} \change{shows} that if $\overline{D}$
satisfies condition~\eqref{eq:conditioninradiussequences}, then
for any $\overline{D}$-small generic sequence $(x_{\ell})_{\ell}$ and
$\overline{E} \in \widehat{\Div}(X)_{\bR}$ we have
\begin{equation*}
  \lim_{\ell \to \infty} h_{\overline{E}}(x_{\ell}) = \lim_{n\to \infty} \, \frac{(\overline{Q}_n^d \cdot \phi_n^* \overline{E})
    - d\,  \mu^{\ess}(\overline{D}) \, (Q_n^{d-1} \cdot \phi_n^* E)}{(Q_n^d)}.
\end{equation*}
\change{Using \eqref{eq:intersectionvarphi} and the definition of pushforward
measures, this implies that} $\overline{D}$ satisfies the
equidistribution property at each place $v \in \mathfrak{M}_K$ and
that its $v$-adic equidistribution measure is given by
\begin{displaymath}
  \nu_{\overline{D},v} = \lim_{n \to \infty} \nu_{n,v},
\end{displaymath}
where  $\nu_{n,v}$ denotes the pushforward to $X_v^{\an}$ of the normalized
Monge-Ampère measure $c_1(\overline{Q}_{n,v})^{\wedge d}/(Q_n^d)$, as
stated in Theorem \ref{thm:MainIntro}.

\begin{remark}\label{rem:conditionMainvolume}
  The inradius $r(Q_n; \phi_n^*D)$ in Theorem \ref{thm:MainSequences}
  measures the bigness of the geometric $\bR$-divisor $Q_n$ for each
  $n$.  For our purposes this invariant is finer than the geometric
  volume $\vol(Q_n) = (Q_n^d)$.  Indeed, for any ample
  $A \in \Div(X)_{\bR}$ with $A - D$ pseudo-effective we have that
  $\phi_n^*A - Q_n$ is pseudo-effective for all $n$, and so
  $ r(Q_n;\phi^*_n D) \ge (Q_n^d) / (d \, (A^d))$ by
  Lemma~\ref{lemma:lowerboundinradiusvol}. Hence any sequence of
  semipositive approximations $(\phi_n, \overline{Q}_n)_{n}$ of
  $\overline{D}$ such that
  \begin{displaymath}
\lim_{n \to \infty} \frac{\mu^{\ess}(\overline{D}) - \mu^{\abs}(\overline{Q}_n)}{(Q_n^d)} = 0
\end{displaymath} 
also satisfies \eqref{eq:conditioninradiussequences}.  When $X$ is a
curve this condition is equivalent to
\eqref{eq:conditioninradiussequences} 
but is stronger in higher dimensions, as it can be seen for instance
in the semiabelian setting (Remark~\ref{rem:16}).
\end{remark}

\begin{remark}\label{rema:hypothesesthm}
  We assume throughout that $X$ is normal because we do not consider
  adelic $\bR$-divisors on an arbitrary projective variety.
  Nevertheless, Theorem \ref{thm:MainSequences} can be applied to
  study the equidistribution properties of an adelic divisor
  $\overline{D}$ on an arbitrary projective variety $X$ over $K$.
  Indeed, since all the data is invariant under birational
  modifications one can reduce to the normal case by working on the
  normalization.
\end{remark}

\subsection{Application to classical equidistribution
  results}\label{subsec:YuanandChen}

Yuan's equidistribution theorem (Theorem \ref{thm:YuanIntro}) is a
direct consequence of Theorem \ref{thm:MainSequences}, which moreover
shows that this result is valid for any adelic $\bR$-divisor with big
geometric $\bR$-divisor and gives the differentiability of the
essential minimum function. These slight improvements could already be
obtained by adapting Yuan's proof.

\begin{corollary}
\label{cor:yuantheorem} 
Assume that $\overline{D}$ is semipositive and
that 
\begin{displaymath}
\mu^{\ess}(\overline{D}) = \frac{(\overline{D}^{d+1})}{(d+1)\, (D^d)}.
\end{displaymath}
Then the essential minimum function is differentiable at $\overline{D}$ and
\begin{displaymath}
  \partial_{\overline{E}}\, \mu^{\ess}(\overline{D})= \frac{(\overline{D}^{d}\cdot \overline{E}) - d \, \mu^{\ess}(\overline{D}) \, (D^{d-1}\cdot E)}{(D^{d})} \quad \text{ for  all }  \overline{E}\in \widehat{\Div}(X)_{\bR}.
\end{displaymath}
In particular, $\overline{D}$ satisfies the $v$-adic equidistribution
property at each $v\in \mathfrak{M}_{K}$ with equidistribution measure
$\nu_{\overline{D},v}= c_{1}(\overline{D}_{v})^{\wedge d}/(D^{d})$.
\end{corollary}

\begin{proof}
  Apply Theorem \ref{thm:MainSequences} to the sequence of
  semipositive approximations 
  $\overline{Q}_{n}=\overline{D}$, $n \in \bN$. By Theorem \ref{thm:criterionextremal} we have 
  \begin{math}
\mu^{\abs}(\overline{D}) =\mu^{\ess}(\overline{D}) 
\end{math}
and so~\eqref{eq:conditioninradiussequences} is
verified.
\end{proof}

The following corollary recovers Chen's equidistribution theorem
\cite[Corollary~5.5]{Chen:diffvol}. As shown in \emph{loc. cit.}, this
result is a consequence of the differentiability of the arithmetic
volume. We shall explain how to deduce it from Theorem
\ref{thm:MainSequences}.

\begin{corollary}\label{coro:Chenequi}
  Assume that $\overline{D}$ is big and that
\begin{displaymath}
\mu^{\ess}(\overline{D}) = \frac{\widehat{\vol}(\overline{D})}{(d+1) \vol(D)}.
\end{displaymath}
Then the essential minimum function is differentiable at $\overline{D}$ and 
\begin{displaymath}
 \partial_{\overline{E}}\, \mu^{\ess}(\overline{D})= \frac{(\langle \overline{D}^{d} \rangle \cdot \overline{E}) - d \, \mu^{\ess}(\overline{D}) \, ( \langle D^{d-1}\rangle \cdot E)}{\vol(D)} \quad \text{ for all }  \overline{E}\in \widehat{\Div}(X)_{\bR}.
\end{displaymath}
In particular, $\overline{D}$ satisfies the $v$-adic equidistribution
property at each $v\in \mathfrak{M}_{K}$, with equidistribution
measure $\nu_{\overline{D},v}=\omega_{\overline{D},v}$.
\end{corollary}

Here $\omega_{\overline{D},v}$ is the probability measure on
$X_{v}^{\an}$ in \eqref{eq:36} and $( \langle D^{d-1}\rangle \cdot E)$ is
the geometric positive intersection number from \cite{BFJ:diffvol}
(see Remark~\ref{rem:geometricpositiveintersection}).  We 
need the next lemma 
 to construct a suitable sequence of semipositive
approximations of~$\overline{D}$.

\begin{lemma}
  \label{lem:goodapprox}
  For every real numbers $t < \mu^{\ess}(\overline{D})$ and
  $\varepsilon >0$ there exists a semipositive approximation
  $(\phi, \overline{Q})$ of $\overline{D}$ such that
\begin{displaymath}
  \mu^{\abs}(\overline{Q}) \ge t, \quad |(Q^d) -\vol(R^t(\overline{D}))|\le  \varepsilon, \quad |(\overline{Q}^{d+1}) -
  \widehat{\vol}(\overline{D}(t)) -(d+1)t\vol(R^t(\overline{D}))| \le \varepsilon.
\end{displaymath}
\end{lemma}

\begin{proof} By Theorem \ref{thm:Essmin}, $\overline{D}(t)$ is
  big. Then by Proposition \ref{prop:orthogonal} applied to
  $\overline{D}(t)$, for any $\varepsilon'>0$ there is a nef adelic
  $\mathbb{R}$-divisor $\overline{P}$ on a normal modification
  $\phi\colon X'\to X$ with $\phi^{*}\overline{D}(t)-\overline{P}$
  pseudo-effective such that
  \begin{equation}
    \label{eq:42}
|(P^d) - \vol(R^t(\overline{D})) | \le \varepsilon' \and |(\overline{P}^{d+1}) -\widehat{\vol}(\overline{D}(t))| \le \varepsilon'.
\end{equation}

Set $\overline{Q} = \overline{P}(-t)=\overline{P}+t[\infty]$. By
construction, $(\phi, \overline{Q})$ is a semipositive approximation
of $\overline{D}$. Since $\overline{P}$ is nef we have
$\mu^{\abs}(\overline{Q}) = \mu^{\abs}(\overline{P}) +t \ge t$, which
gives the first condition. By
\eqref{eq:arithmeticvsgeometricintersection} we also have
\begin{displaymath}
(Q^d) =(P^{d}) \and 
(\overline{Q}^{d+1}) = (\overline{P}^{d+1}) + (d+1)\, t\, (Q^d),
\end{displaymath}
and so the second and third conditions follow from \eqref{eq:42}.
\end{proof}

\begin{proof}[Proof of Corollary \ref{coro:Chenequi}]
  By Theorem \ref{thm:ChenIntegral} and Lemma
  \ref{lemma:propertiesessmin}\eqref{item:essminlowerbound} we have
\begin{displaymath}
 \mu^{\ess}(\overline{D}) =  \frac{\widehat{\vol}(\overline{D})}{(d+1) \vol(D)} = \frac{1}{\vol(D)} \int_0^{\mu^{\ess}(\overline{D})} \vol(R^t(\overline{D}))\, dt.
\end{displaymath}
Since $\vol(R^t(\overline{D})) \le \vol(D)$ we get 
$\vol(R^t(\overline{D})) = \vol(D)$ for all
$0 \le t\le \mu^{\ess}(\overline{D})$. Hence by the same results, for
$t$ in this range we have
\begin{displaymath}
\widehat{\vol}(\overline{D}(t)) = \widehat{\vol}(\overline{D}) - (d+1)\,t \vol(D).
\end{displaymath}
Applying Lemma \ref{lem:goodapprox} we deduce that there is a
sequence $(\phi_n,\overline{Q}_n)_n$ of semipositive approximations
of~$\overline{D}$ such that
\begin{displaymath}
  \lim_{n \to \infty} (\overline{Q}_n^{d+1}) = \widehat{\vol}(\overline{D}), \quad \lim_{n \to \infty} (Q_n^d) = \vol(D), \quad \lim_{n\to \infty} \mu^{\abs}(\overline{Q}_n) = \mu^{\ess}(\overline{D}).
\end{displaymath}
Hence by Remark \ref{rem:conditionMainvolume} the sequence
$(\phi_n, \overline{Q}_n)_{n}$ satisfies condition
\eqref{eq:conditioninradiussequences}. Since
$\mu^{\ess}(\overline{D})> 0$, we have that $\overline{Q}_n$ is nef
for sufficiently large $n$ and therefore $(\phi_n, \overline{Q}_n)_n$
is a Fujita approximation sequence of $\overline{D}$. By the
definition of arithmetic positive intersection numbers and Remark
\ref{rem:geometricpositiveintersection} we have
\begin{displaymath}
\lim_{n\to \infty} (\overline{Q}_n^d \cdot \phi_n^*\overline{E})= (\langle \overline{D}^{d} \rangle \cdot \overline{E}) \and \lim_{n\to \infty} (Q_n^{d-1} \cdot \phi_n^*E) = (\langle D^{d-1}\rangle \cdot E).
\end{displaymath}
The result then follows from Theorem \ref{thm:MainSequences}.
\end{proof}

\begin{remark}
  \label{rem:18}
  Chen's equidistribution theorem implies Yuan's. Indeed, for
  $\overline{D}$ semipositive the statement of the latter is
  invariant under shifts of this adelic $\mathbb{R}$-divisor by
  multiples of $[\infty]$. Hence we can suppose without loss of
  generality that $\overline{D}$ is big and nef, in which case
  Corollary \ref{coro:Chenequi} specializes to Corollary
  \ref{cor:yuantheorem}.
\end{remark}

\subsection{Interpretation in terms of arithmetic positive intersection numbers}\label{subsec:Productversion}

Here we propose a reformulation of Theorem \ref{thm:MainSequences}
that gives a more intrinsic condition for the differentiability of the
essential minimum function and shows that the derivative can be
computed using limits of arithmetic positive intersection numbers. As
before, we let $\overline{D}$ be an $\mathbb{R}$-divisor on $X$ with
$D$ big. \change{Recall that $\overline{D}(t)$ is big for any $t < \mu^{\ess}(\overline{D})$ by Theorem \ref{thm:Essmin}.}

We define the
inradius of a big adelic $\bR$-divisor as the supremum of the
geometric inradii of its nef approximations.

\begin{definition}
  \label{def:rho}
  Let $\overline{B} $ be a big adelic $\bR$-divisor on $X$.  A
  \emph{nef approximation} of $\overline{B}$ is a semipositive
  approximation $(\phi, \overline{P})$ of $\overline{B}$ such that
  $\overline{P}$ is nef.  We denote by $\Theta(\overline{B})$ the set
  of nef approximations of $\overline{B}$. 

  For a big $\bR$-divisor $A$ on $X$, the \emph{inradius} of
  $\overline{B}$ with respect to $A$ is defined~as
\begin{displaymath}
  \rho(\overline{B};A)=\sup \{ r(P; \phi^*A) \, | \ (\phi, \overline{P}) \in \Theta(\overline{B}) \}.
\end{displaymath}
This is a positive real number. We also set
$\rho(\overline{B})= \rho(\overline{B};B)$ for the inradius of
$\overline{B}$ with respect to its geometric $\mathbb{R}$-divisor $B$,
which is big \change{by Lemma \ref{lem:charactbigness}.}
\end{definition}

\change{In the next result the limits $t \to \mu^{\ess}(\overline{D})$
  are taken \emph{from below}.  Indeed the condition
  $t < \mu^{\ess}(\overline{D})$ is needed for the functions inside the
  limits to be well-defined, and so hopefully this will not lead to
  confusion.  The same convention will be adopted throughout for the
  limits where $t$ tends to an essential minimum.}

\begin{theorem}\label{thm:MainProduct}
  Assume that
\begin{equation}\label{eq:conditionMainProduct}
	\liminf_{t \to \mu^{\mathrm{ess}}(\overline{D})} \frac{\mu^{\mathrm{ess}}(\overline{D}) -t}{\rho(\overline{D}(t))}=0.
      \end{equation}
Then the essential minimum function is differentiable at
        $\overline{D}$ and
\begin{equation}\label{eq:derivativeMain2}
  \partial_{\overline{E}}\, \mu^{\ess}(\overline{D}) 
  =\lim_{t \rightarrow \mu^{\mathrm{ess}}(\overline{D})} \frac{(\langle \overline{D}(t)^d \rangle \cdot \overline{E})}{\vol(R^t(\overline{D}))}
  \quad \text{ for all } \overline{E}\in \widehat{\Div}(X)_{\mathbb{R}}.
\end{equation}
\end{theorem}

We can reformulate accordingly the asymptotic behavior of heights and
Galois orbits of the algebraic points in a $\overline{D}$-small
generic sequence: by Proposition \ref{prop:diffvsHCP}, this result
shows that if $\overline{D}$ satisfies 
condition~\eqref{eq:conditionMainProduct} then for any
$\overline{D}$-small generic sequence~$(x_{\ell})_{\ell}$ and
$\overline{E} \in \widehat{\Div}(X)_{\bR}$ we have
\begin{equation*} 
  \lim_{\ell \to \infty} h_{\overline{E}}(x_{\ell}) =
\lim_{t \rightarrow \mu^{\mathrm{ess}}(\overline{D})} \frac{(\langle \overline{D}(t)^d \rangle \cdot \overline{E})}{\vol(R^t(\overline{D}))}.
\end{equation*}
In particular, $\overline{D}$ satisfies the equidistribution property
at each place $v \in \mathfrak{M}_K$ and its  $v$-adic equidistribution
measure is given by
\begin{equation*}
  \nu_{\overline{D},v} = \lim_{t\to\mu^{\ess}(\overline{D})}\omega_{\overline{D}(t),v},
\end{equation*}
where $\omega_{\overline{D}(t),v}$ is the probability measure  on $X_{v}^{\an}$
from \eqref{eq:36}.

We first observe that the conditions in Theorems
\ref{thm:MainSequences} and \ref{thm:MainProduct} are equivalent.

\begin{lemma}\label{lemma:reformulation}
  Condition \eqref{eq:conditionMainProduct} holds if and only if
  there exists a sequence $(\phi_n, \overline{Q}_n)_{n}$ of
  semipositive approximations of $\overline{D}$ satisfying condition
  \eqref{eq:conditioninradiussequences}.
\end{lemma}

\begin{proof}
  First assume that \eqref{eq:conditionMainProduct} holds. Then there
  are sequences of real numbers
  $t_{n}< \mu^{\mathrm{ess}}(\overline{D})$, $n \in \mathbb{N}$, and
  of nef approximations
  $(\phi_n, \overline{P}_n) \in \Theta(\overline{D}(t_n))$,
  $n\in \mathbb{N}$, such that
\begin{displaymath}
\lim_{n \rightarrow \infty} \frac{\mu^{\mathrm{ess}}(\overline{D}) -t_n}{r(P_n;\phi_n^*D)} = 0.
\end{displaymath}
For each $n$ put $\overline{Q}_n = \overline{P}_n (-t_n)$. Then
$(\phi_n,\overline{Q}_n)$ is a semipositive approximation of
$\overline{D}$ with $Q_n = P_n$ and
$\mu^{\mathrm{abs}}(\overline{Q}_n) =
\mu^{\mathrm{abs}}(\overline{P}_n) + t_n \ge t_n$. Therefore the
sequence $(\phi_n, \overline{Q}_n)_{n}$ satisfies
\eqref{eq:conditioninradiussequences}.

Conversely let $(\phi_n, \overline{Q}_n)_{n}$ be a sequence of
semipositive approximations of $\overline{D}$ satisfying
\eqref{eq:conditioninradiussequences}. For each $n$ set
$t_n = \mu^{\abs}(\overline{Q}_n)$ and
$\overline{P}_n = \overline{Q}_n (t_n)$. Then $\overline{P}_{n}$ is
nef since it is semipositive and
$\mu^{\abs}(\overline{P}_n) = \mu^{\abs}(\overline{Q}_n) - t_n = 0$.
Moreover
\begin{displaymath}
\phi_n^*\overline{D}(t_n) - \overline{P}_n = (\phi_n^* \overline{D} - t_n[\infty]) - (\overline{Q}_n - t_n[\infty]) = \phi_n^*\overline{D} - \overline{Q}_n 
\end{displaymath}
is pseudo-effective. Therefore
$(\phi_n,\overline{P}_n) \in \Theta(\overline{D}(t_n))$ and in
particular
\begin{displaymath}
\rho(\overline{D}(t_n)) \ge r(P_n;\phi_n^*D) = r(Q_n;\phi_n^*D).
\end{displaymath}
Finally we have 
\begin{displaymath}
0 \le \liminf_{t \to \mu^{\mathrm{ess}}(\overline{D})} \frac{\mu^{\mathrm{ess}}(\overline{D}) -t}{\rho(\overline{D}(t))} \le \liminf_{n \to \infty} \frac{\mu^{\mathrm{ess}}(\overline{D}) -t_n}{\rho(\overline{D}(t_n))} \le  \lim_{n \to \infty} \frac{\mu^{\mathrm{ess}}(\overline{D}) -\mu^{\abs}(\overline{Q}_n)}{r(Q_n;\phi_n^*D)} = 0.
\end{displaymath}
\end{proof}

Note that for every $t< \mu^{\mathrm{ess}}(\overline{D})$ and $\overline{E }\in \widehat{\Div}(X)_{\mathbb{R}}$ we have 
\begin{displaymath}
  \Omega_{\overline{D}(t)}(\overline{E})= \frac{(\langle \overline{D}(t)^d \rangle \cdot \overline{E})}{\vol(R^t(\overline{D}))}, 
\end{displaymath}
where $\Omega_{\overline{D}(t)} \colon \widehat{\Div}(X)_{\bR} \rightarrow
\bR$ is the linear functional defined in \eqref{eq:41}. 
\begin{lemma}\label{lemma:chapeau}
  For every $\overline{E} \in \widehat{\Div}(X)_{\mathbb{R}}$ we have
  \begin{equation*}
-\partial_{-\overline{E}}\, \mu^{\ess}(\overline{D}) \ge \limsup_{t \to \mu^{\ess}(\overline{D})} \Omega_{\overline{D}(t)}(\overline{E})  \ge \liminf_{t \to \mu^{\ess}(\overline{D})} \Omega_{\overline{D}(t)}(\overline{E}) \ge  \partial_{\overline{E}}\, \mu^{\ess}(\overline{D}).
\end{equation*}
In particular, if the essential minimum function is differentiable at $\overline{D}$ then the limit 
$\lim_{t \to \mu^{\ess}(\overline{D})} \Omega_{\overline{D}(t)}(\overline{E})$
exists and equals $\partial_{\overline{E}}\, \mu^{\ess}(\overline{D})$.
\end{lemma}

\begin{proof}
  Let $ t < \mu^{\mathrm{ess}}(\overline{D})$. \change{By Remark  \ref{rem:14} we have
  $\Omega_{\overline{D}(t)}([\infty]) =1$ and
  $\Omega_{\overline{D}(t)}(\overline{M}) \ge
  \mu^{\ess}(\overline{M})$ for every
  $\overline{M} \in \widehat{\Div}(X)_{\bR}$.}
Hence for any $\lambda > 0$ we have
\begin{equation}\label{eq:lowerboundOmega}
  \Omega_{\overline{D}(t)}(\overline{D}) + \lambda \, \Omega_{\overline{D}(t)}(\overline{E}) =
   \Omega_{\overline{D}(t)}(\overline{D} + \lambda \, \overline{E}) 
 \ge  \mu^{\mathrm{ess}}(\overline{D} + \lambda \overline{E}).
\end{equation}
On the other hand
\begin{displaymath}
 \Omega_{\overline{D}(t)}(\overline{D}) =  \Omega_{\overline{D}(t)}(\overline{D}(t)) + t = \frac{\widehat{\vol}( \overline{D}(t))}{\vol(R^t(\overline{D}))} + t \le  (d+1)(\mu^{\ess}(\overline{D})-t) + t
\end{displaymath}
by the first formula in \eqref{eq:40} and Zhang's inequality (Theorem
\ref{thm:Zhang}). Therefore
\begin{displaymath}
  \lim_{t \to \mu^{\ess}(\overline{D})} \Omega_{\overline{D}(t)}(\overline{D}) = \mu^{\ess}(\overline{D}).
\end{displaymath}
Taking the infimum limit as $t$ approaches
$ \mu^{\mathrm{ess}}(\overline{D})$ from below in
\eqref{eq:lowerboundOmega} then gives
\begin{displaymath}
\liminf_{t \to \mu^{\ess}(\overline{D})} \Omega_{\overline{D}(t)}(\overline{E}) \ge \frac{ \mu^{\mathrm{ess}}(\overline{D} + \lambda \overline{E}) - \mu^{\ess}(\overline{D})}{\lambda},
\end{displaymath}
and we obtain
$\liminf_{t \to \mu^{\ess}(\overline{D})}
\Omega_{\overline{D}(t)}(\overline{E}) \ge \partial_{\overline{E}}\,
\mu^{\ess}(\overline{D})$ by letting $\lambda \to 0$. The rest of the
statement follows by applying this to $-\overline{E}$.
\end{proof}

\begin{proof}[Proof of Theorem \ref{thm:MainProduct}]
  If \eqref{eq:conditionMainProduct} holds, then by Lemma
  \ref{lemma:reformulation} and Theorem \ref{thm:MainSequences} the
  essential minimum function is differentiable at $\overline{D}$. The
  expression for the derivative in \eqref{eq:derivativeMain2} is given
  by Lemma \ref{lemma:chapeau}.
\end{proof}

\change{In Section
  \ref{sec:toric-varieties} we specialize Theorem
  \ref{thm:MainProduct} to the toric setting and translate its
  condition in terms of convex analysis. To this end we need the
  following comparison between the inradius of a big adelic
  $\bR$-divisor and the inradii of the geometric parts of a Fujita
  approximation sequence (Definition \ref{def:9}).

\begin{lemma} \label{lem:4} Let $\overline{B}$ be a big adelic
  $\bR$-divisor on $X$. Then for any
  Fujita approximation sequence $(\phi_n,\overline{P}_n)_n$ of
  $\overline{B}$ and any big $\mathbb{R}$-divisor $A$ on $X$ we have
\begin{displaymath}
  \frac{1}{d+1}\, \rho(\overline{B};A) \le  \liminf_n r(P_n;\phi_n^*A) \le  \limsup_n r(P_n;\phi_n^*A) \le
  \rho(\overline{B};A).
\end{displaymath}
\end{lemma}

\begin{proof} We have
  $ \lim_n \widehat{\vol}(\overline{P}_n) = \lim_n
  (\overline{P}_n^{d+1}) = \widehat{\vol}(\overline{B}) > 0$ by
  Theorem \ref{thm:arithmHodgeMoriwaki} and Definition \ref{def:9}.
  Hence $\overline{P}_n$ is big for every $n$ sufficiently large,
  which implies that  $P_n$ is big and so
  $(\phi_n,\overline{P}_n) \in \Theta(\overline{B})$. This gives the
  third inequality, whereas the second~is~obvious.

To prove the first inequality it suffices to show that for any $(\phi,\overline{P}) \in \Theta(\overline{B})$ we have
\begin{displaymath}
\liminf_n r(P_n;\phi_n^*A) \ge \frac{1}{d+1} r(P;\phi^*A).
\end{displaymath}
Using that the inradius is invariant under normal modifications, we
assume without loss of generality that $\phi = \Id_X$.

Let $n \in \bN$ and $\lambda \in \bR$ with
$0 \le \lambda < (d+1)^{-1}$.  Since $\overline{B} - \overline{P}$ is
pseudo-effective and $\overline{P}_n$ is nef we have
$(\overline{P}_n^d \cdot \phi_n^*\overline{P}) \le (\overline{P}_n^d
\cdot \phi_n^*\overline{B})$ by
Lemma~\ref{lemma:intersectionpseff}. From the arithmetic Siu's
inequality (Theorem \ref{thm:Yuan}) we deduce
\begin{equation}\label{eq:13}
\widehat{\vol}(\overline{P}_n - \lambda \, \phi_n^*\overline{P}) \ge (\overline{P}_n^{d+1}) - (d+1) (\overline{P}_n^d \cdot \phi_n^*\overline{B}) \, \lambda.
\end{equation}
We have
$\lim_n (\overline{P}_n^d \cdot \phi_n^*\overline{B}) =
\widehat{\vol}(\overline{B})$ by Proposition \ref{prop:orthogonal},  and so \eqref{eq:13} implies that 
$\overline{P}_n - \lambda \, \phi_n^*\overline{P}$ is big for every
sufficiently large $n$. Hence $P_n- \lambda \, \phi_n^*P$ is
big and so
\begin{displaymath}
 \liminf_ n r(P_n;\phi_n^*A) \ge \lambda \, r(P;A). 
\end{displaymath}
We conclude by letting $\lambda \to (d+1)^{-1}$.
\end{proof}
}

It would be very interesting to determine whether the condition in
Theorem \ref{thm:MainProduct} is actually a criterion for the
differentiability of the essential minimum.

\begin{question}\label{question:converse} If the essential minimum function is differentiable at $\overline{D}$, then does 
\begin{displaymath}
\liminf_{t \to \mu^{\mathrm{ess}}(\overline{D})} \frac{\mu^{\mathrm{ess}}(\overline{D}) -t}{\rho(\overline{D}(t))}=0
\end{displaymath}
necessarily hold? More optimistically, does this hold as soon as
$\overline{D}$ has the equidistribution property at every place?
\end{question}

In Section \ref{sec:partial-converse} we give a partial answer to this
question.

\section{Proof of  Theorem \ref{thm:MainSequences}  and complements}
\label{sec:Proof}

In this section we prove our main result and give some complements,
including an equidistribution theorem with a more flexible condition
for the sequence of semipositive approximations, a logarithmic
equidistribution theorem, and a partial converse to Theorem
\ref{thm:MainSequences}.

\subsection{A consequence of the arithmetic Siu's inequality}\label{subsec:consequenceYuan}
The following result is
a variant of Lemma \ref{lemma:ineqYuanv0} with an error 
term depending on an inradius.

\begin{proposition}\label{prop:ineqSiu}
  Let $\overline{P}, \overline{E} \in \widehat{\Div}(X)_{\mathbb{R}}$
  with $P$ big and $\overline{P}$ nef. Assume that there exists a nef
  adelic $\bR$-divisor $\overline{A}$ on $X$ such that $A$ is big,
  $\overline{A} + \overline{E}$ is pseudo-effective and
  $\overline{A} - \overline{E}$ is nef. There is a constant $c_d$
  depending only on $d$ such that for every $\lambda \ge 0$
\begin{displaymath}
  \widehat{\vol}_{\chi}(\overline{P} + \lambda \overline{E}) \ge (\overline{P}^{d+1})
  +(d+1) \, (\overline{P}^d\cdot \overline{E})\, \lambda - c_d \, 
   \max  \Big( 1 , \Big( \frac{\lambda}{r(P;A)} \Big)^{d-1} \Big) \, \frac{(\overline{P}^d \cdot \overline{A} )}{r(P;A)}  \, \lambda^{2}.
\end{displaymath}
\end{proposition}

Its proof combines Yuan's arithmetic analogue of Siu's inequality
(Theorem \ref{thm:Yuan}) with the following lemma, itself a
consequence of the arithmetic Hodge index theorem due to Yuan and
Zhang \cite{YuanZhang:hodge}. The first point is proven along the
lines of \cite[Theorem~2.7(2)]{Ikoma:IMRN}.

\begin{lemma}\label{lemma:Hodge} Let $\overline{P}$ and $\overline{A}$
  be nef adelic $\bR$-divisors on $X$ with $P,A$ big. Then
\begin{enumerate}[leftmargin=*]
\item\label{item:Hodgecompare}  for  $i=1, \ldots, d$ we have
\begin{math} 
(P^{d+1-i}\cdot A^{i-1}) \, ( \overline{P}^{d-i}\cdot \overline{A}^{i+1}) \le 2\, (P^{d-i} \cdot A^{i}) \, (\overline{P}^{d+1-i} \cdot \overline{A}^{i} ) ,
\end{math}
\item\label{item:Hodgeinradius} for  $i= 0, \ldots, d$ we have
  \begin{math}
      (\overline{P}^{d-i}\cdot \overline{A}^{i+1}) \le \Big(\dfrac{2}{r(P;A)}\Big)^{i} \,  (\overline{P}^{d} \cdot \overline{A}).
\end{math}
\end{enumerate}
\end{lemma}

\begin{proof}
\change{To prove \eqref{item:Hodgecompare} we start with a
  reduction. Using \cite[Proposition 4.4.2(1)]{Moriwaki:MAMS} and the
  continuity of arithmetic intersection numbers we find two sequences
  $(\overline{P}_n)_n$, $(\overline{A}_n)_n$ of semipositive adelic
  $\bR$-divisors over $P$ and $A$ respectively, with
\begin{displaymath}
\lim_n \,( \overline{P}_n^{d-i}\cdot \overline{A}_n^{i+1}) = (\overline{P}^{d-i}\cdot \overline{A}^{i+1}) \quad \text{ for every } 0 \le i \le d
\end{displaymath}
and such that for each $n \in \bN$ we have 
\begin{itemize}[leftmargin=*]
\item $\overline{P}_n - \overline{P}$ and $\overline{A}_n-\overline{A}$ are effective,
\item there exist models $(\mathcal{X}_n,\mathcal{P}_n)_n$ and $(\mathcal{X}_n,\mathcal{A}_n)_n$ of $(X,P)$ and $(X,A)$ over $\Spec(\cO_K)$ such that $g_{\overline{P}_n,v} =  g_{\mathcal{P}_n,v}$ and $g_{\overline{A}_n,v} =  g_{\mathcal{A}_n,v}$ for every non-Archimedean place $v$.
\end{itemize}
Since $\overline{P}_n - \overline{P}$ and $\overline{A}_n-\overline{A}$ are effective adelic $\bR$-divisors over $0 \in \Div(X)$, we have $h_{\overline{P}_n}(x) \ge h_{\overline{P}}(x)$ and $h_{\overline{A}_n}(x) \ge h_{\overline{A}}(x)$ for every $x \in X(\overline{K})$. It follows that $\overline{P}_n$ and $\overline{A}_n$ are nef for every $n$. Therefore, we can assume without loss of generality that there exist models $(\mathcal{X}, \mathcal{P})$ and $(\mathcal{X}, \mathcal{A})$ of $(X,P)$ and $(X,A)$ that define all the non-Archimedean Green functions of $\overline{P}$ and $\overline{A}$.}

 As $P$ and $A$ are big and nef we have
  $(P^d) = \vol(P) > 0$ and $(A^d) = \vol(A) >0$, which implies that
  $(P^{d+1-i} \cdot A^{i-1}) > 0$ by \cite[Theorem 1.6.1]{LazI}. Set
\begin{displaymath}
 \alpha = \frac{(P^{d-i} \cdot A^{i})}{( P^{d+1-i} \cdot A^{i-1})},
\end{displaymath}
so that  $((A - \alpha P) \cdot P^{d-i}\cdot A^{i-1}) = 0$. By \cite[Theorem 2.7(1)]{Ikoma:IMRN} we have
\begin{equation}\label{eq:ineqYuanZhang}
 ((\overline{A} - \alpha \overline{P})^2 \cdot \overline{P}^{d-i} \cdot \overline{A}^{i-1}) \le 0.
\end{equation}
On the other hand
\begin{displaymath}
((\overline{A} - \alpha \overline{P})^2 \cdot \overline{P}^{d-i}\cdot \overline{A}^{i-1}) = (\overline{P}^{d-i} \cdot \overline{A}^{i+1}) -2\alpha \, (\overline{P}^{d+1-i} \cdot \overline{A}^{i}) + \alpha^2 \, (\overline{P}^{d+2-i} \cdot \overline{A}^{i-1} ) 
\end{displaymath}
and $(\overline{P}^{d+2-i} \cdot \overline{A}^{i-1} )  \ge 0$ since both $\overline{P}$ and $\overline{A}$ are nef.  Therefore \eqref{item:Hodgecompare} follows from \eqref{eq:ineqYuanZhang}. 

Note that \eqref{item:Hodgeinradius} holds trivially for $i =0$. We
deduce the general case by induction on~$i$, applying
\eqref{item:Hodgecompare} and the inequality 
\begin{displaymath}
( P^{d-i} \cdot A^{i}) \le \frac{1}{r(P;A)} (P^{d+1-i}\cdot A^{i-1}),
\end{displaymath}
which follows from \eqref{eq:comparegeometriccap} and the fact that
$P - r(P;A)A$ is pseudo-effective.
\end{proof}

\begin{proof}[Proof of Proposition \ref{prop:ineqSiu}]
  Set $\overline{B} = \overline{A} - \overline{E}$. This is a nef
  adelic $\mathbb{R}$-divisor on $X$ and so by Theorem \ref{thm:Yuan}
  we have
\begin{displaymath}
\widehat{\vol}_{\chi}(\overline{P} + \lambda \overline{E})  = \widehat{\vol}_{\chi}(\overline{P} + \lambda \overline{A} - \lambda \overline{B}) \ge ((\overline{P} + \lambda \overline{A})^{d+1}) - (d+1) \, ((\overline{P} + \lambda \overline{A})^d \cdot \overline{B})\, \lambda .
\end{displaymath}
Expanding the right-hand side we find that it is equal to
\begin{multline*}
(\overline{P}^{d+1}) +(d+1)\, (\overline{P}^d \cdot \overline{E}) \,  \lambda + \sum_{i = 2}^{d+1} \binom{d+1}{i}  (\overline{P}^{d+1-i} \cdot \overline{A}^i)\, \lambda^i\\
-(d+1) \sum_{i = 1}^d \binom{d}{i} (\overline{P}^{d-i}\cdot \overline{A}^i \cdot \overline{B}) \, \lambda^{i+1} .
\end{multline*} 
Since $\overline{P}$ and $\overline{A}$ are nef, the first sum is nonnegative and therefore 
\begin{displaymath}
\widehat{\vol}_{\chi}(\overline{P} + \lambda \overline{E}) \ge (\overline{P}^{d+1}) +(d+1) \, (\overline{P}^d \cdot \overline{E}) \, \lambda  - (d+1)\sum_{i = 1}^d \binom{d}{i}  (\overline{P}^{d-i}\cdot \overline{A}^i \cdot \overline{B}) \, \lambda^{i+1}.
\end{displaymath}
By Lemma \ref{lemma:intersectionpseff} and the fact that  $2\overline{A} - \overline{B} = \overline{A} + \overline{E}$ is
pseudo-effective we have
$(\overline{P}^{d-i}\cdot \overline{A}^i \cdot \overline{B}) \le
2\,(\overline{P}^{d - i} \cdot \overline{A}^{i+1})$, $i =1, \ldots, d$, 
and by Lemma \ref{lemma:Hodge}\eqref{item:Hodgeinradius},
\begin{displaymath}
(\overline{P}^{d - i} \cdot \overline{A}^{i+1}) \le \Big( \frac{2}{r(P;A)} \Big)^i\, (\overline{P}^d \cdot \overline{A}).
\end{displaymath}
 Therefore $\widehat{\vol}_{\chi}(\overline{P} + \lambda \overline{E})$ is bounded from below by
\begin{displaymath}
  (\overline{P}^{d+1}) +(d+1) \, (\overline{P}^d \cdot \overline{E})\, \lambda  - (d+1)\,  (\overline{P}^d \cdot \overline{A})
  \sum_{i = 1}^d \binom{d}{i}  \frac{ 2^{i+1}\lambda^{i+1}}{r(P;A)^{i}},
\end{displaymath}
and the result follows. 
\end{proof}

The following consequence of Proposition \ref{prop:ineqSiu} plays a central role in our proof.

\begin{corollary}\label{coro:ineqSiu}
  Let $\overline{P},\overline{E} \in \widehat{\Div}(X)_{\mathbb{R}}$
  with $\overline{P}$ nef and $P$ big. Assume that there exists a nef
  adelic $\bR$-divisor $\overline{A}$ such that $A$ is big,
  $\overline{A} + \overline{E}$ is pseudo-effective and
  $\overline{A} - \overline{E}$ is nef. There exists a constant $c_d$
  depending only on $d$ such that
 \begin{displaymath}
   \mu^{\ess}(\overline{P} + \lambda \overline{E}) \ge \frac{(\overline{P}^{d+1})}{(d+1)\vol(P + \lambda E)}
   +  \frac{(\overline{P}^d\cdot \overline{E})}{(P^d)} \, \lambda - c_d\, \frac{(\overline{P}^d \cdot \overline{A} )}{(P^d)}\, \frac{\lambda^2}{r(P;A)}
\end{displaymath}
for every $0 \le \lambda < r(P;A)/2$.  In particular, if $E = 0$ then
\begin{equation*}
   \mu^{\ess}(\overline{P} + \lambda \overline{E}) \ge \frac{(\overline{P}^{d+1})}{(d+1)(P^d)}
   +  \frac{(\overline{P}^d\cdot \overline{E})}{(P^d)}\, \lambda - c_d\, \frac{(\overline{P}^d \cdot \overline{A} )}{(P^d)}\, \frac{\lambda^2}{r(P;A)}.
\end{equation*}
\end{corollary}

\begin{proof}
  Let $\lambda$ be a real number with $0 \le \lambda < r(P;A)/2$. By
  Lemma \ref{lemma:comparevolsinradius} we have
\begin{equation}\label{eq:ineqvolsinradius}
\Big( 1 - \frac{\lambda}{r(P;A)} \Big)^d (P^d) \le \vol(P + \lambda E) \le \Big( 1 + \frac{\lambda}{r(P;A)} \Big)^d (P^d).
\end{equation}
In particular, $P + \lambda E$ is big.
  By Zhang's inequality (Theorem \ref{thm:Zhang}) we have 
\begin{displaymath}
\mu^{\ess}(\overline{P} + \lambda \overline{E}) \ge \frac{\widehat{\vol}_{\chi}(\overline{P} + \lambda \overline{E})}{(d+1)\vol(P + \lambda E)}.
\end{displaymath}
Therefore Proposition \ref{prop:ineqSiu} implies that
\begin{equation}\label{eq:ineqcorYuan}
\mu^{\ess}(\overline{P} + \lambda \overline{E}) \ge  \frac{(\overline{P}^{d+1})}{(d+1)\vol(P + \lambda E)} +   \frac{(\overline{P}^d\cdot \overline{E})}{\vol(P + \lambda E)}\, \lambda - c_d   \, \frac{(\overline{P}^d\cdot \overline{A})}{\vol(P + \lambda E)}\, \frac{\lambda^2}{r(P;A)}
\end{equation}
for a constant $c_d > 0$ depending only on $d$.  By
\eqref{eq:ineqvolsinradius} we have
\begin{displaymath}
  \frac{(\overline{P}^d\cdot \overline{A})}{\vol(P + \lambda E)} \le  \Big(1- \frac{\lambda}{r(P;A)}\Big)^{-d} \,  \frac{(\overline{P}^d\cdot \overline{A})}{(P^d)} \le 2^d \,\frac{(\overline{P}^d\cdot \overline{A})}{(P^d)}.
\end{displaymath}
On the other hand, by Lemma \ref{lemma:intersectionpseff} we have
$|(\overline{P}^d \cdot \overline{E} )| \le (\overline{P}^d\cdot
\overline{A})$ since $\overline{P}$ is nef and
$\overline{A} \pm \overline{E}$ are pseudo-effective. Set $a = 1 $ if
$(\overline{P}^d\cdot \overline{E}) \ge 0$ and $a = -1$
otherwise. Then by~\eqref{eq:ineqvolsinradius} we have
\begin{multline*}
  \frac{(\overline{P}^d\cdot \overline{E})}{\vol(P + \lambda E)} \ge \frac{(\overline{P}^d\cdot \overline{E})}{(P^d)} \, \Big(1 +  a\, \frac{\lambda}{r(P;A)}\Big)^{-d} \\ \ge \frac{(\overline{P}^d\cdot \overline{E})}{(P^d)} - c'_d \,\frac{|(\overline{P}^d\cdot \overline{E})|}{(P^d)}\,  \frac{\lambda}{r(P;A)} 
  \ge  \frac{(\overline{P}^d\cdot \overline{E})}{(P^d)} - c'_d \, \frac{(\overline{P}^d\cdot \overline{A})}{(P^d)}\, \frac{\lambda}{r(P;A)} 
\end{multline*}
for some constant $c'_d$ depending only on $d$. The result follows by combining these inequalities with  \eqref{eq:ineqcorYuan}.
\end{proof}

\subsection{Proof of Theorem \ref{thm:MainSequences}}\label{subsec:proofMain}
Let $\overline{D} \in \widehat{\Div}(X)_{\bR}$ with $D$ big and
$(\phi_n, \overline{Q}_n)_{n}$ a sequence of semipositive
approximations of $\overline{D}$ satisfying 
condition~\eqref{eq:conditioninradiussequences}.

\begin{lemma}\label{lemma:condinradius}
  For every big $\bR$-divisor $A$ on $X$ we have
\begin{math}
  \displaystyle{ \lim_{n \rightarrow \infty}
    \frac{\mu^{\mathrm{ess}}(\overline{D}) -
      \mu^{\mathrm{abs}}(\overline{Q}_n)}{r(Q_n;\phi_n^*A)} = 0}.
\end{math}
\end{lemma}

\begin{proof} This a straightforward consequence of \eqref{eq:conditioninradiussequences} thanks to Lemma \ref{lemma:compareinradii}.
\end{proof}

For each $n \in \bN$ set
\begin{displaymath}
 \widetilde{Q}_n = \overline{Q}_n -
\mu^{\mathrm{abs}}(\overline{Q}_n)\,[\infty] \in \widehat{\Div}(X)_{\mathbb{R}}. 
\end{displaymath}
Note that
$\mu^{\mathrm{abs}}(\widetilde{Q}_n) =
\mu^{\mathrm{abs}}(\overline{Q}_n) -
\mu^{\mathrm{abs}}(\overline{Q}_n)=0$ and therefore $\widetilde{Q}_n$
is nef.

\begin{lemma}\label{lemma:expressionlimit}
  Let  $\overline{E} \in \widehat{\Div}(X)_{\mathbb{R}}$. We have  
\begin{displaymath}
\frac{(\widetilde{Q}_n^d \cdot \phi_n^*\overline{E})}{(Q_n^d)} = \frac{(\overline{Q}_n^d \cdot \phi_n^*\overline{E}) -  d\, \mu^{\mathrm{abs}}(\overline{Q}_n) \, (Q_n^{d-1} \cdot \phi_n^* E) }{(Q_n^d)} \quad \text{ for all } n\in \mathbb{N}
\end{displaymath} 
and 
\begin{math}
  \displaystyle{\lim_{n \to \infty} \Big( \frac{(\widetilde{Q}_n^d
      \cdot \phi_n^*\overline{E})}{(Q_n^d)} -\frac{(\overline{Q}_n^d
      \cdot \phi_n^*\overline{E}) - d\,
      \mu^{\mathrm{ess}}(\overline{D}) \, (Q_n^{d-1} \cdot \phi_n^* E)
    }{(Q_n^d)}\Big)=0.}
\end{math}
\end{lemma}

\begin{proof}
  The first equality follows from the multilinearity of the arithmetic
  intersection product and the
  formula~\eqref{eq:arithmeticvsgeometricintersection}.
  For the second set
\begin{displaymath}
\beta_{n} = (\mu^{\ess}(\overline{D}) - \mu^{\abs}(\overline{Q}_n)) \, \frac{(Q_n^{d-1} \cdot \phi_n^* E)}{(Q_n^d)}, \quad n\in \mathbb{N}, 
\end{displaymath}
so that this statement is equivalent to
$\lim_{n \to \infty} \beta_{n} = 0$. To see this, let $A$ be an ample
divisor such that $A \pm E$ are big. We have
$ |(Q_n^d \cdot \phi_n^* E)| \le (Q_n^d \cdot \phi_n^* A)$ by 
inequality~\eqref{eq:comparegeometriccap}, and using Lemma
\ref{lemma:inradiusandcap} we get
\begin{displaymath}
|\beta_{n}| \le (\mu^{\ess}(\overline{D}) - \mu^{\abs}(\overline{Q}_n)) \, \frac{(Q_n^{d-1} \cdot \phi_n^* A)}{(Q_n^d)} \le \frac{\mu^{\ess}(\overline{D}) - \mu^{\abs}(\overline{Q}_n)}{r(Q_n; \phi_n^*A)}.
\end{displaymath}
We conclude with Lemma \ref{lemma:condinradius}.
\end{proof}

\begin{lemma}\label{lemma:sup}
  For every nef
  $\overline{A}\in \widehat{\Div}(X)_{\mathbb{R}}$ with $A$ big we
  have
\begin{math}
\displaystyle{ \sup_{n \in \bN} \frac{(\widetilde{Q}_n^d \cdot \phi_n^*\overline{A})}{(Q_n^d)} < \infty}.
\end{math}
\end{lemma}

\begin{proof}
  Up to replacing $A$ by $\varepsilon A$ for a small $\varepsilon >0$
  we can assume that $D-A$ is big.  Then up to replacing $\overline{A}$
  by $\overline{A}- c\,[\infty]$ for $c \in \bR$ sufficiently large we can furthermore assume that
  $\overline{D} - \overline{A}$ is big thanks to Lemma~\ref{lem:6}.

  For every
  $n \in \bN$ and $ \lambda > 0$ we have
\begin{equation*} 
  \lambda\, (\widetilde{Q}_n^d \cdot \phi_n^*\overline{A}) \le \frac{1}{d+1} ((\widetilde{Q}_n + \lambda\,\phi_n^*\overline{A})^{d+1})
  \le  ((Q_n + \lambda\, \phi_n^*A)^d)\, \mu^{\mathrm{ess}}(\widetilde{Q}_n + \lambda\, \phi_n^*\overline{A}),
\end{equation*}
where the first inequality follows from the fact that both $\widetilde{Q}_n$
and $\overline{A}$ are nef, and the second from Zhang's inequality
(Theorem \ref{thm:Zhang}). 
By Lemma \ref{lemma:propertiesessmin}\eqref{item:essminincreases} and
the fact that $\phi_n^*\overline{D}- \overline{Q}_n$ and
$\overline{D}- \overline{A}$ are pseudo-effective we~have
\begin{displaymath}
\mu^{\mathrm{ess}}(\widetilde{Q}_n + \lambda\, \phi_n^*\overline{A}) = \mu^{\mathrm{ess}}(\overline{Q}_n + \lambda \,\phi_n^*\overline{A}) - \mu^{\mathrm{abs}}(\overline{Q}_n) \le (1+\lambda)\,\mu^{\mathrm{ess}}(\overline{D}) - \mu^{\mathrm{abs}}(\overline{Q}_n).
\end{displaymath}
Set $r_n = r(Q_n;\phi_n^*A)$ for short. Since $Q_n - r_{n} \,\phi_n^*A$ is pseudo-effective we also have
\begin{displaymath}
((Q_n + r_{n}\phi_n^*A)^d) = \vol(Q_n + r_{n} \phi_n^*A) \le \vol(2\, Q_n) = 2^d (Q_n^d).
\end{displaymath}
Combining the previous inequalities for $\lambda=r_{n}$ we get
\begin{displaymath}
  r_{n} \,(\widetilde{Q}_n^d \cdot \phi_n^*\overline{A}) \le 2^d\,((1+r_{n})\, \mu^{\mathrm{ess}}(\overline{D}) -\mu^{\mathrm{abs}}(\overline{Q}_n))
\, (Q_{n}^{d}).
\end{displaymath}
Hence
\begin{displaymath}
  \frac{(\widetilde{Q}_n^d\cdot \phi_n^*\overline{A})}{(Q_n^d)} \le 2^d \,  \frac{\mu^{\mathrm{ess}}(\overline{D})
    - \mu^{\mathrm{abs}}(\overline{Q}_n)}{r_{n}}  + 2^d \mu^{\mathrm{ess}}(\overline{D}).
\end{displaymath}
By Lemma \ref{lemma:condinradius} the right-hand side is upper-bounded
by some constant independent of $n$, and the result follows.
\end{proof}

\begin{proof}[Proof of Theorem \ref{thm:MainSequences}] Recall from
  Section~\ref{sec:variational} that for every
  $\overline{E} \in \widehat{\Div}(X)_{\bR}$, the one-sided derivative
  $ \partial_{\overline{E}}\, \mu^{\ess}(\overline{D})$ exists and
  $- \partial_{-\overline{E}}\, \mu^{\ess}(\overline{D}) \ge
  \partial_{\overline{E}}\, \mu^{\ess}(\overline{D})$. We claim that
\begin{equation}\label{eq:ineqproofsequences}
 \partial_{\overline{E}}\, \mu^{\ess}(\overline{D}) \ge \limsup_{n \rightarrow  \infty} \frac{(\widetilde{Q}_n^d \cdot \phi_n^* \overline{E})}{(Q_n^d)}.
\end{equation}
We  first prove this  when $\overline{E}$ is DSP. In that case, by
Lemma \ref{lem:1} there exists a big and nef
$\overline{A} \in \widehat{\Div}(X)_{\mathbb{R}}$ such that
$\overline{A} \pm \overline{E}$ are nef.  By Lemma \ref{lemma:sup},
\begin{displaymath}
\kappa \coloneqq \sup_{n \in \bN} \frac{(\widetilde{Q}_n^d\cdot \phi_n^*\overline{A})}{(Q_n^d)}
\end{displaymath}
is a real number. For each  $n \in \bN$ and any
$\lambda \ge 0$ such that $D + \lambda E$ is big we have
\begin{displaymath}
  \mu^{\ess}(\overline{D} + \lambda\overline{E}) - \mu^{\abs}(\overline{Q}_n) \ge \mu^{\ess}(\overline{Q}_n + \lambda\,\phi_n^*\overline{E}) - \mu^{\abs}(\overline{Q}_n) = \mu^{\ess}(\widetilde{Q}_n + \lambda\,\phi_n^*\overline{E}) 
\end{displaymath}
by Lemma \ref{lemma:propertiesessmin}\eqref{item:essminincreases}.
Since $\widetilde{Q}_n$ is nef we have
$(\widetilde{Q}_n)^{d+1} \ge 0$, and so by Corollary
\ref{coro:ineqSiu} applied to $\overline{P} = \widetilde{Q}_n$ we get
  \begin{displaymath}
\mu^{\mathrm{ess}}(\overline{D} + \lambda \overline{E})  - \mu^{\mathrm{abs}}(\overline{Q}_n)
 \ge  
\frac{(\widetilde{Q}_n^d \cdot \phi_n^* \overline{E})}{(Q_n^d)}\,  \lambda - \frac{c_d\, \kappa}{r(Q_n;\phi_n^*A)} \, \lambda^2
\end{displaymath}
for every $0 < \lambda < r(Q_n;\phi_n^*A)/2$, where $c_d$ is a
constant depending only on $d$. Rearranging this we obtain
  \begin{equation}\label{eq:ineqproofmain}
    \frac{\mu^{\mathrm{ess}}(\overline{D} + \lambda \overline{E}) - \mu^{\mathrm{ess}}(\overline{D})}{\lambda} \ge
    \frac{(\widetilde{Q}_n^d \cdot \phi_n^* \overline{E})}{(Q_n^d)} -\frac{\mu^{\mathrm{ess}}(\overline{D}))-\mu^{\mathrm{abs}}(\overline{Q}_n)}{\lambda}- \frac{c_d\, \kappa}{r(Q_n;\phi_n^*A)} \, \lambda.
\end{equation}

Set 
\begin{displaymath}
  \gamma_{n} =
   \frac{\mu^{\mathrm{ess}}(\overline{D}) - \mu^{\mathrm{abs}}(\overline{Q}_n)}{r(Q_n;\phi_n^*A)}   \quad  \text{ if } \mu^{\abs}(\overline{Q}_n) \ne \mu^{\ess}(\overline{D}) \and \gamma_{n} =   \frac{1}{n} \quad  \text{ otherwise},
\end{displaymath}
and then $ \lambda_n = r(Q_n;\phi_n^*A) \, \gamma_{n}^{1/2}$.
 By Lemma \ref{lemma:condinradius} we have
$\lim_{n\to\infty} \gamma_{n} = 0$, and so
$\lim_{n\to\infty} \lambda_n = 0$ \change{because
  $r(Q_n;\phi_n^*A) \le r(D;A)$ for every $n$.  Applying
  \eqref{eq:ineqproofmain} with $\lambda = \lambda_n$ for $n$ large
  enough} gives
  \begin{displaymath}
    \frac{\mu^{\mathrm{ess}}(\overline{D} + \lambda_n \overline{E}) - \mu^{\mathrm{ess}}(\overline{D})}{\lambda_n}
    \ge   \frac{(\widetilde{Q}_n^d \cdot \phi_n^* \overline{E})}{(Q_n^d)} -  \gamma_{n}^{1/2} - c_d \, \kappa \, \gamma_{n}^{1/2},
\end{displaymath}
and we obtain \eqref{eq:ineqproofsequences} by letting $n\to \infty$. 

We now consider the general case. By Lemma \ref{lemma:approachDSP},
for each $\varepsilon > 0$ there is a DSP
$\overline{E}' \in \widehat{\Div}(X)_{\bR}$ with
$\overline{E} - \overline{E}'$ and
$\overline{E}' - \overline{E} + \varepsilon[\infty]$
pseudo-effective. Then
$ \partial_{\overline{E}}\, \mu^{\ess}(\overline{D}) \ge
\partial_{\overline{E}'}\, \mu^{\ess}(\overline{D})$ by Lemma
\ref{lemma:propertiesessmin}\eqref{item:essminincreases}, and
Lemma \ref{lemma:intersectionpseff} together with the formula
\eqref{eq:arithmeticvsgeometricintersection} gives
\begin{displaymath}
  (\widetilde{Q}_n^d \cdot \overline{E}') \ge ( \widetilde{Q}_n^d \cdot ( \overline{E} - \varepsilon\, [\infty]))
  = (\widetilde{Q}_n^d \cdot \overline{E}) - \varepsilon\, (Q_n^d) 
\end{displaymath}
for all $n \in \bN$. By the DSP case we obtain
\begin{displaymath}
\partial_{\overline{E}}\, \mu^{\ess}(\overline{D}) \ge  \partial_{\overline{E}'}\, \mu^{\ess}(\overline{D}) \ge \limsup_{n\to \infty} \frac{(\widetilde{Q}_n^d \cdot \phi_n^* \overline{E}')}{(Q_n^d)} \ge  \limsup_{n\to \infty} \frac{(\widetilde{Q}_n^d \cdot \phi_n^* \overline{E})}{(Q_n^d)} - \varepsilon,
\end{displaymath}
and we conclude by letting $\varepsilon\to 0$.

Finally, applying \eqref{eq:ineqproofsequences} to $-\overline{E}$ and
$\overline{E}$ we get
\begin{displaymath}
\liminf_{n \rightarrow  \infty} \frac{(\widetilde{Q}_n^d \cdot \phi_n^* \overline{E})}{(Q_n^d)} \ge -\partial_{-\overline{E}}\, \mu^{\ess}(\overline{D}) \ge \partial_{\overline{E}}\, \mu^{\ess}(\overline{D}) \ge \limsup_{n \rightarrow  \infty} \frac{(\widetilde{Q}_n^d \cdot \phi_n^* \overline{E})}{(Q_n^d)}.
\end{displaymath}
Hence
$- \partial_{-\overline{E}}\, \mu^{\ess}(\overline{D}) =
\partial_{\overline{E}}\, \mu^{\ess}(\overline{D})$, and we conclude
with Lemmas \ref{lemma:diffconcave} and~\ref{lemma:expressionlimit}.
\end{proof}

\begin{remark}
  \label{rem:derivative}
  In the setting of Theorem \ref{thm:MainSequences}, it follows from
  Lemma \ref{lemma:expressionlimit} that if  condition
  \eqref{eq:conditioninradiussequences} is satisfied then the
  derivatives of the essential minimum function at~$\overline{D}$ can
  be alternatively expressed as
  \begin{displaymath}
    \partial_{\overline{E}}\,  \mu^{\ess}(\overline{D}) 
    = \lim_{n \rightarrow \infty}\frac{(\overline{Q}_n^d \cdot \phi_n^* \overline{E}) - d\,  \mu^{\abs}(\overline{Q}_n) \, (Q_n^{d-1} \cdot \phi_n^* E)}{(Q_n^d)} \quad \text{ for all } \overline{E} \in \widehat{\Div}(X)_{\bR}.   
  \end{displaymath}
\end{remark}

\subsection{A variant of Theorem \ref{thm:MainSequences}}\label{subsec:variantMain}
In the course of the proof, when applying Corollary~\ref{coro:ineqSiu}
to produce the lower bound \eqref{eq:ineqproofmain} we neglected the
term
\begin{displaymath}
\frac{(\widetilde{Q}_n^{d+1})}{\vol(Q_n + \lambda\, \phi_n^*E)}.
\end{displaymath}
As it turns out, taking this term into
account gives no improvement for an arbitrary adelic divisor
$\overline{E}$, though it does when $E = 0$.
This leads to the following slight refinement in this situation.
  
\begin{theorem}\label{thm:VariantSequencesEP}
  \change{Let $\overline{D}$ be an adelic $\bR$-divisor on $X$ with
    $D$ big,} and assume that there exists a sequence
  $(\phi_n, \overline{Q}_n)_{n}$ of semipositive approximations of
  $\overline{D}$ such that
  \begin{equation}
    \label{eq:52}
    \lim_{n \rightarrow \infty} \frac{1}{r(Q_n;\phi_n^*D)} \Big(\mu^{\ess}(\overline{D}) - \frac{(\overline{Q}_n^{d+1})}{(d+1)(Q_n^d)}\Big) = 0, \quad  \sup_{n \in \bN} \frac{\mu^{\ess}(\overline{D}) - \mu^{\abs}(\overline{Q}_n)}{r(Q_n;\phi_n^*D)} < \infty.
\end{equation}
Then $\overline{D}$ satisfies the equidistribution property at every
place $v \in \mathfrak{M}_K$, and its $v$-adic equidistribution
measure is $\nu_{\overline{D},v} = \lim_{n\to \infty} \nu_{n,v}$ with
$\nu_{n,v}$ the pushforward to $X_v^{\an}$ of the normalized
Monge-Ampère measure ${c_1(\overline{Q}_{n,v})^{\wedge d}}/{(Q_n^d)}$.
\end{theorem}

We just outline the proof of this result, as it is almost the same as
that of Theorem~\ref{thm:MainSequences}. Let
$(\phi_n, \overline{Q}_n)_{n}$ be a sequence satisfying the conditions of
Theorem \ref{thm:VariantSequencesEP} and set
$\widetilde{Q}_n = \overline{Q}_n - \mu^{\abs}(\overline{Q}_n) \,
[\infty]$ for each $n$. With this notation, the proof of Lemma
\ref{lemma:sup} remains valid thanks to the second condition in
\eqref{eq:52}, and so
\begin{displaymath}
\kappa \coloneqq \sup_{n \in \bN} \frac{(\widetilde{Q}_n^d \cdot \phi_n^*\overline{A})}{(Q_n^d)} < \infty.
\end{displaymath}  
By Proposition \ref{prop:diffvsEP} and Lemma \ref{lemma:diffconcave}\eqref{item:diffconcavealldirections},
it suffices to show that
\begin{equation}\label{eq:goalproofE=0}
-\partial_{-\overline{E}}\, \mu^{\ess}(\overline{D}) =  \partial_{\overline{E}}\, \mu^{\ess}(\overline{D}) = \lim_{n \to \infty} \frac{(\overline{Q}_n^d\cdot \phi_n^*\overline{E})}{(Q_n^d)}
\end{equation}
for any $\overline{E}\in \widehat{\Div}(X)_{\bR}$ over  \change{$E=0$.} We only
treat the case where $\overline{E}$ is DSP, as the general one follows
by density as in the proof of Theorem \ref{thm:MainSequences}.
Let $\overline{A} \in \widehat{\Div}(X)_{\bR}$ be big and nef with
$\overline{A} \pm \overline{E}$ nef.  By Corollary \ref{coro:ineqSiu},
there exists a constant $c_d$ such that
\begin{equation}\label{eq:ineqproofE=0}
  \mu^{\ess}(\overline{D} + \lambda \overline{E}) - \mu^{\abs}(\overline{Q}_n) \ge \frac{(\widetilde{Q}_n^{d+1})}{(d+1) \, (Q_n^d)}
  +  \frac{(\widetilde{Q}_n^d\cdot \phi_n^*\overline{E})}{(Q_n^d)} \, \lambda -  \frac{c_d \,\kappa}{r(Q_n; \phi_n^*A)} \, \lambda^2
\end{equation}
for every $n$ and $0 < \lambda \le r(Q_n;\phi_n^*A)/2$. Since $E = 0$,
by the formula \eqref{eq:arithmeticvsgeometricintersection} we have
\begin{displaymath}
  (\widetilde{Q}_n^{d+1}) = (\overline{Q}_n^{d+1}) - (d+1) \, (Q_n^d) \, \mu^{\abs}(\overline{Q}_n)
  \and  (\widetilde{Q}_n^d\cdot \phi_n^*\overline{E}) = (\overline{Q}_n^d\cdot \phi_n^*\overline{E}) .
\end{displaymath}
Combining this with \eqref{eq:ineqproofE=0} and dividing by $\lambda$ gives 
\begin{displaymath}
  \frac{\mu^{\ess}(\overline{D} + \lambda \overline{E}) - \mu^{\ess}(\overline{D})}{\lambda}
  \ge
 \frac{(\overline{Q}_n^d\cdot \phi_n^*\overline{E})}{(Q_n^d)} +
 \Big(\frac{(\overline{Q}_n^{d+1})}{(d+1)(Q_n^d)} -\mu^{\ess}(\overline{D}) \Big) \frac{1}{\lambda} - \frac{c_{d}\, \kappa}{r(Q_n; \phi_n^*A)} \, \lambda.
\end{displaymath}
As in the proof of Theorem \ref{thm:MainSequences}, a suitable choice of $\lambda = \lambda_n$ permits to
conclude that
\begin{displaymath}
\partial_{\overline{E}}\, \mu^{\ess}(\overline{D}) \ge \limsup_{n\to \infty} \frac{(\overline{Q}_n^d\cdot \phi_n^*\overline{E})}{(Q_n^d)}
\end{displaymath}
using the first condition in \eqref{eq:52}, and we obtain
\eqref{eq:goalproofE=0} by applying this to $-\overline{E}$.

\begin{remark}
  \label{rem:variant}
  This result gives more flexibility to construct the sequence of
  semipositive approximations $(\phi_n,\overline{Q}_n)_n$. For
  example, one can deduce Yuan's equidistribution theorem directly
  from Theorem \ref{thm:VariantSequencesEP} without using Theorem
  \ref{thm:criterionextremal}.

  However, Theorem~\ref{thm:VariantSequencesEP} is not more general
  than Theorem~\ref{thm:MainSequences}. In fact, starting with a
  sequence $(\phi_n,\overline{Q}_n)_n$ satisfying conditions
  \eqref{eq:52} one can modify it to construct another sequence
  satisfying condition \eqref{eq:conditioninradiussequences}, using
  arguments similar to those in the proof of Lemma
  \ref{lem:goodapprox}. Since we do not need this in the remainder of
  the text, we skip the proof of this technical claim.
\end{remark}

\subsection{Logarithmic equidistribution}\label{sec:logequi}
Let $\overline{D}$ be an adelic $\bR$-divisor on $X$ with $D$ big such
that there exists a sequence $(\phi_n, \overline{Q}_n)_n$ of
semipositive approximations of $\overline{D}$ satisfying the condition
of Theorem~\ref{thm:MainSequences}, \change{namely
\begin{displaymath}
 \lim_{n \rightarrow \infty} \frac{\mu^{\ess}(\overline{D}) - \mu^{\abs}(\overline{Q}_n)}{r(Q_n;\phi_n^*D)} = 0.
\end{displaymath}
}
By this result we have that $\overline{D}$ satisfies the
equidistribution property at every $v \in \mathfrak{M}_K$ with
equidistribution measure
\begin{equation}\label{eq:measurelogeq}
  \nu_{\overline{D},v} =\lim_{n \to \infty} \nu_{n,v},
\end{equation}
where $\nu_{n,v}$ denotes the pushforward to $X_v^{\an}$ of the
normalized $v$-adic Monge-Ampère measure of $\overline{Q}_n$.

In this
section we show that this property extends to functions with
logarithmic singularities along effective divisors satisfying a certain
numerical condition. Our presentation follows closely that of
Chambert-Loir and Thuillier in \cite{CLT}, adapting their arguments to
our setting.

\begin{definition} \label{def:3} Let $E$ be an effective divisor on
  $X$ and $v \in \mathfrak{M}_K$. A function
  $\varphi \colon X_v^{\an} \to \bR \cup \{\pm \infty \}$ has \emph{at
    most logarithmic singularities along $E$} if it is a real-valued
  continuous function on $X_v^{\an} \setminus \supp(E)_v^{\an}$ and
  every $x \in X_v^{\an}$ has a neighborhood $U \subset X_v^{\an}$
  together with an equation $f_{U}$ of $E_v^{\an}$ on $U$ and a real
  number $c_U$ such that 
  \begin{displaymath}
|\varphi|_v \le c_U \log |f_{U}|_v^{-1} \quad \text{ on } U. 
  \end{displaymath}
\end{definition}

Equidistribution measures can integrate functions with at most
logarithmic singularities along a divisor.

\begin{proposition}
\label{prop:Greenfunctionsintegrable}   
Let $v\in \mathfrak{M}_{K}$ and
$\varphi \colon X_v^{\an} \to \bR \cup \{\pm \infty \}$ a function
with at most logarithmic singularities along an effective divisor on
$X$. Then $\varphi$ is integrable with respect
to~$\nu_{\overline{D},v}$.
\end{proposition}

For the proof of this proposition we need the next auxiliary result.

\begin{lemma}\label{lemma:comparederivativeintegral}
  For every $\overline{E} \in \widehat{\Div}(X)_{\bR}$ with $E$
  effective we have
\begin{displaymath}
\partial_{\overline{E}} \, \mu^{\ess}(\overline{D}) \ge \sum_{v \in \mathfrak{M}_K} n_v   \int_{X_v^{\an}} g_{\overline{E},v} \, d\nu_{\overline{D},v}.
\end{displaymath}
\end{lemma}

\begin{proof} Let $n \in \bN$. By Lemma \ref{lemma:ineqheightnef} we have  
  \begin{equation}
    \label{eq:32}
  h_{\overline{Q}_n}([\phi_n^*E]) \ge d\, \mu^{\abs}(\overline{Q}_n) \, (Q_n^{d-1} \cdot \phi_n^*E),
\end{equation}
where the left-hand side denotes the height with respect to
$\overline{Q}_n$ of the $\bR$-Weil divisor associated to $\phi_n^*E$.
Applying the arithmetic Bézout formula \eqref{eq:Bezout} we deduce from this
\begin{displaymath}
  \frac{(\overline{Q}_n^d \cdot \phi_n^*\overline{E}) - d\, \mu^{\abs}(\overline{Q}_n)\, (Q_n^{d-1} \cdot \phi_n^*E) }{(Q_n^d)}
  \ge  \sum_{v \in \mathfrak{M}_K} n_v \int_{X_v^{\an}} g_{\overline{E},v} \, d\nu_{n,v}.  
\end{displaymath}
Letting $n\to \infty$, by Theorem
\ref{thm:MainSequences} and Remark \ref{rem:derivative} we have
\begin{displaymath}
  \partial_{\overline{E}} \, \mu^{\ess}(\overline{D}) \ge \liminf_{n
    \to \infty} \sum_{v \in \mathfrak{M}_K} n_v \int_{X_v^{\an}}
  g_{\overline{E},v} \, d\nu_{n,v} \\ \ge \sum_{v \in \mathfrak{M}_K}
  n_v \, \liminf_{n \to \infty} \int_{X_v^{\an}} \min (c,
  g_{\overline{E},v}) \, d\nu_{n,v} 
\end{displaymath}
for any $c \in \bR$. Since $E$ is effective  we have that
$g_{\overline{E},v}$ is bounded from below and so
$\min (c, g_{\overline{E},v}) \in C(X_v^{\an})$ for every $v$. Then
\begin{displaymath}
  \sum_{v \in \mathfrak{M}_K}
  n_v \, \liminf_{n \to \infty} \int_{X_v^{\an}} \min (c,
  g_{\overline{E},v}) \, d\nu_{n,v} =  \sum_{v \in \mathfrak{M}_K} n_v
  \int_{X_v^{\an}} \min (c, g_{\overline{E},v}) \,
  d\nu_{\overline{D},v}
\end{displaymath}
by the equidistribution property at every place. The statement
follows by letting $c\to\infty$ and applying the monotone convergence
theorem.
\end{proof}

\begin{proof}[Proof of Proposition~\ref{prop:Greenfunctionsintegrable}]
  Since $X_{v}^{\an}$ is compact, we can assume without loss of
  genericity that $\varphi=g_{\overline{E},v}$ for an adelic divisor
  $\overline{E}$ over an effective $E \in \Div(X)$.  Up to adding an
  adelic divisor over $0\in \Div(X)$ we can furthermore
  assume that $\overline{E}$ is effective. In this situation we have
  $g_{\overline{E},w} \ge 0$ for every $w \in \mathfrak{M}_K$ and so
  Lemma \ref{lemma:comparederivativeintegral} implies
\begin{displaymath}
   \infty > \partial_{\overline{E}} \, \mu^{\ess}(\overline{D}) \ge  n_v\int_{X_v^{\an}} g_{\overline{E},v} \, d\nu_{\overline{D},v},
\end{displaymath}
and so $ g_{\overline{E},v}$ is integrable with respect
to~$\nu_{\overline{D},v}$.
\end{proof}

The following is the main result of this section.

\begin{theorem}\label{thm:logequi} \change{Let $\overline{D}$ be an adelic $\bR$-divisor
    on $X$ with $D$ big satisfying the condition of
    Theorem~\ref{thm:MainSequences},} and let
  $\overline{E} \in \widehat{\Div}(X)$ with $E$ effective such that
\begin{equation}\label{eq:condlogequi}
\partial_{\overline{E}} \, \mu^{\ess}(\overline{D}) = \sum_{v \in \mathfrak{M}_K} n_v   \int_{X_v^{\an}} g_{\overline{E},v} \, d\nu_{\overline{D},v}.
\end{equation}
Then for every $\overline{D}$-small generic sequence
$(x_{\ell})_{\ell}$ in $X(\overline{K})$ and $v \in \mathfrak{M}_K$ we
have
 \begin{displaymath}
\lim_{\ell \to \infty}\int_{X_v^{\an}} \varphi \, d\delta_{O(x_\ell)_v} = \int_{X_v^{\an}} \varphi \, d\nu_{\overline{D},v}
 \end{displaymath}
 for any function $\varphi \colon X_v^{\an} \rightarrow \bR \cup \{\pm \infty \}$ with at most logarithmic singularities along $E$.
\end{theorem} 

\begin{proof} 
  By Proposition \ref{prop:Greenfunctionsintegrable} and \cite[Lemma
  6.3]{CLT} it suffices to consider the case
  $\varphi = g_{\overline{E},v}$.  Then by
  Proposition~\ref{prop:diffvsHCP} and Theorem \ref{thm:MainSequences}
  we have
 \begin{displaymath}
  \partial_{\overline{E}} \, \mu^{\ess}(\overline{D}) =  \lim_{\ell \to \infty} h_{\overline{E}}(x_{\ell}) = \lim_{\ell \to \infty} \sum_{v\in \mathfrak{M}_K} n_v \int_{X_v^{\an}} g_{\overline{E},v} \, d\delta_{O(x_{\ell})_v},
 \end{displaymath}
 and so  condition \eqref{eq:condlogequi} implies
 \begin{multline}
   \label{eq:54}
  \sum_{v\in \mathfrak{M}_K}
  n_v \int_{X_v^{\an}} g_{\overline{E},v} \, d\nu_{\overline{D},v}
  = \lim_{\ell \to \infty} \sum_{v\in \mathfrak{M}_K} n_v \int_{X_v^{\an}} g_{\overline{E},v} \, d\delta_{O(x_{\ell})_v}
 \\  \ge  \sum_{v\in \mathfrak{M}_K} n_v \liminf_{\ell \to \infty} \int_{X_v^{\an}} g_{\overline{E},v} \,  d\delta_{O(x_{\ell})_v}.   
 \end{multline}
 On the other hand, for every $v$ and any $c\in \mathbb{R}$ we have $ \min(c,g_{\overline{E},v})\in C(X_v^{\an})$ and so 
\begin{displaymath}
  \liminf_{\ell \to \infty} \int_{X_v^{\an}} g_{\overline{E},v} \,  d\delta_{O(x_{\ell})_v}
  \ge  \liminf_{\ell \to \infty} \int_{X_v^{\an}} \min(c, g_{\overline{E},v}) \,  d\delta_{O(x_{\ell})_v} = \int_{X_v^{\an}} \min(c, g_{\overline{E},v}) \,  d\nu_{\overline{D},v}
\end{displaymath}
by the equidistribution property. Letting $c\to \infty$ we get that
\begin{displaymath}
\liminf_{\ell \to \infty} \int_{X_v^{\an}} g_{\overline{E},v} \,  d\delta_{O(x_{\ell})_v} \ge \int_{X_v^{\an}} g_{\overline{E},v} \, d\nu_{\overline{D},v}.
\end{displaymath}
Summing up these inequalities over all the places and taking
\eqref{eq:54} into account we deduce that they are in in fact
equalities.  We conclude by observing that such equality remains true
when $(x_{\ell})_{\ell}$ is replaced by an arbitrary subsequence,
since the latter remains generic and $\overline{D}$-small. Therefore
\begin{displaymath}
\lim_{\ell \to \infty} \int_{X_v^{\an}} g_{\overline{E},v} \,  d\delta_{O(x_{\ell})_v} = \int_{X_v^{\an}} g_{\overline{E},v} \, d\nu_{\overline{D},v}
\end{displaymath}
as desired. 
\end{proof}

\begin{remark}
  \label{rem:1}
  Condition \eqref{eq:condlogequi} is independent of the choice of
  the adelic structure over $E$. Indeed, let $\overline{E}'$ be
  another adelic divisor over $E$. Then there exists a finite set
  $\mathfrak{S} \subset \mathfrak{M}_K$ and a collection
  $\varphi_v \in C(X_v^{\an})^{G_v}$, $v\in \mathfrak{S}$, such that
  $\overline{E}' - \overline{E} = \sum_{v \in \mathfrak{S}}
  \overline{0}^{\varphi_v}$. Since the essential minimum function is
  differentiable at $\overline{D}$, we have
\begin{displaymath}
\partial_{\overline{E}'} \, \mu^{\ess}(\overline{D}) = \partial_{\overline{E}} \, \mu^{\ess}(\overline{D})  + \sum_{v \in \mathfrak{S}}\partial_{\overline{0}^{\varphi_v}} \, \mu^{\ess}(\overline{D}).
\end{displaymath}
Setting $\varphi_v = 0$ for
$v \in \mathfrak{M}_K \setminus \mathfrak{S}$, Proposition
\ref{prop:diffvsEP} together with \eqref{eq:condlogequi} then gives
\begin{displaymath}
\partial_{\overline{E}'} \, \mu^{\ess}(\overline{D}) = \sum_{v \in \mathfrak{M}_K} n_v  \int_{X_v^{\an}} (g_{\overline{E},v} + \varphi_v)\, d\nu_{\overline{D},v} =  \sum_{v \in \mathfrak{M}_K} n_v   \int_{X_v^{\an}} g_{\overline{E}',v} \, d\nu_{\overline{D},v}
\end{displaymath}
and so $\overline{E}'$ also verifies this condition.
\end{remark}

If the sequences of probability measures approaching the
equidistribution measures are eventually constant, we can rephrase the
condition in Theorem \ref{thm:logequi} in terms of the gaps in Zhang's
lower bound for the heights of Weil divisors in \eqref{eq:32}.

\begin{corollary}
  \label{cor:logequi} 
  \change{Let $\overline{D}$ be an adelic $\bR$-divisor on $X$ with $D$ big such
that there exists a sequence $(\phi_n, \overline{Q}_n)_n$ of
semipositive approximations of $\overline{D}$ satisfying the condition
of Theorem~\ref{thm:MainSequences}. 
With notation as in \eqref{eq:measurelogeq}, assume that there exists
$n_0 \in \bN$ such that
  \begin{displaymath}
  \nu_{n,v} = \nu_{\overline{D},v} \quad \text{ for all }  n \ge n_0 \text{ and } v\in \mathfrak{M}_K.
  \end{displaymath}}
 Let
  $E$ be an effective divisor on $X$ such that
\begin{displaymath}
  \lim_{n \to \infty}  \frac{h_{\overline{Q}_n}([\phi_n^*E]) - d \, \mu^{\abs}(\overline{Q}_{n}) \, (Q_n^{d-1}\cdot \phi_n^*E)}{(Q_n^d)} = 0.
\end{displaymath}
Then for every $\overline{D}$-small generic sequence
$(x_{\ell})_{\ell}$ in $X(\overline{K})$ and $v \in \mathfrak{M}_K$ we
have
 \begin{displaymath}
\lim_{\ell \to \infty}\int_{X_v^{\an}} \varphi \, d\delta_{O(x_\ell)_v} = \int_{X_v^{\an}} \varphi \, d\nu_{\overline{D},v}
 \end{displaymath}
 for any function
 $\varphi \colon X_v^{\an} \rightarrow \bR \cup \{\pm \infty \}$ with
 at most logarithmic singularities along $E$.
\end{corollary}

\begin{proof}
  Let $\overline{E}$ be an adelic divisor over $E$. By Remark
  \ref{rem:derivative} and the arithmetic Bézout formula
  \eqref{eq:Bezout} we have
\begin{align*}
  \partial_{\overline{E}} \, \mu^{\ess}(\overline{D})& =  \lim_{n \to \infty} \Big( \frac{h_{\overline{Q}_n}([\phi_n^*E]) - d \, \mu^{\abs}(\overline{Q}_{n}) \, (Q_n^{d-1}\cdot \phi_n^*E)}{(Q_n^d)} + \sum_{v \in \mathfrak{M}_K} n_v \int_{X_v^{\an}} g_{\overline{E},v} \, d\nu_{n,v}\Big)\\
                                                     & =  \lim_{n \to \infty} \frac{h_{\overline{Q}_n}([\phi_n^*E]) - d \, \mu^{\abs}(\overline{Q}_{n}) \, (Q_n^{d-1}\cdot \phi_n^*E)}{(Q_n^d)} + \sum_{v \in \mathfrak{M}_K} n_v \int_{X_v^{\an}} g_{\overline{E},v} \, d\nu_{\overline{D},v},
\end{align*}
and so $\overline{E}$ verifies condition \eqref{eq:condlogequi}.
\end{proof}

\begin{remark}
\label{rem:21}
The assumption in Corollary \ref{cor:logequi} that the sequences of
probability measures $(\nu_{n,v})_{n}$ are eventually constant is
verified in the setting of dynamical systems (Theorem
\ref{thm:dynlogEP}). It would be interesting to know whether this
corollary remains valid without this technical assumption.
\end{remark}

Corollary \ref{cor:logequi} allows to recover the logarithmic
equidistribution theorem from \cite[Theorem~1.2]{CLT} as follows.  Let
$\overline{D}$ be a semipositive adelic $\mathbb{R}$-divisor on $X$
with $D$ ample such~that
\begin{equation}
  \label{eq:67}
 \mu^{\ess}(\overline{D}) = \frac{(\overline{D}^{d+1})}{(d+1)(D^d)},
 \end{equation}
 and let $\overline{E} \in \widehat{\Div}(X)$ with $E$ is effective
 such that
 \begin{displaymath}
 \frac{h_{\overline{D}}([E])}{d \, (D^{d-1}\cdot E)} =  \mu^{\ess}(\overline{D}).
 \end{displaymath}
 \change{Equality
   \eqref{eq:67} implies
   $\mu^{\ess}(\overline{D}) = \mu^{\abs}(\overline{D})$ by Theorem \ref{thm:criterionextremal}.  Therefore
    condition \eqref{eq:conditioninradiussequences} in Theorem
   \ref{thm:MainSequences} as well as that of Corollary
   \ref{cor:logequi} are} trivially satisfied for the constant
 sequence $(\phi_n, \overline{Q}_n)=(\Id_X,\overline{D})$,
 $n\in \mathbb{N}$, thus giving the stated equidistribution for
 functions with at most logarithmic singularities along $E$.

\subsection{A partial converse}\label{sec:partial-converse}

\change{Let $\overline{D}$ be an adelic $\mathbb{R}$-divisor on $X$ with $D$
big.} The next result answers Question \ref{question:converse} under
an additional technical assumption, roughly saying that $\overline{D}$
has a suitable upper bound for which Zhang's inequality is an
equality. As we will see in Section \ref{sec:toric-varieties}, this
assumption is always satisfied for semipositive toric adelic
$\bR$-divisors (Proposition \ref{prop:15}), which will allow us to
give an affirmative answer to this question in this setting (Theorem
\ref{thm:5}).

\change{Recall that for any big $\overline{B}\in \widehat{\Div}(X)_{\bR}$ we
denote by $\rho(\overline{B})$ the inradius of $\overline{B}$ with
respect to $B$ (Definition \ref{def:rho}).}

\begin{proposition}\label{prop:conversequasican}
  Assume that there exists a semipositive adelic $\bR$-divisor
  $\overline{D}'$ over~$D$ such that $\overline{D}' - \overline{D}$ is
  pseudo-effective and
  $\mu^{\ess}(\overline{D}') = \mu^{\abs}(\overline{D}') =
  \mu^{\ess}(\overline{D})$. Then the following conditions are
  equivalent:
\begin{enumerate}[leftmargin=*]
\item \label{item:16} $\displaystyle{\lim_{t \to \mu^{\ess}(\overline{D})} \frac{\mu^{\ess}(\overline{D})-t}{ \rho(\overline{D}(t))} = 0}$,
\item \label{item:18} the essential minimum function is differentiable at $\overline{D}$,
\item \label{item:19} $\overline{D}$ has the equidistribution property
  at every place $v \in \mathfrak{M}_K$.
\end{enumerate}
\end{proposition} 

We deduce this result as a special case of the next lemma.

\begin{lemma}\label{lem:converse} Assume that there exists a nef
  $\overline{A} \in \widehat{\Div}(X)_{\bR}$ with $A$ big such that
  \begin{math}
  \partial_{\overline{A}}\, \mu^{\ess}(\overline{D}) = 0.
  \end{math} Then the following conditions are equivalent:
\begin{enumerate}[leftmargin=*]
\item \label{item:converselim}
  $\displaystyle{\lim_{t \to \mu^{\ess}(\overline{D})}
    \frac{\mu^{\ess}(\overline{D})-t}{\rho(\overline{D}(t))} = 0}$,
\item \label{item:conversediff} the essential minimum function is
  differentiable at $\overline{D}$,
\item \label{item:conversediffA}
  $\partial_{-\overline{A}}\, \mu^{\ess}(\overline{D}) = 0$.
\end{enumerate}

If moreover $A = D$, then they are equivalent to the condition:
\begin{enumerate}[leftmargin=*] \setcounter{enumi}{3}
\item\label{item:converseEPA} $\overline{D}$ has the equidistribution
  property at every place $v \in \mathfrak{M}_K$.
\end{enumerate}
\end{lemma}

\begin{proof}
  We have that \eqref{item:converselim} implies
  \eqref{item:conversediff} by Theorem \ref{thm:MainProduct}.
  \change{If the essential minimum function is differentiable at
    $\overline{D}$ then
    $\partial_{-\overline{A}}\, \mu^{\ess}(\overline{D}) =
    -\partial_{\overline{A}}\, \mu^{\ess}(\overline{D}) = 0$, and so
    \eqref{item:conversediff} implies~\eqref{item:conversediffA}. Now
  if $\eqref{item:conversediffA}$ holds} then by Lemma
  \ref{lemma:chapeau}
\begin{equation}\label{eq:converselimOmega}
  \lim_{t \to \mu^{\ess}(\overline{D})} \frac{(\langle \overline{D}(t)^d \rangle \cdot \overline{A})}{\vol(R^t(\overline{D}))} =0.
\end{equation}
Fix a real number $t < \mu^{\ess}(\overline{D})$ and let
$(\phi_n, \overline{P}_n)_n$ be a Fujita approximation sequence of
$\overline{D}(t)$. Then by Proposition \ref{prop:orthogonal}  
\begin{equation}\label{eq:Pnconverse}
 \lim_{n \to \infty} \frac{(\overline{P}_n^d \cdot \phi_n^*\overline{A})}{(P_n^d)} = \frac{(\langle \overline{D}(t)^d \rangle \cdot \overline{A})}{\vol(R^t(\overline{D}))}  \and     \lim_{n\to \infty} \mu^{\ess}(\overline{P}_n) = \mu^{\ess}(\overline{D}) -t.
\end{equation}
Set
$\overline{P}'_n = \overline{P}_n - \mu^{\ess}(\overline{P}_n) \,
[\infty]$, $n\in \mathbb{N}$. Then $\overline{P}'_n$ is
pseudo-effective by Theorem \ref{thm:Essmin}, and by the formula
\eqref{eq:arithmeticvsgeometricintersection} we have
\begin{displaymath}
  \frac{(\overline{P}_n^d \cdot \change{\phi_n^*} \overline{A})}{(P_n^d)} = \frac{(\overline{P}_n^{d-1} \cdot \phi_n^*\overline{A} \cdot \overline{P}'_n)}{(P_n^d)}
  +\mu^{\ess}(\overline{P}_n) \, \frac{(P_n^{d-1}\cdot \phi_n^*A)}{(P_n^d)}.
\end{displaymath}
Since $\overline{P}_n$ and $\overline{A}$ are nef, by Lemma
\ref{lemma:intersectionpseff} the first summand is non-negative and so
\begin{displaymath}
\frac{(\overline{P}_n^d \cdot \phi_n^*\overline{A})}{(P_n^d)} \ge \mu^{\ess}(\overline{P}_n) \, \frac{(P_n^{d-1}\cdot \phi_n^*A)}{(P_n^d)} \ge \frac{\mu^{\ess}(\overline{P}_n)}{d \, r(P_n;\phi_n^*A)} \ge \frac{\mu^{\ess}(\overline{P}_n)}{d \, \rho(\overline{D}(t);A)},
\end{displaymath}
where the second inequality follows from Lemma
\ref{lemma:inradiusandcap} and the third  from the definition
of the inradius of $\overline{D}(t)$ with respect to $A$. Letting
$n \to \infty$ and applying \eqref{eq:Pnconverse} we get
\begin{displaymath}
 \frac{(\langle \overline{D}(t)^d \rangle \cdot \overline{A})}{\vol(R^t(\overline{D}))}\ge \frac{\mu^{\ess}(\overline{D})-t}{d \, \rho(\overline{D}(t);A)}.
\end{displaymath}
By Lemma \ref{lemma:compareinradii} there is $c > 0$ such that
$\rho(\overline{D}(t);A) \le c \, \rho(\overline{D}(t);D) =
c \,\rho(\overline{D}(t))$ for every $t <
\mu^{\ess}(\overline{D})$. Therefore \eqref{item:converselim} follows
by letting $t \to \mu^{\ess}(\overline{D})$ and using
\eqref{eq:converselimOmega}.

For the last statement, by Proposition \ref{prop:diffvsEP} and Lemma
\ref{lemma:diffconcave}\eqref{item:diffconcavealldirections} 
condition \eqref{item:converseEPA} is equivalent to the fact that the
essential minimum function is differentiable along the subspace of
adelic divisors on $X$ over the zero divisor.  In particular, it is
implied by~\eqref{item:conversediff}.

\change{Finally} assume that $A = D$ and that \eqref{item:converseEPA}
holds.  Then $ \overline{E}\coloneqq \overline{A} - \overline{D}$ is
an adelic divisor over $E = 0$ and so the essential minimum function
is differentiable at $\overline{D}$ along $\overline{E}$. Since this
function is clearly differentiable along $\overline{D}$ \change{(with
  derivative $\mu^{\ess}(\overline{D})$),} by
Lemma~\ref{lemma:diffconcave}\eqref{item:diffconcavealldirections} it
is also differentiable along $\overline{A}$. \change{Therefore
  $\partial_{-\overline{A}}\, \mu^{\ess}(\overline{D}) =
  -\partial_{\overline{A}}\, \mu^{\ess}(\overline{D}) = 0$, giving
  \eqref{item:conversediffA} and completing the proof.}  
\end{proof}

\begin{proof}[Proof of Proposition \ref{prop:conversequasican}] Let
  $\overline{A} = \overline{D}' - \mu^{\ess}(\overline{D}) \,
  [\infty]$. Then
  $\mu^{\ess}(\overline{A}) = \mu^{\abs}(\overline{A}) = 0$, and in
  particular $\overline{A}$ is nef. Thus by Lemma \ref{lem:converse}
  it suffices to show that
  $\partial_{\overline{A}}\, \mu^{\ess}(\overline{D}) = 0$. This is
  clear, since by Lemma
  \ref{lemma:propertiesessmin}\eqref{item:essminsuperadd} for every
  $\lambda > 0$ we have
\begin{displaymath}
  \mu^{\ess}(\overline{D}+\lambda \overline{A}) \ge \mu^{\ess}(\overline{D}) + \lambda \mu^{\ess}(\overline{A}) = \mu^{\ess}(\overline{D}),
\end{displaymath}
whereas Lemma \ref{lemma:propertiesessmin}\eqref{item:essminincreases}
gives the converse inequality, namely
\begin{displaymath}
\mu^{\ess}(\overline{D}) = \mu^{\ess}(\overline{D}') =  \mu^{\ess}(\overline{D}' + \lambda \overline{D}') - \lambda \mu^{\ess}(\overline{D}') = \mu^{\ess}(\overline{D}' + \lambda \overline{A}) \ge \mu^{\ess}(\overline{D} + \lambda \overline{A}).
\end{displaymath}
\end{proof}

\section{Toric varieties}
\label{sec:toric-varieties}

Here we study the differentiability of the essential minimum function
in the toric setting and its consequences for the equidistribution of
the Galois orbits of small generic sequences of algebraic points.  To
this end, first we review the algebraic and Arakelov geometries of
toric varieties following
\cite{BPS:asterisque,BPS:smthf,BMPS:positivity, BPRS:dgopshtv} and
study the inradii and positive intersection numbers of (adelic)
$\mathbb{R}$-divisors on toric varieties.  We then prove our toric
differentiability results (Theorems~\ref{thm:1} and ~\ref{thm:5}) and
extend the corresponding equidistribution properties to test functions
with logarithmic singularities along special hypersurfaces (Theorem
\ref{thm:2} and Corollary~\ref{cor:3}).

\subsection{Geometric and arithmetic aspects}
\label{sec:toric-setting}
Let $\mathbb{T}\simeq \Gm^{d}$ be a split $d$-dimensional torus
over~$K$ and set
\begin{displaymath}
  M=\Hom(\mathbb{T},\Gm) \and  N=\Hom(\Gm,\mathbb{T})
\end{displaymath}
for its lattices of characters and of co-characters. These are both
isomorphic to $ \mathbb{Z}^{d}$ and dual of each other, that is
$M=N^{\vee}$ and $N=M^{\vee}$. Set then
$N_{\mathbb{R}}=N\otimes_{\mathbb{Z}}\mathbb{R}$ and
$M_{\mathbb{R}}=M\otimes_{\mathbb{Z}}\mathbb{R}$. These vector spaces
are also dual of each other, and for $u\in N_{\mathbb{R}}$ and
$x\in M_{\mathbb{R}}$ we denote their pairing by
$\langle u,x \rangle$.  Let also $K[M]$ be the group algebra of~$M$,
and for each $m\in M$ let $\chi^{m}\in K[M]$ be the corresponding
monomial.

Let $X$ be a \emph{projective toric variety} with torus $\mathbb{T}$ and
$D$ a \emph{toric $\mathbb{R}$-divisor} on it. By this we mean a normal
projective variety over $K$ containing $\mathbb{T}$ as an open subset
and equipped with an action of this torus extending its action onto
itself by translations, together with an $\mathbb{R}$-divisor that is
invariant under this action.

Classically toric varieties and $\mathbb{R}$-divisors are constructed
and classified with polyhedral objects. Thus to $X$ and $D$
respectively correspond a \emph{fan} $\Sigma_{X}$ and an
\emph{$\mathbb{R}$-virtual support function}~$\Psi_{D}$.  The fan
$\Sigma_{X}$ is a polyhedral complex of strongly convex cones defined
over $N$ covering the whole of $N_{\mathbb{R}}$, whereas the
$\mathbb{R}$-virtual support function~$\Psi_{D}$ is a real-valued
function on~$N_{\mathbb{R}}$ that is linear on each of the cones of
this fan.  We also associate to $D$ the subset of $M_{\mathbb{R}}$
defined as
\begin{equation*}
  \Delta_{D}=\{x\in M_{\mathbb{R}}\mid \langle u,x\rangle \ge \Psi_{D}(u) \text{ for every } u\in N_{\mathbb{R}}\}.
\end{equation*}
It  is a \emph{quasi-rational polytope}, that is a polytope with rational
slopes.

The positivity invariants and properties of $D$ can be read from its
$\mathbb{R}$-virtual support function and polytope.  For instance, the
volume of $D$ is given by
\begin{displaymath}
\vol(D)=d!\vol_{M}(\Delta_{D})  
\end{displaymath}
where $\vol_{M}$ denotes the Haar measure on $M_{\mathbb{R}}$
normalized so that $M$ has covolume $1$. In particular, if $D$ is nef
then $(D^{d})=d! \vol_{M}(\Delta_{D})$. More generally, for a family
$D_{i}$, $i=1,\dots, d$, of nef toric $\mathbb{R}$-divisors on $X$ we
have
\begin{equation}
  \label{eq:61}
(D_{1} \cdots D_{d}) = \MV_{M}(\Delta_{D_{1}},\dots, \Delta_{D_{d}}),
\end{equation}
where $\MV_{M}$ denotes the mixed volume function with respect to
$\vol_{M}$.

The $\mathbb{R}$-divisor $D$ is pseudo-effective if and only if
$\Delta_{D}\ne \emptyset$, and is big if and only
if~$\dim(\Delta_{D})=d$. In addition, $D$ is nef if and only if
$\Psi_{D}$ is concave. For a nef toric $\mathbb{R}$-divisor $E$ we
have that $D-E$ is pseudo-effective if and only if there exists
$x \in M_{\mathbb{R}}$ such that $x+\Delta_{E}\subset \Delta_{D}$.
All of this can be found in \cite[Section~4]{BMPS:positivity}.

To study the arithmetic counterpart of these constructions and
results, for each place $v\in \mathfrak{M}_{K}$ we denote
by~$\mathbb{S}_{v}$ the \emph{compact torus} of the $v$-adic analytic
torus~$\mathbb{T}_{v}^{\an}$ \cite[Section 4.2]{BPS:asterisque}. In the Archimedean case $\mathbb{S}_{v}$
is isomorphic to the real torus $(S^{1})^{d}$, whereas in the
non-Archimedean case it is an analytic subgroup of
$\mathbb{T}_{v}^{\an}$ in the sense of Berkovich.  We also consider
the \emph{valuation~map}
\begin{displaymath}
  \val_{v}\colon \mathbb{T}_{v}^{\an}\longrightarrow N_{\mathbb{R}}.
\end{displaymath}
With  a splitting of the torus, we can identify the dense subset
$\mathbb{T}_{v}^{\an}(\mathbb{C}_{v})$ with
$(\mathbb{C}_{v}^{\times})^{d}$ and the vector space $N_{\mathbb{R}}$
with $\mathbb{R}^{d}$. In these coordinates, the valuation map writes
down as
\begin{displaymath}
  \val_{v}(x_{1},\dots, x_{d})=(-\log|x_{1}|_{v}, \dots, -\log|x_{d}|_{v}).
\end{displaymath}

Now let $\overline{D}$ be a \emph{toric adelic $\mathbb{R}$-divisor} on $X$, 
that is an adelic $\mathbb{R}$-divisor on $X$ whose geometric
$\mathbb{R}$-divisor $D$ is toric and whose $v$-adic Green function
$g_{\overline{D},v}$ is invariant under the action of~$\mathbb{S}_{v}$
for every $v$. Toric adelic $\mathbb{R}$-divisors over $D$ can be
constructed and classified with adelic families of functions on
$N_{\mathbb{R}}$ whose behavior at infinity is governed by the
$\mathbb{R}$-virtual support function~$\Psi_{D}$ \cite[Proposition
4.3.10]{BPS:asterisque}, \cite[Proposition 4.16]{BMPS:positivity}.
Accordingly we denote by
\begin{displaymath}
\psi_{\overline{D},v}\colon N_{\mathbb{R}}\longrightarrow \mathbb{R}, \quad {v\in \mathfrak{M}_{K}},
\end{displaymath}
the family of \emph{metric functions} associated to $\overline{D}$.
For each $v$, the $v$-adic metric function is defined as
\begin{equation}
  \label{eq:72}
 \psi_{\overline{D},v}(u) =-g_{\overline{D},v}(x) \quad \text{ for every }
 u\in N_{\mathbb{R}} \text{ and } x\in \val_{v}^{-1}(u).
\end{equation}
It is continuous and has bounded difference with respect to $\Psi_{D}$
for every $v$, and it is equal to~$\Psi_{D}$ for all but a finite number of
places.

We also associate to $\overline{D}$ its family of
\textit{local roof functions}
\begin{displaymath}
  {\vartheta_{\overline{D},v}\colon \Delta_{D}\longrightarrow \mathbb{R}}, \quad {v\in \mathfrak{M}_{K}}.
\end{displaymath}
For each $v$, the {$v$-adic roof function} is a continuous concave
function on the polytope that is defined as
\begin{equation*}
  \vartheta_{\overline{D},v}(x)=\inf_{u\in N_{\mathbb{R}}}\langle u,x\rangle-\psi_{\overline{D},v}(u) \quad \text{ for every } x\in \Delta_{D}.
\end{equation*}
These functions are zero  for all but a finite number of places. We
consider then the \emph{global roof function}
$\vartheta_{\overline{D}}\colon \Delta_{D}\to \mathbb{R}$, defined as
the weighted sum
\begin{equation*}
  \vartheta_{\overline{D}}=\sum_{v\in \mathfrak{M}_{K}}n_{v}\vartheta_{\overline{D},v}.
\end{equation*}
We also consider the compact convex set where this concave function is
nonnegative:
  \begin{equation}
    \label{eq:2}
    \Gamma_{\overline{D}}= \{x\in \Delta_{D}\mid
    \vartheta_{\overline{D}}(x)\ge 0\}.   
  \end{equation}

  In analogy with the geometric case, the positivity invariants and
  properties of $\overline{D}$ can be read from its metric and roof
  functions. For instance, the essential minimum of $\overline{D}$ is
  the maximum of its global roof
  function~\cite[Theorem~1.1]{BPS:smthf}:
  \begin{equation}
    \label{eq:69}
  \mu^{\ess}(\overline{D})=\max_{x\in \Delta_{D}}\vartheta_{\overline{D}}(x). 
\end{equation}
Moreover, if $\overline{D}$ is semipositive then
$ \mu^{\abs}(\overline{D})=\min_{x\in
  \Delta_{D}}\vartheta_{\overline{D}}(x)$
\cite[Remark~3.15]{BPS:smthf}.

The volumes of $\overline{D}$ can be computed
as~\cite[Theorem~5.6]{BMPS:positivity}
\begin{displaymath}
  \widehat{\vol}(\overline{D})= (d+1)! \int_{\Gamma_{\overline{D}}} \vartheta_{\overline{D}} \, d  \hspace*{-1.5pt} \vol_{M}
  \and
    \widehat{\vol}_{\chi}(\overline{D})= (d+1)! \int_{\Delta_{D}} \vartheta_{\overline{D}} \, d  \hspace*{-1.5pt} \vol_{M}.
\end{displaymath}
In particular, if $\overline{D}$ is semipositive then
$(\overline{D}^{d+1})= (d+1)! \int_{\Delta_{D}}
\vartheta_{\overline{D}} \, d \hspace*{-1.5pt}\vol_{M} $.  More
generally, the arithmetic intersection number of a family of
semipositive toric adelic $\mathbb{R}$-divisors $\overline{D}_{i}$,
$i=0,\dots, d$, can be computed as
\begin{equation}
  \label{eq:45}
 (  \overline{D}_{0}\cdots  \overline{D}_{d})= \sum_{v\in \mathfrak{M}_{K}}n_{v} \MI_{M}(\vartheta_{\overline{D}_{0},v}, \dots, \vartheta_{\overline{D}_{d},v}), 
\end{equation}
where $\MI_{M}$ denotes the mixed integral function with respect to the
Haar measure $\vol_{M}$~on~$M_{\mathbb{R}}$
\cite[Theorem~5.2.5]{BPS:asterisque}.

We have that $\overline{D}$ is semipositive if and only if
$\psi_{\overline{D},v}$ is concave for every $v$ \cite[Proposition
4.19]{BMPS:positivity}.  By Theorem 6.1 in \emph{loc. cit.} we have
that $\overline{D}$ is pseudo-effective if and only if there exists
$x\in \Delta_{D}$ such that ${\vartheta_{\overline{D}}(x)\ge 0}$ or
equivalently, if and only if $\Gamma_{\overline{D}}\ne \emptyset$. We
also have that $\overline{D}$ is big if and only if
$\dim(\Delta_D) = d$ and there exists $x\in \Delta_{D}$ such that
$\vartheta_{\overline{D}}(x)> 0$, in which case
$\Gamma_{\overline{D}}$ is a convex body.  By the same result, when
$\overline{D}$ is semipositive then it is nef if and only \change{if}
$\vartheta_{\overline{D}}(x)\ge 0$ for every $x\in \Delta_{D}$.

For a semipositive adelic $\mathbb{R}$-divisor $\overline{E}$ on $X$
we have that $\overline{D}-\overline{E} $ is pseudo-effective if and
only if $\Delta_{E} \subset \Delta_{D}$ and
$\vartheta_{\overline{E},v}(x)\le \vartheta_{\overline{D},v}(x)$ for
every $v\in \mathfrak{M}_{K}$ and $x\in \Delta_{E}$ \cite[Proposition
6.4 and Theorem 7.2(1)]{BMPS:positivity}.

\change{When $\overline{D}$ is big we can take any sequence
  $(\Lambda_{n})_{n}$ of quasi-rational polytopes uniformly
  approximating the convex body $ \Gamma_{\overline{D}}$ from inside,
  in the sense that there exists a sequence of positive real numbers
  $(\varepsilon_{n})_{n}$ with $\lim_{n\to \infty}\varepsilon_{n}=0$ such that
  \begin{displaymath}
\Lambda_{n}  \subset   \Gamma_{\overline{D}} \subset (1+\varepsilon_{n})\, \Lambda_{n} \quad \text{ for all } n \in \mathbb{N}.
  \end{displaymath}
  Following~\cite[Theorem~7.2]{BMPS:positivity} and its proof, we can
 associate to this datum a Fujita approximation sequence of
  $\overline{D}$
\begin{equation}
  \label{eq:38}
  (\phi_{n}\colon X_{n}\to X,\overline{P}_{n})_{n}
  \end{equation}
  such that for each $n$ the modification $\phi_{n}\colon X_{n}\to X$
  is \emph{toric}. This means that $X_{n}$ is also a toric variety
  with the same torus~$\mathbb{T}$ and the restriction of $\phi_{n}$
  to this torus is the identity. At the combinatorial level
  this corresponds to a (regular) refinement of the fan
  $\Sigma_{X}$. Moreover $\overline{P}_{n}$ is a toric adelic
  $\mathbb{R}$-divisor on $X_n$ with polytope equal to $\Lambda_{n}$
  and local roof functions equal to those of~$\overline{D}$ restricted
  to this polytope, that is}
  \begin{equation}
    \label{eq:53}
    \Delta_{P_{n}}=\Lambda_{n} \and \vartheta_{\overline{P}_{n},v}=\vartheta_{\overline{D},v}\big|_{\Lambda_{n}} \quad \text{ for all } v\in \mathfrak{M}_{K}.
  \end{equation}

  \subsection{Inradii and positive intersection numbers}
\label{sec:inrad-posit-arithm}

In \cite{Teissier}, Teissier first pointed out the relation between
the inradius of toric line bundles and the inradius of the associated
polytopes in the sense of convex geometry.
The next statement puts this observation into the setting of
$\mathbb{R}$-divisors.

\begin{proposition}
  \label{prop:1}
  Let $D$ and $A$ be toric $\mathbb{R}$-divisors on $X$ such that $D$
  is big and $A$ is big and nef. Then
  \begin{displaymath}
 r(D;A)= r(\Delta_{D};\Delta_{A}),
\end{displaymath}
\change{where $ r(D;A)$ and $r(\Delta_{D};\Delta_{A})$ denote the
  inradii in the sense of Definitions \ref{def:8} and~\ref{def:1},
  respectively.}
\end{proposition}

\begin{proof}
  Since $A$ is nef, for each $\lambda\in \mathbb{R}$ \change{with $\lambda\ge 0$} we have that
  $D-\lambda A$ is pseudo-effective if and only if there exists
  $x\in M_{\mathbb{R}}$ such that
  $x+\lambda\, \Delta_{A}\subset \Delta_{D}$. The equality between the
  two inradii follows then directly from their definitions.
\end{proof}

Now let $\overline{D}$ be a big toric adelic $\mathbb{R}$-divisor on
the projective toric variety $X$.  From the existence and properties
of its toric Fujita approximation sequences \eqref{eq:38} we will
\change{derive bounds for the inradius of $\overline{D}$ and a formula
for  its arithmetic positive intersection numbers.}
  
\begin{proposition}
  \label{prop:5}
  Let  $A$ be a toric $\mathbb{R}$-divisor on $X$ that is big
  and nef.  Then
  \change{
    \begin{displaymath}
 r(\Gamma_{\overline{D}};\Delta_{A}) \le \rho(\overline{D};A)\le (d+1) \, r(\Gamma_{\overline{D}};\Delta_{A}),
  \end{displaymath}
where $\rho(\overline{D};A)$ denotes the inradius in the sense of
  Definition \ref{def:rho}.}
\end{proposition}

\begin{proof}
\change{Let $(\Lambda_{n})_{n}$ be a sequence of quasi-rational polytopes
  uniformly approximating~$\Gamma_{\overline{D}}$ from inside and
  $(\phi_n,\overline{P}_n)_{n}$ the associated toric Fujita
  approximation sequence of $\overline{D}$ as in \eqref{eq:38}.
  By Proposition \ref{prop:1} we have
  \begin{displaymath}
 \lim_{n\to \infty}r(P_{n};\phi_{n}^{*}A)= \lim_{n\to \infty}r(\Lambda_{n};\Delta_{A}) =r(\Gamma_{\overline{D}};\Delta_{A}).
  \end{displaymath}
  The statement then follows from Lemma \ref{lem:4}.}
\end{proof}

\begin{proposition}
  \label{prop:9}
  Let $\overline{E}$ be a semipositive toric adelic
  $\mathbb{R}$-divisor on $X$. Then
\begin{equation*}
 \vol (R^{0}(\overline{D}))=d! \vol(\Gamma_{\overline{D}}), \quad   (\langle \overline{D}^{d}\rangle \cdot \overline{E}) = \sum_{v\in \mathfrak{M}_{K}}n_{v} \MI_{M}(\vartheta_{\overline{D},v}|_{\Gamma_{\overline{D}}}, \dots,
\vartheta_{\overline{D},v}|_{\Gamma_{\overline{D}}},  \vartheta_{\overline{E},v}).
\end{equation*}  
\end{proposition}

  \begin{proof} 
     Let $(\phi_n,\overline{P}_n)$ be a toric Fujita approximation sequence of $\overline{D}$ as in \eqref{eq:38}. By Proposition
    \ref{prop:orthogonal} and the formula for toric intersection
    numbers \eqref{eq:61} we have
  \begin{displaymath}
    \vol(R^{0}(\overline{D}))= \lim_{n\to \infty} (P_{n}^{d}) = \lim_{n\to \infty} d! \vol_{M}(\Lambda_{n}) =d! \vol(\Gamma_{\overline{D}}).
  \end{displaymath}

  For the second formula, by Definition \ref{def:13} we have
  $ (\langle \overline{D}^{d}\rangle \cdot \overline{E}) =\lim_{n\to
    \infty} ( \overline{P}_{n}^{d} \cdot \phi_{n}^{*}\overline{E}) $,
  whereas by~\eqref{eq:53} and the formula for toric arithmetic
  intersection numbers \eqref{eq:45}~we~get
\begin{displaymath}
      ( \overline{P}_{n}^{d} \cdot \phi_{n}^{*}\overline{E}) = \sum_{v\in \mathfrak{M}_{K}}n_{v} \MI_{M}(\vartheta_{\overline{D},v}|_{\Delta_{n}}, \dots,
      \vartheta_{{\overline{D},v}|_{\Delta_{n}}},  \vartheta_{\overline{E},v}) \quad \text{ for every } n\in \mathbb{N}.
  \end{displaymath}
  We conclude by taking the limit $n\to \infty$ and applying
  Lemma~\ref{lem:9}.
\end{proof}

Following \cite[Definition 4.3.3]{BPS:asterisque}, given an arbitrary
adelic structure over a toric $\mathbb{R}$-divisor one can construct a
toric one by an averaging process. To describe it,
for each $v\in \mathfrak{M}_{K}$ and $u\in N_{\mathbb{R}}$ we recall
the probability measure $\eta_{v,u}$ on~$X_{{v}}^{\an}$ from
\cite[Definition 5.1]{BPRS:dgopshtv}, which is defined as:
  \begin{enumerate}[leftmargin=*]
  \item \label{item:38} if $v \in \mathfrak{M}_{K}^{\infty}$ then
    $\eta_{v,u}$ is the translation by any point
    $x\in \val_{v}^{-1}(u) \subset \mathbb{T}_{v}^{\an} \simeq
    (\mathbb{C}^{\times})^{d}$ of the Haar probability measure of the
    compact torus~$\mathbb{S}_{v} \simeq (S^{1})^{d}$,
  \item \label{item:40} if
    $v \in \mathfrak{M}_{K}\setminus \mathfrak{M}_{K}^{\infty}$ then
    $\eta_{v,u}$ is the Dirac measure at the point
    $\zeta_{v}(u) \in \val_{v}^{-1}(u) \subset \mathbb{T}_{v}^{\an}$
    corresponding to the multiplicative seminorm on $K[M]$ defined as
\begin{displaymath}
  |f|_{\zeta_{v}(u)}= \max_{m\in M} |\alpha_{m}|_{v} \, e^{-\langle u,m\rangle} \quad \text{ for every } f=\sum_{m\in M}\alpha_{m}\chi^{m}\in K[M].
\end{displaymath}
    
  \end{enumerate}

\begin{definition}
  \label{def:7}
  Let $\overline{E}$ be an adelic $\mathbb{R}$-divisor on $X$ with $E$
  toric.  For each $v\in \mathfrak{M}_{K}$ let
  $\widehat{g}_{v}\colon X_{v}^{\an} \setminus E_{v}^{\an}\to
  \mathbb{R}$ be the function defined as
  \begin{displaymath}
  \widehat{g}_{v}(x)= \int_{X_{v}^{\an}} g_{\overline{E},v}  \,
  d \eta_{v,\val_{v}(x)}
  \end{displaymath}
and set
  \begin{math}
    \Etor=(E, (\widehat{g}_{v})_{v\in \mathfrak{M}_{K}}).
  \end{math}
\end{definition}  

\change{
  \begin{proposition}
  We have that $\Etor$ is a toric adelic $\mathbb{R}$-divisor on $X$.
\end{proposition}

\begin{proof}
  With notation as in Definition \ref{def:7}, for each
  $v \in \mathfrak{M}_{K}$ we have that $\widehat{g}_{v}$ is a toric
  $v$-adic Green function for $E$, and if $g_{\overline{E},v}$ is
  already toric then $\widehat{g}_{v}=g_{\overline{E},v}$.

  For each non-Archimedean place $v$ the canonical $v$-adic Green
  function $g_{\Ecan,v}$ is toric and induced by the canonical model
  of the pair $(X,E)$ over $\Spec(\mathcal{O}_{K})$ \cite[Example
  4.5.4]{BPS:asterisque}. Since $\overline{E}$ is an adelic
  $\mathbb{R}$-divisor we have that $g_{\overline{E},v}=g_{\Ecan,v}$
  for all but a finite number of $v$'s, and so
  $\widehat{g}_{v}=g_{\Ecan,v}$ for these places. Hence $\Etor$ is a
  toric adelic $\mathbb{R}$-divisor~on~$X$.
  \end{proof}
}

We need the following invariance of arithmetic positive intersection
numbers with respect to this averaging process.

\begin{proposition}
  \label{prop:6}
  For all $\overline{E}\in \widehat{\Div}(X)_{\mathbb{R}}$ \change{with $E$ toric} we have
  $ (\langle \overline{D}^{d}\rangle \cdot \overline{E}) = (\langle
  \overline{D}^{d}\rangle \cdot \Etor)$. 
\end{proposition}

\begin{proof}
  Let $\overline{P}$ be a \change{semipositive} toric adelic $\mathbb{R}$-divisor on $X$.
  With notations as in Definition \ref{def:7}, by the arithmetic
  B\'ezout~formula we have
  \begin{equation}\label{eq:31}
    ( \overline{P}^{d} \cdot \overline{E}) -    ( \overline{P}^{d} \cdot \Etor)
    =    \sum_{v\in \mathfrak{M}_{K}}n_{v} \int_{X_{v}^{\an}} ({g}_{\overline{E},v}- \widehat{g}_{v}) \, c_{1}(\overline{P}_{v})^{\wedge d}.
  \end{equation}

  Let $v\in \mathfrak{M}_{K}$. If $v$ is Archimedean then 
  \begin{align*}
    \int_{X_{v}^{\an}} \widehat{g}_{v} \, c_{1}(\overline{P}_{v})^{\wedge d} & =\int_{X_{v}^{\an}} \bigg(\int_{\mathbb{S}_{v}} g_{\overline{E},v} (t\cdot x) \,
    d \eta_{v,0}(t)\bigg)  \, c_{1}(\overline{P}_{v})^{\wedge d}(x)\\
  &  = \int_{\mathbb{S}_{v}} \bigg( \int_{X_{v}^{\an}} g_{\overline{E},v} (t\cdot x) \,
    c_{1}(\overline{P}_{v})^{\wedge d}(x) \bigg)    d \eta_{v,0}(t)
 = \int_{X_{v}^{\an}} g_{\overline{E},v} \,
     c_{1}(\overline{P}_{v})^{\wedge d}
  \end{align*}
  by Fubini's theorem 
  and the invariance of the $v$-adic Monge-Ampère measure of
  $\overline{P}$ under the action of $\mathbb{S}_{v}$. On the other
  hand, when $v$ is non-Archimedean we have
  \begin{displaymath}
    \int_{X_{v}^{\an}} \widehat{g}_{v} \, c_{1}(\overline{P}_{v})^{\wedge d}=\int_{X_{v}^{\an}}  g_{\overline{E},v} ( \zeta_{v}(\val_{v}(x)))
    \, c_{1}(\overline{P}_{v})^{\wedge d}                                                    = \int_{X_{v}^{\an}} g_{\overline{E},v} (x) \,     c_{1}(\overline{P}_{v})^{\wedge d}
  \end{displaymath}
  by the characterization of the Monge-Ampère measures of semipositive
  toric adelic divisors in \cite[Theorem 4.8.11]{BPS:asterisque}.
  Combining this with \eqref{eq:31} we get
  \begin{equation}
    \label{eq:43}
    ( \overline{P}^{d} \cdot \overline{E}) =    ( \overline{P}^{d} \cdot \Etor).
  \end{equation}

  Now let $(\phi_{n},\overline{P}_{n})_{n}$ be a Fujita approximation
  sequence of $\overline{D}$ as in \eqref{eq:38}.  By~\eqref{eq:43}
\begin{displaymath}
(\overline{P}_{n}^{d}\cdot \phi_{n}^{*}\overline{E}) 
  =   (\overline{P}_{n}^{d}\cdot (\phi_{n}^{*}\overline{E})\vphantom{L}^{\mathrm{tor}})=    (\overline{P}_{n}^{d}\cdot \phi_{n}^{*}(\Etor)) \quad \text{ for every } n\in \mathbb{N}
\end{displaymath}
because $\overline{P}_n$ is toric and the averaging process commutes
with the toric modification~$\phi_{n}$. Indeed, this process occurs on
the open subset $\mathbb{T} \subset X,X_{n}$, which remains unchanged
under this modification. We conclude by taking the limit as
$n\to \infty$.
\end{proof}

\subsection{Equidistribution on  toric varieties}
\label{sec:toric-height-conv}

Let $X$ be a projective toric variety with torus $\mathbb{T}$ and
$\overline{D}$ a toric adelic $\mathbb{R}$-divisor on~$X$ with $D$
big. For each $t\le \mu^{\ess}(\overline{D})$ we set
\begin{displaymath}
S_{t}(\vartheta_{\overline{D}})= \{x\in \Delta_{D}\mid
\vartheta_{\overline{D}}(x)\ge t\}
\end{displaymath}
for the corresponding {sup-level set} of the global roof function of
$\overline{D}$.  It is a nonempty compact convex subset of
$\Delta_{D}$ that is $d$-dimensional whenever
$t< \mu^{\ess}(\overline{D})$. Set also
$\Delta_{D,\max}=S_{\mu^{\ess}(\overline{D})}(\vartheta_{\overline{D}})$.

The global roof function is said to be \emph{wide} if after fixing an
arbitrary norm on~$M_{\mathbb{R}}$, the width 
of these sup-level sets remains relatively large as the level
approaches its maximum (Definition \ref{def:2}). By Proposition
\ref{prop:concavefunctions}, this is equivalent to the fact that the
inradius of these sup-level sets with respect to any fixed convex body
remains relatively large as the level approaches its maximum. By the
same result, it is also equivalent to the fact that for any
$x_{0} \in \Delta_{D,\max}$ we have that $0 \in N_{\mathbb{R}}$ is a
vertex of the {sup-differential}
\begin{math}
\partial \vartheta_{\overline{D}} (x_{0}) \subset N_{\mathbb{R}}  
\end{math} (Definition~\ref{def:4}).  When this condition holds, by
Proposition \ref{prop:7} one can associate to
$ \vartheta_{\overline{D}} $ a unique \textit{balanced family of
  sup-gradients}
with respect to its decomposition into local roof functions.
This is a family of vectors
\begin{displaymath}
u_{v}\in N_{\mathbb{R}}, \quad v\in \mathfrak{M}_{K},   
\end{displaymath}
such that $u_{v} \in \partial \vartheta_{\overline{D},v}(x_{0})$ for
every $v$ with $u_{v}=0$ for all but a finite number of places and
verifying the balancing condition
$ \sum_{v\in \mathfrak{M}_{K}}n_{v}u_{v}=0$.  We let
\begin{displaymath}
  \eta_{\overline{D},v}=
\eta_{v,u_{v}}
\end{displaymath}
be the probability measure on $X_{v}^{\an}$ from Section
\ref{sec:inrad-posit-arithm} for the point $u_{v}\in N_{\mathbb{R}}$.

The following is the main result of this section.  It is an
application of Theorem~\ref{thm:MainSequences}, or rather of its
reformulation in Theorem \ref{thm:MainProduct} in terms of arithmetic
positive intersection numbers, together with the constructions and
results from Sections~\ref{sec:toric-setting} and
\ref{sec:inrad-posit-arithm} and from the
Appendix.
 
\begin{theorem}
  \label{thm:1}
  \change{Let $X$ be a projective toric variety and $\overline{D}$ a
    toric adelic $\mathbb{R}$-divisor on~$X$ with $D$ big.} If
  $\vartheta_{\overline{D}}$ is wide then the essential minimum
  function is differentiable at~$\overline{D}$ and
  \begin{equation*}
\partial_{\overline{E}} \mu^{\ess}(\overline{D}) = \sum_{v\in \mathfrak{M}_{K}}n_{v}\int_{X_{v}^{\an}}g_{\overline{E},v} \, d \eta_{\overline{D},v}
\quad \text{ for every } \overline{E}\in \widehat{\Div}(X)_{\mathbb{R}} \text{ with } E \text{ toric}.
\end{equation*}
In particular, in this case $\overline{D}$ satisfies the
equidistribution property at every $v\in \mathfrak{M}_{K}$ with
${\nu_{\overline{D},v}=\eta_{\overline{D},v}}$.
\end{theorem}

\begin{proof}
  Let $(u_{v})_{v}$ be the balanced family of sup-gradients of
  $\vartheta_{\overline{D}}$, so that
  $\eta_{\overline{D},v}=\eta_{v,u_{v}}$ for every $v$. Set
  $\mu=\mu^{\ess}(\overline{D})$ for short.  \change{Let $A$ be a
    toric divisor on $X$ that is nef and big. Since
    $\vartheta_{\overline{D}}$ is wide we have}
  \begin{equation}
    \label{eq:29}
    \lim_{t\to \mu} \frac{ \mu-t}{r(S_{t}(\vartheta_{\overline{D}}); \change{\Delta_{A}})}=0.
  \end{equation}
  For each $t<\mu$ the global roof functions of $\overline{D}$ and its
  shift by $t$ are related by
  $\vartheta_{\overline{D}}=\vartheta_{\overline{D}(t)}+t$, and so
  $S_{t}(\vartheta_{\overline{D}})=\Gamma_{\overline{D}(t)}$ for the
  convex body defined in~\eqref{eq:2}. Combining this with  Proposition~\ref{prop:5}  and Lemma
  \ref{lemma:compareinradii} we have
  $ \change{r(S_{t}(\vartheta_{\overline{D}}); \Delta_{A}) \le   \rho(\overline{D}(t);A) \le c \,
  \rho(\overline{D}(t))}$ for a constant $c>0$ not depending on
  $t$. Hence
  \begin{displaymath}
    \lim_{t\to \mu}\frac{\mu-t}{\rho(\overline{D}(t))}=0
  \end{displaymath}
  and so by Theorem~\ref{thm:MainProduct}  the essential
  minimum function is differentiable at~$\overline{D}$~with
  \begin{equation} \label{eq:47}    
    \partial_{\overline{E}} \mu^{\ess}(\overline{D})= \lim_{t\to \mu}  \frac{(\langle \overline{D}(t)^{d} \rangle \cdot \overline{E})}{\vol (R^{t}(\overline{D}))}
    \quad \text{ for every } \overline{E} \in \widehat{\Div}(X)_{\mathbb{R}}.
\end{equation}
For the formula for the derivative, we first consider the case when
$\overline{E}$ is  toric and semipositive.

Let $\mathfrak{S}\subset \mathfrak{M}_{K}$ be a finite set of places
 such that $ u_{v}=0$, $ \psi_{\overline{E},v}=\Psi_{E}$ and
$ \vartheta_{\overline{E},v}=0|_{\Delta_{E}}$ for all
$ v \in \mathfrak{M}_{K}\setminus \mathfrak{S}$. Let $t < \mu$ and
set for short $S_{t}=S_{t}(\vartheta_{\overline{D}})$. For each
$v\in \mathfrak{M}_{K}$ let $\theta_{v,t} \colon S_{t}\to \mathbb{R}$
be the restriction of the concave function
$ \vartheta_{\overline{D},v}-\varepsilon_{v} t$ to this convex body,
with $\varepsilon_{v}=1$ if $v$ is Archimedean and $\varepsilon_{v}=0$
otherwise. Then by Proposition \ref{prop:9} we have
\begin{equation} 
\label{eq:58}
 \vol (R^{t}(\overline{D}))=d! \vol_{M}(S_{t}) \and 
 ( \langle \overline{D}(t)^{d} \rangle \cdot \overline{E}) = \sum_{v\in \mathfrak{S}}
 n_{v} \MI_{M}(\theta_{v,t},\dots, \theta_{v,t}, \vartheta_{\overline{E},v}).
\end{equation}

Choose $x_{0}\in \Delta_{D,\max}$ and set
$c_{v}=\vartheta_{\overline{D},v}(x_{0})-\langle u_{v},x_{0}\rangle$,
$v\in \mathfrak{S}$.  By \eqref{eq:69} and the balancing condition of
$(u_{v})_{v}$ we have
\begin{displaymath}
  \sum_{v\in \mathfrak{S}}n_{v} \, c_{v}= \sum_{v \in \mathfrak{M}_{K}}n_{v}\,
(\vartheta_{\overline{D},v}(x_{0})-\langle u_{v},x_{0}\rangle) =
\vartheta_{\overline{D}}(x_{0}) = \mu.
\end{displaymath}
For each $v\in \mathfrak{S}$ we have 
$ \vartheta_{\overline{D},v}(x) \le \langle u_{v},x-x_{0}\rangle
+\vartheta_{\overline{D},v}(x_{0})= \langle u_{v},x\rangle + c_{v}$
for all $x\in \Delta_{D}$ because 
$u_{v}\in \partial \vartheta_{\overline{D},v}(x_0)$. 
Setting
\begin{math}
  \kappa_{v}=\max_{x \in S_{t}} (\langle u_{v},x\rangle +c_{v} - \vartheta_{\overline{D},v}(x)) \ge 0
\end{math}~we~have
\begin{equation}
  \label{eq:55}
  \langle u_{v},x\rangle + c_{v}-\kappa_{v} -\varepsilon_{v} t \le \theta_{v,t}(x) \le   \langle u_{v},x\rangle + c_{v} -\varepsilon_{v} t \quad \text{ for every } x\in S_{t}.
\end{equation}

From the upper bound in \eqref{eq:55} and the monotonicity of the
mixed integral we get
\begin{displaymath}
  \MI_{M}(\theta_{v,t},\dots, \theta_{v,t}, \vartheta_{\overline{E},v})   \le 
  \MI_{M}((\langle u_{v},x\rangle + c_{v} -\varepsilon_{v} \, t)|_{S_{t}},\dots, (\langle u_{v},x\rangle + c_{v}
  -\varepsilon_{v} \, t)|_{S_{t}} , \vartheta_{\overline{E},v}).
\end{displaymath}
By Lemma \ref{lem:12} and Remark \ref{rem:12}, the right-hand side of
this inequality can be computed as
\begin{displaymath}
-d! \vol(S_{t}) \, \psi_{\overline{E},v}(u_{v}) +   d \, (c_{v}-\varepsilon_{v}t) \MV_{M}(S_{t}, \dots, S_{t}, \Delta_{E}) +\langle u_{v},x_{1}\rangle
\end{displaymath}
for a point $x_{1}\in M_{\mathbb{R}}$ not depending on $v$, because
$\psi_{\overline{E},v}$ coincides with the Legendre-Fenchel dual of
$\vartheta_{\overline{E},v}$ as defined in \eqref{eq:59}.  Summing
over all these places, we deduce from \eqref{eq:58} and the balancing
condition of $(u_{v})_{v}$
\begin{equation}
  \label{eq:70}
   \frac{(\langle \overline{D}(t)^{d} \rangle \cdot \overline{E})}{\vol (R^{t}(\overline{D}))} \le -\sum_{v\in \mathfrak{S}}n_{v} \psi_{\overline{E},v}(u_{v}) 
+ d \, (\mu-t)  \, \frac{\MV_{M}(S_{t}, \dots, S_{t}, \Delta_{E})}{d! \vol_{M}(S_{t}) } .
\end{equation}

For the converse inequality, for each $v \in \mathfrak{S}$ and any
$x\in S_{t}$ we have
\begin{displaymath}
  \mu-t \ge \mu-\vartheta_{\overline{D}}(x)= \sum_{w \in \mathfrak{S}}n_{w} (c_{w}+\langle u_{w},x\rangle - \vartheta_{\overline{D},w}(x))
  \ge n_{v}(c_{v}+\langle u_{v},x\rangle - \vartheta_{\overline{D},v}(x)) 
\end{displaymath}
using again the balancing condition of $(u_{v})_{v}$ together with the
previous upper bound for the $w$-adic roof functions for $w\ne
v$. Since this holds for every $x\in S_{t}$ we deduce
$ n_{v} \kappa_{v}\le \mu-t$, and in particular
\begin{displaymath}
\sum_{v \in \mathfrak{S}}n_{v}\kappa_{v}\le \#\mathfrak{S}\,(\mu-t).  
\end{displaymath}
Combining this with the lower bound in \eqref{eq:55} we similarly
obtain
\begin{equation}
  \label{eq:71}
  \frac{(\langle \overline{D}(t)^{d} \rangle \cdot \overline{E})}{\vol (R^{t}(\overline{D}))} \ge -\sum_{v\in \mathfrak{S}}n_{v} \psi_{\overline{E},v}(u_{v}) 
- (\# \mathfrak{S}-1) \, d\, (\mu-t) \, \frac{\MV_{M}(S_{t}, \dots, S_{t}, \Delta_{E})}{d! \vol_{M}(S_{t}) } .
\end{equation}

Now choose $c>0$ such that $x+c\, \Delta_{E}\subset \Delta_{D}$ for
some $x\in M_{\mathbb{R}}$. Then there exists $x'\in M_{\mathbb{R}}$
such that
\begin{math}
  x'+c\, r(S_{t};\Delta_{D})\, \Delta_{E}\subset S_{t}
\end{math}, and so by the monotonicity and the multilinearity of the
mixed volume function we have
\begin{displaymath}
c\, r(S_{t};\Delta_{D}) \,  \MV_{M}(S_{t}, \dots, S_{t}, \Delta_{E}) \le  \MV_{M}(S_{t}, \dots, S_{t}, S_{t}) = d! \vol_{M}(S_{t}).
\end{displaymath}
Hence with the limit \eqref{eq:29} we deduce that the error terms in
\eqref{eq:70} and \eqref{eq:71} vanish as $t\to \mu$. It then follows from
the expression \eqref{eq:47}
\begin{equation}
\label{eq:60}
\partial_{\overline{E}} \mu^{\ess}(\overline{D})= -\sum_{v\in \mathfrak{S}}n_{v} \psi_{\overline{E},v}(u_{v})
= -\sum_{v\in \mathfrak{M}_{K}}n_{v} \psi_{\overline{E},v}(u_{v}).
\end{equation}
By the additivity of the derivative and the metric functions, this
formula readily extends to \change{DSP toric adelic
  $\mathbb{R}$-divisors,} and by density to any toric adelic
$\mathbb{R}$-divisor on $X$.

For an arbitrary adelic $\mathbb{R}$-divisor $\overline{E}$ on $X$
with $E$ toric we apply the averaging process in Definition
\ref{def:7}. By the invariance of the arithmetic positive
intersection numbers with respect to this process (Proposition
\ref{prop:6}) we deduce from \eqref{eq:47} and~\eqref{eq:60}
\begin{displaymath}
  \partial_{\overline{E}} \mu^{\ess}(\overline{D})=   \partial_{\Etor} \mu^{\ess}(\overline{D})= 
  \sum_{v\in \mathfrak{M}_{K}}n_{v} \, \widehat{g}_{v}(x_{v})=
  \sum_{v\in \mathfrak{M}_{K}}n_{v}  \int_{X_{v}^{\an}}g_{\overline{E},v} \, d \eta_{v,u_{v}}
\end{displaymath}
with $\widehat{g}_{v}$ as in Definition \ref{def:7} and any
$x_{v}\in \val_{v}^{-1}(u_{v}) \subset X_{v}^{\an}$, using the
relation between Green functions and metric functions in
\eqref{eq:72}. This completes the proof of the first statement.  The
second follows readily from this and Proposition~\ref{prop:diffvsEP}.
\end{proof}

Now we assume that $\vartheta_{\overline{D}}$ is wide. For every
$f\in K[M]\setminus \{0\}$ we introduce the~quantity
\begin{equation*}
  m_{\overline{D}}(f) =\sum_{v\in \mathfrak{M}_{K}}n_{v}  \int_{X_{v}^{\an}}\log|f|_{v} \, d \eta_{\overline{D},v} \in \mathbb{R}.
\end{equation*}
Recall that for each $v$ we have that
$\eta_{\overline{D},v}=\eta_{v,u_{v}}$ for the component
$u_{v}\in N_{\mathbb{R}}$ of the balanced family of sup-gradients of
$\vartheta_{\overline{D}}$.  Taking into account the definition of
these probability measures and writing
$f=\sum_{m\in M} \alpha_{m} \chi^{m}$ we have
\begin{equation} \label{eq:57}
  m_{\overline{D}}(f) = \sum_{v \in \mathfrak{M}_{K}^{\infty}} n_{v}\int_{\mathbb{S}_{v}} \log|f( t \cdot x_{v})|_{v} \, d \eta_{v,0}(t)  +
\hspace{-3mm}  \sum_{v \in \mathfrak{M}_{K}\setminus \mathfrak{M}_{K}^{\infty}} \hspace{-3mm}  n_{v}\log \max_{m} (e^{-\langle u_{v},m\rangle } |\alpha_{m}|_{v} ),
\end{equation}
where for $v$ Archimedean we denote by $x_{v}$ any point in the fiber
$\val_{v}^{-1}(u_{v}) \subset \mathbb{T}_{v}^{\an}$, and  $\eta_{v,0}$ is
the Haar probability measure of $\mathbb{S}_{v}$. Hence this
quantity is an extension of the classical logarithmic Gauss-Mahler
measure of a Laurent polynomial, which in our setting corresponds to
the case where $u_{v}=0$ for every~$v$.

\begin{lemma}
  \label{lem:13}
The following properties hold:
  \begin{enumerate}[leftmargin=*]
  \item \label{item:4} for every $m\in M$ and
    $\alpha \in K^{\times}$ we have $m_{\overline{D}}(\alpha \chi^{m})=0$,
  \item \label{item:36} for every  $f\in K[M]\setminus \{0\}$ we have $m_{\overline{D}}(f)\ge 0$,
  \item \label{item:37} for every
    $m\in M\setminus \{0\}$ and $\gamma \in K^{\times}$ we have
    \begin{displaymath}
      m_{\overline{D}}(\chi^{m}-\gamma)=\sum_{v\in \mathfrak{M}_{K}}n_{v}\max(0,  \langle u_{v},m\rangle + \log |\gamma|_{v}).
    \end{displaymath}
  \end{enumerate}  
\end{lemma}

\begin{proof}
  For \eqref{item:4}, for each $v \in \mathfrak{M}_{K}$ we have
  ${ \int_{X_{v}^{\an}}\log|\alpha\chi^{m}|_{v} \, d \eta_{v,u_{v}}
    = \log|\alpha|_{v} -\langle u_{v}, m\rangle}$ from the explicit
  expression of this local term in \eqref{eq:57}. Hence
 \begin{displaymath}
   m_{\overline{D}}(\alpha \chi^{m})
   = 
   \sum_{v\in \mathfrak{M}_{K}} n_{v}
   (   \log|\alpha|_{v} -\langle u_{v}, m\rangle) = 0
  \end{displaymath}
  by the product formula and the fact that $(u_{v})_{v}$ is balanced.

  For \eqref{item:36}, 
  choose a vertex $m\in M$ of the Newton polytope of $f$. For each
  $v\in \mathfrak{M}_{K}$ we have
  \begin{displaymath}
    \int_{X_{v}^{\an}}\log|f|_{v}
    \, d \eta_{\overline{D},v}  \ge  \log|\alpha_{m}|_{v} -\langle u_{v}, m\rangle.
  \end{displaymath}
  This follows again from the expression of this term in
  \eqref{eq:57}, using in the Archimedean case the fact that the
  Mahler measure of a Laurent polynomial is bounded below by the
  absolute value of any of \change{its vertex coefficients~\cite[Corollary
  2]{Smyth}.} Together with~\eqref{item:4} this implies
  \begin{displaymath}
    m_{\overline{D}}(f)
    \ge \sum_{v\in \mathfrak{M}_{K}}n_{v} \, ( \log|\alpha_{m}|_{v} -\langle u_{v}, m\rangle)=     m_{\overline{D}}(\alpha_{m} \chi^{m})=0.
  \end{displaymath}

  For \eqref{item:37}, for each $v$ we have
\begin{displaymath}
  \int_{X_{v}^{\an}}\log|\chi^{m}-\gamma|_{v}
  \, d \eta_{\overline{D},v}
  = \log \max( e^{-\langle u_{v},m\rangle} ,  |\gamma|_{v})= - \langle u_{v},m\rangle + \max(0,  \langle u_{v},m\rangle + \log |\gamma|_{v})
\end{displaymath}
using again the explicit expression \eqref{eq:57} together with
Jensen's formula for the Mahler measure in the Archimedean case.  The
statement follows by considering the weighted sum of these terms and
the fact that the family $(u_{v})_{v}$ is balanced.
\end{proof}

For an arbitrary $\overline{E}\in \widehat{\Div}(X)_{\mathbb{R}}$ one
can compute the corresponding derivative of the essential minimum
function by reducing to the situation considered in
Theorem~\ref{thm:1}.  Since the $K$-algebra
$\mathcal{O}_X(\mathbb{T})=K[M]$ is factorial, for any
$E\in \Div(X)_{\mathbb{R}}$ we can choose
\begin{displaymath}
  f_{E}\in \Rat(X)^{\times}_{\mathbb{R}}
\end{displaymath}
defining the restriction of $E$ to the torus.  By Lemma
\ref{lem:13}\eqref{item:4}, the quantity $m_{\overline{D}}(f_{E})$
does not depend on the choice of this equation.

\begin{corollary}
\label{cor:5}
With notations and assumptions as in Theorem \ref{thm:1},
\begin{displaymath}
  \partial_{\overline{E}} \mu^{\ess}(\overline{D})
  =  m_{\overline{D}}(f_{E}) +   \sum_{v\in \mathfrak{M}_{K}}n_{v}  \int_{X_{v}^{\an}}g_{\overline{E},v} \, d \eta_{\overline{D},v} 
  \quad \text{ for every } \overline{E}\in \widehat{\Div}(X)_{\mathbb{R}}.
\end{displaymath}
\end{corollary}

\begin{proof}
  We have that $\overline{E}-\widehat{\div}(f_{E})$ is an adelic toric
  $\mathbb{R}$-divisor on $X$ whose geometric $\mathbb{R}$-divisor
  $E-\div(f_E)$ is toric. Hence by  the invariance of the essential minimum
  with respect to linear equivalence and Theorem \ref{thm:1} we obtain
 \begin{multline*}
   \partial_{\overline{E}} \mu^{\ess}(\overline{D})=
   \partial_{\overline{E}-\widehat{\div}(f_{E})} \mu^{\ess}(\overline{D})=    \sum_{v\in \mathfrak{M}_{K}}n_{v}
   \int_{X_{v}^{\an}}g_{\overline{E}-\widehat{\div}(f_{E}),v} \, d \eta_{\overline{D},v} \\
=     \sum_{v\in \mathfrak{M}_{K}}n_{v}  \int_{X_{v}^{\an}}\log|f_{E}|_{v} \, d \eta_{\overline{D},v}+
    \sum_{v\in \mathfrak{M}_{K}}n_{v}  \int_{X_{v}^{\an}}g_{\overline{E},v} \, d \eta_{\overline{D},v},
  \end{multline*}
  which gives the statement.
\end{proof}

We also have the following converse of Theorem \ref{thm:1} in the
semipositive situation.

\begin{theorem}
  \label{thm:5}
  \label{cor:2}
  \change{Let $X$ be a projective toric variety and $\overline{D}$ a
    semipositive toric adelic $\mathbb{R}$-divisor on~$X$ with $D$
    big.}  The following conditions are equivalent:
  \begin{enumerate}[leftmargin=*]
  \item \label{item:13} $\vartheta_{\overline{D}}$ is wide, 
  \item \label{item:29} the essential minimum function is
    differentiable at $\overline{D}$,
  \item \label{item:15}  $\overline{D}$ satisfies the equidistribution property at every place $v$. 
  \end{enumerate}
\end{theorem}

To prove it, we need the next result showing that in the
semipositive toric setting it is always possible to find a sharp upper
bound as that required by~Proposition~\ref{prop:conversequasican}.

\begin{proposition}
  \label{prop:15}
  If $\overline{D}$ is semipositive then there exists a semipositive
  toric adelic $\mathbb{R}$-divisor $\overline{D}'$ over $D$ with
  $\overline{D}'-\overline{D} $ pseudo-effective and
  $\mu^{\ess}( \overline{D}')= \mu^{\abs}( \overline{D}')= \mu^{\ess}(
  \overline{D})$.
  \end{proposition}

\begin{proof}
  Choose a point $x_{0}\in \Delta_{D,\max}$ and let $(u_{v})_{v}$ be a
  balanced family of sup-gradients for $\vartheta_{\overline{D}}$,
  which always exists thanks to Proposition \ref{prop:7}. For each
  $v \in \mathfrak{M}_{K}$ we have
  \begin{equation}
    \label{eq:63}
    \vartheta_{\overline{D},v}(x) \le \langle u_{v},x\rangle +c_{v}
    \quad \text{   for every } x\in \Delta_{D}
  \end{equation}
  with
  $c_{v}=\vartheta_{\overline{D},v}(x_{0}) -\langle
  u_{v},x_{0}\rangle$.

  Using the correspondence in \cite[Proposition
  4.9.2(2)]{BPS:asterisque}, set $\overline{D}'$ for the semipositive
  toric adelic
  $\mathbb{R}$-divisor over $D$ with local roof functions equal to the
  affine functions in the right-hand side of \eqref{eq:63}.  Since the
  family $(u_{v})_{v}$ is balanced we have
  \begin{displaymath}
    \vartheta_{\overline{D}'}(x)=  \sum_{v\in \mathfrak{M}_{K}} n_{v}    \vartheta_{\overline{D}',v}(x)=  \sum_{v\in \mathfrak{M}_{K}} n_{v}(\langle u_{v},x\rangle +c_{v}) =c     \quad \text{   for every } x\in \Delta_{D},
  \end{displaymath}
  for the constant
  $c= \sum_{v\in \mathfrak{M}_{K}}n_{v}c_{v}\in \mathbb{R}$.  This
  implies $ \mu^{\ess}(\overline{D}')=\mu^{\abs}(\overline{D}')=c$.
  Since $\overline{D}$ is semipositive, inequality \eqref{eq:63}
  implies that $\overline{D}'-\overline{D} $ is
  pseudo-effective. Furthermore, by \eqref{eq:69} and the balancing
  condition for $(u_{v})_{v}$ we have
\begin{displaymath}
c= \sum_{v\in \mathfrak{M}_{K}}n_{v}(\vartheta_{\overline{D},v}(x_{0}) -\langle
  u_{v},x_{0}\rangle) = \vartheta_{\overline{D}}(x_{0}) = \mu^{\ess}(\overline{D})
\end{displaymath}
and so
$\mu^{\ess}( \overline{D}')= \mu^{\abs}( \overline{D}')= \mu^{\ess}(
\overline{D})$, as stated.
\end{proof}

\begin{proof}[Proof of Theorem \ref{thm:5}]
\change{The implications \eqref{item:13} $\Rightarrow$ \eqref{item:29} and
  \eqref{item:29} $\Rightarrow$ \eqref{item:15} are given by
  Theorem~\ref{thm:1} and Proposition \ref{prop:diffvsEP} respectively. Now if \eqref{item:15} holds then by
  Propositions \ref{prop:conversequasican} and \ref{prop:15} we have
  \begin{displaymath}
\lim_{t \to \mu^{\ess}(\overline{D})} \frac{\mu^{\ess}(\overline{D})-t}{ \rho(\overline{D}(t))} = 0.
\end{displaymath}
 Let $A$ be a toric divisor on $X$ that is nef and big. 
 By Lemma \ref{lemma:compareinradii} and Proposition \ref{prop:5}  there exists $c>0$ with
$\rho(\overline{D}(t))\le c\, r(S_{t}(\vartheta_{\overline{D}});\Delta_{A})$ for every $t < \mu^{\ess}(\overline{D})$ and so 
  \begin{displaymath}
\lim_{t \to \mu^{\ess}(\overline{D})} \frac{\mu^{\ess}(\overline{D})-t}{r(S_{t}(\vartheta_{\overline{D}});\Delta_{A})} = 0,
\end{displaymath}
which gives \eqref{item:13}.}
\end{proof}

  \begin{remark}
    \label{rem:9}
    The toric equidistribution theorem from \cite{BPRS:dgopshtv}
    states that in the semipositive case, the toric adelic
    $\mathbb{R}$-divisor $\overline{D}$ verifies the equidistribution
    property at every place of $K$ if and only if it is
    \emph{monocritical}, in the sense that an associate functional on
    a space of adelic measures has a unique global minimum.  By
    Proposition~4.15 in \textit{loc.cit.}, this condition can be
    reformulated in simpler terms as the fact that $0$ is not a vertex
    of $\partial \vartheta_{\overline{D}}(x_{0})$ for any
    $x_{0} \in \Delta_{D,\max}$.  Proposition
    \ref{prop:concavefunctions} shows that it is also equivalent to
    the fact that the global roof function is wide, and so
    Theorem~\ref{thm:5} recovers this toric equidistribution~theorem.

    On the other hand, Theorem \ref{thm:1} extends the sufficient
    condition in this theorem to the situation where $\overline{D}$ is
    not necessarily semipositive and strengthens its conclusion to
    include the differentiability of the essential minimum
    function. 
  \end{remark}

  Combining the previous results with those from Section
  \ref{sec:logequi} we reinforce the toric equidistribution property
  of $\overline{D}$ to include test functions with logarithmic
  singularities along effective divisors satisfying a numerical
  condition.  \change{Recall that for $E\in \Div(X)_{\mathbb{R}}$ we
    denote by $f_{E}$ an $\mathbb{R}$-rational function of $X$
    defining the restriction of $E$ to the torus.}
  
\begin{theorem}
  \label{thm:2}
  \change{Let $X$ be a projective toric variety and $\overline{D}$ a
     toric adelic $\mathbb{R}$-divisor on~$X$ with $D$
    big.}  Assume that  $\vartheta_{\overline{D}}$ is wide and let 
  $E$ be an effective divisor on $X$ such that 
  $m_{\overline{D}}(f_E)=0$. Then for every
  $\overline{D}$-small generic sequence
  $(x_{\ell})_{\ell}$ in $X(\overline{K})$ and $v \in
  \mathfrak{M}_K$ we have
 \begin{displaymath}
   \lim_{\ell \to \infty}\int_{X_v^{\an}} \varphi \, d\delta_{O(x_\ell)_v} =
\int_{X_{v}^{\an}}\varphi \, d \eta_{\overline{D}, v}
 \end{displaymath}
 for any function
 $\varphi \colon X_v^{\an} \rightarrow \bR \cup \{\pm \infty \}$ with
 at most logarithmic singularities along $E$.

 In particular, this holds if each irreducible component of the Weil
 divisor $[E]$ is  either contained in $X\setminus \mathbb{T}$ or is
 the closure of the zero set of an irreducible binomial
 $\chi^{m}-\gamma$ with $m\in M\setminus \{0\}$ and
 $\gamma \in K^{\times}$ such that
 $\log|\gamma|_{v}=-\langle u_{v},m\rangle $ for every $v$.
\end{theorem}

\begin{proof}
  Since $\vartheta_{\overline{D}}$ is wide we have that $\overline{D}$
  satisfies the condition in Theorem \ref{thm:MainProduct} and \emph{a
    fortiori} that in Theorem \ref{thm:MainSequences}. The first
  statement is then a direct application of Corollary~\ref{cor:5} and
  Theorem~\ref{thm:logequi}.

\change{For the second statement we consider the decomposition into
  irreducible components $E=\sum_{i\in I} k_{i} V_{i}$. In the present
  situation, for each $i$ we can choose $f_{V_{i}}=1$ if
  $V_{i}\subset X\setminus \mathbb{T}$, and
  $f_{V_{i}} =\chi^{m_{i}}-\gamma_{i}$ with
  $m_{i}\in M\setminus \{0\}$ and $\gamma_{i} \in K^{\times}$ such
  that $\log|\gamma_{i}|_{v}=-\langle u_{v},m_{i}\rangle $ for every
  $v \in \mathfrak{M}_{K}$ otherwise.}
By Lemma
\ref{lem:13}\eqref{item:37} we have
      \begin{displaymath}
    m_{\overline{D}}(f_{E})= \sum_{i\in I}k_{i} \, m_{\overline{D}}(f_{V_{i}})=0,
  \end{displaymath}
  so this statement follows from the first.
\end{proof}

It is natural to try to interpret in terms of heights the numerical
condition imposed on the effective divisor $E$ by the previous
theorem. To this end, first note that for every point
$x\in \mathbb{T}(\overline{K})$ we have
$h_{\overline{D}}(x)\ge \mu^{\ess}(\overline{D})$ \cite[Lemma
3.8(1)]{BPS:smthf}, and so for every subvariety $V\subset X$ that is
not contained in the boundary $X\setminus \mathbb{T}$ we have
\begin{displaymath}
  \mu^{\ess}(\overline{D}|_{V}) \ge   \mu^{\ess}(\overline{D}).
\end{displaymath}
Following \cite[Definition
5.10]{BPRS:dgopshtv}, we say that $V$ is
\textit{$\overline{D}$-special} if this lower bound is an equality.

Using the characterization of the Bogomolov property for monocritical
semipositive toric adelic $\mathbb{R}$-divisors in \cite[Section
5]{BPRS:dgopshtv} we derive the following logarithmic equidistribution
theorem for the semipositive case.

\begin{corollary}
\label{cor:3}
Assume that $\overline{D}$ is semipositive and that
$\vartheta_{\overline{D}}$ is wide. Let $E$ be an effective divisor on
$X$ such that each of the irreducible components of
$[E]$ is either contained in
$X\setminus \mathbb{T}$ or is the
closure of a $\overline{D}$-special hypersurface of
$\mathbb{T}$. Then for every $\overline{D}$-small
generic sequence $(x_{\ell})_{\ell}$ in $X(\overline{K})$ and
$v \in \mathfrak{M}_K$ we have
 \begin{displaymath}
   \lim_{\ell \to \infty}\int_{X_v^{\an}} \varphi \, d\delta_{O(x_\ell)_v} =
\int_{X_{v}^{\an}}\varphi \, d \eta_{\overline{D},v}
 \end{displaymath}
 for any function
 $\varphi \colon X_v^{\an} \rightarrow \bR \cup \{\pm \infty \}$ with
 at most logarithmic singularities along $E$.
\end{corollary}

\begin{proof} Since $\overline{D}$ is semipositive and
  $\vartheta_{\overline{D}}$ is wide we have that $\overline{D}$ is
  monocritical in the sense of \cite{BPRS:dgopshtv}, see Remark
  \ref{rem:9}.  Let $V$ be an irreducible component of $[E]$ that is
  $\overline{D}$-special.  After extending $K$ is necessary we assume
  that $V$ is geometrically irreducible.  By the Bogomolov property
  for monocritical adelic $\mathbb{R}$-divisors \cite[Theorem
  5.12]{BPRS:dgopshtv}, $V_0=V\cap \mathbb{T}$ is the translate of a
  subtorus.  Since $V_0$ is a hypersurface, there exist $m \in M$ and
  $x_0 \in \mathbb{T}(K)$~such~that
  \begin{displaymath}
V_0 = Z(\chi^{m}-1) \cdot x_{0}.   
  \end{displaymath}
 Note that
  $V_0 = Z(\chi^m - \gamma)$ for $\gamma = \chi^{m}(x_0)$. 
  By \cite[Proposition~5.14(1)]{BPRS:dgopshtv}, the fact that $V$ is
  $\overline{D}$-special implies that
    \begin{displaymath}
u_{v}\in m^{\bot}_{\mathbb{R}}+\val_{v}(x_{0})     \quad \text{ for every } v \in \mathfrak{M}_{K},
  \end{displaymath}
which is  equivalent to the fact that
  $\langle u_{v},m\rangle=\langle \val_{v}(x_{0}) ,m\rangle=
- \log|\gamma|_{v}$ for every~$v$. We conclude with Theorem \ref{thm:2}.
\end{proof}

\section{Dynamical systems and semiabelian varieties}
\label{sec:dynamical-heights}

In this section we study adelic $\mathbb{R}$-divisors that are sums of
several canonical adelic {$\mathbb{R}$-divisors} with different
regimes with respect to an algebraic dynamical system.  In this
setting Zhang's lower bound for the essential minimum might be strict,
in which case Yuan's equidistribution theorem cannot be applied. We
show that in spite of this, the essential minimum function is
differentiable at these adelic $\bR$-divisors, and for every place the
Galois orbits of small generic sequences of algebraic points converge
towards the equilibrium measure (Theorem~\ref{thm:equidyn}). We also
show that this convergence still holds with respect to test functions
with logarithmic singularities along hypersurfaces containing a dense
subset of preperiodic points (Theorem \ref{thm:dynlogEP}).

These results apply in the setting of semiabelian varieties, giving
the differentiability of the essential minimum function and  recovering Kühne's semiabelian equidistribution theorem (Theorem
\ref{thm:semiab}). They also imply that this equidistribution also
holds with respect to functions with logarithmic singularities along
torsion hypersurfaces~(Theorem~\ref{thm:logequisemiab}).

\subsection{Canonical adelic $\bR$-divisors}\label{subsec:canonical}

Canonical metrized line bundles for algebraic dynamical systems were
introduced by Zhang \cite{Zhang:spam} and extended to adelic
$\mathbb{R}$-divisors by Chen and Moriwaki
\cite{ChenMoriwaki:DysDirichlet}. Here we recall this notion and study
some of its positivity~properties.

Let~$X$ be a normal projective variety over $K$ and
${ \phi \colon X \rightarrow X}$ a surjective endomorphism. Then
$\phi$ is finite \cite[Lemma 5.6]{Fakhruddin:qsmag} and we denote by
$\deg(\phi)$ its degree.  Let~$D$ be an $\mathbb{R}$-divisor on $X$
such that ${\phi^* D \equiv qD}$ for a real number $q>1$.

\begin{definition}
  \label{def:5}
  The \emph{canonical} adelic $\bR$-divisor of $D$, denoted
  by~$\Dcan$, is any adelic $\mathbb{R}$-divisor on $X$ such that
  \begin{equation}
    \label{eq:5}
   \phi^{*}\Dcan\equiv q \Dcan  \quad  \text{ on } \widehat{\Div}(X)_{\mathbb{R}}.
 \end{equation}
\end{definition}

To construct it, choose $f \in \Rat(X)_{\bR}^{\times}$ such that
  \begin{math}
    \phi^*D = qD + \div(f).
  \end{math}
  Starting from any adelic $\mathbb{R}$-divisor over $D$ and applying
  Tate's limit argument, it can be shown that there exists a unique
  $\overline{D} \in \widehat{\Div}(X)_{\mathbb{R}}$ such that
  \cite[Section 4]{ChenMoriwaki:DysDirichlet}
\begin{equation}
  \label{eq:17}
  \phi^{*} \overline{D}= q \overline{D}+\widehat{\div}(f).
\end{equation}
In particular $\overline{D}$ is canonical in the sense of Definition
\ref{def:5}.  Now if
$\overline{D}' \in \widehat{\Div}(X)_{\mathbb{R}}$ is another
canonical adelic $\mathbb{R}$-divisor over $D$ then
$\phi^{*} \overline{D}'= q \overline{D}'+\widehat{\div}(f')$ with
$f' \in \Rat(X)_{\bR}^{\times}$.  Necessarily $f = \gamma f'$ with
$\gamma \in K^{\times}_{\bR}$, and from the uniqueness
of~\eqref{eq:17} we get
$\overline{D}'= \overline{D}-\widehat{\div}(\lambda)$ with
$\lambda = \gamma^{1/(q-1)}$.  Hence the canonical adelic
$\mathbb{R}$-divisor of $D$ exists and is unique up to a summand of
the form $\widehat{\div}(\lambda)$ with
$\lambda\in K^{\times}_{\mathbb{R}}$.

The associated height function is not affected by this indeterminacy
thanks to the product formula, and by~\eqref{eq:5} it verifies
\begin{equation}
  \label{eq:14}
  h_{\Dcan}(\phi(x))=  q\, h_{\Dcan}(x) \quad \text{ for every } x\in X(\overline{K}).
\end{equation}
A point $x\in X(\overline{K})$ is \emph{preperiodic} if its orbit with
respect to $\phi$ is finite or equivalently, if there are positive
integers $j<k$ such that $\phi^{\circ j}(x)=\phi^{\circ k}(x)$. The
functoriality~\eqref{eq:14} implies that $h_{\Dcan}(x)=0$ whenever
$x$ is preperiodic.

\change{We next give a condition} ensuring the pseudo-effectivity of
the canonical adelic $\mathbb{R}$-divisor and the vanishing of its
essential minimum.

\begin{proposition}
  \label{prop:dynamicalminima}
  If $R(D)\ne \{0\}$ then $\Dcan$ is pseudo-effective and
  $\mu^{\ess}(\Dcan)=0$.
\end{proposition}

\begin{proof}
  \change{Since $R(D)\ne \{0\}$ we have
    $\mu^{\ess}(\Dcan)\in \mathbb{R}$ \cite[Proposition 2.6]{BC}}. Hence
  \begin{displaymath}
\mu^{\ess}(\Dcan)= \mu^{\ess}(\phi^{*}\Dcan)= \mu^{\ess}(q \Dcan) =q \,\mu^{\ess}(\Dcan)
\end{displaymath}
because $\phi$ is dominant and finite, and
$\phi^{*}\Dcan \equiv q\Dcan$. Thus $\mu^{\ess}(\Dcan)=0$, as stated.

The pseudo-effectivity of $\Dcan$ then \change{follows from Theorem
  \ref{thm:Essmin}, see Remark
  \ref{rema:pseffmax}.} 
Nevertheless we give a self-contained proof of this statement. \change{For
simplicity we assume that $\phi^{*}\Dcan=q \Dcan +\widehat{\div}(f)$
with $f\in \Rat(X)^{\times}$.}

Let $s=(f_{0},e D) $ be a nonzero global section of $e D$ for an
integer $e\ge 1$, which exists because $R(D)\ne \{0\}$. Up to
multiplying $s$ by a nonzero scalar we can suppose that
$\|s\|_{\overline{D},v,\sup}\le 1$ for every non-Archimedean~$v$.
Then given $\varepsilon>0$ we take $k\ge 1$ such that
$ \log \|s\|_{\overline{D},v,\sup} \le \varepsilon \, e\, q^{k}$ for
every Archimedean $v$. The pullback
$\phi^{\circ k,*} s =(\phi^{\circ k,*} f_{0}, e \, \phi^{\circ k,*}
D)$ is a nonzero global section of $e\, \phi^{\circ k, *}D$ and since
$\phi$ is surjective, it has the same $v$-adic sup-norms as $s$.  Our
assumption then implies that there exists a nonzero global section $s_{k}$
of $e\, q^{k} D$ with the same $v$-adic sup-norms. Hence
  \begin{displaymath}
    \log \|s_{k}\|_{eq^k\Dcan , v,\sup}= \log \|s\|_{\Dcan, v,\sup} \le
    \begin{cases}
      \varepsilon  \, e\,  q^{k} & \text{ if } v \in \mathfrak{M}_{K}^{\infty},\\
      0 & \text{ if } v \in\mathfrak{M}_{K}\setminus \mathfrak{M}_{K}^{\infty},
    \end{cases}
\end{displaymath}
\change{and so
  $s_{k} \in R^{-\varepsilon}(\Dcan) = R^0(\Dcan +
  \varepsilon \,[\infty]) $. Therefore $\Dcan + \varepsilon \, [\infty]$ is
  pseudo-effective for every $\varepsilon>0$, and $\Dcan$ is
  pseudo-effective by Lemma~\ref{lem:pseffclosed}.}
\end{proof}

\subsection{Equidistribution for sums of canonical adelic $\bR$-divisors}
\label{sec:equid-sums-canon-1}
Let $\phi $ 
be a surjective endomorphism of a normal
projective variety $X$ over $K$ of dimension $d\ge 1$. 
 For $i=1,\dots, s$ let
$D_{i}\in \Div(X)_{\mathbb{R}}$ with $\phi^*D_i \equiv q_iD_i$ for a
real number $q_{i}>1$ and set
\begin{equation*}
\overline{D} = \sum_{i = 1}^s \overline{D}_i^{\can}.
\end{equation*}
Up to reordering we assume that $1 < q_1 \le q_2 \le \cdots \le
q_s$.  We also assume that
\begin{enumerate}[leftmargin=*]
\item \label{item:31} $R(D_i) \ne \{0\}$ for every $i$,
\item \label{item:32}  $D$ is big and semiample.
\end{enumerate}

When $D$ is ample and $D_{i}$ is nef for every $i$ we have that
$\phi^{*}D-D$ is ample, which by Fakhruddin's theorem \cite[Theorem
5.1]{Fakhruddin:qsmag} ensures that the set of periodic points of
$\phi$ is dense. Together with Proposition \ref{prop:dynamicalminima}
this easily implies that the essential minimum of $\overline{D}$
vanishes.
The next result shows that this property also holds in our
more general setting.

\begin{proposition}\label{prop:essminzero}
  We have $\mu^{\ess}(\overline{D}) = 0$. 
\end{proposition}

\begin{proof}
  By Proposition \ref{prop:dynamicalminima}, the fact that
  $R(D_{i})\ne \{0\}$ implies that $\mu^{\ess}(\overline{D}_i^{\can}) = 0$
  for every $i$.  Hence by Lemma
  \ref{lemma:propertiesessmin}\eqref{item:essminsuperadd} we have
  $\mu^{\ess}(\overline{D}) \ge \sum_{i =1}^s
  \mu^{\ess}(\overline{D}_i^{\can}) = 0$.  On the other~hand,
\begin{displaymath}
  \phi^*\overline{D}
  \equiv \sum_{i = 1}^s q_i \overline{D}_i^{\can} =
 q_1 \overline{D} + \sum_{i = 2}^s (q_i - q_1) \overline{D}_i^{\can}.
\end{displaymath}
Since $q_i \ge q_1$ and $\mu^{\ess}(\overline{D}_i^{\can}) = 0$ for
every $i$, applying again Lemma
\ref{lemma:propertiesessmin}\eqref{item:essminsuperadd} and the fact
that $\phi$ is a finite morphism we obtain
$\mu^{\ess}(\overline{D}) = \mu^{\ess}(\phi^*\overline{D}) \ge q_1
\mu^{\ess}(\overline{D})$. Hence $\mu^{\ess}(\overline{D}) \le 0$,
which gives the statement.
\end{proof}

Since $\phi$ is finite, for every $v \in \mathfrak{M}_K$ and any
measure $\nu$ on $X_v^{\an}$ we can consider the pullback
$\phi_{v}^{\an, *}\nu$ by the $v$-adic analytification of $\phi$, as
explained in \cite[Section~2.8]{Chambert-Loir:Mesuresetequi}. For any
$\overline{A}_1, \ldots, \overline{A}_d \in
\widehat{\mathrm{DSP}}(X)_{\bR}$ we have
\begin{equation}
  \label{eq:73}
  \phi_{v}^{\an,*}(c_1(\overline{A}_{1,v}) \wedge \cdots \wedge c_1(\overline{A}_{d,v})) = c_1(\phi^*\overline{A}_{1,v}) \wedge \cdots
  \wedge c_1(\phi^*\overline{A}_{d,v}).
\end{equation}

The following is the central result of this section. To state it, we
denote by $\mathbb{N}^{s}_{d}\subset\mathbb{N}^{s}$ the set of
$s$-tuples of nonnegative integers whose components sum up to~$d$. For
each $a\in \mathbb{N}^{s}_{d}$ we set
$q^{a}= \prod_{i=1}^{s}q_{i}^{a_{i}}$ and consider the subset
\begin{equation*}
 I=\{a\in \mathbb{N}^{s}_{d} \mid q^{a}=\deg(\phi) \}.
\end{equation*}
Note that there is always a semipositive adelic $\mathbb{R}$-divisor
$\overline{D}'$ over $D$ by semiampleness.

\begin{theorem}\label{thm:equidyn}
  Let $v \in \mathfrak{M}_K$.

\begin{enumerate}[leftmargin=*]
\item\label{item:dyndiff}
  The essential minimum function is differentiable at $\overline{D}$, and for
  any semipositive adelic $\bR$-divisor $\overline{D}'$ over $D$ we
  have
\begin{displaymath}
 \partial_{\overline{E}}\, \mu^{\ess}(\overline{D}) = \frac{1}{(D^d)} \lim_{n \to \infty} \frac{((\phi^{\circ n,*} \overline{D}')^d \cdot \overline{E})}{\deg(\phi)^{n}} \quad \text{for every } \overline{E} \in \widehat{\Div}(X)_{\bR}.
\end{displaymath}
In particular, $\overline{D}$ has the equidistribution property
at  $v$ with
\begin{equation*}
  \nu_{\overline{D},v} = \frac{1}{(D^d)} \lim_{n\to \infty} \frac{\phi_{v}^{\circ n,\an, *}c_1(\overline{D}'_v)^{\wedge d}}{\deg(\phi)^n}.
\end{equation*}
\item\label{item:dynderiv} If each $\Dcan_i$ is DSP and $\overline{D}$ is semipositive then
\begin{displaymath}
  \partial_{\overline{E}}\, \mu^{\ess}(\overline{D})  = \frac{\sum_{ a\in I}\binom{d}{a} (\overline{E}\cdot
    \prod_{i=1}^{s}(\Dcan_{i})^{a_{i}})}{\sum_{a \in I}\binom{d}{a}
(\prod_{i =1}^s D_i^{a_i})} \quad \text{for every } \overline{E} \in \widehat{\Div}(X)_{\bR},
\end{displaymath}
\change{where $\binom{d}{a}$ denotes the multinomial coefficient}.
\item\label{item:dynMA} If each $g_{\Dcan_i,v}$ is semipositive then $\nu_{\overline{D},v} = c_1(\overline{D}_v)^{\wedge d}/(D^d)$.
\end{enumerate} 
\end{theorem}

The next lemma gives the specific sequence of semipositive
approximations that will allow us to deduce this result from Theorem
\ref{thm:MainSequences}.

\begin{lemma}\label{lemma:dynSPapprox}
  Let $\overline{D}'$ be a semipositive adelic $\bR$-divisor over $D$
  such that $\overline{D} - \overline{D}'$ is pseudo-effective, and for
  every $n \in \bN$ set
  $\overline{Q}_n = q_s^{-n}\, \phi^{\circ n,*}\overline{D}' \in
  \widehat{\Div}(X)_{\bR}$. Then
\begin{enumerate}[leftmargin=*]
\item \label{item:33}   $\overline{Q}_n$ is a semipositive
  approximation of $\overline{D}$, 
\item\label{item:Qnr} $r(Q_n;D) \ge q_s^{-n} q_1^n$,
\item\label{item:Qndeg} $(Q_n^d) = q_s^{-dn}\deg(\phi)^n(D^d)$,
\item\label{item:Qnminabs} $\mu^{\abs}(\overline{Q}_n) = q_s^{-n} \mu^{\abs}(\overline{D}')$.
\end{enumerate}
\end{lemma}

\begin{proof} We have that $\overline{Q}_n$ is semipositive and $Q_n$
  is big because these properties are preserved under pullback with
  respect to a finite morphism. Set
  $\overline{F}_n = \phi^{\circ n,*}(\overline{D} - \overline{D}')$,
  which is pseudo-effective since so is
  $\overline{D} - \overline{D}'$. We have
\begin{displaymath}
\overline{D} - \overline{Q}_n = \overline{D} - \frac{1}{q_s^{n}}\phi^{\circ n,*}\overline{D}' =  \overline{D} -\frac{1}{q_s^{n}}\phi^{\circ n,*}\overline{D} + \frac{1}{q_s^{n}}\, \overline{F}_n \equiv \sum_{i=1}^s \Big(1 - \frac{q_{i}^{n}}{q_{s}^n}\Big)\, \overline{D}_i^{\can} + \frac{1}{q_s^{n}}\, \overline{F}_n.
\end{displaymath}
Hence $\overline{D}-\overline{Q}_n$ is pseudo-effective, because
$q_s \ge q_i$ and by Proposition \ref{prop:dynamicalminima} we have
that $\overline{D}_i^{\can}$ is pseudo-effective for every $i$. Thus
$\overline{Q}_n$ is a semipositive approximation of~$\overline{D}$,
proving \eqref{item:33}. Moreover \change{each $D_{i}$ is pseudo-effective and so this is also the case for}  
\begin{displaymath}
  Q_n - \frac{q_{1}^{n}}{q_{s}^{n}} D \equiv \sum_{i = 1}^s \frac{q_{i}^{n}-q_{1}^{n}}{q_{s}^{n}}\, D_i,
\end{displaymath}
 and therefore $r(Q_n;D) \ge q_s^{-n}q_1^n$ as
stated in \eqref{item:Qnr}.  Finally the formulae \eqref{item:Qndeg}
and \eqref{item:Qnminabs} are respectively given by the projection
formula \cite[Chapter~2, Proposition~2.3(c)]{Fulton:it} and the
invariance of the absolute minimum with respect to pullback by a
surjective morphism.
\end{proof}

We also need the next auxiliary result.

\begin{lemma}\label{lemma:dynMA}
  Let $ {I}=\{a\in \mathbb{N}^{s}_{d} \mid q^{a}=\deg(\phi) \}$ as
  before.   
 \begin{enumerate}[leftmargin=*]
 \item \label{item:42} We have
   $ \displaystyle{(D^d) = \sum_{a \in {I}}\binom{d}{a}
     \Big(\prod_{i =1}^s D_i^{a_i}\Big)}$.
 \item \label{item:43} Let $v \in \mathfrak{M}_K$ and assume that $g_{\Dcan_i,v}$ is
   semipositive for every $i$. Then
\begin{displaymath}
c_1(\overline{D}_v)^{\wedge d} = \sum_{a \in {I}} \binom{d}{a} \bigwedge_{i = 1}^s c_1(\overline{D}^{\can}_{i,v})^{\wedge a_i} \and \phi_{v}^{\an,*}c_1(\overline{D}_v)^{\wedge d} = \deg(\phi) \, c_1(\overline{D}_v)^{\wedge d}.
\end{displaymath}
 \end{enumerate}
\end{lemma} 

\begin{proof}
 For each $a \in \bN^s_{d}$ we have 
 \begin{displaymath}
 q^a \Big(\prod_{i=1}^{s}{D}_{i}^{a_{i}} \Big) = \Big(\prod_{i=1}^{s}(\phi^*{D}_{i})^{a_{i}} \Big) =  \deg(\phi)\, \Big(\prod_{i=1}^{s}{D}_{i}^{a_{i}} \Big) 
 \end{displaymath}
 by the projection formula. Therefore this quantity vanishes unless
 $a \in I$, and~\eqref{item:42} follows by the multilinearity of the
 intersection product.

Now assume that $g_{\overline{D}^{\can}_{i},v}$ is semipositive for
every $i$. The multilinearity of the Monge-Ampère operator gives
\begin{equation*}
c_{1}(\overline{D}_{v})^{\wedge d}=   \sum_{a\in \bN_d^s}\binom{d}{a}
    \bigwedge_{i=1}^{s}c_{1}(\Dcan_{i,v})^{\wedge a_{i}}.
\end{equation*} 
By semipositivity, 
$\bigwedge_{i=1}^{s}c_{1}(\Dcan_{i,v})^{\wedge a_{i}}$ is a
measure for each $a\in \mathbb{N}^{s}_{d}$. Since its total mass is $(\prod_{i=1}^{s}{D}_{i}^{a_{i}} )$,
this measure is zero unless $a\in I$. This gives the first formula in \eqref{item:43}. Hence 
\begin{multline*}
  \phi_{v}^{\an,*}c_1(\overline{D}_v)^{\wedge d} = \sum_{a\in I} \binom{d}{a}  \bigwedge_{i=1}^{s}c_{1}(\phi^*\Dcan_{i,v})^{\wedge a_{i}}
  \\ = \sum_{a\in I} \binom{d}{a} \, q^a \bigwedge_{i=1}^{s}c_{1}(\Dcan_{i,v})^{\wedge a_{i}}=\deg(\phi) \, c_1(\overline{D}_v)^{\wedge d} 
\end{multline*}
by the functoriality \eqref{eq:73}, thus giving the second formula.
\end{proof}

\begin{proof}[Proof of Theorem \ref{thm:equidyn}] 
  For \eqref{item:dyndiff} we consider first the case where
  $\overline{D} - \overline{D}'$ is pseudo-effective.  Then for each
  $n \in \bN$ we let 
  $\overline{Q}_n = q_s^{-n}\phi^{\circ n,*}(\overline{D}')$ be the
  semipositive approximation of $\overline{D}$ given by Lemma
  \ref{lemma:dynSPapprox}.  By Lemma
  \ref{lemma:propertiesessmin}\eqref{item:essminincreases} and
  Proposition \ref{prop:essminzero} we have
  $\mu^{\abs}(\overline{Q}_n) \leq \mu^{\ess}(\overline{Q}_n) \le
  \mu^{\ess}(\overline{D}) = 0$, and so by Lemma
  \ref{lemma:dynSPapprox}
  \begin{equation*}
0 \le \frac{\mu^{\ess}(\overline{D}) - \mu^{\abs}(\overline{Q}_n)}{r(Q_n;D)}=  \frac{ - \mu^{\abs}(\overline{Q}_n)}{r(Q_n;D)} \le \frac{-q_{s}^{-n}\mu^{\abs}(\overline{D}')}{q_{s}^{-n}q_{1}^{n} }= 
 \frac{-\mu^{\abs}(\overline{D}')}{q_{1}^{n}}.
\end{equation*}
We  have $ \mu^{\abs}(\overline{D}')>-\infty$ because $D$ is
semiample and so 
\begin{displaymath}
\change{\lim_{n\to\infty} \frac{\mu^{\ess}(\overline{D}) - \mu^{\abs}(\overline{Q}_n)}{r(Q_n;D)}=0.}
\end{displaymath}
Hence by Theorem \ref{thm:MainSequences} the essential minimum
function is differentiable at $\overline{D}$ and for every
$\overline{E} \in \widehat{\Div}(X)_{\bR}$ we have
\begin{displaymath}
  \partial_{\overline{E}}\, \mu^{\ess}(\overline{D}) = \lim_{n\to\infty} \frac{(\overline{Q}_n^{d}\cdot \overline{E})}{(Q_n^d)} = \lim_{n\to\infty} \frac{q_s^{-nd}((\phi^{\circ n,*}\overline{D}')^{d}\cdot \overline{E})}{q_s^{-nd} \deg(\phi)^n(D^d)} = \frac{1}{(D^d)} \lim_{n \to \infty} \frac{((\phi^{\circ n,*} \overline{D}')^d \cdot \overline{E})}{\deg(\phi)^{n}}.
\end{displaymath}
This proves the
first part of the statement in this case.

Now let $\overline{D}'$ be any semipositive adelic $\mathbb{R}$-divisor
over $D$. Take then $\lambda \in K^{\times}$ such that
$\|\lambda\|_{\overline{D}-\overline{D}',v,\sup} \le 1$ for every
non-Archimedean place $v$. It follows that
$\lambda \in \widehat{\Gamma}(X,\overline{D}-\overline{D}' +
t[\infty])$ for any sufficiently large $t \in \bR$, and so
$\overline{D}-(\overline{D}' - t\,[\infty])$ is pseudo-effective. By the
previous case we have
\begin{multline*}
 \partial_{\overline{E}}\, \mu^{\ess}(\overline{D})  =  \frac{1}{(D^d)} \lim_{n \to \infty} \frac{((\phi^{\circ n,*} \overline{D}' - t\,[\infty])^d \cdot \overline{E})}{\deg(\phi)^{n}}\\
  =  \frac{1}{(D^d)} \lim_{n \to \infty} \left( \frac{((\phi^{\circ n,*} \overline{D}')^d \cdot \overline{E})}{\deg(\phi)^{n}} - d\, t\, \frac{((\phi^{\circ n,*} D)^{d-1} \cdot E)}{\deg(\phi)^{n}} \right)
\end{multline*}
using the formula \eqref{eq:arithmeticvsgeometricintersection}. Since
the left-hand side is independent of $t$, it follows that
\begin{displaymath}
\lim_{n \to \infty} \frac{((\phi^{\circ n,*} D)^{d-1} \cdot E)}{\deg(\phi)^{n}} = 0,
\end{displaymath}
completing the proof of this first part. The second part is a direct
consequence of this one using Proposition \ref{prop:diffvsEP}.

For \eqref{item:dynderiv} let
$\overline{E} \in \widehat{\Div}(X)_{\bR}$. Since $\overline{D}$ is
semipositive we can apply \eqref{item:dyndiff} with
$\overline{D}'=\overline{D} $. Since
$\phi^{\circ n,*}\overline{D} \equiv \sum_{i=1}^s
q_i^{n}\overline{D}_i^{\can}$ we obtain
  \begin{equation*}
  \partial_{\overline{E}}\, \mu^{\ess}(\overline{D})  = \frac{1}{(D^d)} \lim_{n \to \infty} \sum_{a \in \bN_d^s} \binom{d}{a} \Big(\frac{q^a}{\deg(\phi)}\Big)^n \, \Big( \overline{E} \cdot \prod_{i=1}^s (\Dcan_i)^{a_i} \Big) 
\end{equation*}
by the multilinearity of the arithmetic intersection product.  The
formula follows then from the existence of this limit together with
Lemma \ref{lemma:dynMA}\eqref{item:42}.

For \eqref{item:dynMA} note first that $g_{\overline{D},v}$ is
semipositive, being a sum of semipositive $v$-adic Green functions.
Take a semipositive adelic $\bR$-divisor $\overline{D}'$ be over $D$
with $g_{\overline{D}',v} = g_{\overline{D},v}$, so that
$c_1(\overline{D}'_v)^{\wedge d} = c_1(\overline{D}_v)^{\wedge d}$.
By Lemma \ref{lemma:dynMA}\eqref{item:43} we have
$\phi_{v}^{\circ n,\an, *}c_1(\overline{D}_v)^{\wedge d} = \deg(\phi)^n
c_1(\overline{D}_v)^{\wedge d}$ for every $n\in \bN$, and so  the
statement follows from~\eqref{item:dyndiff}.
\end{proof}

This result allows to introduce a natural notion of equilibrium
measure in our present setting.

\begin{definition}
  \label{def:14}
  Let $v \in \mathfrak{M}_K$, choose a semipositive adelic
  $\mathbb{R}$-divisor $\overline{D}'$ over $D$ and set
  $\mu_{v}={c_1(\overline{D}'_v)^{\wedge d}}/{(D^{d})}$.  The
  \emph{$v$-adic equilibrium measure} of $\phi$ with respect to $D$ is
  the probability measure on $X_{v}^{\an}$ defined as
  \begin{displaymath}
    \mu_{\phi, D,v}= \lim_{n\to \infty}
      \frac{\phi_{v}^{\circ n,\an, *} \mu_{v} }{\deg(\phi)^n}.
    \end{displaymath}
\end{definition}

Theorem \ref{thm:equidyn} ensures that this limit exists and coincides
with the $v$-adic equidistribution measure $\nu_{\overline{D},v}$. In
particular it does not depend on the choice of $\overline{D}'$. By
construction, the $v$-adic equilibrium measure is fully invariant in
the sense that
\begin{displaymath}
  \phi_{v}^{\an,*} \mu_{\phi,D,v}= \deg(\phi) \, \mu_{\phi,D,v}.
 \end{displaymath}

 \begin{remark}
   \label{rem:20}
   When $D$ is ample and $D_{i}$ is nef for every $i$ , the
   preperiodic points of $\phi$ form a dense subset of
   $X(\overline{K})$ of points of height zero with respect to
   $\overline{D}$. Hence \change{by equidistribution,} in this case the
   $v$-adic equilibrium measure does not depend on $D$.
 \end{remark}
 
We also have the following logarithmic equidistribution result.
Recall that a subvariety $Y\subset X$ is
\emph{preperiodic} if there are two  positive integers $j< k$
such that $\phi^{\circ j} (Y) = \phi^{\circ k}(Y)$.

\begin{theorem}
  \label{thm:dynlogEP} 
  Assume that $\Dcan_{i}$ is semipositive for every $i$.  Let
  $(x_{\ell})_{\ell}$ be a $\overline{D}$-small generic sequence in
  $X(\overline{K})$ and $E$ an effective divisor on $X$ such that
  every irreducible component of its Weil divisor $[E]$ contains a
  dense subset of preperiodic points. Then for every
  $v\in \mathfrak{M}_{K}$ we have
   \begin{displaymath}
     \lim_{\ell \to \infty}\int_{X_v^{\an}} \varphi \, d\delta_{O(x_\ell)_v} = \int_{X_v^{\an}} \varphi \, \frac{c_1(\overline{D}_v)^{\wedge d}}{(D^{d})}
   \end{displaymath}
   for any function
   ${\varphi \colon X_v^{\an} \rightarrow \bR \cup \{\pm \infty \}}$
   with at most logarithmic singularities along $E$.  In particular,
   this holds when $D$ is ample and every irreducible component of
   $[E]$ is preperiodic.
\end{theorem}

\begin{proof}
  Note that $\Dcan$ is semipositive, being a sum of semipositive
  adelic $\mathbb{R}$-divisors.  For every $n \in \bN$ let
  $\overline{Q}_n = q_s^{-n} \phi^{\circ n,*} \overline{D}$ be the
  semipositive approximation of $\overline{D}$ from
  Lemma~\ref{lemma:dynSPapprox}. By this result and Lemma
  \ref{lemma:dynMA}\eqref{item:43} we have
  \begin{equation}
    \label{eq:dynlogEP}
    \lim_{n\to \infty} \frac{\mu^{\abs}(\overline{Q}_{n})}{r(Q_{n};D)}=0 \and 
    \frac{c_{1}(\overline{Q}_{n,v})^{\wedge d} }{(Q_{n}^{d})}
    =\frac{ c_{1}(\overline{D}_{v})^{\wedge d}}{(D^{d})} \quad \text{ for every } v\in \mathfrak{M}_{K}.
    \end{equation}

    Let $Y$ be an irreducible component of $[E]$. Up to switching to
    linearly equivalent divisors, we can suppose that $Y$ is not
    contained in the support of any of the $D_{i}$'s, and so we can
    consider the restriction of $\overline{Q}_{n}$ to $Y$. Then 
    \begin{displaymath}
      h_{\overline{Q}_n}(Y) \le d \, \mu^{\ess}(\overline{Q}_{n}|_{Y}) \, (Q_{n}^{d-1}\cdot Y)\le 0.
    \end{displaymath}
    by Zhang's inequality (Theorem~\ref{thm:Zhang}) \change{together
      with Theorem \ref{thm:arithmHodgeMoriwaki}} and the fact that
    the set of preperiodic points of $Y(\overline{K})$ is dense.  On
    the other hand, let $A $ be an ample divisor on $X$ such that
    $A - Y$ is pseudo-effective. Since $Q_{n}$ is nef we have
    $(Q_n^{d-1} \cdot Y) \le (Q_n^{d-1} \cdot A) \le (Q_{n}^{d}) /
    r(Q_{n};A)$ by inequality \eqref{eq:comparegeometriccap} and
    Lemma \ref{lemma:inradiusandcap}. Using this and Lemma
    \ref{lemma:ineqheightnef} we obtain
\begin{displaymath}
  0\le  \frac{h_{\overline{Q}_n}(Y)- d \, \mu^{\abs}(\overline{Q}_n)\, (Q_{n}^{d-1}\cdot Y)}{(Q_{n}^{d})}
  \le \frac{ - d \, \mu^{\abs}(\overline{Q}_n)\, (Q_{n}^{d-1}\cdot Y)}{(Q_{n}^{d})} \le \frac{ - d \, \mu^{\abs}(\overline{Q}_n)}{r(Q_{n};A)}.
\end{displaymath}
These inequalities with the limit in \eqref{eq:dynlogEP} imply that
this quantity vanishes as $n\to \infty$. Since this holds for every
$Y$, the condition of Corollary \ref{cor:logequi} is verified and so
this result gives the first statement.

Finally, assume that $D$ is ample and let $Y$ be a preperiodic irreducible
component of $[E]$.  Then  there is an
integer $j>0$ such that $Y'\coloneqq \phi^{\circ j}(Y)$ is periodic
with period $k_{0}>0$, and so the iteration $\phi^{\circ k_{0}}$
induces a dynamical system on~$Y'$. Up to linear equivalence we can
restrict $D$ to $Y'$ and we have
\begin{displaymath}
\phi^{\circ k_{0},*}{D}|_{Y'}-{D}|_{Y'} \equiv \sum_{i = 1}^s (q_i^{k_0}-1)D_i|_{Y'} = (q_1^{k_0}-1)D|_{Y'} + \sum_{i=1}^s (q_i^{k_0} - q_1^{k_0})D_i|_{Y'}.
\end{displaymath}
The semipositivity assumption implies that $D_i$ is nef for every
$i$. Then $\phi^{\circ k_{0},*}{D}|_{Y'}-{D}|_{Y'}$ is ample, being
the sum of an ample $\bR$-divisor and a nef one. By Fakhruddin's
theorem \cite[Theorem 5.1]{Fakhruddin:qsmag} the set of periodic
points of $Y'(\overline{K})$ is dense, and so $Y(\overline{K})$
contains a dense subset of preperiodic points.
\end{proof}

\subsection{Equidistribution on semiabelian varieties}
\label{sec:semi-vari}

Here we specialize the results of the previous section in the
semiabelian setting. We first recall the basic constructions and
properties that are needed to this end, referring to
\cite{Chambert-Loir:pphvs, Kuhne:pshsv} for the proofs and more
details.

\change{Following \cite[Section 4]{Kuhne:pshsv} we consider a semiabelian variety $G$
over $K$ that is the extension of an abelian variety~$A$ of dimension
$g$ by a split torus $\Gm^{r}$, so that there is an exact sequence of
commutative algebraic groups over $K$
\begin{displaymath}
  0\longrightarrow \Gm^{r} \longrightarrow G \longrightarrow A \longrightarrow 0.
\end{displaymath}
The toric compactification
$ \Gm^{r} \hookrightarrow (\mathbb{P}^{1})^{r}$ induces a
compactification $\overline{G}$ of this semiabelian variety.}  To
construct it one endows the product variety
$G\times (\mathbb{P}^{1})^{r}$ with the action of this split torus
defined at the level of points as
\begin{displaymath}
  t\cdot (x,y)= (t\cdot_{G} x, t^{-1}\cdot_{ (\mathbb{P}^{1})^{r}} y)
\end{displaymath}
and defines $\overline{G}$ as the categorical quotient
$G\times (\mathbb{P}^{1})^{r}/\Gm ^{r}$. It is a smooth variety
over~$K$ containing $G$ as a dense open subset, and the projection
$G\to A$ extends to a morphism
\begin{math}
  \pi\colon \overline{G}\rightarrow A
\end{math}
allowing to consider this compactification as a
$ (\mathbb{P}^{1})^{r}$-bundle over~$A$.

For a given integer $\ell>1$ the multiplication-by-$\ell$ on $G$
extends to a morphism
\begin{math}
[\ell]_{\overline{G}}\colon \overline{G}\rightarrow \overline{G}
\end{math}
of degree $ \ell^{r+2g}$. If we denote
by~$[\ell]_{A}$ the multiplication-by-$\ell$ on~$A$, then there is a
commutative diagram
\begin{equation}
  \label{eq:20}
    \xymatrix{
\overline{G}\ar[r]^{[\ell]_{\overline{G}}} \ar[d]^{\pi}& 
\overline{G}  \ar[d]^{\pi}\\
A\ar[r]^{[\ell]_{A}}& A}
\end{equation}

The boundary $\overline{G}\setminus G$ is an effective Weil divisor,
and we denote by $M$ its associated (Cartier) divisor on
$\overline{G}$. It is relatively ample with respect to $\pi$ and
verifies
\begin{equation*}
   [\ell]_{\overline{G}}^{*}M = \ell M \quad \text{ on }  {\Div}(\overline{G}).
\end{equation*}
Let $N$ be an ample symmetric divisor on $A$, which therefore verifies
that $ [\ell]_{A}^{*}N \equiv \ell^{2}N$ on $\Div(A)$. Then its pullback
$\pi^{*}N$ is semiample, and by  \eqref{eq:20} it verifies
\begin{displaymath}
  [\ell]_{\overline{G}}^{*}\pi^{*}N \equiv \ell^{2}\pi^{*}N   \quad \text{ on }  {\Div}(\overline{G}).
\end{displaymath} 
Furthermore, the sum
\begin{math}
  D=M+\pi^{*}N
\end{math}
is an ample divisor on $\overline{G}$.

Let $\Mcan \in \widehat{\Div}(\overline{G}) $ and
$\Ncan \in \widehat{\Div}(A)$ be the canonical adelic divisors of $M$
and~$N$ for the surjective endomorphisms $[\ell]_{\overline{G}}$ and
$[\ell]_{A}$, respectively. By
\cite[Proposition~3.4]{Chambert-Loir:pphvs} the adelic divisor $\Mcan$
does not depend of the choice of $\ell$, and the same holds
for~$\Ncan$. By the commutativity in \eqref{eq:20} we have
\begin{equation}
\label{eq:23}
[\ell]_{\overline{G}}^{*}\Mcan = \ell \Mcan \and [\ell]_{\overline{G}}^{*}\pi^{*}\Ncan\equiv
\ell^{2}  \pi^{*}\Ncan  \quad \text{ on }  \widehat{\Div}(\overline{G}).
\end{equation}
In particular, $\pi^{*}\Ncan$ is the canonical adelic divisor of
$\pi^{*}N$ for $[\ell]_{\overline{G}}$.

We have that $\Mcan$ is semipositive, as shown by Chambert-Loir in
\cite[Proposition 3.6]{Chambert-Loir:pphvs} relying on some specific
regular models of abelian varieties constructed by
K\"unnemann. 
The adelic divisor
$\Ncan$ on $A$ is semipositive because $N$ is ample, and so this is
also the case for $\pi^{*}\Ncan$.
Finally set
\begin{equation}
  \label{eq:24}
\change{ \overline{D}=\Mcan+\pi^{*}\Ncan \in \widehat{\Div}(\overline{G}) }
\end{equation}
By
\eqref{eq:23}, its height function verifies
\begin{displaymath}
h_{\overline{D}}([\ell]_{\overline{G}} x)=\ell\, h_{\Mcan}(x)+\ell^{2}\,
h_{\Ncan}(\pi(x)) \quad \text{ for } x\in \overline{G}(\overline{K}).
\end{displaymath}
It is nonnegative on $G(\overline{K})$ and vanishes on the torsion
points, and so ${\mu^{\ess}(\overline{D})=0}$. \change{On the other
  hand, this height function might in principle take negative values
  at the points in the boundary $\overline{G}\setminus G$, as
  suggested by Chambert-Loir's computation in a similar situation
  \cite[Corollaire~4.6]{Chambert-Loir:pphvs}. If this occurs then
  $\mu^{\abs}(\overline{D})<0$ and $\overline{D}$ would be outside the
  scope of Yuan's equidistribution theorem.}

The next result is a direct application of Theorem \ref{thm:equidyn}.

\begin{theorem}
  \label{thm:semiab}
The  essential minimum
function is differentiable at $\overline{D}$ with
\begin{equation*}
  \partial_{\overline{E}}\, \mu^{\ess}(\overline{D})
=\frac{((\Mcan)^{r} \cdot (\pi^{*} \Ncan)^{g}\cdot \overline{E})}{(M^{r} \cdot \pi^{*}N^{g})}
  \quad \text{ for every } \overline{E}\in \widehat{\Div}(\overline{G})_{\mathbb{R}}.  
\end{equation*}
In particular,  $\overline{D}$ satisfies the $v$-adic equidistribution
  property at every $v\in \mathfrak{M}_{K}$ with
\begin{displaymath}
\nu_{\overline{D},v} =\frac{ c_{1}(\Mcan_{v})^{\wedge r} \wedge c_{1}(\pi^{*} \Ncan_{v})^{\wedge g}}{ ({M}^{r} \cdot \pi^{*}{N}^{g})} = \frac{ c_{1}(\overline{D}_{v})^{\wedge r+g}}{(D^{r+g})}.
\end{displaymath}
\end{theorem}

\begin{proof} 
  We have $R(M)\ne \{0\}$ because $M$ is effective, and
  ${R(\pi^{*}N)\ne \{0\}}$ because~$\pi^{*}N$ is semiample.  As
  explained,  $D$ is ample and both $\Mcan$ and $\pi^{*}\Ncan$ are
  semipositive. Then Theorem~\ref{thm:equidyn} gives the stated
  differentiability for the essential minimum function.
  
  To apply the formula of Theorem
  \ref{thm:equidyn}\eqref{item:dynderiv} for the derivative
  $ \partial_{\overline{E}}\, \mu^{\ess}(\overline{D})$ we need to
  determine the elements $a\in \mathbb{N}^{2}_{r+g}$ for
  which \begin{math}
    \ell^{a_{1}+2a_2}=\deg([\ell]_{\overline{G}})=\ell^{r+2g}.
\end{math}
The only one is $a=(r,g)$, and so we obtain the desired expression. 
The formulae for the $v$-adic equidistribution measure
then follow from Proposition \ref{prop:diffvsEP} and Theorem \ref{thm:equidyn}\eqref{item:dynMA}.
\end{proof}

\begin{remark}
\label{rem:11}
When $v$ is Archimedean, the equidistribution measure in this result
coincides with the Haar probability measure on the maximal compact
subgroup $\mathbb{S}_{v}\simeq (S^{1})^{r+2g}$ of $G_{v}^{\an}$, see
for instance \cite[Lemma 5.2]{Kuhne:pshsv}.

When $v$ is non-Archimedean, the description of this measure seems
more complicated.  For abelian varieties, they were described by
Gubler in terms of convex geometry \cite{Gubler} but the extension to
the semiabelian case is still pending.
\end{remark}

\begin{remark}\label{rem:16}
  In our current semiabelian setting, the sequence of semipositive
  approximations of $\overline{D}$ from Lemma \ref{lemma:dynSPapprox}
  applied with $\overline{D}' = \overline{D}$ verifies 
  \begin{displaymath}
     \overline{Q}_{n} \equiv \ell^{-n}\Mcan+\pi^{*}\Ncan, \quad n\in \mathbb{N},
  \end{displaymath}
  and for the corresponding inradius, degree and absolute minimum we
  have $ r(Q_{n};D) \ge \ell^{-n}$,
  $ (Q_{n}^{r+g})= \ell^{-rn}(D^{r+g})$ and
  $ \mu^{\abs}(\overline{Q}_{n})=\ell^{-2n}\mu^{\abs}(\overline{D})$
  for each $n$. Hence
  \begin{displaymath}
    0\le \frac{\mu^{\ess}(\overline{D})-\mu^{\abs}(\overline{Q}_{n})}{    r(Q_{n};D)}=
    \frac{-\mu^{\abs}(\overline{Q}_{n})}{    r(Q_{n};D)} \le \frac{-\mu^{\abs}(\overline{D})}{\ell^{n}}
  \end{displaymath}
  and so  condition \eqref{eq:conditioninradiussequences} is
  satisfied. On the other hand, this is not the case for the stronger
  condition from Remark~\ref{rem:conditionMainvolume} as soon as
  $r\ge 2$.
\end{remark}

We next extend this equidistribution result to the closure of a
subvariety of~$G$ with vanishing essential minimum.  By the
semiabelian Bogomolov conjecture, proved by David and Philippon
\cite{DP:stvs}, these subvarieties are translates of semiabelian
subvarieties by torsion points, and so they do not provide examples of
equidistribution phenomena beyond those already obtained.  However
this extension is the centerpiece of K\"uhne's approach to this
conjecture \cite[Proposition~4.1]{Kuhne:pshsv} and so it is worth
showing that it can also be derived from our results.

Let $Y\subset \overline{G}$ be the closure of a subvariety of $G$, and
set $e=\dim(Y)$ and $e'=\dim(\pi(Y))$.  Then $Y$ is not contained in
the support of $M$, and after possibly replacing the divisor
$N \in \Div(A)$ by a linearly equivalent one, we assume without loss
of generality that $Y$ is neither contained in the support
of~$\pi^{*} N$.  Hence we can consider the restriction of
$\overline{D}$ to this subvariety.
\change{For a place $v \in \mathfrak{M}_K$, recall that 
\begin{displaymath}
  c_{1}(\Mcan_{v})^{\wedge e-e'}\wedge c_{1}(\pi^{*}\Ncan_{v})^{\wedge e'} \wedge \delta_{Y_{v}^{\an}}
\end{displaymath}
is the signed measure on $\overline{G}_{v}^{\an}$ introduced in
\eqref{eq:22}. Since both $\overline{M}$ and $\pi^*\overline{N}$ are
semipositive, it is a measure. It is supported on $Y_{v}^{\an}$ and
has total mass $(M^{e-e'}\cdot \pi^{*}{N}^{e'} \cdot Y)$.}

\begin{proposition}
  \label{prop:equisemiabsubvar}
  With notation as above, assume that $\mu^{\ess}(\overline{D}|_Y) =0$. Then
  $\overline{D}|_{Y}$ satisfies the $v$-adic equidistribution property
  for every $v\in \mathfrak{M}_{K}$~with
\begin{equation*} 
  \nu_{\overline{D}|_{Y},v}=\frac{c_{1}(\Mcan_{v})^{\wedge e-e'}\wedge c_{1}(\pi^{*}\Ncan_{v})^{\wedge e'} \wedge \delta_{Y_{v}^{\an}}}
  {(M^{e-e'}\cdot \pi^{*}{N}^{e'} \cdot Y)}.
\end{equation*}
\end{proposition}

\begin{proof} 
  For each ${n\in \mathbb{N}}$ let
  $\overline{Q}_n = \ell^{-n} \Mcan + \pi^* \Ncan$.  Then $Q_{n}$ is
  ample, $\overline{Q}_{n}$ is semipositive and
  $\overline{D}-\overline{Q}_{n}$ is effective. Let $ \widetilde{Y}$
  be the normalization of the subvariety $Y$, and denote by
  $\overline{D}|_{\widetilde{Y}}$ and
  $\overline{Q}_{n}|_{\widetilde{Y}}$ the adelic $\mathbb{R}$-divisors
  on $\widetilde{Y}$ obtained by pullback. Since the normalization
  morphism is birational and $Y$ is not contained in the support of
  $M$ and $\pi^{*}N$, we have that $\overline{Q}_{n}|_{\widetilde{Y}}$
  is a semipositive approximation of $\overline{D}|_{\widetilde{Y}}$.
  Its absolute minimum can be estimated as
\begin{displaymath}
  0 = \mu^{\ess}(\overline{D}|_{Y})= \mu^{\ess}(\overline{D}|_{\widetilde{Y}}) \ge  \mu^{\abs}(\overline{Q}_n|_{\widetilde{Y}}) \ge \mu^{\abs}(\overline{Q}_n)
  = \ell^{-2n}\mu^{\abs}(\overline{D}),
\end{displaymath}
where the last equality comes from the fact that $\overline{Q}_n \equiv \ell^{-2n} ([\ell]_{\overline{G}}^{\circ n,*}\overline{D})$ and the invariance of the absolute minimum with respect to pullback by surjective morphisms. On the other hand we have that 
$Q_n|_{\widetilde{Y}} - \ell^{-n}D|_{\widetilde{Y}} = (1- \ell^{-n})\pi^*N |_{\widetilde{Y}}$
is effective and so
\begin{displaymath}
  r(Q_n|_{\widetilde{Y}};D|_{\widetilde{Y}}) \ge \ell^{-n}.
\end{displaymath}
Hence
\begin{math}
\lim_{n\to \infty} {\mu^{\abs}(\overline{Q}_n|_{\widetilde{Y}})}/{r(Q_n|_{\widetilde{Y}};D|_{\widetilde{Y}})}=0
\end{math}, and so Theorem \ref{thm:MainSequences}  and
Remark~\ref{rema:hypothesesthm} imply that $\overline{D}|_{{Y}}$
satisfies the equidistribution property for
every~$v\in \mathfrak{M}_{K}$~with
\begin{equation*}
  \nu_{\overline{D}|_{Y},v}=\lim_{n\to \infty} \frac{ c_{1}((\overline{Q}_{n}|_{Y})_{v})^{\wedge e}}{((Q_{n}|_{Y})^{e})}
  =\lim_{n\to \infty} \frac{ c_{1}(\overline{Q}_{n,v})^{\wedge e} \wedge \delta_{Y_{v}^{\an}}}{(Q_{n}^{e}\cdot Y)}.
\end{equation*}

To compute this limit, first note that
\begin{equation*}
  (Q_{n}^e\cdot {Y}) = \sum_{j=0}^{e} \ell^{-n(e-j)}  \binom{e}{j} (M^{e-j} \cdot \pi^{*}N^{j} \cdot  {Y}).
\end{equation*}
For each $j$ consider the intersection product $[M^{e-j} \cdot {Y}]$
in the Chow group of $j$-dimensional cycles of~${Y}$.  By the
projection formula we have
\begin{displaymath}
(M^{e-j} \cdot \pi^{*}N^{j} \cdot {Y})= 
  (\pi_*[M^{e-j}\cdot {Y}] \cdot N^{j}),
\end{displaymath}
where the left intersection number is computed over ${Y}$ and the
right over $\pi({Y})$.  In particular  this quantity vanishes for
$j > e' = \dim(\pi(Y))$. On the other hand, for $j = e'$ it is equal to
$(M^{e-e'} \cdot F) \, (N^{e'})$ for a general fiber $F$ of the
projection $ Y \rightarrow \pi(Y)$, and therefore it is positive because
$M$ is relatively ample and $N$ is ample.  Hence
\begin{equation}
\label{eq:18}
(Q_{n}^e\cdot Y) =  \ell^{-n(e-e')}  \binom{e}{e'} (M^{e-e'} \cdot \pi^{*}N^{e'} \cdot Y) + O(\ell^{-n(e-e'+1)} ),
\end{equation}
and the dominant term in this asymptotics is positive.

Furthermore the measure
$c_{1}(\overline{Q}_{n,v})^{\wedge e } \wedge \delta_{Y_{v}^{\an}}$ is
zero whenever $j>e'$ because its total mass vanishes, and therefore
\begin{displaymath}
c_{1}(\overline{Q}_{n,v})^{\wedge e} \wedge \delta_{Y_{v}^{\an}}= 
  \sum_{j=0}^{e'} \ell^{-n(e-j)}\binom{e}{j} 
    c_{1}(\Mcan_{v})^{\wedge e-j}\wedge c_{1}(\pi^{*}\Ncan_{v})^{\wedge j}   \wedge \delta_{Y_{v}^{\an}} .
\end{displaymath}
The statement then follows by taking the limit for $n\to \infty$ of the
ratio between this asymptotics and that in \eqref{eq:18}.
\end{proof}

Finally we strengthen the semiabelian equidistribution property to
include test functions with logarithmic singularities along the
closure of a torsion hypersurface or an irreducible component of the
boundary.  Recall that a hypersurface of ${G}$ is \emph{torsion} if it
is the translate of a semiabelian hypersurface of $G$ by a torsion
point.  \change{We also recall that $\overline{D}=\Mcan+\pi^{*}\Ncan $
  denotes the adelic divisor on $\overline{G}$ in \eqref{eq:24}.}

\begin{theorem}
  \label{thm:logequisemiab}
  Let $(x_{\ell})_{\ell}$ be a $\overline{D}$-small generic sequence
  in $\overline{G}(\overline{K})$ and $E$ an effective divisor on
  $\overline{G}$ such that each irreducible component of $[E]$ is
  either the closure of a torsion hypersurface of ${G}$ or an
  irreducible component of $\overline{G}\setminus G$.  Then for every
  $v\in \mathfrak{M}_{K}$ and any function
  $\varphi \colon \overline{G}_v^{\an} \rightarrow \bR \cup \{\pm
  \infty \}$ with at most logarithmic singularities along $E$ we have
   \begin{displaymath}
     \lim_{\ell \to \infty}\int_{\overline{G}_v^{\an}} \varphi \, d\delta_{O(x_\ell)_v} = \int_{\overline{G}_v^{\an}} \varphi \, \frac{c_1(\overline{D}_v)^{\wedge r+g}}{(D^{r+g})} .
 \end{displaymath}
\end{theorem}

\begin{proof}
  This follows from Theorem \ref{thm:dynlogEP} noting that every irreducible component  of $[E]$ is a
  preperiodic hypersurface for the endomorphism
  $[\ell]_{\overline{G}}$ for any $\ell>1$.
\end{proof}

\section{Quasi-projective varieties}\label{sec:quasiproj}

In this section we extend our study to the setting of adelic line
bundles on quasi-projective varieties in the sense of Yuan and Zhang
\cite{YuanZhang:quasiproj}.  We start by recalling the elements of
this theory in terms of adelic $\mathbb{R}$-divisors on
quasi-projective varieties, following the presentation of Burgos and
Kramer in~\cite[Section 3]{BurgosKramer}.  Once the basic constructions
and facts are achieved, our arguments can be applied in a rather
direct way. For brevity we focus on the variant from Section
\ref{subsec:variantMain}, whose extension
(Theorem~\ref{thm:equiquasiproj}) generalizes Yuan and Zhang's
quasi-projective equidistribution theorem.

\subsection{Adelic $\mathbb{R}$-divisors on quasi-projective
  varieties}
\label{sec:adel-divis-quasi}

First we consider the geometric case.  Let $X$ be a normal projective
variety over $K$ of dimension $d \ge 1$ and $B $ an effective divisor
on $X$. Set $U = X \setminus \supp(B)$ and let $R(X,U)$ be the
category of normal modifications of $X$ which are isomorphisms over
$U$. Given such a modification $\pi \colon X_\pi \rightarrow X$, we
write $(X_{\pi},\pi)$ or simply $\pi$ for the corresponding object in
$R(X,U)$.  The space of \emph{model $\bR$-divisors} on~$U$ is 
defined as the direct limit
\begin{displaymath}
\Div(U)_{\bR}^{\mathrm{mod}} = \varinjlim_{\pi \in R(X,U)} \Div(X_{\pi})_{\bR}.
\end{displaymath}
Given $D,D'\in \Div(U)_{\bR}^{\mathrm{mod}}$ we write $D \ge D'$ or
$D' \le D$ if there exists $(X_{\pi}, \pi) \in R(X,U)$ such that
$D, D' \in \Div(X_{\pi})_{\bR}$ and $D - D'$ is effective.  The
$B$-adic norm on $\Div(U)_{\bR}^{\mathrm{mod}}$ (with possibly
infinite values) is defined as
\begin{displaymath}
\|D\|_{B} = \inf \{ \varepsilon \in \bR_{> 0} \ | \ - \varepsilon B \le D \le \varepsilon B\} \quad \text{ for every } D \in \Div(U)_{\bR}^{\mathrm{mod}}.
\end{displaymath}
The space $\Div(U)_{\bR}^{\adel}$ is then defined as the completion of
$\Div(U)_{\bR}^{\mathrm{mod}}$ for the $B$-adic topology, and its
elements are called \emph{adelic $\mathbb{R}$-divisors} on $U$.  This
space depends on the open subset $U$ but not on the effective
divisor~$B$.

Let $D\in \Div(U)_{\bR}^{\adel}$ and $(D_i)_{i}$  a Cauchy sequence in
$ \Div(U)_{\bR}^{\model}$ representing this adelic $\mathbb{R}$-divisor. The
\emph{volume} of $D$ is defined as
\begin{displaymath}
  \vol(D)= \lim_{i \to \infty} \vol(D_i).
\end{displaymath}
It follows from \cite[Theorems 5.2.1 and 5.2.9]{YuanZhang:quasiproj}
that this limit exists in $\bR$ and does not depend on the choice of
the approximating sequence.

We say that $D$ is \emph{big} if $\vol(D) > 0$.  We also say that $D$
is \emph{nef} if the sequence $(D_i)_{i}$ can be chosen such that
$D_i$ is nef for every $i$. We then say that $D$ is \emph{integrable}
if it can be written as $D = A_1 - A_2$ with
$A_1, A_2 \in \Div(U)_{\bR}^{\adel}$ nef.  The subspace of integrable
adelic $\mathbb{R}$-divisors on $U$ is denoted
by~$\Div(U)_{\bR}^{\integ}$.

The \emph{intersection product} of integrable adelic $\mathbb{R}$-divisors is
the symmetric multilinear map from \cite[Theorem 3.7]{BurgosKramer}
\begin{displaymath}
(D_{1}, \dots, D_{d}) \in (  \Div(U)_{\bR}^{\integ})^{d}  \longmapsto  (D_{1} \cdots D_{d}) \in \mathbb{R}.
\end{displaymath}
For $j=1,\dots, d$ let $D_{j}$ be a nef adelic $\mathbb{R}$-divisor on
$U$ and choose a Cauchy sequence $(D_{j,i})_{i}$ of nef model
$\mathbb{R}$-divisors on $U$ representing $D_{j}$. Then
\begin{displaymath}
  (D_1 \cdots D_d) = \lim_{i \to \infty} (D_{1,i} \cdots D_{d,i})
\end{displaymath}
with the intersection products in the right-hand side computed on
common models.  By \cite[Theorem 3.7]{BurgosKramer}, this limit exists
in $\mathbb{R}$ and does not depend on the choice of the sequences.
We have $ (D^d)=\vol(D) $ for every nef $D \in \Div(U)_{\bR}^{\adel}$.

\begin{definition}
  \label{def:16}
  Let $P, A$ be big adelic $\mathbb{R}$-divisors on $U$.  The \emph{inradius}
  of $P$ with respect to~$A$ is defined as
\begin{displaymath}
r(P;A) = \sup \{ \lambda \in \bR \ | \ P-\lambda A \text{ is big}\}.
\end{displaymath}  
\end{definition}

\begin{lemma}
  \label{lem:3}
  Let $P, A$ be big adelic $\mathbb{R}$-divisors on $U$, and let
  $(P_{i})_{i}$ and $(A_{i})_{i}$ be Cauchy sequences in
  $\Div(U)^{\model}$ representing $P$ and $A$. Then
  \begin{equation*}
r(P;A) = \lim_{i \to \infty} r(P_i;A_i).
\end{equation*}
Moreover, $r(P;A)$ is a positive real number.
\end{lemma}

\begin{proof}
By definition we have
  \begin{equation}
    \label{eq:3}
    \lim_{i\to \infty}\vol(P_{i}-\lambda  A_{i})=\vol(P-\lambda A) \quad  \text{ for every } \lambda \in \bR.    
  \end{equation}
  Let $\lambda<r(P;A)$. Then $P-\lambda A$ is big and so the
  right-hand side of \eqref{eq:3} is strictly positive, which implies
  that $P_{i}-\lambda A_{i}$ is big for every $i$ large enough. Since
  $\lambda$ is arbitrary, we obtain
  \begin{displaymath}
    \liminf_{i\to \infty}r(P_{i};A_i) \ge r(P;A).
  \end{displaymath}
  Now let $\lambda<\limsup_{i\to \infty} r(P_{i};A_{i})$. Then there
  are subsequences $(P_{i_{k}})_{k}$ and $(A_{i_{k}})_{k}$ and a
  constant $c>0$ such that
  $(P_{i_{k}} -\lambda A_{i_{k}}) - c A_{i_{k}}$ is big
  for every $k$. Hence the left-hand side of \eqref{eq:3} is strictly
  positive and $P-\lambda  A$ is big.  Since $\lambda$
  is~arbitrary,~we~get
  \begin{displaymath}
    \limsup_{i\to \infty}r(P_{i};A_i) \le r(P;A)
  \end{displaymath}
  thus completing the proof of the first statement.

  For the second, since bigness is an open condition we have
  $r(P;A) > 0$.  This condition also implies that there exist
  $P',A' \in \Div(X)_{\bR}$ with $P'-P$ and $A-A'$ big, and so
  $r(P;A) \le r(P';A') < \infty$.
\end{proof}

Next we consider the arithmetic case. An \emph{arithmetic variety}
$\mathcal{X}$ over $\mathcal{O}_{K}$ is a flat integral scheme over
$\Spec(\mathcal{O}_{K})$.  Assume that $\mathcal{X}$ is normal and
projective of dimension $d+1$ and denote by $X$ its generic fiber,
which is a normal projective variety over $K$ of dimension $d$. We
denote by $ \Div(\mathcal{X})_\bR $ the space of $\bR$-divisors on
$\mathcal{X}$, and for $\mathcal{D}\in \Div(\mathcal{X})_{\mathbb{R}}$
we denote by $\mathcal{D}|_{X}\in \Div(X)_{\mathbb{R}}$ the
restriction to  $X$. 

An \emph{arithmetic $\bR$-divisor} on $\mathcal{X}$ is a pair
$\overline{\mathcal{D}} = (\mathcal{D},(g_v)_{v \in
  {\mathfrak{M}_{K}^{\infty}}})$ where $\mathcal{D}$ is an
$\bR$-divisor on $\mathcal{X}$ and $g_v$  a continuous $v$-adic
Green function for $\mathcal{D}|_{X}$ for every
$v \in {\mathfrak{M}_{K}^{\infty}}$. The space of arithmetic
$\bR$-divisors on $\mathcal{X}$ is denoted by
$\widehat{\Div}(\mathcal{X})_{\bR}$.  We say that
$\overline{\mathcal{D}}$ is \emph{effective} (respectively
\emph{strictly effective}) if $\mathcal{D}$ is effective and
$g_v \ge 0$ (respectively $g_v > 0$) on
$X_v^{\mathrm{an}} \setminus \supp(D)_v^{\an}$ for every
$v$.

Let $\overline{\mathcal{B}} = (\mathcal{B},(g_{\overline{\mathcal{B}},v})_v)$ be a strictly effective arithmetic
divisor on $\mathcal{X}$ and set
$\mathcal{U} = \mathcal{X} \setminus \supp(\mathcal{B})$. Consider
also the underlying divisor and the open subset
\begin{displaymath}
B = \mathcal{B}|_{X} \in \Div(X) \and U = X \setminus \supp(B) \subset X.  
\end{displaymath}
We denote by $R(\mathcal{X},\mathcal{U})$ the category of normal
modifications $\pi\colon \mathcal{X}_{\pi} \to \mathcal{X}$ that are
isomorphisms over $\mathcal{U}$. Such a normal modification is denoted
by $(\pi, \mathcal{X}_{\pi})$ or simply by~$\pi$. The space of
\emph{model arithmetic $\bR$-divisors} on $\mathcal{U}$ is 
the direct limit
\begin{displaymath}
\widehat{\Div}(\mathcal{U})_{\bR}^{\mathrm{mod}}  = \varinjlim_{\pi \in R(\mathcal{X}, \mathcal{U})} \widehat{\Div}(\mathcal{X}_\pi)_{\bR}.
\end{displaymath} 

Given
$\overline{\mathcal{D}}_1, \overline{\mathcal{D}}_2 \in
\widehat{\Div}(\mathcal{U})_{\bR}^{\mathrm{mod}}$ we write
$\overline{\mathcal{D}}_1 \ge \overline{\mathcal{D}}_2$ or
$\overline{\mathcal{D}}_2 \le \overline{\mathcal{D}}_1$ whenever
$\overline{\mathcal{D}}_1 - \overline{\mathcal{D}}_2$ is effective on
a model where both $\overline{\mathcal{D}}_1$ and
$\overline{\mathcal{D}}_2$ are defined.  The
$\overline{\mathcal{B}}$-adic norm is~defined~as
\begin{displaymath}
  \|\overline{\mathcal{D}}\|_{\overline{\mathcal{B}}} = \inf \{ \varepsilon \in \bR_{>0} \ | \ -\varepsilon \overline{\mathcal{B}}
  \le \overline{\mathcal{D}} \le \varepsilon \overline{\mathcal{B}} \} \quad \text{ for } \overline{\mathcal{D}}
  \in \widehat{\Div}(\mathcal{U})_{\bR}^{\mathrm{mod}}.
\end{displaymath}
The space $\widehat{\Div}(\mathcal{U})_{\bR}^{\mathrm{adel}}$ of
\emph{adelic $\bR$-divisors} on $\mathcal{U}$ is then defined as the
completion of $\widehat{\Div}(\mathcal{U})_{\bR}^{\mathrm{mod}}$ with
respect to the $\overline{\mathcal{B}}$-adic topology. 
As in the geometric case, it  depends only on the open subset
$\mathcal{U}$. 

\begin{remark}
\label{rem:3}
In \cite[Section 3.3]{BurgosKramer} arithmetic varieties are required
to have smooth generic fiber, and arithmetic $\mathbb{R}$-divisors are
assumed to be of smooth type.  Nevertheless, our space
$\widehat{\Div}(\mathcal{U})_{\mathbb{R}}^{\adel}$ coincides with that
in \emph{loc. cit.} up to possibly shrinking the open subset
$\mathcal{U}$, by the existence of compactifications with smooth
generic fiber and the density of Green functions of smooth type among
those of continuous type \cite[Remark 3.14]{BurgosKramer}.
\end{remark}

Given
$\overline{\mathcal{D}} \in
\widehat{\Div}(\mathcal{U})_{\mathbb{R}}^{\adel}$ we denote by
$(\overline{\mathcal{D}}_i)_{i}$ a Cauchy sequence in
$\widehat{\Div}(\mathcal{U})_{\bR}^{\mathrm{mod}}$ representing this
adelic $\mathbb{R}$-divisor.  For convenience, we assume that such
sequences have bounded differences even for small indices, namely that
there exists $c\in \mathbb{R}_{\ge 0}$ with
\begin{equation}
  \label{eq:76}
  \|\overline{\mathcal{D}}_{i}-\overline{\mathcal{D}}_{j}\|_{\overline{\mathcal{B}}} \le c \quad \text{ for every } i,j. 
\end{equation}

For each $i$ choose
$(\pi_i, \mathcal{X}_i) \in R(\mathcal{X}, \mathcal{U})$ such that
$\overline{\mathcal{D}}_i \in \widehat{\Div}(\mathcal{X}_i)_{\bR}$,
set $X_i$ for the generic fiber of $ \mathcal{X}_i$ and then
$D_i = {\mathcal{D}_i}|_{X_i}$.  We have that $(D_i)_{i}$ is a
sequence in $\Div(U)_{\bR}^{\mathrm{mod}}$ that is Cauchy for the
$B$-adic topology and so defines an element
$D \in \Div(U)_{\bR}^{\adel}$, called the \emph{geometric} adelic
$\mathbb{R}$-divisor of $\overline{\mathcal{D}}$.

Also for each $i$ we write
$\overline{\mathcal{D}}_{i}=(\mathcal{D}_{i}, (g_{i,v})_{v\in
  \mathfrak{M}_{K}^{\infty}})\in \widehat{\Div}(\mathcal{X}_{i})_{\mathbb{R}} $
and we denote by
\begin{equation}
  \label{eq:4}
 \overline{D}_i=(D_{i}, (g_{i,v})_{v\in \mathfrak{M}_{K}})\in \widehat{\Div}(X_{i})_{\mathbb{R}} 
\end{equation}
the adelic $\mathbb{R}$-divisor on $X_{i}$ in the sense of Definition
\ref{def:adelicdiv} obtained by adding the non-Archimedean Green
functions induced by $\mathcal{D}_{i}$, as explained in
Section~\ref{subsec:adelicdef}.

\begin{remark}
  \label{rem:6}
  We denote the elements of $\widehat{\Div}(\mathcal{X})_{\mathbb{R}}$
  by overlined calligraphic letters and following the pattern
  described above, the akin elements of
  $\Div(\mathcal{X})_{\mathbb{R}}$, $\widehat{\Div}(X)_{\mathbb{R}}$
  and ${\Div}({X})_{\mathbb{R}}$ are denoted by either the same
  calligraphic letter or the corresponding Roman one, and either
  keeping or not the overline. Similarly, the elements of
  $\widehat{\Div}(\mathcal{U})_{\mathbb{R}}^{\adel}$ are denoted by
  overlined calligraphic letters and those in
  $\Div(U)_{\mathbb{R}}^{\adel}$ by the corresponding non-overlined
  Roman letter.
\end{remark}

We say that $\overline{\mathcal{D}}$ is \emph{pseudo-effective} (respectively
\emph{semipositive}, respectively \emph{nef}) if the sequence
$(\overline{\mathcal{D}}_i)_{i}$ can be chosen such that
$\overline{\mathcal{D}}_i $ is pseudo-effective (respectively semipositive,
respectively nef) for every $i$.  We say that $\overline{\mathcal{D}}$ is
\emph{integrable} if it is the difference of two nef adelic
$\bR$-divisors on $\mathcal{U}$. The subspace of integrable adelic
$\bR$-divisors on $\mathcal{U}$ is denoted by
$\widehat{\Div}(\mathcal{U})_{\bR}^{\mathrm{int}}$.

\begin{remark}
\label{rem:10}
Our definition of nef adelic $\bR$-divisors on $\mathcal{U}$ coincides
with that in~\cite{BurgosKramer} but differs slightly from the one in
\cite{YuanZhang:quasiproj}, where nef adelic divisors in the above
sense are called strongly nef. However, the two definitions coincide
after possibly shrinking the open subset $\mathcal{U}$ \cite[Remarks
3.5, 3.16 and 3.39]{BurgosKramer}.
\end{remark}

When $\overline{\mathcal{D}} $ is integrable, for each
$v \in \mathfrak{M}_K$ we denote by
$c_1(\overline{\mathcal{D}}_v)^{\wedge d}$ the signed measure on
$U_v^{\an}$ defined in \cite[Section 3.6.7]{YuanZhang:quasiproj} and
extended to integrable adelic $\bR$-divisors by multilinearity.  By
\cite[Lemma 5.4.4]{YuanZhang:quasiproj} it has total mass $(D^d)$.  If
$\overline{\mathcal{D}} $ is nef and $\overline{\mathcal{D}}_i$ is
also nef for every $i$, then
 \begin{equation}\label{eq:defMongeAmperequasiproj}
 \int_{U_v^{\an}} \varphi \, c_1(\overline{\mathcal{D}}_v)^{\wedge d} = \lim_{i \to \infty} \int_{X_{i,v}^{\an}} \varphi \, c_1(\overline{D}_{i,v})^{\wedge d}
 \end{equation}
 for every continuous function
 $\varphi \colon U_v^{\an} \rightarrow \bR$ with compact support. Here
 we view $\varphi$ as a function on $X_{i,v}^{\an}$ via the open
 immersion $U_v^{\an} \hookrightarrow X_{i,v}^{\an}$.

 For each $i$ we denote by
 $h_{\overline{\mathcal{D}}_{i}} \colon X_{i}(\overline{K})
 \rightarrow \bR$ the height function of
 $\overline{D}_{i} \in \widehat{\Div}(X_{i})_{\mathbb{R}}$.  The fact
 that $(\overline{\mathcal{D}}_i)_{i}$ is a Cauchy sequence readily
 implies that $\lim_{i \to \infty} h_{\overline{\mathcal{D}}_i}(x)$
 exists for every $x \in U(\overline{K})$. This limit does not depend
 on the choice of the sequence, and so we define the height function
 $ h_{\overline{\mathcal{D}}} \colon U(\overline{K}) \longrightarrow
 \bR$ by setting
\begin{displaymath}
  h_{\overline{\mathcal{D}}}(x) = \lim_{i\to
    \infty}h_{\overline{\mathcal{D}}_i}(x) \quad \text{ for every }
  x \in U(\overline{K}).  
\end{displaymath}

The arithmetic intersection product 
of integrable adelic $\mathbb{R}$-divisors is
the symmetric multilinear map from \cite[Theorem 3.37]{BurgosKramer}
\begin{displaymath}
  (\overline{\mathcal{D}}_{1}, \dots, \overline{\mathcal{D}}_{d+1}) \in (  \widehat{\Div}(\mathcal{U})_{\bR}^{\integ})^{d+1}  \longmapsto
(\overline{\mathcal{D}}_{1}\cdots \overline{\mathcal{D}}_{d+1})\in \mathbb{R}.
\end{displaymath}
For $j=1,\dots, d+1$ let $\overline{\mathcal{D}}_{j}$ be a nef adelic
$\mathbb{R}$-divisor on $\mathcal{U}$ and choose a Cauchy sequence
$(\overline{\mathcal{D}}_{j,i})_{i}$ of nef model
$\mathbb{R}$-divisors on $\mathcal{U}$ representing
$\overline{\mathcal{D}}_{j}$. Then
\begin{displaymath}
(\overline{\mathcal{D}}_1 \cdots \overline{\mathcal{D}}_{d+1}) = \lim_{i \to \infty} (\overline{D}_{1,i} \cdots  \overline{D}_{d+1,i}),
\end{displaymath}
where $\overline{D}_{i,j}$, $j=1,\dots, d+1$, are the associated
adelic $\mathbb{R}$-divisors as in \eqref{eq:4} and the arithmetic
intersection products in the right-hand side are computed in common
models.

For any dominant morphism
$\phi \colon \mathcal{U}' \rightarrow \mathcal{U}$ of normal
quasi-projective arithmetic varieties there is a pullback map
$\phi^* \colon \widehat{\Div}(\mathcal{U})_{\bR}^{\mathrm{adel}}
\rightarrow \widehat{\Div}(\mathcal{U}')_{\bR}^{\mathrm{adel}}$
\cite[Section 3.5]{BurgosKramer}. If $\phi$ is birational then
\begin{displaymath}
  (\overline{\mathcal{D}}_1 \cdots \overline{\mathcal{D}}_{d+1}) =(\phi^*\overline{\mathcal{D}}_1 \cdots \phi^*\overline{\mathcal{D}}_{d+1}).
\end{displaymath}

We denote by $[\infty] \in \widehat{\Div}(\mathcal{X})_{\bR} $ the
arithmetic divisor over the zero divisor on $\mathcal{X}$ with
$g_v = 1$ for every $v \in \mathfrak{M}_K^{\infty}$. We have
\begin{equation*}
 (\overline{\mathcal{D}}_1 \cdots \overline{\mathcal{D}}_{d}\cdot [\infty]) = (D_{1}  \cdots D_{d}).
\end{equation*}

 \subsection{Essential and absolute minima}
 \label{sec:essent-absol-minima}

\change{As in the previous section, let $\mathcal{X}$ be a normal projective arithmetic
 variety over $\mathcal{O}_{K}$ of dimension $d+1$ equipped with a
 strictly effective arithmetic divisor $\mathcal{B}$ and set
 $\mathcal{U}=\mathcal{X}\setminus \supp(\mathcal{B})$.  We also set
 $X$ for the generic fiber of $\mathcal{X}$, which is a normal
 projective variety over $K$ of dimension $d\ge 1$, and then
 $B=\mathcal{B}|_{X}$ and $U=X\setminus \supp(B)$.}

 Let
 $\overline{\mathcal{D}} \in
 \widehat{\Div}(\mathcal{U})_{\bR}^{\mathrm{adel}}$ with big
 $D \in \Div(U)^{\adel}_{\bR}$. The \emph{essential minimum} of
 $\overline{\mathcal{D}}$ is defined in the expected way as the
 quantity
\begin{displaymath}
  \mu^{\mathrm{ess}}(\overline{\mathcal{D}}) 
= \sup_{V \subsetneq U} \inf_{x \in (U\setminus V)(\overline{K})} h_{\overline{\mathcal{D}}}(x),
\end{displaymath}
the supremum being over all the proper closed subsets
$V \subsetneq U$.  \change{When $\overline{\mathcal{D}}$ is semipositive we define
  its \emph{absolute minimum} as the quantity
\begin{displaymath}
\mu^{\abs}(\overline{\mathcal{D}}) = \sup \{ \lambda \in \bR \ | \ \overline{\mathcal{D}} - \lambda [\infty] \text{ is nef}\}. 
\end{displaymath}
This definition is modeled on \eqref{eq:28} and will be suitable for our
purposes.}

We next establish some auxiliary results.  To this end, let
$(\overline{\mathcal{D}}_i)_{i}$ be a Cauchy sequence in
$\widehat{\Div}(\mathcal{U})_{\bR}^{\mathrm{mod}}$ representing
$\overline{\mathcal{D}}$, and for each $i$ let
$\overline{{D}}_i \in \widehat{\Div}(X_{i})_{\mathbb{R}}$ as in
\eqref{eq:4}.

\begin{lemma}\label{lemma:nefquasiproj}
  If $\overline{\mathcal{D}}$ is semipositive and
  $\mu^{\abs}(\overline{\mathcal{D}}) > - \infty$, then
  $\overline{\mathcal{D}} -
  \mu^{\abs}(\overline{\mathcal{D}})[\infty]$ is nef.
\end{lemma}

\begin{proof}
  Let $(\lambda_n)_{n}$ be a sequence of real numbers converging to
  $\mu^{\abs}(\overline{\mathcal{D}})$ from below. Then for every $n $
  there exists a Cauchy sequence $(\overline{\mathcal{D}}_{n,i})_i$ in
  $\widehat{\Div}(\mathcal{U})_{\bR}^{\mathrm{mod}}$ representing
  $\overline{\mathcal{D}} - \lambda_n[\infty]$ with
  $\overline{\mathcal{D}}_{n,i}$ nef for every
  $i$. Since $\overline{\mathcal{B}}$ is strictly effective we have that
  $\|[\infty]\|_{\overline{\mathcal{B}}}$ is a real number.
  Then for any $\varepsilon > 0$ and every $n$ and $i$ sufficiently
  large we have
\begin{displaymath}
  \|\overline{\mathcal{D}}_{n,i} - (\overline{\mathcal{D}} - \mu^{\abs}(\overline{\mathcal{D}})\, [\infty])\|_{\overline{\mathcal{B}}} \le \|\overline{\mathcal{D}}_{n,i} - (\overline{\mathcal{D}} - \lambda_n[\infty])\|_{\overline{\mathcal{B}}} + (\mu^{\abs}(\overline{\mathcal{D}}) - \lambda_n) \, \|[\infty]\|_{\overline{\mathcal{B}}} < \varepsilon.
\end{displaymath}
It follows that
$\overline{\mathcal{D}} - \mu^{\abs}(\overline{\mathcal{D}})[\infty]$
can be represented by a Cauchy sequence of nef model arithmetic
$\bR$-divisors, and so it is nef.
\end{proof}

\begin{lemma}\label{lemma:essminlimitquasiproj}
 We have
\begin{math}
\displaystyle{\mu^{\ess}(\overline{\mathcal{D}}) = \lim_{i \to \infty} \mu^{\ess}(\overline{{D}}_i)}.
\end{math}
\end{lemma}

\begin{proof}
  For any $k \in \bN$ and $t \in \bR$ we have
\begin{displaymath}
  \mu^{\ess}(k\,\overline{\mathcal{D}} + t \,[\infty]) = k\, \mu^{\ess}(\overline{\mathcal{D}}) + t \and  \mu^{\ess}(k \overline{D}_i + t\, [\infty]) = k\, \mu^{\ess}(\overline{D}_i) + t, \quad i\in \mathbb{N},
\end{displaymath}
and so we can replace without loss of generality
$\overline{\mathcal{D}}$ by $k\, \overline{\mathcal{D}} + t\,[\infty]$
and $\overline{\mathcal{D}}_{i}$ by
$k\, \overline{\mathcal{D}}_{i} + t\,[\infty]$.  Since $D$ is big,
taking $k$ and $t$ sufficiently large and applying Lemma \ref{lem:6}
\change{together with the continuity of the arithmetic volume} we can
assume that $\overline{\mathcal{B}} \le \overline{\mathcal{D}}_i $,
first for \change{$i=i_{0}$ with $i_{0}\in \mathbb{N}$ sufficiently
  large, and then for every $i\ge i_{0}$} using the assumption~\eqref{eq:76}.

Since $(\overline{\mathcal{D}}_i)_{i}$ is Cauchy, there exists a
sequence of positive real numbers~$(\varepsilon_i)_{i}$ converging to
zero such that
   \begin{equation}
     \label{eq:25}
     \overline{\mathcal{D}}_i - \varepsilon_i \overline{\mathcal{B}}
     \le \overline{\mathcal{D}}_j \le \overline{\mathcal{D}}_i +
     \varepsilon_i \overline{\mathcal{B}} \quad \text{ for every } \change{i_{0}} \le i\le j.        
   \end{equation}
   Since the support of
   $ \pm(D_j - D_i) +
   \varepsilon_i B$ does not intersect $U$,
   we get from \eqref{eq:25}
   \begin{equation}
     \label{eq:82}
  h_{\overline{\mathcal{D}}_i}(x) - \varepsilon_i h_{\overline{\mathcal{B}}}(x) \le h_{\overline{\mathcal{D}}_{j}}(x) \le h_{\overline{\mathcal{D}}_i}(x) + \varepsilon_i h_{\overline{\mathcal{B}}}(x) \quad \text{ for every }  x \in U(\overline{K}). 
\end{equation}
Since $\overline{\mathcal{B}} \le \overline{\mathcal{D}}_i$, there
exists a dense open subset $U_i \subset U$ such that
$h_{\overline{\mathcal{B}}}(x) \le h_{\overline{\mathcal{D}}_i}(x)$
for every $x \in U_i(\overline{K})$.  Using this and taking the limit
for $j\to \infty$ in \eqref{eq:82} we deduce
\begin{displaymath}
(1- \varepsilon_i) h_{\overline{\mathcal{D}}_i}(x)   \le h_{\overline{\mathcal{D}}}(x) \le (1+ \varepsilon_i)h_{\overline{\mathcal{D}}_i}(x)  \quad \text{ for every }  x \in U_{i}(\overline{K}). 
\end{displaymath}
Since $U_i$ is dense, this implies that
$(1- \varepsilon_i) \mu^{\ess}(\overline{{D}}_i) \le
\mu^{\ess}(\overline{\mathcal{D}}) \le (1+
\varepsilon_i)\mu^{\ess}(\overline{{D}}_i)$, and the lemma follows
by letting $i\to \infty$.
\end{proof}

\begin{lemma}\label{lemma:increaseessminquasiproj}
  Let
  $\overline{\mathcal{E}} \in \widehat{\Div}(\mathcal{U})_{\bR}^{\mathrm{adel}}$
  such that $E \in \Div(U)^{\adel}_{\bR}$ is big and
  $\overline{\mathcal{D}} - \overline{\mathcal{E}}$ is pseudo-effective. Then
  $\mu^{\ess}(\overline{\mathcal{D}}) \ge \mu^{\ess}(\overline{\mathcal{E}})$.
\end{lemma}

\begin{proof}
  Let $ (\overline{\mathcal{E}}_i)_i$ and
  $(\overline{\mathcal{M}}_i)_{i}$ be Cauchy sequences in
  $\widehat{\Div}(\mathcal{U})_{\bR}^{\mathrm{mod}}$ representing
  respectively $\overline{\mathcal{E}}$ and
  $\overline{\mathcal{D}}-\overline{\mathcal{E}}$, with
  $\overline{\mathcal{M}}_i$ pseudo-effective for every $i$. Then there is a
  sequence $(\varepsilon_i)_i$ of real numbers converging to zero such
  that
\begin{equation}\label{eq:pseffquasiproj}
  \overline{\mathcal{M}}_i \le \overline{\mathcal{D}}_i - \overline{\mathcal{E}}_i + \varepsilon_i \overline{\mathcal{B}}
  \quad \text{ for every } i.
\end{equation}
Then
$(\overline{\mathcal{D}}_i + \varepsilon_i \overline{\mathcal{B}})_i$
is a Cauchy sequence representing $\overline{\mathcal{D}}$ which
by~\eqref{eq:pseffquasiproj} satisfies that
$(\overline{\mathcal{D}}_i + \varepsilon_i \overline{\mathcal{B}})-
\overline{\mathcal{E}}_i$ is pseudo-effective for every $i$. By Lemmas
\ref{lemma:propertiesessmin}\eqref{item:essminincreases}
and~\ref{lemma:essminlimitquasiproj} we have
 \begin{displaymath}
   \mu^{\ess}(\overline{\mathcal{D}}) = \lim_{i \to \infty} \mu^{\ess}(\overline{D}_i + \varepsilon_i
\overline{B}) \ge \lim_{i \to \infty} \mu^{\ess}(\overline{{E}}_i)
   = \mu^{\ess}(\overline{\mathcal{E}}).
 \end{displaymath}
\end{proof}

The next result is the quasi-projective version of Corollary
\ref{coro:ineqSiu}, and will be the key ingredient in the proof of our
quasi-projective equidistribution theorem.

\begin{proposition}\label{prop:keyquasiproj} Let
  $\overline{\mathcal{P}}, \overline{\mathcal{E}} \in
  \widehat{\Div}(\mathcal{U})_{\bR}^{\mathrm{adel}}$ with
  $\overline{\mathcal{P}}$ nef and
  $P \in \Div(U)_{\mathbb{R}}^{\adel}$ big. Assume that there exists a
  nef
  $\overline{\mathcal{A}} \in
  \widehat{\Div}(\mathcal{U})_{\bR}^{\mathrm{adel}}$ such that
  $\overline{\mathcal{A}} \pm \overline{\mathcal{E}}$ are nef and
  $A \in \Div(U)_{\mathbb{R}}^{\adel}$ is big. There exists a constant
  $c_d$ depending only on $d$ such that
\begin{displaymath}
\mu^{\ess}(\overline{\mathcal{P}} + \lambda \overline{\mathcal{E}}) \ge \frac{(\overline{\mathcal{P}}^{d+1})}{(d+1)\vol(P+\lambda E)} +  \frac{(\overline{\mathcal{P}}^d \cdot \overline{\mathcal{E}})}{(P^d)}\, \lambda - c_d \, \frac{(\overline{\mathcal{P}}^d \cdot \overline{\mathcal{A}})}{(P^d)} \, \frac{\lambda^2}{r(P;A)}
\end{displaymath}
for every $0 \le \lambda < r(P;A)/2$. In particular, if
$E=0 \in \Div(U)^{\adel}_{\bR}$ then
\begin{displaymath}
\mu^{\ess}(\overline{\mathcal{P}} + \lambda \overline{\mathcal{E}}) \ge \frac{(\overline{\mathcal{P}}^{d+1})}{(d+1)(P^d)} +  \frac{(\overline{\mathcal{P}}^d \cdot \overline{\mathcal{E}})}{(P^d)} \, \lambda - c_d \, \frac{(\overline{\mathcal{P}}^d \cdot \overline{\mathcal{A }})}{(P^d)} \, \frac{\lambda^2}{r(P;A)}.
\end{displaymath}
\end{proposition}

\begin{proof} 
  Let $(\overline{\mathcal{P}}_i)_{i}, (\overline{\mathcal{M}}_i)_i$
  and $(\overline{\mathcal{N}}_i)_i$ be Cauchy sequences in
  $\widehat{\Div}(\mathcal{U})_{\bR}^{\mathrm{mod}}$ representing
  respectively
  $ \overline{\mathcal{P}}, \overline{\mathcal{A}}+
  \overline{\mathcal{E}}$ and
  $\overline{\mathcal{A}}-\overline{\mathcal{E}}$ and such that
  $\overline{\mathcal{P}}_i, \overline{\mathcal{M}}_i$ and
  $\overline{\mathcal{N}}_i$ are nef and defined on the same
  projective arithmetic variety $\mathcal{X}_i$ for every $i$. Set
  \begin{displaymath}
\overline{\mathcal{A}}_i =
\frac{\overline{\mathcal{M}}_i+\overline{\mathcal{N}}_i}{2} \and 
\overline{\mathcal{E}}_i = \frac{\overline{\mathcal{M}}_i-\overline{\mathcal{N}}_i}{2} .
  \end{displaymath}
  Then $(\overline{\mathcal{A}}_i)_i$ and
  $ (\overline{\mathcal{E}}_i)_i$ are Cauchy sequences representing
  $\overline{\mathcal{A}}$ and $\overline{\mathcal{E}}$, and we have
  that $\overline{A}_i$ and $\overline{A}_i \pm \overline{E}_i$ are
  nef for every $i$.

  Let $\lambda \in [0, r(P;A)/2)$. By Lemma \ref{lem:3} we have
  $\lambda \in [0, r(P_i;A_i)/2)$ for every sufficiently large $i$. In
  particular
\begin{displaymath}
P_i + \lambda E_i = P_i- \lambda A_i + \lambda(A_i+E_i)  
\end{displaymath}
is big because $ P_i- \lambda A_i$ is big and $A_i+E_i$ is nef, hence
pseudo-effective.  Then by Corollary~\ref{coro:ineqSiu} there exists a
constant $c_d$ depending only on $d$ such that
\begin{displaymath}
  \mu^{\ess}(\overline{P}_i + \lambda \overline{E}_i) \ge  \frac{(\overline{P}_i^{d+1})}{(d+1)\vol(P_i+\lambda E_i)}+  \frac{(\overline{P}_i^d \cdot \overline{E}_i)}{(P_i^d)}\, \lambda - c_d  \frac{(\overline{P}_i^d \cdot \overline{A}_i)}{(P_i^d)}  \frac{\lambda^2}{r(P_i;A_i)}.
\end{displaymath}
We conclude by letting $i\to\infty$.
\end{proof}

\subsection{Equidistribution on quasi-projective varieties} \label{sec:equid-quasi-proj}

\change{We keep the notation of the previous section. In particular, let
$\overline{\mathcal{D}} \in
\widehat{\Div}(\mathcal{U})_{\bR}^{\mathrm{adel}}$ with big
$D \in \Div(U)^{\adel}_{\bR}$.}

\begin{definition} \label{def:17} A \emph{semipositive approximation}
  of $\overline{\mathcal{D}}$ is a pair
  $(\phi, \overline{\mathcal{Q}})$ where
\begin{enumerate}[leftmargin=*]
\item $\phi\colon \mathcal{U}' \rightarrow \mathcal{U}$ is a birational morphism of normal quasi-projective arithmetic~varieties,
\item $\overline{\mathcal{Q}} $ is a semipositive adelic $\bR$-divisor
  on $\mathcal{U}'$ with  big $Q  \in \Div(U)^{\adel}_{\bR}$,
\item  $\phi^*\overline{\mathcal{D}}-\overline{\mathcal{Q}}$ is pseudo-effective.
\end{enumerate}
\end{definition}

A sequence $(x_{\ell})_{\ell}$ in $U(\overline{K})$ is called
\emph{generic} if for every closed subset $V \varsubsetneq U$ there
exists $\ell_0 \in \bN$ such that $x_{\ell} \notin V(\overline{K})$
for every $\ell\ge \ell_{0}$. A generic sequence $(x_{\ell})_{\ell}$
is called \emph{$\overline{\mathcal{D}}$-small} if
\begin{displaymath}
\lim_{\ell \to \infty} h_{\overline{\mathcal{D}}}(x_{\ell}) =
\mu^{\ess}(\overline{\mathcal{D}}).
\end{displaymath}

For every $x \in U(\overline{K})$ the Galois orbit
$O(x) \subset X(\overline{K})$ lies in $U(\overline{K})$. Thus for
every $v\in \mathfrak{M}_K$ the $v$-adic Galois orbit $O(x)_v$ lies in
$ U_v^{\an}$, and in particular $\delta_{O(x)_v}$ is a probability
measure on $U_v^{\an}$.

The following is the quasi-projective version of Theorem
\ref{thm:VariantSequencesEP}.

\begin{theorem}\label{thm:equiquasiproj}
  Assume that there exists a sequence
  $(\phi_n \colon \mathcal{U}_n \rightarrow \mathcal{U},
  \overline{\mathcal{Q}}_n)_{n}$ of semipositive approximations of
  $\overline{\mathcal{D}}$ such that
  \begin{equation}
    \label{eq:34}
    \lim_{n \rightarrow \infty} \frac{1}{r(Q_n;\phi_n^*D)} \Big(\mu^{\ess}(\overline{\mathcal{D}}) - \frac{(\overline{\mathcal{Q}}_n^{d+1})}{(d+1)\, (Q_n^d)}\Big) = 0, \ \sup_{n \in \bN} \frac{\mu^{\ess}(\overline{\mathcal{D}}) - \mu^{\abs}(\overline{\mathcal{Q}}_n)}{r(Q_n;\phi_n^*D)} < \infty.
\end{equation}
Let $v \in\mathfrak{M}_K$, and for each $n \ge 1$ let $\nu_{n,v}$ be
the pushforward to $U_v^{\an}$ of the normalized $v$-adic Monge-Ampère
measure $c_1(\overline{\mathcal{Q}}_{n,v})^{\wedge d}/(Q_n^d)$ on
$U_{n,v}^{\an}$. Then
\begin{enumerate}[leftmargin=*]
\item the sequence $(\nu_{n,v})_n$ converges weakly to a probability measure $\nu_{\overline{\mathcal{D}}, v}$ on $U_v^{\an}$,  
\item for every $\overline{\mathcal{D}}$-small generic sequence
  $(x_{\ell})_{\ell}$ in $U(\overline{K})$, the sequence of probability measures
  $(\delta_{O(x_{\ell})_{v}})_{\ell}$ on $U_v^{\an}$ converges weakly
  to~$\nu_{\overline{\mathcal{D}},v}$. 
\end{enumerate}
\end{theorem}

\begin{proof}
  Let $v \in \mathfrak{M}_K$ and
  ${\varphi \colon U_v^{\an} \rightarrow \bR}$ a continuous function
  with compact support, and let $(x_{\ell})_{\ell}$ be a
  $\overline{\mathcal{D}}$-small generic sequence in
  $U(\overline{K})$. 
We need to show that
\begin{equation}\label{eq:goalquasiproj}
\lim_{\ell \to \infty} \int_{U_v^{\an}} \varphi \, d\delta_{O(x_{\ell})_v} = \lim_{n\to \infty} \int_{U_{v}^{\an}} \varphi \, d\nu_{n,v}.
\end{equation}

Let $\varepsilon > 0$. Since $\varphi$ has compact support, we can
view it as an element of $C(X_v^{\an})$. By \cite[Proposition 2.11, 
Theorem 2.13 \change{and Lemma 2.16}]{GualdiMartinez}, after possibly extending the base field
$K$, there is a $\Gal(\overline{K}_v/K_v)$-invariant
$\varphi_{\varepsilon}\in C( X_v^{\an})$ such that
$|\varphi_{\varepsilon} - \varphi| < \varepsilon$ on $X_v^{\an}$ and
such that the adelic divisor
$\overline{E} \coloneqq \overline{0}^{\varphi_{\varepsilon}}$ as in
\eqref{eq:27} is DSP.  Then by Lemma~\ref{lem:1} there exists
$\overline{A}\in \widehat{\Div}(X)_{\mathbb{R}}$ such that both
$\overline{A}$ and $\overline{A} \pm \overline{E}$ are nef and
$A \in \Div(X)_{\bR}$ is big. Shrinking $\mathcal{U}$ if necessary, we
view these adelic $\mathbb{R}$-divisors on $X$ as elements of
$\widehat{\Div}(\mathcal{U})_{\bR}^{\mathrm{int}}$, in which case we
respectively denote them by $\overline{\mathcal{E}} $ and
$\overline{\mathcal{A}}$.

For each $n \in \bN$ set
$\widetilde{\mathcal{Q}}_n = \overline{\mathcal{Q}}_n -
\mu^{\abs}(\overline{\mathcal{Q}}_n)\, [\infty]$.  By Lemma
\ref{lemma:nefquasiproj} we have that $\widetilde{\mathcal{Q}}_n$ is nef.
Moreover, the second condition in \eqref{eq:34} implies
\begin{displaymath}
\kappa \coloneqq \sup_{n \in \bN} \frac{(\widetilde{\mathcal{Q}}_n^d \cdot \phi_n^* \overline{\mathcal{A}})}{(Q_n^d)} < \infty.
\end{displaymath}
We omit the proof, as it is identical to that for Lemma
\ref{lemma:sup}. In this respect, note that Zhang's inequality remains
valid in the quasi-projective setting: this is \cite[Theorem
5.3.3]{YuanZhang:quasiproj}, and alternatively it follows from
Proposition \ref{prop:keyquasiproj} applied with
$\lambda = 0$. 

Let $n \in \bN$ and $\lambda \in (0, r(Q_n; \phi_n^*A)/2)$. By
Proposition \ref{prop:keyquasiproj} applied with
$\overline{\mathcal{P}} = \widetilde{\mathcal{Q}}_n$, there exists a constant $c_d$
depending only on $d$ such that
\begin{displaymath}
\mu^{\ess}(\widetilde{\mathcal{Q}}_n + \lambda \, \phi_n^*\overline{\mathcal{E}}) \ge \frac{(\widetilde{\mathcal{Q}}_n^{d+1})}{(d+1)\, (Q_n^d)}+  \frac{(\widetilde{\mathcal{Q}}_n^d \cdot \phi_n^*\overline{\mathcal{E}})}{(Q_n^d)} \, \lambda - c_d\, \kappa \, \frac{\lambda^2}{r(Q_n; \phi_n^*A)}.
\end{displaymath}
Since  $(\widetilde{\mathcal{Q}}_n^{d+1}) = (\overline{\mathcal{Q}}_n^{d+1}) - (d+1)\,  \mu^{\abs}(\overline{\mathcal{Q}}_n) \, (Q_n^d)$ and 
\begin{math}
  \mu^{\ess}(\widetilde{\mathcal{Q}}_n + \lambda \, \phi_n^*\overline{\mathcal{E}}) = \mu^{\ess}(\overline{\mathcal{Q}}_n + \lambda \, \phi_n^*\overline{\mathcal{E}}) - \mu^{\abs}(\overline{\mathcal{Q}}_n)
\end{math}
we obtain
\begin{displaymath}
  \mu^{\ess}(\overline{\mathcal{Q}}_n + \lambda  \, \phi_n^*\overline{\mathcal{E}}) \ge \frac{(\overline{\mathcal{Q}}_n^{d+1})}{(d+1) \, (Q_n^d)}
  +  \frac{(\widetilde{\mathcal{Q}}_n^d \cdot \phi_n^*\overline{\mathcal{E}})}{(Q_n^d)} \lambda - c_d\,\kappa \,  \frac{\lambda^2}{r(Q_n; \phi_n^*A)}.
\end{displaymath}
On the other hand we have  
\begin{displaymath}
  \mu^{\ess}(\overline{\mathcal{D}}) + \lambda \liminf_{\ell \to \infty} h_{\overline{\mathcal{E}}}(x_\ell) = \liminf_{\ell \to \infty} h_{\overline{\mathcal{D}} + \lambda \overline{\mathcal{E}}}(x_{\ell}) \ge \mu^{\ess}(\overline{\mathcal{D}} + \lambda \overline{\mathcal{E}}) \ge \mu^{\ess}(\overline{\mathcal{Q}}_n
  + \lambda \,  \phi_n^*\overline{\mathcal{E}}),
\end{displaymath}
where the last inequality is given by  Lemma \ref{lemma:increaseessminquasiproj}. Therefore 
\begin{equation*}
  \liminf_{\ell \to \infty} h_{\overline{\mathcal{E}}}(x_\ell) \ge  \frac{(\widetilde{\mathcal{Q}}_n^d
    \cdot \phi_n^*\overline{\mathcal{E}})}{(Q_n^d)} -  \Big(\mu^{\ess}(\overline{\mathcal{D}})- \frac{(\overline{\mathcal{Q}}_n^{d+1})}
  {(d+1)(Q_n^d)} \Big) \, \frac{1}{\lambda}  -  \frac{c_d\, \kappa }{r(Q_n; \phi_n^*A)} \, \lambda.
\end{equation*}
\change{Letting 
\begin{displaymath}
\gamma_n = \max \Big(\frac{1}{r(Q_n;\phi_n^*A)}\Big(\mu^{\ess}(\overline{\mathcal{D}})- \frac{(\overline{\mathcal{Q}}_n^{d+1})}
  {(d+1)(Q_n^d)}\Big) ,  \frac{1}{n} \Big),
\end{displaymath}
we then have
\begin{equation}
  \label{eq:75}
  \liminf_{\ell \to \infty} h_{\overline{\mathcal{E}}}(x_\ell) \ge  \frac{(\widetilde{\mathcal{Q}}_n^d
    \cdot \phi_n^*\overline{\mathcal{E}})}{(Q_n^d)} -   \frac{r(Q_n; \phi_n^*A)\, \gamma_n}{\lambda}  -  \frac{c_d\, \kappa }{r(Q_n; \phi_n^*A)} \, \lambda
\end{equation}
for every $n \in \bN$ and $\lambda \in (0, r(Q_n; \phi_n^*A)/2)$. 
 It follows from the first condition in~\eqref{eq:34} and Lemma
\ref{lemma:compareinradii} that $\lim_n \gamma_n = 0$. 
 Applying inequality \eqref{eq:75} with $\lambda = r(Q_n;\phi_n^*A)\,\gamma_n^{1/2}$ for $n$ large enough gives 
\begin{displaymath}
 \liminf_{\ell \to \infty} h_{\overline{\mathcal{E}}}(x_\ell) \ge  \frac{(\widetilde{\mathcal{Q}}_n^d
    \cdot \phi_n^*\overline{\mathcal{E}})}{(Q_n^d)} - \gamma_n^{1/2} - c_d\, \kappa\, \gamma_n^{1/2},
\end{displaymath}
and so }
\begin{equation}\label{eq:ineqproofquasiproj}
\liminf_{\ell \to \infty} h_{\overline{\mathcal{E}}}(x_\ell) \ge \limsup_{n\to \infty}  \frac{(\widetilde{\mathcal{Q}}_n^d \cdot \phi_n^*\overline{\mathcal{E}})}{(Q_n^d)}.
\end{equation}

To conclude we adapt the arguments in the proof of \cite[Theorem
5.4.3]{YuanZhang:quasiproj}. Let $n \in
\bN$ and choose a Cauchy sequence
$(\overline{\mathcal{Q}}_{n,i})_i$ in
$\widehat{\Div}(\mathcal{U}_n)^{\model}$ representing
$\widetilde{\mathcal{Q}}_n$ and such that
$\overline{\mathcal{Q}}_{n,i}$ is nef for every
$i$. Let
$\mathcal{X}_{n,i}$ be a projective arithmetic variety on which
$\overline{\mathcal{Q}}_{n,i}$ is defined, denote by
$X_{n,i}$ its generic fiber, and set $\overline{Q}_{n,i} \in
\widehat{\Div}(X_{n,i})_{\mathbb{R}}$ as in \eqref{eq:4}.

Let $v\in \mathfrak{M}_{K}$.  By \eqref{eq:defMongeAmperequasiproj},
the $v$-adic Monge-Ampère measures $c_1(\overline{Q}_{n,i,v})^{\wedge d}$
converge to that of $\widetilde{\mathcal{Q}}_n$, which coincides with
that of $\overline{\mathcal{Q}}_{n}$.  Then
\begin{align*}
    \frac{(\widetilde{\mathcal{Q}}_n^d \cdot \phi_n^*\overline{\mathcal{E}})}{(Q_n^d)}
  & = \frac{n_v}{(Q_n^d)} \lim_{i \to \infty} \int_{X_{n,i,v}^{\an}} \varphi_{\varepsilon}\, c_1(\overline{Q}_{n,i,v})^{\wedge d} \\
  & \ge \frac{n_v}{(Q_n^d)} \lim_{i \to \infty} \int_{X_{n,i,v}^{\an}} \varphi \, c_1(\overline{Q}_{n,i,v})^{\wedge d} - n_v \varepsilon\\ 
 & = \frac{n_v}{(Q_n^d)} \int_{U_{n,v}^{\an}} \varphi \, c_1(\overline{\mathcal{Q}}_{n,v})^{\wedge d} - n_v \varepsilon = n_v \int_{U_{v}^{\an}} \varphi \, d\nu_{n,v} - n_v\varepsilon,
\end{align*}
where in these 
integrals  we write $\varphi_{\varepsilon}$ and $\varphi$ for  their pullbacks to $X_{n,i,v}^{\an}$ and $U_{n,v}^{\an}$.
Moreover
 \begin{displaymath}
   n_v\varepsilon + n_v \int_{U_v^{\an}} \varphi \, d\delta_{O(x_{\ell})_v} \ge n_v \int_{X_v^{\an}} \varphi_{\varepsilon} \, d\delta_{O(x_{\ell})_v} = h_{\overline{\mathcal{E}}}(x_\ell)  \quad \text{ for every }  \ell \in \bN. 
 \end{displaymath}
 Combining this with \eqref{eq:ineqproofquasiproj} and letting $\varepsilon\to 0$ we obtain
 \begin{displaymath}
 \liminf_{\ell \to \infty} \int_{U_v^{\an}} \varphi \, d\delta_{O(x_{\ell})_v} \ge  \limsup_{n\to \infty} \int_{U_{n,v}^{\an}} \varphi \, d\nu_{n,v},
 \end{displaymath}
 and we deduce \eqref{eq:goalquasiproj} by applying this to
 $-\varphi$.
\end{proof}

The next consequence is the number field case of the Yuan and Zhang's
equidistribution theorem in the quasi-projective setting~\cite[Theorem
5.4.3]{YuanZhang:quasiproj}.

\begin{corollary}
\label{coro:equiYZ}
Assume that $\overline{\mathcal{D}}$ is nef and that
\begin{displaymath}
\mu^{\ess}(\overline{\mathcal{D}}) = \frac{(\overline{\mathcal{D}}^{d+1})}{(d+1)(D^d)}.
\end{displaymath}
Then for every $v \in \mathfrak{M}_K$ and every
$\overline{\mathcal{D}}$-small generic sequence $(x_{\ell})_{\ell}$ in
$U(\overline{K})$ the sequence of probability measures
$(\delta_{O(x_{\ell})_{v}})_{\ell}$ on $U_v^{\an}$ converges weakly to
$c_1(\overline{\mathcal{D}}_{v})^{\wedge d}/(D^d)$.
\end{corollary}

\begin{proof}
  Apply Theorem \ref{thm:equiquasiproj} to the constant sequence
  $(\phi_n, \overline{\mathcal{Q}}_n) =
  (\Id_\mathcal{U}, \overline{\mathcal{D}})$, $n\in \mathbb{N}$.
\end{proof}

\begin{remark}
  \label{rem:19}
  We assume  throughout that $\mathcal{U}$ is normal to be able to work
  with $\bR$-divisors.  Nevertheless, both Theorem
  \ref{thm:equiquasiproj} and Corollary \ref{coro:equiYZ} can be
  applied to an adelic divisor $\overline{\mathcal{D}}$ on an arbitrary
  quasi-projective arithmetic variety $\mathcal{U}$ over
  $\Spec(\mathcal{O}_{K})$ just shrinking to a normal open subset.
\end{remark}

Recently, Biswas proved the differentiability of the arithmetic volume
function and deduced a quasi-projective version of Chen's
equidistribution theorem~\cite{Biswas}. 
It would be interesting to check
if this result also follows from Theorem~\ref{thm:equiquasiproj}
by adapting the arguments we used in the proof of Corollary
\ref{coro:Chenequi}.

\appendix
\section*{Appendix: Auxiliary results on convex analysis}\label{sec:prel-conv-analys}

\renewcommand{\thesection}{A}
\renewcommand{\theequation}{A.\arabic{equation}}

\setcounter{theorem}{0}
\setcounter{equation}{0}

Here we recall the constructions and properties from convex analysis
that are used in our study of toric varieties in Section
\ref{sec:toric-varieties}.  We also establish some auxiliary results,
most notably Proposition~\ref{prop:concavefunctions} concerning the
rate of the decay of the sup-level sets of a concave function as the
level approaches the maximum value.

Fix an integer $d\ge 1$ and let $C\subset \mathbb{R}^{d}$ be a
\emph{convex body}, that is a compact convex subset with nonempty
interior.

\begin{definition}
  \label{def:1} 
  For a linear functional $u\in (\mathbb{R}^{d})^{\vee}$ we denote by
  $\rmw(C,u) $ the length of the interval ${u}(C) \subset
  \mathbb{R}$. The \emph{width} of $C$ is defined as
\begin{displaymath}
  \rmw(C)=\inf_{u\in S^{d-1}} \rmw(C,u),
\end{displaymath}
where $S^{d-1}$ denotes the unit sphere of $(\mathbb{R}^{d})^{\vee}\simeq \mathbb{R}^{d}$.
  
For another convex body   $B \subset \mathbb{R}^{d}$, the \emph{inradius} of
$C$ with respect to $B$ is defined~as 
\begin{equation*}
\rmr (C;B) =\sup\{ \lambda\in \mathbb{R}_{>0} \mid \exists \, x\in \mathbb{R}^{d} \text{ with } \lambda B +x \subset C\}.
\end{equation*}  
When $B$ is the unit ball of $\mathbb{R}^{d}$, it is the classical
inradius from Euclidean geometry.
\end{definition}

\change{Since $C$ is compact both the width and the inradius are attained:}
\begin{displaymath}
  \change{\rmw(C)=\min_{u\in S^{d-1}} \rmw(C,u) \and \rmr (C;B) =\max\{ \lambda\in \mathbb{R}_{>0}
  \mid \exists \, x\in \mathbb{R}^{d} \text{ with } \lambda B +x \subset C\}.}
\end{displaymath}
\change{These quantities} can be compared up to scalar factors: there
are constants $c_{1},c_{2} >0$ depending only on $d$ and $B$ such that
\begin{equation}  \label{eq:comparewidthinradius}
  c_{1}\, \rmw(C) \le \rmr(C;B) \le c_{2}\, \rmw(C).
\end{equation}
The first inequality comes from \cite[page 86, inequality
(9)]{BonnesenFenchel:tcb} whereas the second is clear from the
definitions.

Let $f \colon C \to \mathbb{R}$ be a concave function and set
$\mu=\sup_{x\in C}f(x)$.  For each $t \le \mu$ we denote by 
\begin{equation*}
  S_{t}(f)=\{x\in C \mid f(x)\ge t\}
\end{equation*}
the corresponding sup-level set.  It is a nonempty compact convex
subset of $C$ that is a convex body whenever $t<\mu$.  We also set
$C_{\max}=S_{\mu}(f)$.

We next introduce the basic objects of the differential analysis of
concave functions.

\begin{definition}
  \label{def:4}
For $x_{0}\in C$, the \emph{sup-differential} of $f$ at
$x_{0}$ is the closed convex subset of $(\mathbb{R}^{d})^{\vee}$
defined as
\begin{equation*}
  \partial f (x_{0}) = \{u \in (\bR^d)^{\vee} \mid  \langle u, x-x_{0} \rangle \ge f (x) - f (x_{0})  \text{ for every } x\in C\}.
\end{equation*}  
Its elements are called the
\emph{sup-gradients} of $f$ at $x_{0}$.
\end{definition}

A point $x_{0} \in C$ lies in $ C_{\max}$ if and only if
$0\in \partial f(x_{0})$. When this is the case, the point $0$ is
\emph{not} a vertex of $ \partial f(x_{0})$ if and only if there
exists $u\in (\mathbb{R}^{d})^{\vee}\setminus \{0\}$ such that both
$u$ and $-u$ belong to this sup-differential or equivalently, if and
only if
\begin{equation} \label{eq:vertexsupdiff}
f(x) \le \mu-|\langle u, x-x_{0} \rangle|  \quad \text{ for every } x\in C.
\end{equation}
This condition does not depend on the choice of 
$x_{0} \in C_{\max}$: if \eqref{eq:vertexsupdiff} holds then
$\langle u, x_{1}-x_{0}\rangle =0$ for every $x_{1}\in C_{\max}$, and so
this inequality also holds with $x_{0}$ replaced by $x_{1}$.

The next proposition is a rigidity result that allows to determine
when a concave function admits a ``Canadian tent'' upper bound
like \eqref{eq:vertexsupdiff} in terms of the rate of decay of the
inradius or the width of its sup-level sets as the level approaches
its~maximum.

\begin{proposition}
  \label{prop:concavefunctions}
  The following conditions are equivalent:
  \begin{enumerate}[leftmargin=*]
  \item \label{item:inradius} for any convex body
    $B \subset \mathbb{R}^{d}$ we have
    $\displaystyle{\lim_{t\to\mu^{-}} \frac{\mu-t}{\rmr(S_{t}(f);B)}=0}$,
    \item \label{item:width} $\displaystyle{\lim_{t\to\mu^{-}} \frac{\mu-t}{\rmw(S_{t}(f))}=0}$, 
    \item \label{item:widthu} for every $u\in (\mathbb{R}^{d})^{\vee} \setminus \{0\}$ we have $\displaystyle{\lim_{t\to\mu^{-}} \frac{\mu-t}{\rmw(S_{t}(f),u)}=0}$, 
    \item \label{item:vertex} for any $x_{0}\in C_{\max}$ we have
      that $0 \in (\mathbb{R}^{d})^{\vee}$ is a vertex of $\partial f(x_{0})$.
  \end{enumerate}
\end{proposition}

Its proof relies on the next two lemmas.

\begin{lemma}
  \label{lem:decreasingeta}
  With the previous notation, we have
  \begin{enumerate}[leftmargin=*]
  \item \label{item:decreaseetawidthu} for every $u\in (\mathbb{R}^{d})^{\vee} \setminus \{0\}$ the function
  \begin{math}
\displaystyle{    t\in (-\infty,\mu) \mapsto \frac{\mu-t}{\rmw(S_{t}(f),u)}}
  \end{math}
  is non-increasing,
  \item \label{item:decreaseetawidth} the function
  \begin{math}
\displaystyle{          t\in (-\infty,\mu) \mapsto \frac{\mu-t}{\rmw(S_{t}(f))}}
  \end{math}
  is non-increasing.
  \end{enumerate}
\end{lemma}

\begin{proof}
  First suppose that $d=1$. Then choose $x_{0}\in C_{\max}$ and for
  each pair $t,t' \in \mathbb{R}$ with $t'<t<\mu$ consider the affine
  map $ \iota \colon \mathbb{R} \rightarrow \mathbb{R}$ defined as
  \begin{displaymath}
\iota( x)= \frac{t-t'}{\mu-t'} \, x_{0}+\frac{\mu-t}{\mu-t'}\, x.
  \end{displaymath}
  It follows from the concavity of $f$ that
  $\iota(S_{t'}(f))\subset S_{t}(f)$.  Denoting by $\ell$ the Lebesgue
  measure on $\mathbb{R}$, this gives
  \begin{equation}
    \label{eq:d=1}
  \frac{\mu-t}{\mu-t'}  \, \ell(S_{t'}(f)) \le  \ell(S_{t}(f)) .
  \end{equation}

  Now let $d$ be any positive integer. Take
  $u\in (\mathbb{R}^{d})^{\vee} \setminus \{0\}$ and consider the
  direct image of $f$ with respect to $u$, which is the concave
  function $u_{*}f \colon u(C) \to \mathbb{R}$ defined~as
  \begin{displaymath}
    u_{*}f(y) = \sup\{f(x) \mid x\in C \text{ such that } \langle u,x\rangle =y\} \quad \text{ for every } y\in  u(C).
  \end{displaymath}
  Clearly $ \sup_{y\in u(C)} u_{*} f (y)= \mu$ and
  $ S_{t}(u_{*}f)=u(S_{t}(f))$ for every $ t \le \mu$.
Then for any  $t'<t<\mu$ inequality \eqref{eq:d=1} gives 
\begin{equation}
  \label{eq:decreaseeta}
  \frac{\mu-t'}{\rmw(S_{t'}(f),u)}=
  \frac{\mu-t'}{\ell(S_{t'}(u_{*}f))}
  \ge
    \frac{\mu-t}{\ell(S_{t}(u_{*}f))} =
    \frac{\mu-t}{\rmw(S_{t}(f),u)}, 
  \end{equation}
  proving \eqref{item:decreaseetawidthu}.  The statement
  \eqref{item:decreaseetawidth} follows by choosing $u\in S^{d-1}$
  such that $\rmw(S_{t}(f),u)=\rmw(S_{t}(f))$ and
  applying~\eqref{eq:decreaseeta} to show 
  \begin{displaymath}
          \frac{\mu-t'}{\rmw(S_{t'}(f))} \ge
      \frac{\mu-t'}{\rmw(S_{t'}(f) ,u)} \ge
    \frac{\mu-t}{\rmw(S_{t}(f),u)} =      \frac{\mu-t}{\rmw(S_{t}(f))},
  \end{displaymath}
  as stated.
  \end{proof}

  \begin{lemma}
    \label{lem:canadiantent}
    Let $x_{0}\in C_{\max}$ and $t<\mu$. Then for every
    $u\in (\mathbb{R}^{d})^{\vee}\setminus \{0\}$ we have
    \begin{displaymath}
 f(x)  \le       \mu-\frac{\mu-t}{\rmw(S_{t}(f),u)} \, |\langle u,x-x_{0}\rangle | \quad \text{ for every } x\in C\setminus S_{t}(f).
    \end{displaymath}
  \end{lemma}

  \begin{proof}
    Let $x\in C\setminus S_{t}(f)$ and set
    $t'=f(x)<t$. Since both $x$ and $ x_{0}$ lie in
    $ S_{t'}(f)$ we have 
    $|\langle u,x-x_{0}\rangle | \le \rmw(S_{t'}(f),u)$. Combining this 
    with Lemma \ref{lem:decreasingeta}\eqref{item:decreaseetawidthu} we get
    \begin{displaymath}
      \mu-f(x) \ge \frac{\mu-t'}{\rmw(S_{t'}(f),u)} \, |\langle u,x-x_{0}\rangle |  \ge  \frac{\mu-t}{\rmw(S_{t}(f),u)} \, |\langle u,x-x_{0}\rangle | ,
        \end{displaymath}
        which gives the statement.
  \end{proof}

  \begin{proof}[Proof of Proposition \ref{prop:concavefunctions}]
    The equivalence between \eqref{item:inradius} and
    \eqref{item:width} follows from the inequalities in
    \eqref{eq:comparewidthinradius}, and clearly \eqref{item:width}
    implies \eqref{item:widthu}.

    Now assume \eqref{item:widthu}. If \eqref{item:vertex} does not
    hold, then there exists  $x_{0}\in C_{\max}$ such that
    $0\in \partial f(x_{0})$ is not a vertex of this convex
    subset, and so we can take  $u\in (\mathbb{R}^{d})^{\vee}\setminus \{0\}$ such that 
    \begin{math}
      \mu-|\langle u, x-x_{0} \rangle| \ge  f (x)
    \end{math}
    for every $ x\in C$ as in \eqref{eq:vertexsupdiff}.
    
    \change{For each $t<\mu$ we have that
    $u(x_{0}) \in S_{t}(u_{*}f) = u(S_t(f)) $ and so we can choose $y$
    in this interval} such that
\begin{displaymath}
   |y- \langle u, x_{0} \rangle| \ge \frac{1}{2}\, \rmw(S_{t}(f),u).
\end{displaymath}
Taking $x\in S_{t}(f)$ with $\langle u,x\rangle=y$ we have
$ |\langle u, x-x_{0} \rangle| \ge \rmw(S_{t}(f),u)/2$ and so
    \begin{displaymath}
\frac{\mu-t}{ \rmw(S_{t}(f),u)} \ge \frac{\mu-f(x)}{ \rmw(S_{t}(f),u)}\ge  \frac{|\langle u, x-x_{0} \rangle|}{ \rmw(S_{t}(f),u)} \ge \frac{1}{2},
    \end{displaymath}
which contradicts  \eqref{item:widthu} and thus implies \eqref{item:vertex}.

To close the loop, we show that \eqref{item:vertex} implies
\eqref{item:width}. For this suppose that \eqref{item:width} does not
hold, which by Lemma \ref{lem:decreasingeta} implies that there exists
$c>0$ such that
\begin{displaymath}
\frac{\mu-t}{\rmw(S_{t}(f))}\ge c \quad \text{ for every } t<\mu.
    \end{displaymath}
    In particular $\dim(C_{\max})<d$ since otherwise
    $\rmw(S_{t}(f))\ge \rmw(C_{\max})>0$ for every $t<\mu$.

    Take sequences $(t_{k})_{k}$ in $(-\infty,\mu)$ and $(u_{k})_{k}$
    in $S^{d-1}$ with $\lim_{k\to \infty} t_{k}= \mu$ such that
    $\rmw(S_{t_{k}}(f))= \rmw(S_{t_{k}}(f),u_{k})$ for every $k$.  By
    the \change{compactness} of $S^{d-1}$ we can assume that
    $\lim_{k\to \infty}u_{k}=u$ for a point $u\in S^{d-1}$. Take also
    $x_{0}\in C_{\max}$.  By Lemma \ref{lem:canadiantent} we have
    \begin{displaymath}
      \mu-c\,  |\langle u_{k}, x-x_{0}\rangle | \ge f(x) \quad \text{ for every } x\in C\setminus S_{t_{k}}(f).
    \end{displaymath}

    Now let $x\in C \setminus C_{\max}$. Then
    $x \notin S_{t_{k}}(f)$ for $k\gg 1$ and so
    \begin{displaymath}
 \mu-c\,  |\langle u, x-x_{0}\rangle | =           \lim_{k\to \infty} \mu-c\,  |\langle u_{k}, x-x_{0}\rangle | \ge f(x).
    \end{displaymath}
    Since $\dim(C_{\max})<d$, this inequality extends to
    $x\in C_{\max}$ by continuity.  Therefore $0$ is not a vertex of
    $\partial f(x_{0})$ \change{as we have seen in \eqref{eq:vertexsupdiff},
    and so condition} \eqref{item:vertex} does not hold.
  \end{proof}

  \begin{definition}
    \label{def:2}
    The concave function $f$ is said to be \emph{wide (at its
      maximum)} if it verifies any of the equivalent conditions in
    Proposition \ref{prop:concavefunctions}.
  \end{definition}

  Now assume that the considered concave function decomposes as a
  finite sum
  \begin{equation}
    \label{eq:50}
    f=\sum_{i\in I} n_{i} f_{i}
  \end{equation}
  where each $n_{i}$ is a positive real number and
  $f_{i}\colon C\to \mathbb{R}$ a concave function.
  
  \begin{definition}
    \label{def:6}
    A \emph{balanced family of sup-gradients} for the decomposition
    \eqref{eq:50} is a family of vectors
    \begin{displaymath}
      u_{i}\in (\mathbb{R}^{n})^{\vee},  \quad i\in I,       
    \end{displaymath}
    such that there exists $x_{0}\in C_{\max}$ with
    $u_{i}\in \partial f_{i}(x_{0})$ for every $i$ and
    $\sum_{i\in I}n_{i}u_{i}=0$.
  \end{definition}

  \begin{proposition}
    \label{prop:7}
    The decomposition $f=\sum_{i\in I}n_{i} f_{i}$ admits a balanced
    family of sup-gradients. If $f$ is wide then this family is
    unique.
  \end{proposition}

  \begin{proof}
    For the first statement, the decomposition of $f$ implies the decomposition of 
    its sup-differential at a point $x_{0}\in C$ as the Minkowski
    sum
  \begin{equation}
    \label{eq:16}
  \partial f(x_{0})= \sum_{i\in I} n_{i} \, \partial f_{i}(x_{0}),
\end{equation}
see for instance \cite[Proposition 2.3.9]{BPS:asterisque}. If
$x_{0}\in C_{\max}$ then $0\in \partial f(x_{0})$, and so we obtain a
balanced family of sup-gradients by considering any decomposition of
this vector according to \eqref{eq:16}. This proves the first statement. 

If $f$ is wide then $0\in \partial f(x_{0})$ is a vertex, and so the
second statement is given by \cite[Proposition 3.15]{BPRS:dgopshtv}.
  \end{proof}

  We denote by $\MV(C_{1},\dots, C_{d})$ the mixed volume of a family
  of $d$ convex bodies of $\mathbb{R}^{d}$, and by
  $\MI(f_{0},\dots, f_{d})$ the mixed integral of a family of concave
  functions on convex bodies of $\mathbb{R}^{d}$, both with respect to
  the Lebesgue measure of $\mathbb{R}^{d}$. They are respectively
  defined as alternating sums of Minkowski sums of convex bodies and
  sup-convolutions of concave functions, see  \cite[Definitions~2.7.14 and
  2.7.16]{BPS:asterisque} for precisions. 

  The next lemma gives the continuity of the mixed integral with
  respect to the approximation of the domains of the involved concave
  functions. The proof is straightforward from the behavior of
  sup-convolutions with respect to restrictions of domains and the
  continuity of the integral of a  concave function on a
  convex body with respect to the approximation of its domain.
  
    \begin{lemma}
\label{lem:9}
For $i=0, \dots, d$ let $f_{i}\colon C_{i} \to \mathbb{R}$ be a
concave function on a convex body, and $(C_{i,n})_{n}$ a sequence of
convex bodies approaching $C_{i}$ uniformly from inside. Then
    \begin{displaymath}
\lim_{n\to \infty}      \MI (f_{0}|_{C_{0,n}}, \dots, f_{d}|_{C_{d,n}}) = \MI(f_{0},\dots, f_{d}).
    \end{displaymath}
  \end{lemma}

  Finally, the next result allows to compute the mixed integral when
  all but one of the involved concave functions are equal and affine.
  Recall that the \emph{Legendre-Fenchel dual} of a concave function
  on a convex body $f \colon C\to \mathbb{R}$ is the concave function
  $f^{\vee}\colon (\mathbb{R}^{d})^{\vee}\to \mathbb{R}$ defined as
  \begin{equation}
    \label{eq:59}
    f^{\vee}(u)=\inf_{x\in C} \langle u,x\rangle -f(x).
  \end{equation}
  
\begin{lemma}
  \label{lem:12}  
  Let $f \colon C\to \mathbb{R}$ \change{be} an affine function on a
  convex body with linear part $u\in (\mathbb{R}^{d})^{\vee}$ and
  constant $c\in \mathbb{R}$. Then for any concave function on a
  convex body $g \colon B\to \mathbb{R}$ we have
  \begin{displaymath}
    \MI(f, \dots, f, g)=    \MI(u|_{C},\dots, u|_{C}, u|_{B})+c \, d \MV (C,\dots, C, B)
    -d! \vol(C) \, g^{\vee}(u).
  \end{displaymath}
\end{lemma}

\begin{proof}
  The proof is based on \cite[Section 1.3]{Gualdi:hhtv}
  and we will freely use the notation therein. This requires that both
  $C$ and $B$ are lattice polytopes, which we now suppose.

  By Corollary 1.10 in \textit{loc. cit.} we have
  \begin{equation}
    \label{eq:65}
  \MI(f,\dots, f, g)=
  \MI(u|_{C},\dots, u|_{C}, g)+ c\, d \MV (C,\dots, C, B),
\end{equation}
and by Proposition 1.5 in \textit{loc. cit.}, the mixed real
Monge-Ampère measure of $u|_{C}$ is
\begin{displaymath}
  \textrm{M}\mathcal{M}(\change{(u|_{C})^{\vee},\dots, (u|_{C})^{\vee}}) = d! \vol(C) \, \delta_{u}
\end{displaymath}
with $\delta_{u}$ the Dirac measure at the point
$u \in (\mathbb{R}^{d})^{\vee}$.  Hence applying the recursive formula
of Theorem~1.6 in \textit{loc. cit.}  to $g$ and to $u|_{B}$ we get
\begin{equation}
  \label{eq:66}
    \MI(u|_{C},\dots, u|_{C}, g)-  \MI(u|_{C},\dots, u|_{C}, u|_{B}) = d! \vol(C) \, ((u|_{B})^{\vee}(u)-g^{\vee}(u)) .
\end{equation}
The statement  follows in this case from \eqref{eq:65} and \eqref{eq:66} together
with the fact that
\begin{math}
(u|_{B})^{\vee}(u) = \inf_{x\in B } (\langle u,x\rangle - \langle u,x\rangle)=0.
\end{math}

The case when $C$ and $B$ are arbitrary convex bodies is deduced from
the previous using  the invariance of the formula with respect to
homothecies and its continuity with respect to uniform
approximations.
\end{proof}

\begin{remark}
  \label{rem:12}
  For any $u\in (\mathbb{R}^{d})^{\vee}$ and convex bodies
  $C_{1},C_{2}\subset \mathbb{R}^{d}$, the sup-convolution of the
  restrictions $u|_{C_{1}}$ and $u|_{C_{2}}$ coincides with the
  restriction $u|_{C_{1}+C_{2}}$, that is
  \begin{displaymath}
u|_{C_{1}} \boxplus u|_{C_{2}}= u|_{C_{1}+C_{2}}.
  \end{displaymath}
  Hence the mixed integral $\MI(u|_{C},\dots, u|_{C}, u|_{B})$ can be
  written as an alternating sum of integrals of the linear function
  $u$ on the Minkowski sum of several copies $C$ and $B$. In
  particular, the map $u\mapsto \MI(u|_{C},\dots, u|_{C}, u|_{B})$ is
  linear.
\end{remark}

\providecommand{\bysame}{\leavevmode\hbox to3em{\hrulefill}\thinspace}
\providecommand{\MR}{\relax\ifhmode\unskip\space\fi MR }
\providecommand{\MRhref}[2]{%
  \href{http://www.ams.org/mathscinet-getitem?mr=#1}{#2}
}
\providecommand{\href}[2]{#2}

\end{document}